\documentclass[11pt,a4paper]{article}
\usepackage{ifpdf}
\usepackage[utf8]{inputenc}
\usepackage[T1]{fontenc}
\usepackage{graphicx}
\usepackage{color}
\usepackage{textcomp}
\usepackage{amsmath}
\usepackage{amssymb}
\usepackage{bbm}
\usepackage{hyperref}
\usepackage{mathrsfs}
\usepackage{dsfont}
\usepackage{amsthm}
\usepackage{geometry}
\usepackage[all]{xy}
\usepackage[usenames,dvipsnames]{pstricks}
\usepackage{epsfig}
\usepackage{pst-grad} 
\usepackage{pst-plot} 
\usepackage{amsfonts}

\usepackage[normalem]{ulem}
\title{Stability of Minkowski Space-time with a translation space-like Killing field}
\author{Cécile Huneau}

\ifpdf
\fi

\newtheorem{thm}{Theorem}[section]
\newtheorem{prp}[thm]{Proposition}
\newtheorem{cor}[thm]{Corollary}
\newtheorem{lm}[thm]{Lemma}
\newtheorem{df}[thm]{Definition}
\newtheorem{rk}[thm]{Remark}

\newcommand{\m}[1]{\mathbb{#1}}
\newcommand{\ld}[1]{\left\|#1\right\|_{L^2}}
\newcommand{\q}[1]{\mathcal{#1}}

\newcommand{\wht}[1]{\widetilde{#1}}
\newcommand{\gra}[1]{\mathbf{#1}}

\newcommand{\ba}[1]{\underline{#1}}
\newcommand{\ep}{\varepsilon}
\newcommand{\ch}{\mathbbm{1}}

\newcommand{\Up}{\Upsilon}
\newcommand{\Upr}{\Upsilon\left(\frac{r}{t}\right)}
\newcommand{\Uprt}{\Upsilon\left(\frac{r}{\tau}\right)}

\numberwithin{equation}{section}

\makeatletter
\renewcommand*{\eqref}[1]{%
	\hyperref[{#1}]{\textup{\tagform@{\ref*{#1}}}}%
}
\makeatother

\begin{document}
 
\maketitle

\begin{abstract}
	In this paper we prove the nonlinear stability of Minkowski space-time with a translation Killing field.  In the presence of such a symmetry, the $3+1$ vacuum Einstein equations reduce to the $2+1$ Einstein equations with a scalar field. We work in generalised wave coordinates. In this gauge Einstein's equations can be written as a system of quasilinear quadratic wave equations. The main difficulty in this paper is due to the decay in $\frac{1}{\sqrt{t}}$ of free solutions to the wave equation in $2$ dimensions, which is weaker than in $3$ dimensions. This weak decay seems to be an obstruction for proving a stability result in the usual wave coordinates. In this paper we construct a suitable generalized wave gauge in which our system has a "cubic weak null structure", which allows for the proof of global existence.
\end{abstract}

\section{Introduction}

In this paper, we address the stability of the Minkowski solution to the Einstein vacuum equations with 
a translation space-like Killing field. 
More precisely, we look for solutions of the $3+1$ vacuum Einstein equations, on manifolds of the form $\Sigma \times \m R_{x_3} \times \m R_t$, where $\Sigma$ is a $2$ dimensional manifold,
equipped with a metric of the form
$$\gra g= e^{-2\phi} g + e^{2\phi}(dx^3)^2, $$
where  $\phi$ is a scalar function, and $g$ a Lorentzian metric on $\Sigma \times \m R$, all quantities being independent of $x^3$. 
For these metrics, Einstein vacuum equations are equivalent to the $2+1$ dimensional system
\begin{equation}\label{sys}
\left\{ \begin{array}{l}
\Box_g \phi = 0\\
R_{\mu \nu} =2 \partial_\mu  \phi \partial_\nu \phi,\\
\end{array} \right.
\end{equation}
where $R_{\mu \nu}$ is the Ricci tensor associated to $g$. Choquet-Bruhat and Moncrief studied the case where $\Sigma$ is compact of genus $G \geq 2$. In \cite{choquet}, they proved the stability of a particular expanding solution.
In this paper we work in the case $\Sigma= \m R^2$. A particular solution is then given by Minkowski solution itself. It corresponds to $\phi=0$ and $g=m$, the Minkowski metric in dimension $2+1$.
A natural question one can ask in this setting is the nonlinear stability of this solution.

In the $3+1$ vacuum case, the stability of Minkowski space-time has been proven in the celebrated work of Christodoulou and Klainerman \cite{ck} in a maximal foliation. It has then been proven by Lindblad and Rodnianski using wave coordinates in \cite{lind}. Their proof extends also to Einstein equations coupled to a scalar field.
Let us note that the perturbations of Minkowski solution considered in our paper are not asymptotically flat in $3+1$ dimension, due to the presence of a translation Killing field. Consequently they are not included in \cite{ck} and \cite{lind}.

In \cite{qstab} we already proved quasistability of the solution $(\phi=0,g=m)$ : the perturbed solutions exist in exponential time : more precisely we show that the solutions exist up to time $t\sim e^{\frac{1}{\sqrt{\ep}}}$ where $\ep$ measures the size of the initial data. In both \cite{qstab} and this paper, we work in generalized wave coordinates. Consequently the method we use is more in the spirit of \cite{lind} than in the spirit of \cite{ck}.

\subsection{Einstein equations in wave coordinates}
Wave coordinates $(x^\alpha)$ are required to satisfy $\Box_g x^\alpha=0$. In these coordinates \eqref{sys} reduces to the following system of quasilinear wave equations

\begin{equation}
\left\{ \begin{array}{l}
\Box_g \phi = 0,\\
\Box_g g_{\mu \nu}= -4\partial_\mu \phi \partial_\nu \phi +P_{\mu \nu}(\partial g, \partial g),
\end{array}
\right.
\end{equation}
where $P_{\mu \nu}$ is a quadratic form.
To understand the difficulty, let us first recall known results in $3+1$ dimensions. In $3+1$ dimensions, a semi linear system of wave equations of the form
\[ \Box u^i = P^i (\partial u ^j, \partial u^k) \]
is critical in the sense that if there isn't enough structure, the solutions might blow up in finite time (see the counter examples by John \cite{john}). However, if the right-hand side satisfies the null
condition, introduced by Klainerman in \cite{klai}, the system admits global solutions. This condition requires that  $P^i$ is a null form, that is to say a linear combinations of the following forms
$$Q_0(u,v)=\partial_t u \partial_t v-\nabla u. \nabla v, \quad Q_{\alpha \beta}(u,v)=\partial_\alpha u \partial_\beta v-\partial_\alpha v \partial_\beta u.$$
In $3+1$ dimensions, Einstein equations written in wave coordinates do not satisfy the null condition. However, this is not a necessary condition to obtain global existence. An example is provided by the system
\begin{equation}
\label{model}
\left\{ 
\begin{array}{l}
\Box \phi_1 =0, \\
\Box \phi_2=(\partial_t \phi_1)^2.
\end{array}\right.
\end{equation}
The non-linearity does not have the null structure, but thanks to the decoupling there is nevertheless global existence.
In \cite{craslind}, Lindblad and Rodnianski showed that once the semi linear terms involving null forms are removed, Einstein's equations in wave coordinates can be written as a system with the same structure as \eqref{model}. They used the wave condition to obtain better decay for some coefficients of the metric, which allow them to control the quasilinear terms. However the decay they are able to show for the metric coefficients is $\frac{\ln(t)}{t}$, which is slower than the decay for the solution of the wave equation which is $\frac{1}{t}$. An example of a quasilinear scalar wave equation admitting global existence without the null condition, but with a slower decay is also studied by Lindblad in \cite{lindsym} in the radial case, and by Alinhac in \cite{alin2} and Lindblad in \cite{lindblad} in the general case. In \cite{craslind}, Lindblad and Rodnianski introduced the notion of weak-null structure, which gathers all these examples.

In $2+1$ dimensions, to show global existence, one has to be careful with  both quadratic and cubic terms. Quasilinear scalar wave equations in $2+1$ dimensions have been studied by Alinhac in \cite{alin}. He shows global existence for a quasilinear equation of the form
$$\Box u =g^{\alpha \beta}(\partial u)\partial_\alpha \partial_\beta u,$$
if the quadratic and cubic terms in the right-hand side satisfy the null condition (the notion of null form for the cubic terms is defined in \cite{alin}).
Global existence for a semi-linear wave equation with the quadratic and cubic terms satisfying the null condition has been shown by Godin in \cite{godin} using an algebraic trick to remove the quadratic terms, which does however not extend to systems. The global existence in the case of systems of semi-linear wave equations with the null structure has been shown by Hoshiga in  \cite{hoshiga}. It requires the use of $L^\infty-L^\infty$ estimates for the inhomogeneous wave equations, introduced in \cite{kubo}.

To show global existence for our system in wave coordinates, it will therefore be necessary to exhibit structure in quadratic and cubic terms. However, as for the vacuum Einstein equations in $3+1$ dimension in wave coordinates, our system does not satisfy the null structure. In particular it is important to understand what happens for 
a system of the form \eqref{model} in $2+1$ dimensions. For such a system, standard estimates only give an $L^\infty$ bound for $\phi_2$, without decay. Moreover, the growth of the energy of $\phi_2$ is like $\sqrt{t}$. 

One can easily imagine that with more intricate a coupling than for \eqref{model}, it will be very difficult to prove stability without decay for $\phi_2$. It seems that in the usual wave gauge one cannot prove more than existence of the perturbed solutions in time $\frac{1}{\ep^2}$. But it also seems that this problem is only a problem of coordinates. In \cite{qstab} we overcame part of the difficulty by looking at solutions $g=g_{\mathfrak{b}} + \wht g$ with
$$ g_{\mathfrak{b}} = -dt^2 + dr^2 + (r+\chi(q)b(\theta)q)^2 d\theta^2,$$
where $r,\theta$ are the polar coordinate, $q=r-t$ and $\chi$ is a cut-off function such that $\chi(q)=0$ for $q<0$ and $\chi(q)=1$ for $q>1$,
and $b(\theta)$ suitably chosen, depending on $\phi$. We have two complementary points of view on this method. The metric $g_{\mathfrak{b}}$ can be seen as an approximate solution whose role is to tackle the worst terms in \eqref{sys}. Also, since $t,x_1=r\cos(\theta),x_2=r\sin(\theta)$ are not wave coordinates for $g_{\mathfrak{b}}$, this forces us to work in a different gauge, more suited to the geometry of the problem : the procedure can also be seen as choosing the right coordinate system, in which Einstein equations have a better structure.
The condition we imposed on $b(\theta)$ in \cite{qstab} was
$$\left|b(\theta)-\int_0^\infty (\partial_q \phi)^2(r,t,\theta)rdr\right| \lesssim \left(\frac{\ep^2}{\sqrt{1+t}}\right),$$
with $b$ depending only on $\theta$.
However, due to the logarithmic growth in $t$ of the higher energy norms of $\phi$, which seems inherent to such problems, the coefficient $b(\theta)$ was controlled only by restricting to exponential times.

The main idea of this paper to overcome this difficulty is to construct more carefully an approximate solution and gauge choice, noticing that in the metric $g_{\mathfrak{b}}$, the Fourier coefficients
$$\int b(\theta)d\theta,\quad \int b(\theta)\cos(\theta)d\theta, \quad \int b(\theta)\sin(\theta)d\theta,$$
are imposed by the constraint equations, but the other Fourier coefficients of $b(\theta)$ are only a gauge choice in the region $\chi=1$.

\subsection{The initial data}

In this section, we will explain how to choose the initial data for $\phi$ and $g$. We will note $i,j$ the space-like indices and $\alpha,\beta$ the space-time indices. We will work in weighted Sobolev spaces.

\begin{df} Let  $m\in \m N$ and $\delta \in \mathbb{R}$. The weighted Sobolev space $H^m_\delta(\mathbb{R}^n)$ is the completion of $C^\infty_0$ for the norm 
 $$\|u\|_{H^m_\delta}=\sum_{|\beta|\leq m}\|(1+|x|^2)^{\frac{\delta +|\beta|}{2}}D^\beta u\|_{L^2}.$$
The weighted Hölder space $C^m_{\delta}$ is the complete space of $m$-times continuously differentiable functions with norm 
$$\|u\|_{C^m_{\delta}}=\sum_{|\beta|\leq m}\|(1+|x|^2)^{\frac{\delta +|\beta|}{2}}D^\beta u\|_{L^\infty}.$$
Let $0<\alpha<1$. The Hölder space $C^{m+\alpha}_\delta$ is the the complete space of $m$-times continuously differentiable functions with norm 
$$\|u\|_{C^{m+\alpha}_{\delta}}=\|u\|_{C^m_\delta} + \sup_{x \neq y, \; |x-y|\leq 1} \frac{|\partial^m u(x)-\partial^m u(y)|(1+|x|^2)^\frac{\delta}{2}}{|x-y|^\alpha}.$$
\end{df}
We  recall the Sobolev embedding with weights (see for example \cite{livrecb}, Appendix I).

\begin{prp}\label{holder} Let $s,m \in \m N$. We assume $s >1$. Let $\beta \leq \delta +1$ and $0<\alpha<min(1,s-1)$. Then, we have the continuous embedding
$$H^{s+m}_{\delta}(\m R^2)\subset C^{m+\alpha}_{\beta}(\m R^2).$$
\end{prp}

Let $0<\delta <1$ and $N \geq 1$.
The initial data $(\phi_0, \phi_1)$ for $(\phi,\partial_t \phi)|_{t=0}$ are freely given in $H^{N+1}_{\delta}\times H^{N}_{\delta+1}$.
For technical reasons, we will work here with compactly supported initial data for $\phi$ : $(\phi_0,\phi_1) \in H^{N+1}(\m R^2)\times H^N(\m R^2)$ supported in $B(0,R)$. 
The initial data for $(g_{\mu\nu},\partial_t g_{\mu\nu})$ cannot be chosen arbitrarily, they must satisfy the constraint equations.

We recall the constraint equations.  First we write the metric $g$ in the form
$$g = -N^2(dt)^2 +\bar{g}_{ij}(dx^i +\beta^i dt)(dx^j + \beta^j dt),$$
where the scalar function $N$ is called the lapse, the vector field $\beta$ is called the shift and $\bar{g}$ is a Riemannian metric on $\m R^2$. 

We consider the initial space-like surface $\m R ^2 = \{t=0\}$.
We will use the notation
$$\partial_0=\partial_t - \mathcal{L}_{\beta},$$
where $\mathcal{L}_{\beta}$ is the Lie derivative associated to the vector field $\beta$. With this notation, we have the following 
expression for the second fundamental form of $\m R^2$
$$K_{ij}=-\frac{1}{2N}\partial_0 g_{ij}.$$
We will use the notation 
$$\tau=g^{ij}K_{ij}$$
for the mean curvature. We also introduce the Einstein tensor
$$G_{\alpha \beta}=R_{\alpha \beta} - \frac{1}{2}Rg_{\alpha \beta},$$
where $R$ is the scalar curvature $R = g^{\alpha \beta}R_{\alpha \beta}$. 
The constraint equations are given by
\begin{align}
 \label{contrmom} G_{0j} &\equiv N(\partial_j \tau - D^i K_{ij})=2\partial_0\phi \partial_j \phi, \; j=1,2,\\
\label{contrham} G_{00} & \equiv \frac{N^2}{2}(\bar{R}-|K|^2+ \tau^2)= 2(\partial_0 \phi)^2 - {g}_{00} {g}^{\alpha \beta}\partial_\alpha\phi \partial_\beta \phi,
\end{align}
where $D$ and $ \bar{R}$ are respectively the covariant derivative and the scalar curvature associated to $\bar{g}$. 
We have studied the constraint equations in \cite{moi} and \cite{expcont}. The following result is a direct consequence of \cite{expcont} which was proven in Appendix 1 of \cite{qstab}. It gives us the initial data we need.

\begin{thm}\label{thinitial}
Let $0<\delta<1$. Let $(\phi_0, \phi_1)\in H^{N+1}_{\delta}(\m R^ 2)\times H^{N}_{\delta+1}(\m R^ 2)$ 
We assume
$$\|\phi_0\|_{H^{N+1}_\delta}+\|\phi_1\|_{H^N_{\delta+1}}
\lesssim \ep.$$
If $\ep>0$ is small enough, there exists $a_0,a_1,a_2 \in \m R \times \m R \times \m S^1$, $J\in W^{N,2}(\m S^ 1)$ and 
$$ (g_{\alpha \beta})_0,(g_{\alpha \beta})_1 \in H^{N+1}_{\delta}\times H^{N}_{\delta+1}$$ such that the initial data for $g$ given by
$$g=g_{\mathfrak{a}}+g_0, \;\partial_t g = \partial_t g_{\mathfrak{a}} +g_1,$$
where $g_{\mathfrak{a}}$ is defined by 
$$g_{\mathfrak{a}} = -dt^2 + dr^2 + (r+\chi(q)a(\theta)q)^2 d\theta^2+J(\theta)\chi(q)dqd\theta,$$
 with
$$a(\theta)=a_0+a_1\cos(\theta)+a_2\sin(\theta),$$
are such that
\begin{itemize}
\item $g_{ij}, K_{ij}=-\frac{1}{2}\partial_0 g_{ij}$ satisfy the constraint equations \eqref{contrmom} and \eqref{contrham}.
\item the following generalized wave coordinates condition is satisfied at $t=0$
$$g^{\lambda\beta}\Gamma^\alpha_{\lambda \beta}=g_{\mathfrak{a}}^{\lambda\beta}(\Gamma_a)^\alpha_{\lambda \beta},$$
\end{itemize}
where $\Gamma_a$ denotes the Christoffel symbols of $g_{\mathfrak{a}}$, expressed in the coordinates $t,x_1=r\cos(\theta), x_2=r\sin(\theta)$.
Moreover, we have the estimates
$$\|J\|_{W^{N,2}(\m S^1)}+\|g_0\|_{H^{N+1}_{\delta}} +\|g_1\|_{H^N_{\delta+1}} \lesssim \ep^2,$$
\begin{align*}
a_0&=\frac{1}{4\pi}\int \left(\dot{\phi}^2+|\nabla \phi|^2\right)dx +O(\ep^4),\\
a_1&=\frac{1}{\pi}\int \dot{\phi}\partial_1 \phi dx +O(\ep^4),\\
a_2&=\frac{1}{\pi}\int \dot{\phi}\partial_2 \phi dx +O(\ep^4).
\end{align*}
\end{thm}

Let us note that in the resolution of the constraint equations, the only free data in the metric is in the choice of $\tau$ and corresponds to what hypersurface in the space-time will be "$t=0$".

Before stating our main result, we will recall some notations and basic tools in the study of wave equations.

\subsection{Some basic tools}

We will use the notation $a\lesssim b$ when there exists a numerical constant $C$ such that $a\leq Cb$.
\paragraph{Coordinates and frames}
\begin{itemize}
\item We note $x^\alpha$ the standard space-time coordinates, with $t=x^0$.
We note $(r,\theta)$ the polar space-like coordinates, and $s=t+r$, $q=r-t$ the null coordinates. The associated one-forms are
$$ds=dt+dr, \quad dq=dr-dt,$$
and the associated vector fields are
$$\partial_s = \frac{1}{2}(\partial_t+\partial_r), \quad \partial_q =\frac{1}{2}(\partial_r -\partial_t).$$
\item We note $\partial$ the space-time derivatives, $\nabla$ the space-like derivatives, and by 
$\bar{\partial}$ the derivatives tangent to the future directed light-cone in Minkowski, that is to say $\partial_t +\partial_r$ and
$\frac{\partial_\theta}{r}$.
\item We introduce the null frame $L=\partial_t + \partial_r$, $\ba L = \partial_t -\partial_r$, $U =\frac{\partial_\theta}{r}$. In this frame, the Minkowski metric takes the form
$$m_{L\ba L}= -2, \; m_{UU}= 1, \; m_{LL}=m_{\ba L \ba L}=m_{L U}= m_{\ba L U}=0.$$
The collection $\q T =\{U,L\}$ denotes the vector fields of the frame tangent to the outgoing light-cone, and the collection 
$\q V = \{U,L,\ba L\}$ denotes the full null frame.
\item When it is omitted, the volume form is $dx$, the Lebesgue measure for the background coordinates, and the domain of integration is $\m R^2$. The $L^p$ spaces are also always considered with respect to the Lebesgue measure for the background coordinates.
\end{itemize}

\paragraph{The flat wave equation}
Let $\phi$ be a solution of
\begin{equation}\label{eqphi}
 \left\{ \begin{array}{l} \Box \phi = 0,\\
	 (\phi, \partial_t \phi)|_{t=0} =(\phi_0, \phi_1).\\
        \end{array}
\right.
\end{equation}
The following proposition establishes decay for the solutions of $2+1$ dimensional flat wave equation.
\begin{prp}[Proposition 2.1 in \cite{kubota}]\label{flat1}
Let $\mu >\frac{1}{2}$. 
We have the estimate
\[|\phi(x,t)|\lesssim M_\mu(\phi_0,\phi_1)\frac{(1+|t-r|)^{[1-\mu]_{+}}}{\sqrt{1+t+r}\sqrt{1+|t-r|}}\]
where
\[M_\mu(\phi_0, \phi_1)=\sup_{y \in \m R^2}
(1+|y|)^{\mu}|\phi_0(y)|+(1+|y|)^{\mu +1}(|\phi_1(y)|+|\nabla \phi_0(y)|)\]
and where we used the notation $A^{[\alpha]_{+}}=A^{\max(\alpha,0)}$ if $\alpha \neq 0$ and 
 $A^{[0]_{+}}=\ln(A)$.
\end{prp}

\paragraph{Minkowski vector fields} We will rely in a crucial way on the
 Klainerman vector field method. We introduce the following family of vector fields
$$\q Z =\left\{\partial_\alpha, \Omega_{\alpha \beta}=-x_\alpha \partial_\beta + x_\beta \partial_\alpha, S=t\partial_t + r\partial_r \right\},$$
where $x_\alpha = m_{\alpha \beta}x^\beta$. These vector fields satisfy the commutation property 
$$[\Box, Z]=C(Z)\Box,$$
where 
$$C(Z)=0,\; Z \neq S, \quad C(S)=2.$$
Moreover some easy calculations give

\begin{align*}
& \partial_t + \partial_r = \frac{S+\cos(\theta)\Omega_{0,1}+\sin(\theta)\Omega_{0,2}}{t+r},\\
&\frac{1}{r}\partial_\theta = \frac{\Omega_{1,2}}{r}=\frac{\cos(\theta) \Omega_{0,2} -\sin(\theta)\Omega_{0,1}}{t},\\
& \partial_t-\partial_r  = \frac{S-\cos(\theta)\Omega_{0,1}-\sin(\theta)\Omega_{0,2}}{t-r}.
\end{align*}
With these calculations, and the commutations properties in the region
$-\frac{t}{2}\leq r \leq 2t$
$$[Z,\partial]\sim \partial, \; [Z,\bar{\partial}] \sim \bar{\partial},$$
we obtain
\begin{equation}
\label{important}
|\partial^k \bar{\partial^l}u|\leq \frac{1}{(1+|q|)^k (1+s)^l}|Z^{k+l} u|,
\end{equation}
where here and in the rest of the paper,
$Z^I u$ denotes any product of $I$ or less of the vector fields of $\q Z$.
Estimates \eqref{important} and Proposition \ref{flat1} yield

\begin{cor}\label{cordec}Let $\phi $ be a solution of \eqref{eqphi}. We have the estimate
\[|\partial^k \bar{\partial}^l \phi(x,t)|
\lesssim M^{k+l}_\mu(\phi_0, \phi_1) 
\frac{(1+|t-r|)^{[1-\mu]_{+}}}{(1+t+r)^{l+\frac{1}{2}}(1+|t-r|)^{k+\frac{1}{2}}}\]
where 
\[M^j_\mu(\phi_0, \phi_1)  =\sup_{y \in \m R^2}
(1+|y|)^{\mu+j}|\nabla^s \phi_0(y)|+(1+|y|)^{\mu +1+j}(|\nabla^s \phi_1(y)|+|\nabla^{1+j} \phi_0(y)|).\]
\end{cor}

\paragraph{Weighted energy estimate}
We consider a weight function $w (q)$, where $q=r-t$, such that $w'(q)>0$ and 
$$\frac{w(q)}{(1+|q|)^{1+\mu}}\lesssim w'(q)\lesssim \frac{w(q)}{1+|q|},$$
for some $0<\mu<\frac{1}{2}$.

\begin{prp}\label{energy}
We assume that $\Box \phi= f$. Then we have
\begin{align*}
&\frac{1}{2}\partial_t \int_{\m R^2} w(q)\left((\partial_t \phi)^2+|\nabla \phi|^2 \right)dx
+\frac{1}{2}\int_{\m R^2} w'(q)\left((\partial_s \phi)^2 + \left(\frac{\partial_\theta \phi}{r}\right)^2\right)dx \\
&\lesssim \int_{\m R^2} w(q)|f\partial_t \phi|dx. 
\end{align*}
\end{prp}
For the proof of Proposition \ref{energy}, we refer to the proof of Proposition \ref{prpweighte}.

\paragraph{Weighted Klainerman-Sobolev inequality} The following proposition allows us to obtain $L^\infty$ estimates from the energy estimates. It is proved in Appendix 5 of \cite{qstab}. The proof is inspired from the corresponding $3+1$ dimensional proposition (Proposition 14.1 in \cite{lind}).
\begin{prp}\label{prpweight}
We denote by $v$ any of our weight functions.
 We have the inequality
$$|f(t,x)v^\frac{1}{2}(|x|-t)|\lesssim \frac{1}{\sqrt{1+t+|x|}\sqrt{1+||x|-t|}}\sum_{|I|\leq 2}\|v^\frac{1}{2}(.-t) Z^I f\|_{L^2}.$$
\end{prp}

\paragraph{Weighted Hardy inequality}
If $u$ is solution of $\Box u=f$, the energy estimate allows us to estimate the $L^2$ norm of $\partial u$. To estimate the $L^2$ norm of $u$, we will use a weighted Hardy inequality.
\begin{prp}\label{hardy}Let $\alpha<1$ and $\beta>1$. We have, with $q=r-t$
$$\left\|\frac{v(q)^\frac{1}{2}}{(1+|q|)}f\right\|_{L^2} \lesssim\|v(q)^\frac{1}{2}\partial_r f\|_{L^2}, $$
where
\begin{align*}
v(q)&=(1+|q|)^\alpha, \; for\; q<0,\\
v(q)&=(1+|q|)^\beta, \; for\; q>0.\\
\end{align*}
\end{prp}
This is proven in Appendix 4 of \cite{qstab}. The proof is inspired from the $3+1$ dimensional analogue (Lemma 13.1 in \cite{lind}).

\paragraph{$L^\infty-L^\infty$ estimate}  With the condition $w'(q)>0$ for the energy inequality, we are not allowed to take weights of the form $(1+|q|)^\alpha$, with $\alpha>0$ in the region $q<0$. Therefore, Klainerman-Sobolev inequality can not give us more than the estimate
$$|\partial u|\lesssim \frac{1}{\sqrt{1+|q|}\sqrt{1+s}},$$
in the region $q<0$, for a solution of $\Box u= f$. However, we know that for suitable initial data, the solution of the wave equation $\Box u=0$ satisfies
$$|u|\lesssim \frac{1}{\sqrt{1+|q|}\sqrt{1+s}}, \; 
|\partial u|\lesssim \frac{1}{(1+|q|)^\frac{3}{2}\sqrt{1+s}}.$$
To recover some of this decay we will use the following proposition
 
\begin{prp}\label{inhom}
Let u be a solution of
\begin{equation*}
 \left\{ \begin{array}{l} 
         \Box u = F, \\
	 (u, \partial_t u)|_{t=0}=(0,0).
        \end{array}
\right.
\end{equation*}
For $\mu>\frac{3}{2} , \nu >1$ we have the following $L^\infty-L^\infty$ estimate
$$|u(t,x)|( 1+t+|x|)^\frac{1}{2} \leq C(\mu, \nu) M_{\mu, \nu}(F) (1+|t-|x|||)^{-\frac{1}{2} + [2-\mu]_{+}},$$
where 
$$M_{\mu, \nu}(F) = \sup (1+|y|+s)^\mu (1+|s-|y||)^\nu F(y,s),$$
and where we used the convention $A^{[\alpha]_+}=A^{\max(\alpha,0)}$ if $\alpha \neq 0$ and $A^{[0]_{+}}=\ln(A)$. 
\end{prp}
This is proven in Appendix 3 of \cite{qstab}. This inequality has been introduced by Kubo and Kubota in \cite{kubo}.

\paragraph{An integration lemma} The following lemma will be used many times in the proof of Theorem \ref{main}, to obtain estimates for $u$ when we only have estimates for $\partial u$.
\begin{lm}
\label{lmintegration}
Let $\alpha,\beta,\gamma \in \m R$ with $\beta<-1$. We assume that the function $u: \m R^{2+1} \rightarrow \m R$ satisfies
$$|\partial u|\lesssim (1+s)^\gamma(1+|q|)^\alpha, \; for\; q<0, \quad
|\partial u|\lesssim (1+s)^\gamma (1+|q|)^\beta \; for \; q>0,$$
and for $t=0$
$$|u|\lesssim (1+r)^{\gamma+\beta}.$$
Then we have the following estimates
$$| u|\lesssim (1+s)^\gamma\max(1,(1+|q|)^{\alpha+1}), \; for\; q<0, \quad
|u|\lesssim (1+s)^\gamma (1+|q|)^{\beta+1} \; for \; q>0.$$
\end{lm}
\begin{proof}
We assume first $q>0$. We integrate the estimate
$$|\partial_q u|\lesssim (1+s)^\gamma (1+|q|)^\beta,$$
from $t=0$. We obtain, since $\beta<-1$, for $q>0$
$$|u|\lesssim (1+s)^\gamma (1+|q|)^{\beta+1}.$$
Consequently, we have, for $q=0$, $|u|\lesssim (1+s)^\gamma$.
We now assume $q<0$. We integrate
$$ |\partial_q u|\lesssim (1+s)^\gamma(1+|q|)^\alpha,$$
from $q=0$. We obtain
$$| u|\lesssim (1+s)^\gamma\max(1,(1+|q|)^{\alpha+1}).$$
This concludes the proof of Lemma \ref{lmintegration}.
\end{proof}

\paragraph{Generalized wave coordinates}
In a coordinate system $x^\alpha$, the Ricci tensor is given by
\begin{equation}\label{calcricci}
R_{\mu \nu}=-\frac{1}{2}g^{\alpha \rho}\partial_\alpha \partial_\rho g_{\mu \nu} +\frac{1}{2} H^\rho \partial_\rho g_{\mu \nu}
+\frac{1}{2}\left( g_{\mu \rho}\partial_\nu H^\rho + g_{\nu \rho}\partial_\mu H^\rho \right) + \frac{1}{2}P_{\mu \nu}(g)(\partial g, \partial g),
\end{equation}
where $P_{\mu \nu}(g)(\partial g, \partial g)$ is a quadratic form in $\partial g$ and 
\begin{equation}\label{wavecond}H^\alpha =-\Box_g x^\alpha= -\partial_\lambda g^{\lambda \alpha}-\frac{1}{2}g^{\lambda \mu} \partial^\alpha g_{\lambda \mu}.
\end{equation}
The wave coordinate condition (respectively the generalized wave coordinate condition) consists in imposing $H^\alpha= 0$ (respectively $H^\alpha= F^\alpha$ a fixed function, which may depend on $g$ but not on its derivatives).
\begin{prp}
	If the coupled system of equations
	\begin{equation}
	\label{wcoor}\left\{\begin{array}{l}
	-\frac{1}{2}g^{\alpha \rho}\partial_\alpha \partial_\rho g_{\mu \nu} +\frac{1}{2} F^\rho \partial_\rho g_{\mu \nu}
	+\frac{1}{2}\left( g_{\mu \rho}\partial_\nu F^\rho + g_{\nu \rho}\partial_\mu F^\rho \right) +\frac{1}{2} P_{\mu \nu}(g)(\partial g, \partial g)=2\partial_\mu \phi \partial_\nu \phi\\
	g^{\alpha \rho}\partial_\alpha\partial_\rho \phi - F^\rho\partial_\rho \phi = 0
	\end{array}\right.
	\end{equation}
	with $F$ a function which may depend on $\phi, g$,
	is satisfied on a time interval $[0,T]$ with $T>0$, if the initial induced Riemannian metric and second fundamental form $(\bar{g}, K)$ satisfy the constraint equations, and if the initial compatibility condition
	\begin{equation}\label{ft0}
	F^\alpha|_{t=0} = H^\alpha|_{t=0},
	\end{equation}
	is satisfied, then the equations \eqref{sys} are satisfied on $[0,T]$, together with the wave coordinate condition
	$$F^\alpha = H^\alpha.$$
\end{prp}
For a proof of this result, we refer to \cite{wald}, or Appendix 2 of \cite{qstab}.

\subsection{Main Result}
We introduce another cut-off function $\Up:\m R_+ \rightarrow \m R_+$ such that $\Up(\rho)=0$ for $\rho\leq \frac{1}{2}$ and $\rho \geq 2$ and $\Up=1$ for 
$\frac{3}{4}\leq \rho \leq \frac{3}{2}$. 
Theorem \ref{main} is our main result, in which we prove stability of Minkowski space-time with a translational symmetry.
We give here a first version, a more precise one will be given in Section \ref{second}.
\begin{thm}\label{main}Let $0<\ep<1$.
Let $\frac{1}{2}<\delta<1$ and $N \geq 25$.
Let $(\phi_0,\phi_1) \in H^{N+2}(\m R^2)\times H^{N+1}(\m R^2)$ compactly supported in $B(0,R)$.  
We assume
$$\|\phi_0\|_{ H^{N+2}}+\|\phi_1\|_{H^{N+1}}
\leq \ep.$$
Let $\ep\ll\rho\ll \sigma \ll \delta$, such that $\delta-2\sigma>\frac{1}{2}$.
If $\ep$ is small enough, there exists a global solution $(g,\phi)$ of \eqref{sys}.
Moreover, if we call $C$ the causal future of $B(0,R)$, and $\bar{C}$ its complement, there exists
a coordinate system $(t,x_1,x_2)$ in $C$ and a coordinate system $(t',x_1',x_2')$ in $\bar{C}$ such that
we have in $C$ :
$$(\phi,\partial_t \phi)|_{t=0}= (\phi_0,\phi_1),$$
$$|g-m|\lesssim \frac{\ep}{(1+s)^{\frac{1}{2}-\rho}}, \quad |\phi| \lesssim \frac{\ep}{(1+s)^{\frac{1}{2}}(1+|q|)^{\frac{1}{2}-4\rho}},$$
$$\|\ch_C \partial Z^{N} \phi\|_{L^2} +\|\ch_C \partial^2 Z^N \phi\|_{L^2} + \|\ch_C(1+|q|)^{-\frac{1}{2}-\sigma }\partial Z^N (g-m)\|_{L^2} \lesssim \ep (1+t)^{2\rho},$$
and we have in $\bar{ C}$
$$\|\ch_{\bar{ C}} (1+|q|)^{1+\delta -2\sigma} \partial Z^N (g-g_{\mathfrak{a}})\|_{L^2}\lesssim \ep(1+t)^{2\rho}.$$
\end{thm}

\paragraph{Comments on Theorem \ref{main}}
\begin{itemize}
	\item We consider perturbations of $3+1$ dimensional Minkowski space-time with a translational space-like Killing field. These perturbations are not asymptotically flat in $3+1$ dimensions, therefore the result of Theorem \ref{main} does not follow from the stability of Minkowski space-time by Christodoulou and Klainerman  \cite{ck}.
	\item As our gauge, we choose  the generalized wave coordinates. Therefore, the method we use has a lot in common with the method of Lindblad and Rodnianski in \cite{lind} where they proved the stability of Minkowski space-time in harmonic gauge. It is an interesting problem to investigate the stability of Minkowski with a translation symmetry using a strategy in the spirit of \cite{ck} or \cite{nicolo}. 
	\item Theorem \ref{main} can be easily generalized to the non polarized case, where the $3+1$ dimensional metric is of the form
	$$\gra g= e^{-2\phi} g + e^{2\phi}(dx^3+A_\alpha dx^\alpha)^2. $$
	In this case, the vacuum Einstein equations reduce to Einstein equations coupled to a wave-map system in $2+1$ dimension. Since the additional equations have the null structure, this does not make any change to the proof given here.
	\item It is conjectured that maximal Cauchy developments of asymptotically flat solutions to the 3+1 vacuum Einstein equations with a spacelike translational Killing field are geodesically complete (see \cite{guda}). This result has been proved in the non polarized case with an additional symmetry assumption in \cite{guda}.
	
	\item We assume more regularity for $\phi$ than for $g$. This is possible in wave coordinates because the equation $\Box_g \phi = 0$ involves only $g$ and not its derivatives. The proof is based on the construction of an approximate solution in the exterior region, and it is for the control of this approximate solution, which involves one derivative of $\phi$ that the additional regularity for $\phi$ is needed. 
\item Our proof restricts to $\phi$ compactly supported. The reason why is the following. Our approximate solution forces us to work in adapted generalized wave coordinates in the exterior. The approximate solution involves one derivatives of $\phi$, so if $\phi$ was not supported only in the interior, the equation $\Box_g \phi=0$ would involve coupling terms between $\phi$ and the approximate solution, at the level of two derivatives of $\phi$. Let us note that in \cite{qstab}, the compact support assumption for $\phi$ is not needed.
\item The space-time constructed in Theorem \ref{main} is the development of the initial data of Theorem \ref{thinitial}. At space-like infinity, the metric converges to $g_{\mathfrak{a}}$ given by the constraint equations. The metric $g_{\mathfrak{a}}$ is Ricci flat in the exterior and has a deficit angle. This behaviour is similar to the behaviour of Einstein-Rosen waves (see \cite{beck} and \cite{asht}) which are radial solutions of \eqref{sys}.
\item Generalized wave coordinates have also been used by Hintz and Vasy in the proof of the nonlinear stability of Kerr-de Sitter black holes (see \cite{hintz}). There seems to be a lot of similarities between the two constructions. In their paper they choose the generalized wave coordinates inductively in order to remove the non physical exponentially growing solutions which appear as solutions to the Einstein equations in wave coordinates. In our paper, solutions to the Einstein equations in wave coordinates may have a growth in $\sqrt{t}$ : we also choose the generalized wave coordinates to remove this pathological behaviour.
\end{itemize}

\subsection{Sketch of the proof}\label{sketch}

To begin with, let us look at the structure of Einstein equations in wave coordinates.
The structure of Einstein equations can be seen when we write them in the null frame
$L$, $\ba L$, $U$.
We decompose the metric into
$$g=m+\wht{g}+\frac{1}{4}g_{\ba L \ba L}dq^2,$$
where $m$ is the Minkowski metric.
Then, if we neglect all the nonlinearities involving a good derivative, we obtain the following model system for \eqref{sys} in wave coordinates
\begin{equation*}
\left\{\begin{array}{l}
\Box \phi + g_{LL}\partial^2_q \phi = 0,\\
\Box  \wht g + g_{LL}\partial^2_q \wht g =0,\\
\Box  g_{\ba L \ba L} + g_{LL}\partial^2_q g_{\ba L \ba L} =-16(\partial_q \phi)^2.
\end{array}
\right.
\end{equation*}
The quadratic terms involving $g_{LL}$ are handled by making use of the wave coordinate condition, as in \cite{lind} :
the condition $H^\alpha=0$ where $H^\alpha$ is defined by \eqref{wavecond} implies $\partial_q g_{LL}\sim \bar{\partial } \wht g$ (more precisely it is implied by $\ba L_{\alpha}H^\alpha=0$). 
Therefore, the quadratic terms involving $g_{LL}$ behave like terms having the null structure.
Consequently, we are left with the model system
\begin{equation*}
\left\{\begin{array}{l}
\Box \phi  = 0,\\
\Box  g_{\ba L \ba L} =-16(\partial_q \phi)^2.
\end{array}
\right.
\end{equation*}
Thanks to the decoupling it is of course possible to solve such a system. However, in $2+1$ dimensions, for initial data of size $\ep$ the energy estimate yields
$$\|\partial g_{\ba L \ba L}\|_{L^2}\lesssim \ep \sqrt{1+t},$$
and the metric coefficient $g_{\ba L \ba L}$ has no decay, not even with respect to $q=r-t$. This is not enough to solve the full coupled system.
However, this seems to be only a coordinate problem.
To see it, let's assume for a moment we had found a coordinate system (not the wave gauge) in which all the metric coefficients have at least the decay of a solution to the free wave equation.
Then we can compute, on the light cone
$$R_{\ba L \ba L }= -2\partial_q^2 g_{UU}+O\left(\frac{\ep}{(1+r)^\frac{3}{2}}\right).$$
Since we also have
$$R_{\ba L \ba L}=8(\partial_q \phi)^2=O\left(\frac{\ep}{1+r}\right),$$
we see that the only term which could balance this behaviour is $\partial_q^2 g_{UU}$.
Consequently, we would like to write
$$\partial_q g_{UU}= -\frac{4}{r}\int (\partial_q \phi)^2rdq +\partial_q \wht g_{UU}.
$$
We can impose to have such a decomposition in $C$, the causal future of $B(0,R)$, by choosing to work in generalized wave coordinates such that
\begin{equation}
\label{wcond}H^{ L}=- L_\alpha \Box_g x^\alpha =  \frac{2}{r}\int_{-\infty}^q (\partial_q \phi)^2rdq.
\end{equation}
Indeed, we will see in Section \ref{secwave} that $L_\alpha \Box_g x^\alpha \sim \frac{1}{2}\partial_q g_{UU}+\bar{\partial} g$. With this gauge, a model equation for $g_{\ba L \ba L}$ will be
$$
\Box  g_{\ba L \ba L} =-16(\partial_q \phi)^2+2g_{\ba L  L}\partial_{\ba L} H^L
\sim 4\wht g_{L \ba L}(\partial_q \phi)^2.
$$
In the right-hand side, instead of having a quadratic nonlinearity without null structure, we now have a cubic nonlinearity without null structure. This leads to a logarithmic loss in the estimate for $g_{\ba L \ba L}$ but as in \cite{lind}, this loss occurs only on "bad" coefficients. Thanks to the structure, "bad" components of the metric interact only with good derivatives. Consequently, this is not an obstruction for proving global existence. 

Let us analyse the consequence of \ref{wcond} in the exterior region. Since the initial data for $\phi$ are compactly supported, outside the causal future of this compact region, we obtain
$$  \partial_q g_{UU}= \frac{1}{r}a(\theta,s) +\partial_q \wht g_{UU},$$
for some function $a$ which can be computed from $\phi$,
and consequently
$$ g_{UU} \sim \frac{q}{r}a(\theta,s).$$
We can compute that the metric
$$-dt^2 +dr^2 +(r+qa(\theta,s))^2d\theta^2,$$
is not Ricci flat when $a$ depends on $s$. This is not compatible with the fact that outside the causal future of $B(0,R)$ we have $R_{\mu \nu}= 2\partial_\mu \phi \partial_\nu \phi=0$.

To overcome this difficulty we will follow the following scheme.
\begin{itemize}
	\item The initial data for $\phi$ are given, compactly supported in $B(0,R)$. The initial data for $g$ are given by Theorem \ref{thinitial}. At the level of the initial data, we can write $g=g_{\mathfrak{a}}+\wht g$, with
	\begin{align*}
	g_{\mathfrak{a}} &= -(dt')^2 + (dr')^2 + (r'+\chi(q')a(\theta')q')^2 (d\theta')^2+J(\theta')\chi(q')dq'd\theta',\\
	a(\theta')&=a_0+a_1\cos(\theta')+a_2\sin(\theta'),\\
	a_0&=\frac{1}{4\pi}\int_{\m R^2} \left(\dot{\phi}^2+|\nabla \phi|^2\right)dx +O(\ep^4),\\
	a_1&=\frac{1}{\pi}\int_{\m R^2} \dot{\phi}\partial_1 \phi dx +O(\ep^4),\\
	a_2&=\frac{1}{\pi}\int_{\m R^2} \dot{\phi}\partial_2 \phi dx+O(\ep^4),
	\end{align*}
	and $(\wht g ,\partial_t \wht g)|_{t=0} \in H^{N+1}_\delta \times H^N_{\delta+1}$ for all $0<\delta<1$.
	\item We can solve \eqref{sys} in generalised wave coordinates $\Box_g x^{(\alpha)}=\Box_{g_{\mathfrak{a}}} x^{(\alpha)}$ up to a time $T$.
	We obtain a solution of the form $g=g_{\mathfrak{a}}+\wht g$.
	Moreover the solution is global outside the causal future of $B(0,R)$ (see Appendix \ref{ext}).
	\item We consider a function $b(\theta,s)$, satisfying a set of hypothesis $H$, and make a change of variable in the region $q>R+1$
	$$s'\sim (1+b(\theta,s))s,$$
	$$q'\sim (1+b(\theta,s))^{-1}q,$$
	$$\theta'\sim \theta + f(\theta,s),$$
	where $f(\theta,s)$ is such that
	$$1+\partial_\theta f(\theta,s)= (1+b(\theta,s))^{-1}.$$
	We use the symbol $\sim$ because we prefer to give a simplified formula at this stage, to enlighten the main contributions. The precise formula is given in next section.
	With this change of variable, we obtain a solution of the form $g=g_{\mathfrak{b}}+ \wht g$,
	where in the region $q>R+1$, the metric $g_{\mathfrak{b}}$ corresponds to $g_{\mathfrak{a}}$, expressed in the new coordinates $q,s,\theta$. By construction $g_{\mathfrak{b}}$ is Ricci flat for $q>R+1$
and it is given by
	$$ g_{\mathfrak{b}} \sim -dt^2+ dr^2 + (r+\chi(q)(a(\theta)-b(\theta,s)-\partial^2_\theta b(\theta,s))q)^2 d\theta^2,$$
	where we have neglected the terms involving $\partial_s b$.
	By looking at the Riemannian metric induced by $g$ on the new hypersurface $t=0$ and at the second fundamental form,  we obtain a solution to the constraint equations with an asymptotic behaviour compatible with $g_{\mathfrak{b}}$ (see Appendix \ref{reguini}).
	\item We now perform a bootstrap argument. We assume that we have a solution $(g,\phi)$ on a time interval $[0,T']$ in  generalized coordinates such that
	in the exterior $\Box_g x^\alpha \sim \Box_{g_{\mathfrak{b}}} x^\alpha$ (plus corrective terms) and in the interior we have \eqref{wcond} (plus corrective terms). We assume that we can write $g=g_{\mathfrak{b}} + \wht g$, with $\wht g$ satisfying estimates similar to the estimates for a free wave (except for $\wht g_{\ba L \ba L}$ which has a logarithmic growth in the energy). We note
	$$h(\theta,s)=a(\theta)+b(\theta,s)+\partial^2_\theta b(\theta,s).$$
	We would like to have $h(\theta,s)=\int_0^\infty (\partial_q \phi(t=\frac{s}{2},r,\theta))^2rdr$\footnote{It is more convenient for the estimates to use an integration along lines of constant $t$ and $\theta$ than lines of constant $s$ and $\theta$. On the light cone, we have $t=\frac{s}{2}$, so it is why we evaluate the integral at $t=\frac{s}{2}$}. However, we can not ask directly for this equality to hold because it would introduce non local terms in the equations. Instead we will use bootstrap assumptions to construct $h$ inductively :
	in the bootstrap assumptions we assume that
	\begin{equation}
	\label{skbootb}\left|\Pi  \left(\int_0^\infty (\partial_q \phi(t=s/2,r,\theta))^2rdr - h(\theta,s)\right)\right|\lesssim \frac{\ep^2}{\sqrt{1+s}},
	\end{equation}
	where $\Pi: H^k(\m S^1)\rightarrow H^k(\m S^1)$ is the projection operator such that 
	$$\int \Pi(u) d\theta = \int \Pi(u) \cos(\theta)d\theta=\int \Pi(u)\sin(\theta) d\theta=0.$$
	\item 
	By integrating the constraint equations on a time slab $t$ constant, we obtain the remaining estimate for $h$
	$$\left|\int_0^\infty (\partial_q \phi(t=s/2,r,\theta))^2rdr - h(\theta,s)\right|\lesssim \frac{\ep^2}{\sqrt{1+s}}.$$
	\item We obtain estimates for $\wht g$ and $\phi$ thanks to $L^\infty-L^\infty$ estimates and energy estimates. This step follows a quite standard vector field method, similar to the one in \cite{lind}. Let us just note that, as in \cite{qstab}, we have to introduce a set of weight functions to be able to use the structure of our equations even at the level of our last energy estimate. We describe briefly this issue in section \ref{noncomu}. We also use the set of weight functions to treat the interaction with the metric $g_{\mathfrak{b}}$. 
	\item To improve estimate \eqref{skbootb}, we set 
	$$h^{(2)}(\theta,s)\sim 2\int_0^\infty (\partial_q \phi(t=\frac{s}{2},r,\theta))^2rdr,$$
	and choose $b^{(2)}$ to be the solution to the elliptic equation
	$$b^{(2)}+\partial^2_\theta b^{(2)} =\Pi h^{(2)}.$$
	We then check that $b^{(2)}$ satisfies the estimates $H$ (for this purpose, some geometric corrective terms have to be added to the formula given here for $h^{(2)}$), return to the third step to construct initial data with the asymptotic behaviour given by $g_{b^{(2)}}$, and solve the evolution problem in coordinates adapted to $g_{b^{(2)}}$, we note $(g^{(2)},\phi^{(2)})$ the solution. We remark that we can go from one solution to the other by a change of coordinates, which can be controlled thanks to the estimates we have on the metric. Consequently $(g^{(2)},\phi^{(2)})$ satisfy the same improved estimates as $(g,\phi)$ and moreover, we have improved \eqref{skbootb}.
	\item By performing an inverse change of variable in $\bar{C}$ we see that $g$ converges to $g_{\mathfrak{a}}$ in the exterior, which is the behaviour given in Theorem \ref{main}.
\end{itemize}

\subsection{Non commutation of the null decomposition with the wave operator}\label{noncomu}
We have seen in the previous section that the coefficient $g_{\ba L \ba L}$ is expected to have a logarithmic growth in the energy
$$\|w^\frac{1}{2}\partial g_{\ba L\ba L}\|_{L^2}\lesssim \ep^2(1+t)^\rho.$$
 We do not want this behaviour to propagate to the other coefficients of the metric. To this end, we will rely on a decomposition of the type
$$g= g_{\mathfrak{b}} +\Up\left(\frac{r}{t}\right)\frac{g_{\ba L \ba L}}{4}dq^2+\wht g_1.$$
However, since the wave operator does not commute with the null decomposition, we have to control the solution $\wht g_1$ of an equation of the form (see the proof of Corollary \ref{corstr})
$$\Box \wht g_1 = \Up \left(\frac{r}{t}\right) \frac{\bar{\partial} g_{\ba L \ba L}}{r}.$$
 When applying the weighted energy estimate for $\wht g_1$, we obtain
$$\frac{d}{dt}\|w(q)^\frac{1}{2} \partial \wht g_1\|^2_{L^2}
\leq \left\|w(q)^\frac{1}{2}\Up\left(\frac{r}{t}\right) \frac{\bar{\partial} g_{\ba L \ba L}}{r}\right\|_{L^2}\ld{w(q)^\frac{1}{2} \partial \wht g_1}.$$
We estimate
\begin{equation}
\label{estdeh}
\left\|w(q)^\frac{1}{2}\Up\left(\frac{r}{t}\right) \frac{\bar{\partial} g_{\ba L \ba L}}{r}\right\|_{L^2}\lesssim \frac{1}{1+t}\|w(q)^\frac{1}{2}\partial g_{\ba L \ba L}\|_{L^2}\lesssim \frac{\ep^2}{(1+t)^{1-\rho}},
\end{equation}
This yields
$$\frac{d}{dt}\|w(q)^\frac{1}{2} \partial \wht g_1\|_{L^2} \leq 
\frac{\ep^2}{(1+t)^{1-\rho}}.$$
So
$$\|w(q)^\frac{1}{2} \partial \wht g_1\|_{L^2} \leq \ep^2(1+t)^\rho,$$
which is precisely the behaviour we are trying to avoid with this decomposition ! However
we have not been able to exploit all the decay in $t$ in \eqref{estdeh} : we could not exploit the good derivative $\bar \partial$ acting on $g_{\ba L \ba L}$. In order to do so, we will use different weight functions for $\wht g_1$ and for $g_{\ba L \ba L}$. If we set 
$$ w(q)=(1+|q|)^{2\sigma} w_1(q),$$
 and we assume that we have
$$\| w(q)^\frac{1}{2}\partial g_{\ba L \ba L} \|_{L^2}\lesssim \ep^2(1+t)^\rho,$$
then we can estimate
$$\left\|w_1(q)^\frac{1}{2}\Up\left(\frac{r}{t}\right) \frac{\bar{\partial} g_{\ba L \ba L}}{r}\right\|_{L^2}\lesssim \frac{1}{1+t} \left\|w(q)^\frac{1}{2}\Up\left(\frac{r}{t}\right) \frac{\bar{\partial} g_{\ba L \ba L}}{(1+|q|)^{\sigma}}\right\|_{L^2}.$$
We write
$$|\bar{\partial} h|\lesssim \frac{1}{1+s}|Z h|
\lesssim \frac{1}{(1+s)^{\sigma}(1+|q|)^{1-\sigma}}|Z h|
,$$
so we obtain
$$\left\|w_1(q)^\frac{1}{2}\Up\left(\frac{r}{t}\right) \frac{\bar{\partial} g_{\ba L \ba L}}{r}\right\|_{L^2}\lesssim \frac{1}{(1+t)^{1+\sigma}} \left\| w(q)^\frac{1}{2} \frac{Zg_{\ba L\ba L} }{1+|q|}\right\|_{L^2}
\lesssim \frac{1}{(1+t)^{1+\sigma}} \|w(q)^\frac{1}{2} \partial Z g_{\ba L \ba L}\|_{L^2},$$
where we used the weighted Hardy inequality.
Consequently, the energy inequality for $\wht g_1$ yields
$$\frac{d}{dt}\|w_1(q)^\frac{1}{2} \partial \wht g_1\|_{L^2} \lesssim \frac{\ep^2}{(1+t)^{1+\sigma-\rho}},$$
and therefore, if $\sigma>\rho$
$$\|w_1(q)^\frac{1}{2} \partial \wht g_1\|_{L^2} \lesssim \ep^2.$$
Recall that the weighted energy inequality forbids weights of the form $(1+|q|)^\alpha$ with $\alpha>0$ in the region $q<0$. Therefore we are forced to make the following choice in the region $q<0$
$$ w(q)= O(1), \quad w_1(q)= \frac{1}{(1+|q|)^{2\sigma}}.$$
Thus, for $\wht g_1$, the $t^\rho$ loss has been replaced by a loss in $(1+|q|)^{\sigma}$.

\section{The background metric}
In this section, we explain the construction of the metric $g_{\mathfrak{b}}$. This metric should be
\begin{itemize}
	\item isometric to the Minkowski metric in the region $q<R$,
	\item isometric to $g_{\mathfrak{a}}$ in the region $q>R+1$,
	\item not flat in the transition region, but the Ricci tensor must not contain terms which can not be handled.
\end{itemize}
Moreover we will need coordinates in which, in the region $q>R+1$, we have
$(g_{\mathfrak{b}})_{UU} \sim \frac{q}{r}h(\theta,s)$. For this, we will write
$$g_{\mathfrak{b}} = \chi(q)g_{\mathfrak{a}}+(1-\chi(q))m,$$
where $g_{\mathfrak{a}}$ is expressed in appropriate coordinates, described in the following section, and $ \chi$ is an appropriate cut-off function.

\subsection{A change of variable}\label{change}
In this section we describe the coordinate change we will use. Corrective terms have to be added compared to what was explained in the introduction because
\begin{itemize}
	\item the metric coefficients expressed in the new coordinates should have enough decay,
	\item the Ricci tensor in the transition region should also have enough decay.
\end{itemize}
Let $a_0,a_1,a_2 \in \m R$ given by Theorem \ref{thinitial}. They satisfy 
$$|a_0|+|a_1|+|a_2|\leq \ep^2,$$
we will note $a(\theta)=a_0+a_1\cos(\theta)+a_2\sin(\theta)$.
We have at this stage to already state the estimates satisfied by $b$. It will allow us to see which terms can be treated as remainders and simplify the exposition. 
Let $b(\theta,s)$ satisfying the following set of hypothesis
\begin{equation}
\label{intb}\int_{\m S^1} \frac{b(\theta,s)}{1+b(\theta,s)}d\theta=0,
\end{equation}
\begin{align}
\label{estb1}&\|  Z^{N-11} b\|_{H^2(\m S^1)}\leq \frac{\ep^2}{(1+s)^{2}},\\
\label{estb2}&\| Z^{N-1} b\|_{H^2(\m S^1)}\leq \ep^2,\\
\label{estb3}&\| \partial_s  Z^{N-1} b\|_{H^1(\m S^1)}\leq \frac{\ep^2}{(1+s)^{2-\frac{1}{2}\sigma}},\\
\label{estb4}&\| Z^{N} b\|_{H^2(\m S^1)} \leq \ep^2(1+s)^\rho,\\
\label{estb5}&\int_0^S(1+s)\|\partial_s Z^N b\|^2_{H^2(\m S^1)}ds \leq \ep^4(1+S)^{2\rho} ,\\
\label{estb5bis}&\int_0^S(1+s)^{1-4\rho}\|\partial_s Z^N b\|^2_{H^2(\m S^1)}ds \leq \ep^4 ,\\
\label{estb5ter}&\int_0^S(1+s)^{3-\sigma}\|\partial^3_s Z^N b\|^2_{L^2(\m S^1)}ds \leq \ep^4(1+S)^{2\rho} ,
\end{align}
and we can write 
\begin{equation}
\label{decb}\partial_s b = f_1+f_2
\end{equation}
with 
\begin{align}
\label{estf1b}&\|Z^Nf_1\|_{L^2(\m S^1)} \leq \ep^2(1+s)^{\frac{1}{2}\sigma-2},\\
\label{estf2b}&\|Z^Nf_2\|_{L^2(\m S^1)} \leq\ep^2(1+s)^{-\frac{3}{2}+\rho},
\end{align}
\begin{align}
\label{estf1}&\int_0^S (1+s)^4\|\partial_s Z^Nf_1\|_{L^2(\m S^1)}^2ds \leq \ep^4(1+S)^{2\rho},\\
\label{estf11}&\int_0^S (1+s)^3\|\partial_s Z^Nf_1\|_{H^1(\m S^1)}^2ds \leq \ep^4(1+S)^{2\rho},\\
\label{estf1bis1}&\int_0^S (1+s)^2\|Z^Nf_1\|_{H^1(\m S^1)}^2ds \leq \ep^4(1+S)^{2\rho},\\
\label{estf2}& \int_0^S (1+s)^{3-2\rho}\|Z^N f_2\|^2_{H^1(\m S^1)}ds \leq \ep^4(1+S)^{2\rho},\\
\label{estf2bis}& \int_0^S (1+s)^{3}\|\partial_s  Z^N f_2\|^2_{H^1(\m S^1)}ds \leq \ep^4(1+S)^{2\rho}.\\
\end{align}
We note $\q H$ the set of hypothesis \eqref{intb} to \eqref{estf1bis1}.

We begin by constructing a Ricci flat metric in the following way : we start with the Ricci flat metric
$${g_{\mathfrak{a}}}=ds'dq' +(r'+a(\theta')q')^2(d\theta')^2+J(\theta')d\theta'dq'.$$
We perform the change of coordinates
$$s'=(1+b(\theta,s))s-(\partial_\theta b(\theta,s))^2(1+b(\theta,s))^{-1}q,$$
$$q'=(1+b(\theta,s))^{-1}q,$$
$$\theta'= \theta - \frac{q}{r}\frac{\partial_\theta b(\theta,s)}{(1+b(\theta,s))^2}+ f(\theta,s),$$
where $f(\theta,s)$ is such that
$$1+\partial_\theta f(\theta,s)= (1+b(\theta,s))^{-1},$$
and we note
$$r= \frac{1}{2}(s+q), \; t =\frac{1}{2}(s-q).$$
In the following proposition, we estimate the coefficients of ${g_{\mathfrak{a}}}$ in the null frame $L=\partial_t +\partial_r$, $\ba L =\partial_t -\partial_r$, $U=\frac{\partial_\theta}{r}$.
\begin{prp}\label{prpgb}
We can write ${g_{\mathfrak{a}}}= \sigma^0+\sigma^1$ where
\begin{align*}
&\sigma^0(L,\ba L)=-2,\\
&\sigma^0(U, L)= s(1+b)\partial_s f,\\
&\sigma^0(U,U)=1+\frac{q}{r}h(\theta,s),
\end{align*}
where
\begin{equation}
\label{eqh}h(\theta,s)=2a(\theta')(1+b)^{-2}+(1+b)^{-2}-1-2\frac{\partial^2_\theta b}{(1+b)}+(1+b)^{-2}(\partial_\theta b)^2,
\end{equation}
and we have the estimates (we denote by $\partial_\theta^\alpha$ any product of $\alpha$ or less vector fields $\partial_\theta$) :
\begin{equation}
\label{calcg0}|Z^I {g_{\mathfrak{a}}}| \lesssim s|\partial_s Z^I b|+
q|\partial_s \partial_\theta Z^I b|+\frac{q}{s}| \partial^2_\theta Z^I b|,
\end{equation}
\begin{equation}
\label{calcsigma1}|Z^I \sigma^1_{\q T \q T}| \lesssim |\partial_s Z^I b|
+\frac{q}{s}|\partial_s \partial_\theta Z^I b|+\frac{q^2}{s^2}|\partial_\theta Z^I b|.
\end{equation}
Moreover we have
\begin{equation}
\label{calcqg}|\partial_q Z^I{g_{\mathfrak{a}}}|\lesssim \frac{1}{(1+|q|)}|Z^I g_0|,
\end{equation}
\begin{equation}
\label{calcsg}|\partial_s Z^I {g_{\mathfrak{a}}}|\lesssim s|\partial^2_s Z^I b|+|\partial_s Z^I b|
+q|\partial^2_s \partial_\theta Z^I b|+
\frac{q}{s}| \partial_s \partial^2_\theta Z^Ib|
+\frac{q}{s^2}|  \partial^2_\theta Z^Ib|.
\end{equation}
\end{prp}

\begin{proof}
We have
\begin{align*}
ds'=&(1+b(\theta,s))ds+\partial_\theta b(\theta,s)sd\theta
+\partial_s b(\theta,s) s ds
-(\partial_\theta b(\theta,s))^2(1+b(\theta,s))^{- 1}dq\\
&-\partial_\theta\left((\partial_\theta b(\theta,s))^2(1+b(\theta,s))^{- 1}\right)qd\theta-
\partial_s\left((\partial_\theta b(\theta,s))^2(1+b(\theta,s))^{- 1}\right)qds\\
dq'&=(1+b(\theta,s))^{-1}dq -q(1+b(\theta,s))^{-2}\partial_\theta b(\theta,s)d\theta -q(1+b(\theta,s))^{-1}\partial_s b(\theta,s)ds,\\
d\theta'&=\left((1+b(\theta,s))^{-1}-\frac{q}{r}\partial_\theta \left(\frac{\partial_\theta b(\theta,s)}{(1+b(\theta,s))^2}\right)\right)d\theta 
+\partial_s f(\theta,s)ds\\
&+\frac{\partial_\theta b(\theta,s)}{(1+b(\theta,s))^2}\left(\frac{q}{2r^2}ds-\frac{s}{2r^2}dq\right) -\frac{q}{r}\partial_s \left(\frac{\partial_\theta b(\theta,s)}{(1+b(\theta,s))^2}\right)ds
\end{align*}
and also
\begin{align*}r'= \frac{1}{2}(s'+q')=&\frac{1}{2}\left((1+b(\theta,s))s-\frac{1}{(1+b(\theta,s))}(\partial_\theta b(\theta,s))^2q+(1+b(\theta,s))^{-1}q\right)\\
=&(1+b)r-\frac{1}{2(1+b(\theta,s))}(\partial_\theta b)^2q
+\frac{1}{2}((1+b)^{-1}-(1+b))q
.
\end{align*}
We note that $\frac{b}{s}$ has at least the same decay in $s$ and is more regular than $s\partial_s b$, so if we are able to estimate the second, we are able to estimate the first. We can also neglect quadratic terms with similar or better behaviour than a term which is already present. Consequently we write
\begin{align*}
ds'&=(1+b+O\left(s\partial_s b\right)+O\left(q\partial_s\partial_\theta b\right))ds-(1+b)^{-1}(\partial_\theta b)^2dq
+\left(s\partial_\theta b+O(\ep^2q\partial^2_\theta b)\right)d\theta,\\
dq'&=O(q\partial_s b)ds+(1+b)^{-1}dq -q(1+b)^{-2}\partial_\theta b d\theta,
\end{align*}
and consequently
\begin{align*}
ds'dq'=&\left(1+O\left(s\partial_s b\right)+O\left(q\partial_s \partial_\theta b\right)\right)dsdq
-(\partial_\theta b)^2(1+b)^{-2}dq^2+O\left(q\partial_s\partial_\theta b\right)ds^2\\
&+\left(-qs (1+b)^{-2}(\partial_\theta b)^2+O(\ep^2q^2\partial^2_\theta b)\right)d\theta^2\\
&+(-q\partial_\theta b(1+b)^{-1}+O(\ep qs\partial_s \partial_\theta b))dsd\theta+\left(s(1+b)^{-1}(\partial_\theta b) + O(\ep^2q\partial^2_\theta b)\right)dqd\theta.
\end{align*}
We also estimate
\begin{align*}&(1+b)^2r^2(d\theta')^2\\
=&\left(r^2-2qr(1+b)\partial_\theta\left(\frac{\partial_\theta b}{(1+b)^2}\right)+O(\ep q\partial^2_\theta b)\right)d\theta^2
+(q\partial_\theta b(1+b)^{-1}+(1+b)r^2\partial_s f)dsd\theta\\
&-s\partial_\theta b(1+b)^{-1}dqd\theta
+O(r^2(\partial_s f)^2)ds^2+\left(((1+b)^{-2}(\partial_\theta b)^2+O\left(\frac{\ep q}{s} \partial_\theta b\right)\right)dq^2
\end{align*}
and
$$(r'+a(\theta')q')^2-(1+b)^2r^2
=rq(2a(\theta+f)-(\partial_\theta b)^2+(1-(1+b)^2))+O(q^2\ep^2\partial_\theta b).$$
Consequently, we can estimate the coefficients of the metric ${g_{\mathfrak{a}}}$ in coordinates $s,q,\theta$
\begin{align*}
(g_{\mathfrak{a}})_{sq}&= 1+O(s\partial_{s}b)+O\left(q\partial_s \partial_\theta b\right),\\
(g_{\mathfrak{a}})_{ss}&=O(q\partial_{s}\partial_\theta b),\\
(g_{\mathfrak{a}})_{qq}&=O\left(\frac{q}{s}(\partial_\theta b)^2\right),\\
(g_{\mathfrak{a}})_{s\theta}&=r^2(1+b)\partial_s f +O\left(sq\partial_s \partial_\theta b\right),\\
(g_{\mathfrak{a}})_{q\theta}&=J(\theta')(1+b)^{-2}+O(\ep^2q\partial_\theta^2b),\\
(g_{\mathfrak{a}})_{\theta \theta}&= r^2+qr\left(2a(\theta')(1+b)^{-2}+(1+b)^{-2}-1-2(1+b)\partial_\theta \left(\frac{\partial_\theta b}{(1+b)^2}\right)-(1+b)^{-2}(\partial_\theta b)^2-2(1+b) ^{-2}(\partial_\theta b)^2\right)
\\&+O(\ep q^2\partial^2_\theta b)\\
&=r^2+qr\left(2a(\theta')(1+b)^{-2}+(1+b)^{-2}-1-2\frac{\partial^2_\theta b}{(1+b)}+(1+b)^{-2}(\partial_\theta b)^2\right)+O(\ep q^2\partial^2_\theta b)
\end{align*}
To obtain the estimates for $Z^I g_{\mathfrak{a}}$, we note that we have the following expression for the commutator of $Z$ with $\partial_s$
$$[S,\partial_s]=\partial_s,\quad [\Omega_{0,1},\partial_s]=\cos(\theta)\partial_s -\frac{q}{2r^2}\sin(\theta)\partial_\theta,
\quad [\Omega_{0,2},\partial_s]=\sin(\theta)\partial_s +\frac{q}{2r^2}\cos(\theta)\partial_\theta.$$
We see that if we isolate the contribution $r^2(1+b)\partial_s f$ in $(g_{\mathfrak{a}})_{s\theta}$ and $rqh$ in $(g_{\mathfrak{a}})_{\theta \theta}$ we obtain the desired estimates for $\sigma^1$.
\end{proof}

We call $g_{\mathfrak{b}}$ the metric whose coefficients in the coordinates $s,q,\theta$ are given by the coefficients above, where the terms involving $b$ or $f$ are multiplied by a cut-off function $\chi(q)$, more precisely
$$g_{\mathfrak{b}} = \chi(q)g_{\mathfrak{a}}+(1-\chi(q))m,$$
where $m$ is the Minkowski metric
$$m=dsdq+r^2d\theta^2.$$

From now on in the paper, $\chi$ will be a cut-off function such that
$$\chi(q)=1, \text{ for }q\geq R+1, \quad \chi(q)=0, \text{ for } q\leq R+\frac{1}{2}.$$
In particular, when $\chi=0$, $g_{\mathfrak{b}}$ is isometric to Minkowski metric and when $\chi=1$, $g_{\mathfrak{b}}$ is isometric to $g_{\mathfrak{a}}$.

\begin{cor}\label{corestb}
	We have for $I\leq N-2$
	\begin{align}
	\label{estsigma0}|Z^I g_{\mathfrak{b}}|&\lesssim \frac{\ep^2}{(1+s)^\frac{3}{4}}+ \frac{\ep^2(1+|q|)}{(1+s)},\\
	\label{estsigma1}|Z^I \sigma^1_{\q T \q T}|&\lesssim \frac{\ep^2q}{(1+s)^\frac{7}{4}}  +\frac{\ep^2q}{(1+s)^2}.
	\end{align}
	For $I\leq N-11$ we have
	\begin{equation}
	\label{estg0}|Z^I g_{\mathfrak{b}}|\lesssim \frac{\ep^2(1+|q|)}{(1+s)}.
	\end{equation}
\end{cor}
\begin{proof}
	We have thanks to \eqref{estb3} and the Sobolev embedding $H^1(\m S^1)\subset L^\infty$
	$$|\partial_s \partial_\theta Z^I b|\lesssim \|\partial_s  Z^{I+1}b\|_{H^1(\m S^1)} \lesssim\frac{\ep^2}{(1+s)^{2-\frac{\sigma}{2}}}\lesssim \frac{\ep^2}{(1+s)^{2-\frac{1}{4}}}, \text{ for } I \leq N-2$$
		and
		$$|\partial^2_\theta Z^I b|\lesssim \|Z^{I+1}b\|_{H^2}\lesssim \ep^2, \text{ for }I\leq N-2.$$
		Consequently, 
		$$|Z^{N-2} g_{\mathfrak{b}}|\lesssim \frac{\ep^2}{(1+s)^{\frac{3}{4}}}+ \frac{\ep^2(1+|q|)}{(1+s)},$$
		$$|Z^{N-2} \sigma^1_{\q T \q T}|\lesssim  \frac{\ep^2}{(1+s)^{\frac{7}{4}}}+ \frac{\ep^2(1+|q|)^2}{(1+s)^2}.$$
	Thanks to \eqref{estb1} we have, for $I\leq N-11$
	$$|\partial_s \partial_\theta Z^I b|\lesssim \|\partial_s  Z^{I}b\|_{H^2(\m S^1)} \lesssim \frac{\ep^2}{(1+s)^{2}}$$
	and consequently
	$$|Z^{N-11}g_{\mathfrak{b}}|\lesssim \frac{\ep^2(1+|q|)}{(1+s)},$$
	$$|Z^{N-11}\sigma_{TT}|\lesssim \frac{\ep^2(1+|q|)^2}{(1+s)^2},$$
	which concludes the proof of Corollary \ref{corestb}.
\end{proof}

\begin{cor}\label{estih}
	We have the estimates 
	$$\|Z^{N-1} h\|_{L^2(\m S^1)} \lesssim \ep^2,$$
	$$\|Z^N h \|_{L^2(\m S^1)} \lesssim \ep^2(1+s)^\rho,$$
	$$\int_0^S(1+s)\|\partial_s Z^N h\|^2_{L^2(\m S^1)}ds \leq \ep^4(1+S)^{2\rho} .$$
\end{cor}
\begin{proof}
	It is a direct consequence of the definition of $h$ \eqref{eqh} and the assumptions for $b$, \eqref{estb2} and \eqref{estb4}.
\end{proof}
\subsection{Calculation of the Ricci tensor}
We now turn to the estimate for the Ricci coefficients of $g_{\mathfrak{b}}$.
\begin{prp}
\label{riccigb}
We can write $R=R^0+R^1$, where
$$R^0_{U\ba L}= -\partial^2_q (\chi(q))\sigma^0_{UL}, \quad R^0_{\ba L \ba L}=4\frac{\partial_q^2(q\chi(q))}{r}h(\theta,s),$$
and
 $$|Z^I R^1|\lesssim \ch_{R\leq q\leq R+1}\left(|s\partial^2_s Z^I b|+
|\partial_s^2 \partial_\theta Z^I b|+|\partial_s \partial_\theta Z^I b|
+\frac{1}{s^2}|\partial^2_\theta Z^I b|\right),$$
 $$|Z^I R^1_{LL}|\lesssim \ch_{R\leq q\leq R+1}\left(
|\partial_s^2 \partial_\theta Z^I b|+|\partial_s \partial_\theta Z^I b|
+\frac{1}{s^2}|\partial^2_\theta Z^I b|\right).$$
\end{prp}

\begin{proof}
	When $\chi=0$, $g_{\mathfrak{b}}$ is isometric to Minkowski metric, and when $\chi=1$, $g_{\mathfrak{b}}$ is isometric to $g_{\mathfrak{a}}$ which is Ricci flat, so the Ricci coefficients are
 non zero only in the region where $\chi$ is non constant, that is to say near the null cone. Since $g_{\mathfrak{a}}-m=O\left(\frac{\ep^2}{r}\right)$ in the region where $\chi$ is non constant, all the quadratic terms are a $O\left(\frac{\ep^4}{r^2}\ch_{R\leq q \leq R+1}\right)$ and are easier to estimate than the non quadratic terms. The non quadratic terms have to involve a term containing $\chi'(q)$.
We calculate in $s,q,\theta$ coordinates. We note that for Minkowski metric the non zero Christoffel symbols are
$$\Gamma^\theta_{\theta s}=\Gamma^\theta_{\theta q}=\frac{1}{2r}, \quad \Gamma^s_{\theta \theta}=\Gamma^q_{\theta \theta}=-\frac{r}{2}.$$
We can consequently neglect the terms involving one of these symbols : they give contributions which are $O\left(\frac{\chi'(q)}{r}(g_{\mathfrak{a}}-m)\right)$. We calculate

\begin{align*}
R_{qq}=&\partial_q \Gamma^q_{qq}+\partial_s \Gamma^s_{qq}+\partial_\theta \Gamma^\theta_{qq}
-\partial_q \Gamma^q_{qq}-\partial_q \Gamma^s_{qs}-\partial_q \Gamma^\theta_{q\theta}+O\left(\frac{\chi'(q)}{r}(g_{\mathfrak{a}}-m)\right)+0\left(\frac{\ep^4\chi'(q)}{r^2}\right)\\
=&\partial_s\partial_q g_{qq}+g^{\theta \theta}\partial_\theta \partial_q g_{q\theta}
-\frac{1}{2}\partial_q\partial_s g_{sq}-\frac{1}{2}g^{\theta\theta}\partial^2_q g_{\theta \theta}
+O\left(\frac{\chi'(q)}{r}(g_{\mathfrak{a}}-m)\right)+0\left(\frac{\ep^4\chi'(q)}{r^2}\right)\\
=&\frac{1}{r}\partial^2_q(q\chi(q))h(\theta,s)+O\left(\frac{\chi'(q)}{r}(g_{\mathfrak{a}}-m)\right)+0\left(\frac{\ep^4\chi'(q)}{r^2}\right),
\end{align*}
\begin{align*}
R_{sq}=&\partial_q \Gamma^q_{sq}+\partial_s \Gamma^s_{sq}+\partial_\theta \Gamma^\theta_{sq}
-\partial_s \Gamma^q_{qq}-\partial_s \Gamma^s_{qs}-\partial_s \Gamma^\theta_{q\theta}
+O\left(\frac{\chi'(q)}{r}(g_{\mathfrak{a}}-m)\right)+0\left(\frac{\ep^4\chi'(q)}{r^2}\right)\\
=&\partial_q\partial_q g_{ss}+\frac{1}{2}g^{\theta \theta}\partial_\theta \partial_q g_{s\theta}
-2\partial_s\partial_q g_{sq}-g^{\theta\theta}\partial_s\partial_q g_{\theta \theta}
+O\left(\frac{\chi'(q)}{r}(g_{\mathfrak{a}}-m)\right)+0\left(\frac{\ep^4\chi'(q)}{r^2}\right)\\
=&0(s\chi'(q)\partial^2_s b)+0(q\chi'(q)\partial^2_s\partial_\theta b)+O\left(\frac{\chi'(q)}{r}(g_{\mathfrak{a}}-m)\right)+0\left(\frac{\ep^4\chi'(q)}{r^2}\right),
\end{align*}
\begin{align*}
R_{ss}=&\partial_q \Gamma^q_{ss}+\partial_s \Gamma^s_{ss}+\partial_\theta \Gamma^\theta_{ss}
-\partial_s \Gamma^q_{sq}-\partial_s \Gamma^s_{ss}-\partial_s \Gamma^\theta_{s\theta}+
O\left(\frac{\chi'(q)}{r}(g_{\mathfrak{a}}-m)\right)+0\left(\frac{\ep^4\chi'(q)}{r^2}\right)\\
=&\partial_q\partial_s g_{ss}
-\partial_s\partial_q g_{ss}+O\left(\frac{\chi'(q)}{r}(g_{\mathfrak{a}}-m)\right)+0\left(\frac{\ep^4\chi'(q)}{r^2}\right)\\
=&O\left(\frac{\chi'(q)}{r}(g_{\mathfrak{a}}-m)\right)+0\left(\frac{\ep^4\chi'(q)}{r^2}\right),
\end{align*}
\begin{align*}
R_{s\theta}=&\partial_q \Gamma^q_{s\theta}+\partial_s \Gamma^s_{s\theta}+\partial_\theta \Gamma^\theta_{s\theta}
-\partial_s \Gamma^q_{\theta q}-\partial_s \Gamma^s_{\theta s}-\partial_s \Gamma^\theta_{\theta\theta}+O\left(\chi'(q)(g_{\mathfrak{a}}-m)\right)+0\left(\frac{\ep^4\chi'(q)}{r}\right)\\
=&\partial_q\partial_\theta g_{ss}-\partial_s\partial_q g_{s\theta}+
O\left(\chi'(q)(g_{\mathfrak{a}}-m)\right)+0\left(\frac{\ep^4\chi'(q)}{r}\right)\\
=&0\left(s^2\chi'(q)\partial^2_s b\right)+0\left(s\chi'(q)\partial^2_s\partial_\theta b\right)+O\left(\chi'(q)(g_{\mathfrak{a}}-m)\right)+0\left(\frac{\ep^4\chi'(q)}{r}\right),
\end{align*}
\begin{align*}
R_{q\theta}=&\partial_q \Gamma^q_{q\theta}+\partial_s \Gamma^s_{q\theta}+\partial_\theta \Gamma^\theta_{q\theta}
-\partial_q \Gamma^q_{\theta q}-\partial_q \Gamma^s_{\theta s}-\partial_q \Gamma^\theta_{\theta\theta}+O\left(\chi'(q)(g_{\mathfrak{a}}-m)\right)+0\left(\frac{\ep^4\chi'(q)}{r}\right)\\
=&g^{\theta \theta}\partial_\theta\partial_q g_{\theta \theta}-\partial_q(\partial_\theta g_{qs}+\partial_s g_{\theta q}-\partial_q g_{s\theta})-\frac{1}{2}g^{\theta \theta}\partial_q \partial_\theta g_{\theta \theta} 
+O\left(\chi'(q)(g_{\mathfrak{a}}-m)\right)+0\left(\frac{\ep^4\chi'(q)}{r}\right)\\
=&\partial^2_q g_{s\theta} +O\left(s^2\chi'(q)\partial^2_s b\right)+O\left(\chi'(q)(g_{\mathfrak{a}}-m)\right)+0\left(\frac{\ep^4\chi'(q)}{r}\right),
\end{align*}
\begin{align*}
R_{\theta\theta}=&\partial_q \Gamma^q_{\theta\theta}+\partial_s \Gamma^s_{\theta\theta}+\partial_\theta \Gamma^\theta_{\theta\theta}
-\partial_\theta \Gamma^q_{\theta q}-\partial_\theta \Gamma^s_{\theta s}-\partial_\theta \Gamma^\theta_{\theta\theta}+O\left(r\chi'(q)(g_{\mathfrak{a}}-m)\right)+0\left(\ep^4\chi'(q)\right)\\
=&\partial_q(2\partial_\theta g_{s\theta}-\partial_s g_{\theta \theta})-\partial_s\partial_q g_{\theta \theta}-\partial_\theta \partial_q g_{\theta s}+ \partial_\theta \partial_q g_{\theta s}
+O\left(r\chi'(q)(g_{\mathfrak{a}}-m)\right)+0\left(\ep^4\chi'(q)\right)\\
=&O\left(s\chi'(q)\partial_s\partial^2_\theta b\right)+O\left(r\chi'(q)(g_{\mathfrak{a}}-m)\right)+0\left(\ep^4\chi'(q)\right).
\end{align*}
We note that all this terms give contributions which are $O\left(\frac{\ep^2\chi'(q)}{r^2}\right)$
except the contribution of $\frac{1}{r}\partial^2_q(q\chi(q))h(\theta,s)$ in $R_{qq}$ and the contribution of
$\partial^2_q g_{s\theta}$ (more precisely the term $\frac{1}{2}\partial_q^2 \chi(q)\sigma^0_{s\theta}$) in $R_{q\theta}$.
\end{proof}

\subsection{The generalized wave coordinates}\label{secgen}

We will look for solutions of the form $g=g_{\mathfrak{b}} + \wht g$.
We will work in generalized wave coordinates, chosen as follow. First we need them to be compatible with our choice of background metric. We will note
$$F^\alpha_b = \Box_{g_{\mathfrak{b}}}x^\alpha.$$
Next we need to get rid of the artificial bad term $\sigma^0_{UL} \partial^2_q \chi(q)$ in $R_{U\ba L}$. If we look at \eqref{calcricci}, we see that this contribution can only come from $\frac{1}{2}g_{UU}\partial_{\ba L}F^U_b$, and consequently from a term of the form $-\sigma^0_{UL}\chi'(q)$ in $U_\alpha F^\alpha$. Consequently, we will modify the wave coordinate condition to remove this term.
We also want to take a coordinate choice in which the quadratic nonlinearities in our system satisfy the null condition (see Section \ref{sketch}). For these two reasons, we introduce the vector-valued function $G$ such that 
\begin{align}
\label{gu}U_\alpha G^\alpha&=\sigma^0_{UL}\chi'(q),\\
\label{gl}L_\alpha G^\alpha& =  \frac{1}{r}\Up\left(\frac{r}{t}\right)\int_{\infty}^r \left(2(\partial_q \phi)^2r -h(\theta,s=2t)\chi'(q)\right)dr,\\
\label{gbal}\ba L_\alpha G^\alpha &= 0
\end{align}
Finally, we will work in generalized wave coordinates such that
\begin{equation}
\label{gwl}H^\alpha = g^{\lambda \beta}\Gamma^\alpha_{\lambda \beta}=F^\alpha_b+G^\alpha +\wht G^\alpha,
\end{equation}
where $\wht G^\alpha$ is defined in the following manner : 
\begin{df}\label{defgtild} $\wht G^\alpha$ is the sum of all the terms in $g^{\lambda \beta}\Gamma^\alpha_{\lambda \beta}$, calculated for $g=g_{\mathfrak{b}} + \wht g$, which are of the form $\wht g \partial^l_s\partial^k_\theta b$, where
$l+k-2\geq 1$ or $l\geq 2$.
\end{df} Proposition \ref{prpcross} is the reason why we add this small modulation to the gauge condition.

In generalized wave coordinates, the expression \eqref{calcricci}  allows us to write the system \eqref{sys} into the form

\begin{equation}\label{gw}
 \left\{ \begin{array}{l}
          \Box_g \phi = 0\\
\Box_g g_{\mu \nu}= -4\partial_\mu \phi \partial_\nu \phi +P_{\mu \nu}(\partial g, \partial g)
+ g_{\mu \rho}\partial_\nu H^\rho + g_{\nu \rho}\partial_\mu H^\rho ,
         \end{array}
\right.
\end{equation}
where
\begin{equation}\label{quadr}
\begin{split}
 P_{\mu \nu}(g)(\partial g, \partial g)
=&\frac{1}{2}g^{\alpha \rho}g^{\beta \sigma}\left(\partial_\mu g_{\rho \sigma}\partial_\alpha g_{\beta \nu}+\partial_\nu g_{\rho \sigma}\partial_\alpha g_{\beta \mu}
-\partial_\beta g_{\mu \rho}\partial_\alpha g_{\nu \sigma}-\frac{1}{2}\partial_\mu g_{\alpha \beta} \partial_\nu g_{\rho \sigma}\right)\\
&+\frac{1}{2}g^{\alpha \beta}g^{\lambda \rho}\partial_\alpha g_{\nu \rho}\partial_\beta g_{\mu \rho}.
\end{split}
\end{equation}

\begin{rk}
 In generalized wave coordinates, the wave operator can be expressed as
$$\Box_g = g^{\alpha \rho}\partial_\alpha \partial_\rho - H_b^\rho\partial_\rho.$$
\end{rk}
The expression \eqref{calcricci} yields also
\begin{equation}
\label{calcrb}
(R_b)_{\mu\nu}=-\frac{1}{2}\Box_{g_{\mathfrak{b}}} (g_{\mathfrak{b}})_{\mu \nu}+\frac{1}{2}P_{\mu \nu}(g_{\mathfrak{b}})(\partial g_{\mathfrak{b}},\partial g_{\mathfrak{b}})+\frac{1}{2}\left((g_{\mathfrak{b}})_{\mu \rho}\partial_\nu F_b^\rho + (g_{\mathfrak{b}})_{\mu \rho}\partial_\mu F^\rho_b\right).
\end{equation}
Therefore, subtracting twice the equation \eqref{calcrb} to the second equation of \eqref{gw} we obtain
\begin{equation}\label{s2}
 \left\{ \begin{array}{l}
          \Box_g\phi = 0,\\
\Box_g\wht g_{\mu \nu} =-4 \partial_\mu \phi\partial_\nu \phi
+2(R_b)_{\mu \nu} + g_{\mu \rho}\partial_\nu G_b^\rho + g_{\mu \rho}\partial_\mu G^\rho_b
+ P_{\mu \nu}(g)(\partial \wht g, \partial \wht g) + \wht P_{\mu \nu} (\wht g, g_{\mathfrak{b}}),
         \end{array}
\right.
\end{equation}
where $P_{\mu \nu}(g)(\partial \wht g, \partial \wht g) $ is defined by \eqref{quadr} and
\begin{equation}
\label{ptilde}
\begin{split}
\wht P_{\mu \nu} (\wht g, g_{\mathfrak{b}})=&
\left(g_{\mathfrak{b}}^{\alpha\beta}-g^{\alpha \beta}\right)\partial_\alpha \partial_\beta (g_{\mathfrak{b}})_{\mu \nu}
+(G^\rho +\wht G^\rho)\partial_\rho (g_{\mathfrak{b}})_{\mu \nu}\\
&+P_{\mu \nu}(g)(\partial g,\partial g)-P_{\mu \nu}(g)(\partial \wht g,\partial \wht g)-P_{\mu \nu}(g_{\mathfrak{b}})(\partial g_{\mathfrak{b}},\partial g_{\mathfrak{b}})\\
&+\wht g_{\mu \rho}\partial_\nu (F^\rho) + \wht g_{\nu \rho}\partial_\mu (F^\rho)
+ g_{\mu \rho}\partial_\nu \wht G^\rho + g_{\nu \rho}\partial_\mu \wht G^\rho.
\end{split}
\end{equation}
Let us note that $\wht P_{\mu \nu} (\wht g, g_{\mathfrak{b}})$ contains only crossed terms between $g_{\mathfrak{b}}$ and $\wht g$. 

\begin{prp}
	\label{prpcross} $\wht P_{\mu \nu} (\wht g, g_{\mathfrak{b}})$ does not contain any term involving $\partial_s^3\partial_\theta b$ nor $\partial_s^3 b$ nor $\partial^4_\theta b$.
\end{prp}

\begin{proof}
	By looking at the decomposition of $g_{\mathfrak{b}}$ we observe that the terms in $\wht P_{\mu \nu} (\wht g, g_{\mathfrak{b}})$ which involve $\partial_s^3\partial_\theta b$ or $\partial_s^3 b$ or $\partial^4_\theta b$, in fact involve $\partial_s^2 (g_{\mathfrak{b}})_{s-}$ or $\partial_\theta^2 (g_{\mathfrak{b}})_{\theta \theta}$, where ${-}$ stands for any index.
	The terms involving two derivatives of $g_{\mathfrak{b}}$ in $\wht P_{\mu \nu} (\wht g, g_{\mathfrak{b}})$ are the same than in
	$$-\Box_g (g_{\mathfrak{b}})_{\mu \nu}+g_{\mu \rho}\partial_{\nu}(F_b^\rho + \wht G^\rho)+g_{\nu \rho}\partial_{\mu}(F_b^\rho + \wht G^\rho).$$
	Our choice of  $\wht G^\rho$ is done precisely in order for the terms involving $\partial_s^3\partial_\theta b$ or $\partial_s^3 b$ or $\partial^4_\theta b$ in the above expression to be the same than in the Ricci tensor of $g_{\mathfrak{g}}+\wht g$, so the same than in the expression
\begin{equation} \label{expr}\partial_\alpha \Gamma^\alpha_{\mu \nu}-\partial_\mu \Gamma^\alpha_{\nu \alpha}.
\end{equation}
	We look for the terms involving $\partial_s^2 (g_{\mathfrak{b}})_{s-}$. When $\mu$ and $\nu$ are equal to $q$ or $\theta$ these terms are not present. If $\mu=s$, the terms involving $\partial_s^2 (g_{\mathfrak{b}})_{s-}$ in \eqref{expr} are the same than in
	$$ \frac{1}{2}g^{s\rho}\partial_s (\partial_s g_{\nu \rho}+\partial_\nu g_{s \rho}-\partial_\rho g_{s\nu})
	-\frac{1}{2}g^{\alpha \rho}\partial_s \partial_\nu g_{\alpha \rho}.$$
If $\nu$ is equal to $q$ or $\theta$, we directly see that the terms involving $\partial_s^2 (g_{\mathfrak{b}})_{s-}$
compensate. If $\nu=s$ these terms are the same than in
	$$ \frac{1}{2}g^{s\rho}\partial_s (\partial_s g_{s \rho}+\partial_s g_{s \rho})-\frac{1}{2}g^{ss}\partial_s \partial_s g_{ss}
	-\frac{1}{2}g^{\alpha \rho}\partial_s \partial_s g_{\alpha \rho},$$
	so again they compensate. The case $\partial_\theta^2 (g_{\mathfrak{b}})_{\theta \theta}$ is similar, so this concludes the proof of Proposition \ref{prpcross}.
\end{proof}

\subsection{Second version of Theorem \ref{main}}\label{second}
We give here a more precise version of Theorem \ref{main}.
\begin{thm}\label{main2}Let $0<\ep<1$.
	Let $\frac{1}{2}<\delta<1$ and $N \geq 25$.
	Let $(\phi_0,\phi_1) \in H^{N+2}(\m R^2)\times H^{N+1}(\m R^2)$ compactly supported in $B(0,R)$.  
	We assume
	$$\|\phi_0\|_{ H^{N+2}}+\|\phi_1\|_{H^{N+1}}
	\leq \ep.$$
	Let $\ep\ll\rho\ll \sigma \ll \delta$ such that $\delta-2\sigma>\frac{1}{2}$.
	If $\ep$ is small enough, there exists a global solution $(g,\phi)$ of \eqref{sys}. More precisely there exists $b(\theta,s)$ satisfying the set of hypothesis $\q H$, and a set of generalized wave coordinates $(t,x^1,x^2)$ defined by \eqref{gwl} such that we can write in these coordinates $g=g_{\mathfrak{b}} + \wht g$, where $g_{\mathfrak{b}}$ is defined in Section \ref{change}.
	We have the estimates for $\wht g,\phi$ 
	$$|\wht g|\lesssim \frac{\ep}{(1+s)^{\frac{1}{2}-\rho}}, \quad |\phi| \lesssim \frac{\ep}{(1+s)^{\frac{1}{2}}(1+|q|)^{\frac{1}{2}-4\rho}},$$
	$$\| \partial Z^{N} \phi\|_{L^2} +\| \partial^2 Z^N \phi\|_{L^2} + \|w_2 \partial Z^N \wht g \|_{L^2} \lesssim \ep (1+t)^{2\rho},$$
	where
$$\left\{\begin{array}{l}
w_2(q)=(1+|q|)^{2+2\delta-4\sigma},\text{ for } q>0\\
w_2(q)=(1+|q|)^{-1-2\sigma}, \text{ for } q<0.
\end{array}\right.$$
Moreover, for $h(\theta,s)$ defined by \eqref{eqh}, we have
$$\int_0^\infty 2(\partial_q \phi(t,r,\theta))^2rdr - h(\theta,s=2t) = O\left(\frac{\ep^2}{\sqrt{1+t}}\right).$$
\end{thm}

\section{Bootstrap assumptions and proof of Theorem \ref{main}}
\subsection{Bootstrap assumptions}
Let $\rho,\sigma,\mu$ such that $\ep \ll \rho \ll \sigma  \ll \delta,$ and 
 $$\delta-2\sigma>\frac{1}{2}, \quad \sigma \leq \frac{1}{4}, \quad \mu \leq \frac{1}{4}$$
The initial data $(\phi_0,\phi_1)$ for $\phi$ are given in $H^{N+2}\times H^{N+1}(\m R^2)$, compactly supported in $B(0,R)$.
Let $b$ which satisfy the set of estimates $\q H$. We construct a metric $g_{\mathfrak{b}}$ as described in Section \ref{change}.
There exists initial data for $g$ (see Appendix \ref{reguini}), such that we can write at $t=0$
$$g= g_{\mathfrak{b}} + \wht h_0, \quad \partial_t g = \partial_t g_{\mathfrak{b}} + \wht h_1$$
with $(\wht h_0,\wht h_1) \in H^{N+1}_\delta \times H^N_{\delta+1}$ and
\begin{itemize}
\item the constraint equations are satisfied at $t=0$,
\item the generalized wave coordinate condition is satisfied at $t=0$.
\end{itemize}
We consider a time $T$ such that there exists a solution $g=g_{\mathfrak{b}}+\wht g, \phi$ on $[0,T]$ of \eqref{s2}. We assume that on $[0,T]$ the following estimates hold

\paragraph{$L^\infty$-based bootstrap assumptions}
For $I\leq N-9$ we assume
\begin{align}
\label{bootphi1}|Z^I \phi| &\leq \frac{2C_0\ep}{\sqrt{1+s}(1+|q|)^{\frac{1}{2}-4\rho}},\\  
\label{bootg1}|Z^I \wht g| &\leq \frac{2C_0\ep}{(1+s)^{\frac{1}{2}-\rho}},
\end{align}
where here and in the following, $C_0$ is a constant depending on $\rho,\sigma,\mu, \delta, N$ such that the inequalities
are satisfied at $t=0$ with $2C_0$ replaced by $C_0$.
For $I\leq N-7$ we assume
\begin{align}
\label{bootphi2}|Z^I \phi| &\leq \frac{2C_0\ep}{(1+s)^{\frac{1}{2}-2\rho}},\\ 
\label{bootg2}|Z^I \wht g| &\leq \frac{2C_0\ep}{(1+s)^{\frac{1}{2}-3\rho}}.
\end{align}

\paragraph{$L^2$-based bootstrap assumptions}

We introduce three weight functions
$$\left\{\begin{array}{l}
w(q)=(1+|q|)^{2+2\delta},\text{ for } q>0\\
w(q)=1+(1+|q|)^{-2\mu}, \text{ for } q<0,
\end{array}\right.$$
$$\left\{\begin{array}{l}
w_1(q)=(1+|q|)^{2+2\delta-2\sigma},\text{ for } q>0\\
w_1(q)=(1+|q|)^{-2\sigma}, \text{ for } q<0,
\end{array}\right.$$
$$\left\{\begin{array}{l}
w_2(q)=(1+|q|)^{2+2\delta-4\sigma},\text{ for } q>0\\
w_2(q)=(1+|q|)^{-1-2\sigma}, \text{ for } q<0.
\end{array}\right.$$
We introduce the following decompositions of the metric
\begin{align}
\label{dec1}g&=g_{\mathfrak{b}}+\wht g,\\
\label{dec2}g&=g_{\mathfrak{b}}+4\Up\left(\frac{r}{t}\right) k dq^2 +\wht g_1,
\end{align}
where $ k$ satisfies
\begin{equation}
\label{eqg2}
\Box_g k = Q_{\ba L \ba L }=\partial_q g_{UU}\partial_q \wht g_{\ba L  L}
+\wht g_{L \ba L }\partial_q G^L.
\end{equation}
We introduce the second decomposition to exploit the weak null structure for cubic terms.
Our $L^2$-based bootstrap assumptions are the following.

Estimates for $\phi$ :
\begin{align}
\label{phiN}&\|w^\frac{1}{2}\partial Z^N \phi\|_{L^2}\leq 2C_0 \ep (1+t)^\rho,\\
\label{dphiN}&\|w^\frac{1}{2}\partial^2 Z^N \phi\|_{L^2}\leq 2C_0 \ep (1+t)^\rho,\\
\label{SphiN}&\|w^\frac{1}{2}\partial (S Z^N\phi-s\partial_q\phi Z^N g_{LL}) \|_{L^2}\leq 2C_0 \ep (1+t)^\rho,\\
\label{ZphiN}&\|w^\frac{1}{2}\partial Z^{N+1} \phi\|_{L^2}\leq 2C_0 \ep (1+t)^{\frac{1}{2}+\rho},\\
\label{phiN1}&\|w^\frac{1}{2}\partial Z^{N-1} \phi\|_{L^2}\leq 2C_0 \ep,
\end{align}
Integrated estimates for $\bar{\partial}\phi$ :
\begin{align}
\label{iphiN}& \int_0^t \int w'(q)(\bar{\partial}Z^N \phi)^2 dxd\tau  \leq 2C_0\ep^2(1+t)^{2\rho} \\
\label{idphiN}& \int_0^t \int w'(q)(\bar{\partial} \partial Z^N \phi)^2 dxd\tau  \leq 2C_0\ep^2(1+t)^{2\rho} \\
\label{iSphiN}& \int_0^t \int w'(q)(\bar{\partial}( S Z^N\phi-s\partial_q\phi Z^N g_{LL}))^2 dxd\tau  \leq 2C_0\ep^2(1+t)^{2\rho} \\
\label{iZphiN}& \int_0^t \int w'(q)(1+\tau)^{-1}(\bar{\partial} Z^{N+1} \phi)^2 dxd\tau  \leq 2C_0\ep^2(1+t)^{2\rho} 
\end{align}
Estimates for $\wht g$ :
\begin{align}
\label{gN}&\|w_2^\frac{1}{2}\partial Z^N \wht g_{1}\|_{L^2}\leq 2C_0 \ep(1+t)^\rho ,\\
\label{g2N}&\|w_1^\frac{1}{2}\partial Z^N \wht g\|_{L^2}\leq 2C_0 \ep (1+t)^{2\rho},\\
\label{g2N2}&\|w^\frac{1}{2}\partial Z^{N-3} \wht g\|_{L^2}\leq 2C_0 \ep (1+t)^{\rho},\\
\label{gN2}&\|w_1^\frac{1}{2}\partial Z^{N-4} \wht g_1\|_{L^2}\leq 2C_0 \ep,\\
\label{gmieux} &\|\partial \wht Z^{N-10}g_1\|_{L^2} \leq 2C_0 \ep.
\end{align}
Integrated estimates for $\bar{\partial}\wht g$
\begin{align}
\label{igN}& \int_0^t \int  w_2'(q)(\bar{\partial}Z^N \wht g_{1})^2 dxd\tau  \leq 2C_0\ep^2(1+t)^{2\rho},\\
\label{ig2N}& \int_0^t \int  w_1'(q)(\bar{\partial}Z^N \wht g)^2 dxd\tau  \leq 2C_0\ep^2(1+t)^{4\rho}, \\
\label{ig2Nbis}& \int_0^t  \int (1+\tau)^{-2\rho}w_1'(q)(\bar{\partial}Z^N \wht g)^2 dxd\tau  \leq 2C_0\ep^2(1+t)^{2\rho}. 
\end{align}

\paragraph{Bootstrap assumptions for $h$}
\begin{equation}
\label{esth}\left\|\Pi  Z^{N-1}\left(\int_0^\infty 2\partial_r \phi(t,r,\theta)\partial_t \phi(t,r,\theta)rdr + h(\theta,2t)\right)\right\|_{L^2(\m S^1)} \leq  2C_0\frac{\ep^2}{(1+t)^{\frac{1}{2}-2\rho}} .
\end{equation}
\begin{equation}
\label{esthbis}\left\|\Pi  Z^{N-5}\left(\int_0^\infty 2\partial_r \phi(t,r,\theta)\partial_t \phi(t,r,\theta)rdr+ h(\theta,2t)\right)\right\|_{L^2(\m S^1)} \leq  2C_0\frac{\ep^2}{(1+t)^{\frac{1}{2}}} .
\end{equation}
We note
\begin{align*}\Delta_h(t) =&\left|\int h(\theta,2t)d\theta -\int_{\m R^2}\left((\partial_t \phi)^2
+|\nabla \phi|^2\right)(t,x)dx\right|+
\left|\int h(\theta,2t)\cos(\theta)d\theta +\int_{\m R^2}2\left(\partial_t \phi \partial_1 \phi\right)(t,x) dx \right|\\
&+
\left|\int h(\theta,2t)\sin(\theta)d\theta +\int_{\m R^2}2\left(\partial_t \phi \partial_2 \phi\right)(t,x) dx\right| . 
\end{align*}
we assume
\begin{equation}
\label{estdelta}
|\Delta_h|\leq 2C_0 \frac{\ep}{\sqrt{1+t}}.
\end{equation}

\subsection{Proof of Theorem \ref{main}}

We have the following improvements for the bootstrap assumptions. The constant $C$ will denote a constant depending only on $\rho,\sigma, \mu,\delta,N$.

\begin{prp}\label{prpangle}
We have
$$|\Delta_h |\leq \frac{C\ep^2}{\sqrt{1+t}}.$$
\end{prp}
The proof of Proposition \ref{prpangle} is the object of Section \ref{angle}.

\begin{prp}\label{prplinf}
We have
\begin{align*}
|Z^{N-9} \phi| &\leq \frac{C_0\ep+C\ep^2}{\sqrt{1+s}(1+|q|)^{\frac{1}{2}-4\rho}},\\  
|Z^{N-9} \wht g| &\leq \frac{C_0\ep+C\ep^2}{(1+s)^{\frac{1}{2}-\rho}},\\
|Z^{N-7} \phi| &\leq \frac{C_0\ep+C\ep^2}{(1+s)^{\frac{1}{2}-2\rho}},\\ 
|Z^{N-7} \wht g| &\leq \frac{C_0\ep+C\ep^2}{(1+s)^{\frac{1}{2}-2\rho}}.
\end{align*}
\end{prp}
The proof of Proposition \ref{prplinf} is the object of Section \ref{seclinf}.
\begin{prp}\label{prphigh} We have the estimates
\begin{align*}
&\|w^\frac{1}{2}\partial Z^N \phi\|_{L^2}\leq C_0 \ep +C\ep^\frac{3}{2}(1+t)^\rho,\\
&\|w^\frac{1}{2}\partial^2 Z^N \phi\|_{L^2}\leq C_0 \ep +C\ep^\frac{3}{2}(1+t)^\rho,\\
&\|w^\frac{1}{2}\partial (S Z^N\phi-s\partial_q\phi Z^N g_{LL}) \|_{L^2}\leq C_0 \ep +C\ep^\frac{3}{2}(1+t)^\rho,\\
&\|w^\frac{1}{2}\partial Z^{N+1} \phi\|_{L^2}\leq C_0 \ep +C\ep^\frac{3}{2}(1+t)^\rho,\\
&\|w_2^\frac{1}{2}\partial Z^N \wht g_{1}\|_{L^2}\leq C_0\ep +C \ep^\frac{5}{4}(1+t)^\rho ,\\
&\|w_1^\frac{1}{2}\partial Z^N \wht g\|_{L^2}\leq  C_0 \ep +C\ep^\frac{3}{2}(1+t)^{2\rho},\\
\end{align*}
and the integrated estimates 
\begin{align*}
& \int_0^t \int w'(q)(\bar{\partial}Z^N \phi)^2 dxd\tau  \leq C_0 \ep^2 + C\ep^3(1+t)^{2\rho}, \\
& \int_0^t \int w'(q)(\bar{\partial} \partial Z^N \phi)^2 dxd\tau  \leq C_0 \ep^2 + C\ep^3(1+t)^{2\rho}, \\
& \int_0^t \int w'(q)(\bar{\partial}( S Z^N\phi-s\partial_q\phi Z^N g_{LL}))^2 dxd\tau  \leq C_0 \ep^2 + C\ep^3(1+t)^{2\rho}, \\
& \int_0^t \int w'(q)(1+\tau)^{-1}(\bar{\partial} Z^{N+1} \phi)^2 dxd\tau  \leq C_0 \ep^2 + C\ep^3(1+t)^{2\rho},\\ 
& \int_0^t \int  w_2'(q)(\bar{\partial}Z^N \wht g_{1})^2 dxd\tau  \leq C_0 \ep^2+C\ep^\frac{5}{2}(1+t)^{2\rho},\\
& \int_0^t \int  w_1'(q)(\bar{\partial}Z^N \wht g)^2 dxd\tau  \leq C_0\ep^2+ 2C\ep^3(1+t)^{4\rho}, \\
& \int_0^t \int  (1+\tau)^{-2\rho}w_1'(q)(\bar{\partial}Z^N \wht g)^2 dxd\tau  \leq  C_0\ep^2+ 2C\ep^3(1+t)^{2\rho}. 
\end{align*}
\end{prp}
The proof of Proposition \ref{prphigh} is the object of Section \ref{sechigh}.

\begin{prp}\label{prplow}
	We have the estimates
	\begin{align*}
	&\|w^\frac{1}{2}\partial Z^{N-1} \phi\|_{L^2}\leq C_0 \ep +C\ep^\frac{3}{2}(1+t)^\rho,\\
&\|w^\frac{1}{2}\partial Z^{N-3} \wht g\|_{L^2}\leq  C_0 \ep +C\ep^\frac{3}{2}(1+t)^\rho,\\
&\|w_1^\frac{1}{2}\partial Z^{N-4} \wht g_1\|_{L^2}\leq C_0 \ep+C\ep^\frac{5}{4},\\
 &\|\partial  Z^{N-10}\wht g_1\|_{L^2} \leq C_0 \ep+C\ep^\frac{5}{4}.
	\end{align*}
\end{prp}
The proof of Proposition \ref{prplow} is the object of Section \ref{seclow}.
To improve the estimate for $h$ we set
$$\check{h}(\theta,s)
=2\int_0^\infty (1+\beta)g^{0\alpha}\sqrt{|\det{g}|}\partial_\alpha \phi \partial_r \phi dr,$$
where the integrand is taken at time $t=\frac{s}{2}$ and $\beta$ is defined by $\beta(r,T,\theta)=0$ and
$$\partial_s \beta +\frac{1}{4}g_{LL}\partial_r \beta= -\frac{1}{2}\frac{\wht g_{L \ba L}}{r}-\frac{1}{2}F_2,$$
where $F_2$ is defined in Corollary \ref{corbeta}.
The additional terms we add (compared to the heuristic choice $-2\int \partial_t \phi \partial_r \phi rdr$ ) are needed for two purposes
\begin{itemize}
	\item $\partial_s\check{ h}$ must be $O\left(\frac{1}{(1+t)^2}\right)$,
	\item $\partial_s \check{h}$ must be at the same level of regularity than $\partial_\theta \partial \phi$ and $\partial_\theta g$.
\end{itemize}

We extend the function $\check{h}$ to all times by
$$ h'(\theta,s)= \psi(s)\check{h}(\theta,s)+(1-\psi(s))\check{h}(\theta,2T),$$
where $\psi$ is a cut-off function such that $\psi=1$ for $s\leq 2T-1$ and $\psi=0$ for $s>2T$.
\begin{prp}\label{prph}
$h'$ satisfy the following estimates
\begin{align}
\label{estih2}&\| Z^{N-1} h'\|_{L^2(\m S^1)}\leq C\ep^2,\\
\label{esth4}&\| Z^{N} h'\|_{L^2(\m S^1)} \leq \ep^2(1+t)^\rho,\\
\label{estih1}&\|\partial_sZ^{N-11} h'\|_{L^2(\m S^1)}\leq \frac{C\ep^2}{(1+t)^2},\\
\label{estih3}&\|\partial_sZ^{N-1} h'\|_{H^{-1}(\m S^1)}\leq \frac{C\ep^2}{(1+t)^{2-\frac{1}{2}\sigma}},\\
\label{esth5}&\int_0^t(1+\tau)\| Z^{N} h'\|^2_{L^2(\m S^1)}d\tau \leq \ep^2(1+t)^{2\rho},
\end{align}
and we can write $\partial_s h' = h'_1+h'_2$ with, 
\begin{align}
\label{estf1h1}&\int_0^t (1+\tau)^4\|\partial_s Z^N h'_1\|_{H^{-2}(\m S^1)}^2d\tau \leq C\ep^4(1+t)^{2\rho},\\
\label{estf1hb}&\| Z^N h'_1\|_{H^{-2}(\m S^1)} \leq C\ep^2(1+t)^{\frac{1}{2}\sigma-2},\\
\label{estf2hb}&\| Z^N h'_2\|_{H^{-2}(\m S^1)} \leq C\ep^2(1+t)^{\frac{3}{2}-\rho},\\
\label{estfh2}& \int_0^t (1+\tau)^{3-2\rho}\| Z^N h'_2\|_{H^{-1}(\m S^1)}^2d\tau \leq C\ep^4(1+t)^{2\rho},\\
\label{estf2hbis}& \int_0^t (1+\tau)^3\|\partial_s  Z^N h'_2\|_{H^{-1}(\m S^1)}^2ds \leq C\ep^4(1+t)^{2\rho},\\
\label{estf1hbis}&\int_0^t (1+\tau)^3\|\partial_s Z^N h'_1\|_{H^{-1}(\m S^1)}^2d\tau \leq C\ep^4(1+t)^{2\rho},\\
\label{estf1bbis}&\int_0^t (1+\tau)^2\| Z^N h'_1\|_{H^{-1}(\m S^1)}^2d\tau \leq C\ep^4(1+t)^{2\rho},
\end{align}
and also
\begin{equation}
\label{esth'}\left\| Z^{N-5}\left(\int_0^\infty \partial_r \phi(t,r,\theta)\partial_t \phi(t,r,\theta)rdr + h'(\theta,2t)\right)\right\|_{L^2(\m S^1)} \leq  C\frac{\ep^3}{\sqrt{1+t}} ,
\end{equation}
\begin{equation}
\label{esth2'}\left\| Z^{N-1}\left(\int_0^\infty \partial_r \phi(t,r,\theta)\partial_t \phi(t,r,\theta)rdr + h'(\theta,2t)\right)\right\|_{L^2(\m S^1)} \leq  C\frac{\ep^3(1+t)^\rho}{\sqrt{1+t}}.
\end{equation}
\end{prp}
The proof of Proposition \ref{prph} is the object of Section \ref{secprph}.
We obtain $b$ thanks to the following proposition

\begin{prp}\label{prpb}
There exists $b_0(s),b_1(s),b_2(s)$ such that there exists a solution $b^{(2)}$ of
\begin{align*}&2a(\theta+f)(1+b^{(2)})^{-2}+(1+b^{(2)})^{-2}-1-2\frac{\partial^2_\theta b^{(2)}}{(1+b^{(2)})}+(1+b^{(2)})^{-2}(\partial_\theta b^{(2)})^2\\
&=\Pi h'(\theta,s)+b_0+b_1\cos(\theta)+b_2\sin(\theta),
\end{align*}
and $b^{(2)}$ satisfy
$$\int_{\m S^1} \frac{b^{(2)}}{(1+b^{(2)})}d\theta=0,$$
$$\|b^{(2)}\|_{H^{l+2}(\m S^1)}\lesssim \|h'\|_{H^{l}(\m S^1)},$$
$$\|\partial^k_s b^{(2)}\|_{H^{l+2}(\m S^1)}\lesssim \|\partial^k_s h'\|_{H^{l}(\m S^1)},$$
and 
$$|b_0-a_0|+|b_1-a_1|+|b_2-a_2|\lesssim \ep^4,$$
$$|\partial^k_s b_0|+|\partial^k_s b_1|+|\partial^k_s b_2|\lesssim \ep^2\|\partial^k_s h'\|_{L^2(\m S^1)}.$$
\end{prp}
The proof of Proposition \ref{prpb} is the object of Section \ref{secprpb}.
\begin{prp}
\label{prpfin}
We set
$$h^{(2)}= \Pi h'(\theta,s)+b_0+b_1\cos(\theta)+b_2\sin(\theta),$$
There exists a solution $(g^{(2)}=g_{{\mathfrak{b}}^{(2)}}+\wht g^{(2)},\phi^{(2)})$ of \eqref{s2}
on $[0,T]\times \m R^2$, in generalized wave coordinates
$$(H^{(2)})^\alpha = (g^{(2)})^{\lambda \beta}(\Gamma^{(2)})^\alpha_{\lambda \beta}=(F^{(2)})^\alpha+(G^{(2)})^\alpha +(\wht G^{(2)})^\alpha,$$
with
$$(F^{(2)})^\alpha = \Box_{g_{\mathfrak{b}^{(2)}}}x^\alpha,$$
\begin{align*}
U_\alpha (G^{(2)})^\alpha&=-s(1+b^{(2)})\partial_s f^{(2)}\chi'(q),\\
L_\alpha (G^{(2)})^\alpha& =  \frac{1}{r}\Up\left(\frac{r}{t}\right)\int_{\infty}^r \left(2(\partial_q \phi^{(2)})^2r -h^{(2)}(\theta,2t)\partial_q^2(q\chi(q))\right)dr,\\
\ba L_\alpha (G^{(2)})^\alpha &= 0,
\end{align*}
with $f^{(2)}$ such that $1+\partial_\theta f^{(2)}=(1+b^{(2)})^{-1},$
and $(\wht G^{(2)})^\alpha$ contains the terms in $ (g^{(2)})^{\lambda \beta}(\Gamma^{(2)})^\alpha_{\lambda \beta}$ of the form $\wht g^{(2)} \partial^l_s\partial^k_\theta b^{(2)}$, where
$l+k-2\geq 1$ or $l\geq 2$.
Moreover $(g^{(2)},\phi^{(2)})$ satisfy the same estimates as $(g,\phi)$, 
$b^{(2)}$ satisfy the estimates $\q H$ and
\begin{equation*}
\left\|\Pi  Z^{N-5}\left(\int_0^\infty (\partial_q \phi^{(2)}(t,r,\theta))^2rdr + h^{(2)}(\theta,2t)\right)\right\|_{L^2(\m S^1)} \leq  C\frac{\ep^3}{\sqrt{1+t}} .
\end{equation*}
\begin{equation*}
\left\|\Pi  Z^{N-1}\left(\int_0^\infty (\partial_q \phi^{(2)}(t,r,\theta))^2rdr + h^{(2)}(\theta,2t)\right)\right\|_{L^2(\m S^1)} \leq  C\frac{\ep^3(1+t)^\rho}{\sqrt{1+t}} .
\end{equation*}

\end{prp}
The proof of Proposition \ref{prpb} is the object of Section \ref{secprpfin}. Combining Propositions \ref{prpangle} to \ref{prpfin} we now give the proof of Theorem \ref{main}
\begin{proof}[Proof of Theorem \ref{main}]
	We choose $\ep$ small enough such that 
	$$C\ep^\frac{1}{4} \leq \frac{C_0}{2}, \quad C\ep \leq \frac{1}{2}$$
	Then Propositions \ref{prplinf}, \ref{prphigh} and \ref{prplow} imply that the bootstrap assumptions for $(\phi,\wht g)$ are true with the constant $2C_0$ replaced with $\frac{3C_0}{2}$. Moreover, thanks to Proposition \ref{prpangle}, the bootstrap assumption \ref{estdelta} is true with $2C_0$ replaced by $C_0$. Moreover, Propositions \ref{prph}, \ref{prpb} and \ref{prpfin} yield the existence of $b^{(2)}$ satisfying the hypothesis $\q H $ and $(g^{(2)}=g_{b^{(2)}}+\wht g^{(2)},\phi^{(2)})$ solution of \eqref{sys} such that the bootstrap assumptions are satisfied by $(\wht g^{(2)},\phi^{(2)})$ with $2C_0$ replaced by $\frac{3C_0}{2}$, and the bootstrap assumptions \eqref{esth},\eqref{esthbis} and \eqref{estdelta} are satisfied by
	$$h^{(2)}=2a(\theta+f)(1+b^{(2)})^{-2}+(1+b^{(2)})^{-2}-1-2\frac{\partial^2_\theta b^{(2)}}{(1+b^{(2)})}+(1+b^{(2)})^{-2}(\partial_\theta b^{(2)})^2,$$
	with $2C_0$ replaced by $C_0$. This concludes the proof of Theorem \ref{main}.
\end{proof}

\subsection{First consequences of the bootsrap assumptions}
Thanks to the weighted Klainerman-Sobolev inequality, the bootstrap assumptions immediately imply the following propositions

\begin{prp}\label{estks}
We have the estimates, for $q<R$
\begin{align}
 \label{ks1}|\partial Z^{N-3} \phi(t,x)|&\lesssim \frac{\ep}{\sqrt{1+|q|}\sqrt{1+s}},\\
  \label{ks3}|\partial Z^{N-5} \wht g(t,x)|&\lesssim \frac{\ep(1+s)^\rho}{\sqrt{1+|q|}\sqrt{1+s}},\\
\label{ks2}|\partial Z^{N-6} \wht{g_1}(t,x)|&\lesssim \frac{\ep}{(1+|q|)^{\frac{1}{2}-\sigma}\sqrt{1+s}},\\
\label{ks2bis}|\partial Z^{N-12} \wht{g_1}(t,x)|&\lesssim \frac{\ep}{\sqrt{1+|q|}\sqrt{1+s}},
\end{align}
and for $q>R$ 
\begin{align}
\label{ks6} |\partial Z^{N-5} \wht g|&\lesssim \frac{\ep}{(1+s)^{\frac{1}{2}-\rho}(1+|q|)^{\frac{3}{2}+\delta}},\\
\label{ks5}|\partial Z^{N-6} \wht{g_1}(t,x)|&\lesssim \frac{\ep}{(1+|q|)^{\frac{3}{2}+\delta-\sigma}\sqrt{1+s}}.
\end{align}
\end{prp}
Thanks to Lemma \ref{lmintegration} we deduce the following corollary

\begin{cor}\label{estksg}
We have the estimates, for $q<R$
\begin{align}
\label{iks1}| Z^{N-3} \phi(t,x)|&\lesssim \frac{\ep\sqrt{1+|q|}}{\sqrt{1+s}},\\
\label{iks3}| Z^{N-5} \wht g(t,x)|&\lesssim \frac{\ep(1+s)^\rho\sqrt{1+|q|}}{\sqrt{1+s}},\\
\label{iks2}| Z^{N-6} \wht{g_1}(t,x)|&\lesssim \frac{\ep(1+|q|)^{\frac{1}{2}+\sigma}}{\sqrt{1+s}},\\
\label{iks2bis}| Z^{N-12} \wht{g_1}(t,x)|&\lesssim \frac{\ep\sqrt{1+|q|}}{\sqrt{1+s}}
\end{align}
and for $q>R$ 
\begin{align}\label{iks6} 
| Z^{N-5} \wht g|&\lesssim \frac{\ep}{(1+s)^{\frac{1}{2}-\rho}(1+|q|)^{\frac{1}{2}+\delta}},\\
\label{iks5}|Z^{N-6} \wht{g_1}(t,x)|&\lesssim \frac{\ep}{(1+|q|)^{\frac{1}{2}+\delta-\sigma}\sqrt{1+s}}.
\end{align}
\end{cor}

To obtain $L^2$ estimates for $Z^I \phi$ and $Z^I \wht g$ we may use the weighted Hardy inequality

\begin{prp}
We have
\begin{align}
\label{hphin}
&\|(1+|q|)^{-1} w^\frac{1}{2}Z^N \phi\|_{L^2}+\|(1+|q|)^{-1} w^\frac{1}{2}(SZ^N \phi-s\partial_q \phi Z^N g_{LL})\|_{L^2}\lesssim \ep (1+t)^{\rho},\\
\label{hphinn}
&\|(1+|q|)^{-1} w^\frac{1}{2}Z^{N+1} \phi\|_{L^2}\lesssim \ep (1+t)^{\frac{1}{2}+\rho},\\
\label{hg1n}&\|(1+|q|)^{-1}w_2^\frac{1}{2}Z^{N} \wht g_1\|_{L^2}\lesssim \ep(1+t)^\rho,\\
\label{hg2n}&\|(1+|q|)^{-1} w_1^\frac{1}{2}Z^{N} \wht g\|_{L^2}\lesssim \ep (1+t)^{2\rho},\\
\label{hphin1}&\|(1+|q|)^{-1} w^\frac{1}{2}Z^{N-1} \phi\|_{L^2} \lesssim \ep,\\
\label{hgn1}&\|(1+|q|)^{-1} w^\frac{1}{2}Z^{N-3} \wht g\|_{L^2} \lesssim \ep(1+t)^\rho,\\
\label{hg1n1}&\|(1+|q|)^{-1}w_1^\frac{1}{2}Z^{N-4} \wht g_1\|_{L^2}\lesssim \ep.
\end{align}
\end{prp}

\begin{proof}
	The only thing we have to check is whether we can apply Proposition \ref{hardy} with our weight functions. In the exterior, the smaller weight is $w_2(q)=(1+|q|)^\beta$ with $\beta=2+2\delta-2\sigma >1$. In the interior, the biggest one is a $O(1)$. Consequently we are in the range of the weighted Hardy inequality.
\end{proof}

\begin{lm}\label{lmh} We have
\begin{equation}
\label{esthq}\left\|\Pi  Z^{N-1}\left(\int_0^\infty 2(\partial_q \phi(t,r,\theta))^2rdr - h(\theta,2t)\right)\right\|_{L^2(\m S^1)} \lesssim \frac{\ep^2}{(1+t)^{\frac{1}{2}-2\rho}} ,
\end{equation}
\begin{equation}
\label{esthqbis}\left\|\Pi  Z^{N-5}\left(\int_0^\infty 2(\partial_q \phi(t,r,\theta))^2rdr-h(\theta,2t)\right)\right\|_{L^2(\m S^1)} \lesssim \frac{\ep^2}{(1+t)^{\frac{1}{2}}} .
\end{equation}
\end{lm}
\begin{proof}
	We can write 
	$$(\partial_q \phi)^2+\partial_r \phi \partial_t \phi = O\left(\partial \phi \bar{\partial} \phi\right).$$
	Consequently, thanks to \eqref{bootphi1} and \eqref{important} we have
	$$|(\partial_q \phi)^2+\partial_r \phi \partial_t \phi|\lesssim \frac{\ep}{(1+s)^\frac{3}{2}(1+|q|)^{\frac{3}{2}-4\rho}}|Z\phi|,$$
	and therefore
	\begin{align*}
	&\left\|\int \left((\partial_q \phi)^2+\partial_r \phi \partial_t \phi\right)rdr\right\|_{L^2(\m S^1)}\\
	&\lesssim \int \frac{\ep}{(1+s)^\frac{1}{2}(1+|q|)^{\frac{3}{2}-4\rho}}\|Z\phi\|_{L^2(\m S^1)}dr\\
	&\lesssim \frac{\ep}{(1+t)}\ld{\frac{Z\phi}{1+|q|}}.
	\end{align*}
	Estimates \eqref{hphin} and \eqref{hphin1} conclude the proof of Lemma \ref{lmh}.
\end{proof}

\section{The wave coordinates condition}\label{secwave}

Similarly to \cite{lind} we use the wave coordinate condition to obtain better decay on some coefficients of the metric. More precisely, since we are in $2+1$ dimensions, the wave coordinate condition gives us three relations, which yield the fact that $\partial_q g_{LL}$, $\partial_q g_{LU}$ and $\partial_q g_{UU}$ have a better decay than expected. In the first part of this section, we calculate the algebraic relations given by the wave coordinate condition, and in the remaining parts, we give the estimates for these good coefficients of metric.

\subsection{Good components of the metric}
The wave coordinates condition yields better decay properties in $s$ for some components of the metric. Since far from a conical neighbourhood of the light cone, we have $|q|\sim s$, this condition will only be relevant near the light cone. It is given by
$$H^\alpha = g^{\lambda \beta}\Gamma^\alpha_{\lambda \beta}=F^\alpha+G^\alpha +\wht G^\alpha,$$
where the terms are defined in Section \ref{secgen}.
\begin{prp}\label{estLL}
We have the following estimate, in the region $\frac{t}{2}\leq r \leq 2t$,
$$
|\partial_q Z^I \wht g_{LL}|\lesssim 
|\bar{\partial}Z^I \wht g_{\ba L L}|+|\bar{\partial}Z^I \wht g_{\q T \q T}|
+\frac{1}{1+s}
\left(|Z^I \wht g_{L\ba L}|+|Z^I \wht g_{\q T \q T}|\right).$$
Moreover, in the region $q\leq R+1$ we can write
$$\partial_q \wht g_{LL} = \frac{1}{2r}\wht g_{L\ba L}+ \wht g_{\q T \q T}\partial_q \wht g_1 + \wht g_1 \partial_{\q T} \wht g_1
+\partial_U \wht g_{\q T \q T} +\frac{1}{r} \wht g_{\q T \q T}.$$
\end{prp}
Let us note that the second part of the Proposition will only be used in Section \ref{secprph}.
\begin{proof}
The wave coordinate condition implies
\begin{align*}
-\ba L_{\alpha}H^\alpha &=\ba L_{\alpha}\left(\frac{1}{\sqrt{|\det(g)|}}\partial_\mu (g^{\mu \alpha}\sqrt{|\det(g)|})\right)\\
=&\frac{ g^{\mu \alpha}}{\sqrt{|\det(g)|}}\ba L_{\alpha}\partial_\mu \sqrt{|\det(g)|}
+\partial_\mu (\ba L_\alpha g^{\mu \alpha})-g^{\mu \alpha}\partial_\mu( \ba L_\alpha)\\
=&\frac{g^{\ba L \mu}}{\sqrt{|\det(g)|}}\partial_{\mu}\sqrt{|\det (g)|}
+\partial_\mu (g^{\ba L \mu})
-\frac{1}{r}g^{UU}\\
=&\frac{g^{\ba L \ba L}}{\sqrt{|\det(g)|}}\partial_{\ba L} \sqrt{|\det (g)|}
+\frac{g^{\ba L \q T}}{\sqrt{|\det(g)|}}\partial_{\q T} \sqrt{|\det (g)|}
+\partial_{\ba L} g^{\ba L \ba L}+\partial_U g^{\ba L U}+\partial_L g^{\ba L L}\\
&+\frac{1}{r}g^{\ba L R}-\frac{1}{r}g^{UU},
\end{align*}
where we have denoted by $R$ the vector field $\partial_r$, and used the following calculations
\begin{align*}
g^{\mu \alpha}\partial_\mu(\ba L_\alpha)=&-g^{\mu \alpha}\partial_\mu( R_\alpha)\\
=&-g^{11}\partial_1 \cos(\theta)-g^{12}(\partial_2 \cos(\theta)-\partial_1 \sin(\theta))-g^{22}\partial_2 \sin(\theta)\\
=&-\frac{g^{UU}}{r},
\end{align*}
\begin{align*}
\partial_\mu g^{\ba L \mu}=&\partial_0 g^{\ba L0} +\partial_1 g^{\ba L 1}
+\partial_2 g^{\ba L 2}\\
=&\partial_0 g^{\ba L0}+\partial_R g^{\ba L R}+\partial_U g^{\ba L U}
+g^{\ba L R}(\partial_1 \cos(\theta)+\partial_2 \sin(\theta))
+g^{\ba L U}(-\partial_1\sin(\theta)+\partial_2 \cos(\theta))\\
=&\partial_{\ba L} g^{\ba L \ba L}+\partial_U g^{\ba L U}+\partial_L g^{\ba L L}+\frac{g^{\ba L R}}{r}.
\end{align*}
Consequently
\begin{equation}\label{condonde1}
\begin{split}
\partial_{\ba L} g^{\ba L \ba L}
=& -\ba L_{\alpha}\left(F^\alpha+G^\alpha + \wht G^\alpha\right)-\frac{g^{\ba L \ba L}}{\sqrt{|\det(g)|}}\partial_{\ba L} \sqrt{|\det (g)|}
-\frac{g^{\ba L \q T}}{\sqrt{|\det(g)|}}\partial_{\q T} \sqrt{|\det (g)|}\\
&-\partial_U g^{\ba L U}-\partial_L g^{\ba L L}
-\frac{1}{r}g^{\ba L R}-\frac{1}{r}g^{UU}.
\end{split}
\end{equation}
Also we have in the basis $L,\ba L, U$
\begin{equation}
\label{determinant}\det(g)|_{L,\ba L,U}=g_{LL}(g_{\ba L \ba L}g_{UU}-(g_{U \ba L})^2) -g_{L\ba L}(g_{L \ba L}g_{UU}-g_{LU}g_{\ba L U} ) + 
g_{LU}(g_{\ba L L}g_{U \ba L}-g_{\ba L \ba L}g_{LU}).
\end{equation}
and we can express
\begin{align*}
g^{\ba L \ba L}=\frac{1}{\det(g)}(g_{ L  L}g_{UU}-(g_{U  L})^2)&=-\frac{1}{4}g_{LL}+ O( g_{\q T \q T})O(g),\\
g^{\ba L U}=\frac{1}{\det(g)}(g_{\ba L  L}g_{LU}-g_{U \ba L}g_{LL})&=\frac{1}{2}g_{LU}+ O(g_{\q T \q T})O(g ),\\
g^{\ba L L}=\frac{1}{\det(g)}(g_{\ba L  L}g_{UU}-g_{U \ba L}g_{UL})&=\frac{1}{g_{L\ba L}}+ O( g_{\q T \q T}),\\
\end{align*}
where we have used the notation $O(g)=O(g-m)$ where $m$ is the Minkowski metric.
To go from the determinant in the basis $L,\ba L, U$ to the determinant in the basis $t,x_1,x_2$ we just have to divide by $4$.
Therefore
$$| \sqrt{|\det(g)|}-\sqrt{|\det(g_{\mathfrak{b}})|}+\frac{1}{2}\wht g_{L\ba L}|\lesssim 
|\wht g_{\q T \q T}|.$$
 We note that in \eqref{condonde1} the terms involving $\partial_{L} g_{L \ba L}$ compensate.
Since in \eqref{condonde1}, by definition of $F^\alpha$  the terms involving only $g_{\mathfrak{b}}$ compensate, we have
$$\partial_q \wht g_{LL}-\frac{1}{2r}\wht g_{L \ba L}=
 \bar{\partial} \wht g_{\q T \q T}+
\frac{1}{1+s} \wht g_{\q T \q T}
+s.t..$$
where $s.t$ denotes similar terms (here these terms are quadratic terms with a better or similar decay), and we have used the fact that in the region $\frac{t}{2}\leq r \leq 2t$, we have $r\sim s$. This prove the second part of the proposition.
Since $[Z,\partial_q]\sim \partial_q$ and $[Z,\bar{\partial}]\sim \bar{\partial}$ we have
$$\left|\partial_q Z^I \wht g_{LL}\right|\lesssim |Z^{I-1}\wht g_{LL}|+|\bar{\partial}Z^I \wht g_{L \ba L}|+|\bar{\partial}Z^I \wht g_{\q T \q T}|+
\frac{1}{1+s}(|Z^I \wht g_{L \ba L}|+|Z^I \wht g_{\q T \q T}|).$$
This concludes the proof of Proposition \ref{estLL}.
\end{proof}

The other two contractions of the wave condition yield better decay on a conical neighbourhood of the light cone for $\wht g_{U L}$ and $\wht g_{UU}$.
\begin{prp}\label{estLU}
We have the following property
\begin{align*}
| Z^I (\partial_q\wht g_{UL}+G^U)|&\lesssim |\overline{\partial}Z^I \wht g_{\q T \q V}| + \frac{1}{1+s} |Z^I \wht g_{\q T \q V}|,\\
\left|Z^I \left(\partial_q \wht g_{UU}+2G^L\right)\right|&\lesssim  |\overline{\partial}Z^I \wht g| + \frac{1}{1+s}|Z^I \wht g|.
\end{align*}
\end{prp}

\begin{proof}
To obtain the first estimate, we contract the wave coordinate condition with the vector field $U$.
\begin{align*}
-U_{\alpha}H^\alpha &= \frac{1}{\sqrt{|\det(g)|}}U_{\alpha}\partial_\mu (g^{\mu \alpha})\sqrt{\det(g)}\\
=&\frac{ g^{\mu \alpha}}{\sqrt{|\det(g)|}}U_{\alpha}\partial_\mu \sqrt{|\det(g)|}
+\partial_\mu (U_\alpha g^{\mu \alpha})+g^{\mu \alpha}\partial_\mu( U_\alpha)\\
=&\frac{g^{U \mu}}{\sqrt{|\det(g)|}}\partial_{\mu}\sqrt{|\det (g)|}
+\partial_\mu (g^{U \mu})
+\frac{1}{r}g^{UR}\\
=&\frac{g^{U \ba L}}{\sqrt{|\det(g)|}}\partial_{\ba L} \sqrt{|\det (g)|}
+\frac{g^{U \q T}}{\sqrt{|\det(g)|}}\partial_{\q T} \sqrt{|\det (g)|}
+\partial_{\ba L} g^{U \ba L}+\partial_U g^{U U}+\partial_L g^{U L}+\frac{1}{r}g^{UR}. 
\end{align*}
Therefore 
\begin{equation*}
\partial_{\ba L} g^{U \ba L}=-U_{\alpha}H_b^\alpha-\frac{g^{U \ba L}}{\sqrt{|\det(g)|}}\partial_{\ba L} \sqrt{|\det (g)|}
-\frac{g^{U \q T}}{\sqrt{|\det(g)|}}\partial_{\q T} \sqrt{|\det (g)|}
-\partial_U g^{U U}-\partial_L g^{U L}-\frac{1}{r}g^{UR},
\end{equation*}
and arguing as in Proposition \ref{estLL} we infer
$$|\partial_q\wht g_{UL}+G^U|\lesssim |\bar{\partial} \wht g_{\q T \q V}|+\frac{1}{1+s}|\wht g_{\q T \q V}|+s.t.$$
Commuting with the vector fields $Z$ as before, we obtain the desired estimate.
To obtain the second one, we contract the wave coordinate condition with $L$
\begin{equation}\label{wcl}\begin{split}
-L_{\alpha}H^\alpha = &\frac{1}{\sqrt{|\det{g}|}}L_{\alpha}\partial_\mu (g^{\mu \alpha})\sqrt{|\det(g)|}.\\
=&\frac{1}{\sqrt{|\det{g}|}}\partial_{\ba L}\left(\sqrt{|\det(g)|}g^{L\ba L}\right) +\frac{1}{\sqrt{|\det{g}|}}\partial_{\q T}\left(\sqrt{|\det(g)|}g^{L\q T}\right)-g^{\mu \alpha}\partial_\mu (L_\alpha). \end{split}
\end{equation}
We note that 
\begin{align*}
\sqrt{|\det(g)|}g^{L\ba L}=&\frac{1}{2\sqrt{|\det(g)|_{L,\ba L,U}|}}(g_{L\ba L}g_{UU}-g_{U\ba L}g_{UL})\\
=&\frac{g_{L\ba L}g_{UU}}{\sqrt{g_{L\ba L}^2g_{UU}+O(\wht g_{\q T \q T})O(g)}}
+O(\wht g_{\q T \q T})O(g)\\
=&-\frac{1}{2}\sqrt{g_{UU}}+ +O(\wht g_{\q T \q T})O(g).
\end{align*}
Therefore \eqref{wcl} yields
$$|\partial_q \wht g_{UU}+2G^{L}|\lesssim |\bar{\partial}\wht g|
+\frac{1}{1+s}|\wht g|.$$
We commute with the vector fields $Z$ to conclude.
\end{proof}

\subsection{Estimate for the good metric component $g_{LL}$}
Thanks to the bootstrap assumptions, we obtain the following corollary.
\begin{cor}\label{estll}
	We have the estimates for $q>R+1$ 
	\begin{equation}
	\label{wextll}|\partial Z^{N-8} \wht g_{LL}|\lesssim \frac{\ep}{(1+s)^\frac{3}{2}(1+|q|)^{\frac{1}{2}+\delta -\sigma}}
	\end{equation}
	and for $q\leq R+1$ 
	\begin{align}
	\label{wintll1}|\partial Z^{N-7} \wht g_{LL}|&\lesssim \frac{\ep(1+|q|)^{\frac{1}{2}+\sigma}}{(1+s)^\frac{3}{2}},\\
	\label{wintll1bis}|\partial Z^{N-7} \wht g_{LL}|&\lesssim \frac{\ep(1+|q|)^{\frac{1}{2}}}{(1+s)^{\frac{3}{2}-\rho}},\\
	\label{wintll2}|\partial Z^{N-8} \wht g_{LL}|&\lesssim \frac{\ep}{(1+s)^{\frac{3}{2}-3\rho}}, \\
	\label{wintll3}|\partial Z^{N-10} \wht g_{LL}|&\lesssim \frac{\ep}{(1+s)^{\frac{3}{2}-\rho}} .
	\end{align}
\end{cor}
\begin{proof}
	First we note that $\wht g$ and $\wht g_1$ differ only by their $\ba L \ba L$ component.
	Thanks to Proposition \ref{estLL} we have
$$|\partial_q Z^I \wht g_{LL}|\lesssim \frac{1}{1+s}|Z^{I+1}\wht g_1|.$$
Then \eqref{iks5} yields \eqref{wextll}, \eqref{iks2} yields \eqref{wintll1}, \eqref{iks3} yields \eqref{wintll1bis}, \eqref{bootg2} yields \eqref{wintll2} and \eqref{bootg1} yields \eqref{wintll3}.
\end{proof}
Thanks to Lemma \ref{lmintegration}, since $\delta-\sigma >\frac{1}{2}$ we obtain the following corollary
\begin{cor}\label{iestll}
	We have the estimates for $q>R+1$ 
	\begin{equation}
	\label{iwextll}|Z^{N-7} \wht g_{LL}|\lesssim \frac{\ep}{(1+s)^\frac{3}{2}(1+|q|)^{-\frac{1}{2}+\delta -\sigma}}, 
	\end{equation}
	and for $q\leq R+1$ 
	\begin{align}
	\label{iwintll1}| Z^{N-7} \wht g_{LL}|&\lesssim \frac{\ep(1+|q|)^{\frac{3}{2}+\sigma}}{(1+s)^\frac{3}{2}}, \\
	\label{iwintll1bis}| Z^{N-7} \wht g_{LL}|&\lesssim \frac{\ep(1+|q|)^{\frac{3}{2}}}{(1+s)^{\frac{3}{2}-\rho}},\\
	\label{iwintll2}|Z^{N-8} \wht g_{LL}|&\lesssim \frac{\ep(1+|q|)}{(1+s)^{\frac{3}{2}-3\rho}},\\
	\label{iwintll3}| Z^{N-10} \wht g_{LL}|&\lesssim \frac{\ep(1+|q|)}{(1+s)^{\frac{3}{2}-\rho}}.
	\end{align}
\end{cor}
We now give $L^2$ estimates for the coefficient $g_{LL}$.

\begin{prp}\label{l2ll}
	We have 
	\begin{align}
	\label{wl2dlln}&
	\int_0^t\left\| w'_2(q)^\frac{1}{2} \partial Z^N \wht g_{LL}\right\|^2_{L^2}ds \lesssim \ep^2(1+t)^{2\rho},\\
	\label{hwl2dlln}&
	\int_0^t\left\| \frac{w_2'(q)^\frac{1}{2}}{(1+|q|)}  Z^N \wht g_{LL}\right\|^2_{L^2}ds \lesssim \ep^2(1+t)^{2\rho},\\
	\label{wl2dll}&\left\| \frac{w_2(q)^\frac{1}{2}}{(1+|q|)} \partial Z^{N-1} \wht g_{LL}\right\|_{L^2}\lesssim \frac{\ep}{(1+t)^{1-\rho}}, \quad \left\| \frac{w_1(q)^\frac{1}{2}}{(1+|q|)} \partial Z^{N-1} \wht g_{LL}\right\|_{L^2}\lesssim \frac{\ep}{(1+t)^{1-2\rho}}\\
\label{hwl2dll}&\left\| \frac{w_2(q)^\frac{1}{2}}{(1+|q|)^2}  Z^{N-1} \wht g_{LL}\right\|^2_{L^2}\lesssim \frac{\ep}{(1+t)^{1-\rho}}, \quad \left\| \frac{w_1(q)^\frac{1}{2}}{(1+|q|)^2}  Z^{N-1} \wht g_{LL}\right\|^2_{L^2}\lesssim \frac{\ep}{(1+t)^{1-2\rho}},\\
	\label{wl2dll2}&\left\| \frac{w_1(q)^\frac{1}{2}}{(1+|q|)} \partial Z^{N-5} \wht g_{LL}\right\|_{L^2}\lesssim \frac{\ep}{(1+t)^{1-\rho}},\\
	\label{hwl2dll2}&\left\| \frac{w_1(q)^\frac{1}{2}}{(1+|q|)^2}  Z^{N-5} \wht g_{LL}\right\|^2_{L^2}\lesssim \frac{\ep}{(1+t)^{1-\rho}}.
	\end{align}
Moreover, the same statement holds true with $w_2$ replaced by $w_1$ and $\rho$ replaced by $2\rho$.
\end{prp}

\begin{proof}
	Proposition \eqref{estLL} implies
	$$|\partial_q Z^N \wht g_{LL}|\lesssim 
	\frac{1}{1+s}|Z^I \wht g_1|
	+|\bar{\partial} Z^N \wht g_{1}|.$$
	Thanks to \eqref{igN} we have
	$$\int_0^t\left\| w_2'(q)^\frac{1}{2} \bar{\partial} Z^N \wht g_{1}\right\|^2_{L^2}ds \lesssim \ep^2(1+t)^{2\rho},$$
	and thanks to \eqref{hg1n} we have
	\begin{align*}
	&\int_0^t\left\| \frac{w_2'(q)^\frac{1}{2}}{1+s}  Z^N \wht g_1\right\|^2_{L^2}ds \\
	&\lesssim \int_0^t\frac{1}{(1+s)}\left\| \frac{w_2(q)^\frac{1}{2}}{1+|q|}  Z^N \wht g_1\right\|^2_{L^2}ds\\ 
	&\lesssim \int_0^t \frac{\ep^2}{(1+s)^{1-2\rho}} \lesssim \ep^2(1+t)^{2\rho},
	\end{align*}
	where we have used \eqref{hg1n}. This concludes the proof of estimate \eqref{wl2dlln}. To prove \eqref{hwl2dlln} we notice that
	thanks to the weighted Hardy inequality
	$$\left\| \frac{w_2'(q)^\frac{1}{2}}{(1+|q|)}  Z^N \wht g_{LL}\right\|_{L^2}\lesssim  \left\| w'_2(q)^\frac{1}{2} \partial Z^N \wht g_{LL}\right\|^2_{L^2}.$$
	Indeed in the exterior $w_2'(q)=(1+|q|)^{\beta}$ with $\beta =1+2\delta-4\sigma>1$, and in the interior, $w_2'(q)=(1+|q|)^\alpha$ with $\alpha=-1-2\sigma-1<1$.
	To prove \eqref{wl2dll} we write
	$$|\partial_q Z^I \wht g_{LL}|\lesssim \frac{1}{1+s}|Z^{I+1}\wht g_1|,$$
	and consequently
	$$\left\| \frac{w_2(q)^\frac{1}{2}}{(1+|q|)} \partial Z^{N-1} \wht g_{LL}\right\|_{L^2}\lesssim
	\frac{1}{1+t}\left\| \frac{w_2(q)^\frac{1}{2}}{(1+|q|)}  Z^{N} \wht g_1\right\|_{L^2} \lesssim \frac{\ep^2}{(1+t)^{1-\rho}},$$
	where we have used \eqref{hg1n}. To prove \eqref{hwl2dll} we notice that 	thanks to the weighted Hardy inequality
	$$\left\| \frac{w_2(q)^\frac{1}{2}}{(1+|q|)^2}  Z^{N-1} \wht g_{LL}\right\|_{L^2}\lesssim  \left\| \frac{w_2(q)^\frac{1}{2}}{(1+|q|)} \partial Z^{N-1} \wht g_{LL}\right\|^2_{L^2}.$$
	Indeed in the exterior $\frac{w_2}{(1+|q|)^2}=(1+|q|)^{\beta}$ with $\beta =2\delta-4\sigma>1$, and in the interior, $\frac{w_2}{(1+|q|)^2}=(1+|q|)^\alpha$ with $\alpha=-1-2\sigma-2<1$.
	Estimates \eqref{wl2dll2} and \eqref{hwl2dll2} are proved in the same way thanks to \eqref{hg1n1}.
\end{proof}
\paragraph{The outgoing null cone}

\begin{prp}
	\label{cone} For $\ep$ small enough, the causal future of $B(0,R)$ is included in the Minkowski cone $q<R+\frac{1}{2}$.
\end{prp}

\begin{proof}
With our estimates for $g$, we obtain that the outgoing solution of the eikonal equation
$$g^{\alpha \beta}\partial_\alpha u \partial_\beta u=0,$$
with initial data $u=r$ is such that
$$\partial_s u =O(g_{LL})=O\left(\frac{\ep q}{(1+s)^{\frac{3}{2}-\rho}}\right),$$
so $$u=q\left(1+O(\ep)\right).$$
The causal future of $B(0,R)$ is the inside of the cone bounded by the hypersurface $u=R$. So in the causal future of $B(0,R)$ we have
$$q\leq R(1+\ep)\leq R+\frac{1}{2}.$$
\end{proof}
As a consequence, $\phi$ is supported in the region $q<R+\frac{1}{2}$.

\subsection{Estimate for the good metric coefficient $g_{UL}$}
We recall the definition of $G^U$ 
$$G^U=\sigma^0_{UL}\chi'(q).$$
The following estimate for $G^U$ is a direct consequence of \eqref{estg0}.
\begin{prp}
	We have
	\begin{equation}
	\label{estgu}|Z^{N-11}G^U|\lesssim \frac{\ep^2\ch_{R+\frac{1}{2}\leq q \leq R+1 }}{1+s}.
	\end{equation}
\end{prp}
We now go to the estimate of $g_{UL}$.
\begin{cor}\label{estul}
	We have the estimates for $q>R+1$ 
	\begin{equation}
	\label{wextlu}
	|\partial Z^I \wht g_{UL}|\lesssim \frac{\ep}{(1+s)^{\frac{3}{2}-\rho}(1+|q|)^{\frac{1}{2}+\delta}}
	\end{equation}
	and for $q\leq R+1$
	\begin{align}
	\label{wintlu1}|\partial_q Z^{N-7}  g_{UL}|&\lesssim \frac{\ep(1+|q|)^{\frac{1}{2}}}{(1+s)^{\frac{3}{2}-\rho}},\\
	\label{wintlu2}|\partial_q Z^{N-8} g_{UL}|&\lesssim \frac{\ep}{(1+s)^{\frac{3}{2}-3\rho}}, \\
	\label{wintlu3}|\partial_q Z^{N-10} g_{UL}|&\lesssim \frac{\ep}{(1+s)^{\frac{3}{2}-\rho}}.
	\end{align}
\end{cor}
\begin{proof}
	In the exterior, $G^U=0$ so thanks to Proposition \ref{estLU} we have
	$$|\partial Z^I \wht g_{UL}|\lesssim \frac{1}{1+s}|Z^{I+1}\wht g|,$$
	and \eqref{wextlu} is a consequence of \eqref{iks6}.
	In the interior Proposition \ref{estLU} yields
	$$|Z^I (\partial_q (\wht g_{UL}+\sigma^0_{UL}\chi(q)))|\lesssim \frac{1}{1+s}|Z^{I+1}\wht g|.$$
	We can write $\wht g_{LU}+\chi(q)\sigma^0_{UL}=g_{UL}-\chi(q)\sigma^1_{UL}$ so thanks to \eqref{estsigma1} for $I\leq N-2$
	$$|\partial_q Z^I g_{UL}|\lesssim \frac{1}{1+s}|Z^{I+1}\wht g|+ \ch_{q>R}|Z^I \sigma^1_{UL}| \lesssim \frac{1}{1+s}|Z^{I+1}\wht g|+\frac{\ep^2}{(1+s)^\frac{7}{4}}.$$
	Consequently \eqref{iks3} yields \eqref{wintlu1}, \eqref{bootg2} yields \eqref{wintlu2} and \eqref{bootg1} yields \eqref{wintlu3}.
\end{proof}

By integrating, we obtain the following corollary
\begin{cor}\label{iestul}
	We have the estimates for $q>R+1$ 
	\begin{equation}
\label{iwextlu}
	|Z^{N-7} \wht g_{\q T \q T}|\lesssim \frac{\ep}{(1+s)^{\frac{3}{2}-\rho}(1+|q|)^{-\frac{1}{2}+\delta}}
	\end{equation}
	and for $q\leq R+1$
	\begin{align}
	\label{iwintlu1}| Z^{N-7}  (\wht g_{UL}-(1-\chi(q))\sigma^0_{UL})|&\lesssim \frac{\ep(1+|q|)^{\frac{3}{2}}}{(1+s)^{\frac{3}{2}-\rho}},\\
	\label{iwintlu2}| Z^{N-8} (\wht g_{UL}-(1-\chi(q))\sigma^0_{UL})|&\lesssim \frac{\ep(1+|q|)}{(1+s)^{\frac{3}{2}-3\rho}}, \\
	\label{iwintlu3}| Z^{N-10} (\wht g_{UL}-(1-\chi(q))\sigma^0_{UL})|&\lesssim \frac{\ep(1+|q|)}{(1+s)^{\frac{3}{2}-\rho}}, \\
	\label{iwintlu4}| Z^{N-11}\wht g_{UL}|&\lesssim \frac{\ep(1+|q|)^{\frac{1}{2}+\rho}}{1+s}.
	\end{align}
\end{cor}
\begin{proof}
	By integrating \eqref{wextlu} in the exterior, since $\delta>\frac{1}{2}$, we obtain \eqref{iwextlu}.
	For $I\leq N-7$, we have 
	$$|\partial_q(Z^I \wht g_{UL} +\chi(q)\sigma^0_{UL})|\lesssim \frac{\ep(1+|q|)^{\frac{1}{2}}}{(1+s)^{\frac{3}{2}-\rho}},$$
	and consequently
	\begin{equation}
	\label{aintegrer}|\partial_q(Z^I \wht g_{UL} +(\chi(q)-1)\sigma^0_{UL})|\lesssim \frac{\ep(1+|q|)^{\frac{1}{2}}}{(1+s)^{\frac{3}{2}-\rho}}.
	\end{equation}
	For $q=R+1$ we can estimate, thanks to \eqref{iwextlu} and the fact that $\chi(q)-1=0$,
	$$|Z^I (\wht g_{UL} +(\chi(q)-1)\sigma^0_{UL})|\lesssim \frac{\ep}{(1+s)^\frac{3}{2}}.$$
	Consequently, integrating \eqref{aintegrer} from $q=R+1$ yields \eqref{iwintlu1}. We obtain \eqref{iwintlu2} and \eqref{iwintlu3} in the same way. Estimate \eqref{iwintlu4} is a direct consequence of \eqref{iwintlu3} and \eqref{estg0}.
\end{proof}

We now give the $L^2$ estimates for the good coefficients of the metric
\begin{prp}\label{l2lu}
	We have 
	\begin{align}
	\label{wl2dlun}&\int_0^t\left\| w_2'(q)^\frac{1}{2} \partial_q Z^N (\wht g_{UL}+\chi(q)\sigma^0_{UL})\right\|^2_{L^2}ds \lesssim \ep^2(1+t)^{2\rho},\\
		\label{hwl2dlun}&\int_0^t\left\| \frac{w_2'(q)^\frac{1}{2}}{1+|q|} Z^N (\wht g_{UL}+\chi(q)\sigma^0_{UL})\right\|^2_{L^2}ds \lesssim \ep^2(1+t)^{2\rho},\\
	\label{wl2dul}&\left\| \frac{w_2(q)^\frac{1}{2}}{(1+|q|)} \partial_q Z^{N-1} (\wht g_{UL}+\chi(q)\sigma^0_{UL})\right\|_{L^2}\lesssim \frac{\ep}{(1+t)^{1-\rho}},\\
	\label{wl2dulbis}&\left\| \frac{w_2(q)^\frac{1}{2}}{(1+|q|)^{\frac{1}{4}+\sigma}} \partial_q Z^{N-1} \wht g_{UL}\right\|_{L^2}\lesssim \frac{\ep}{(1+t)^\frac{1}{4}},\\
	\label{hwl2dul}&\left\| \frac{w_2(q)^\frac{1}{2}}{(1+|q|)^2}  Z^{N-1} (\wht g_{UL}+(1-\chi(q))\sigma^0_{UL})\right\|^2_{L^2}\lesssim \frac{\ep}{(1+t)^{1-\rho}},\\
	\label{hwl2dulbis}&\left\| \frac{w_1(q)^\frac{1}{2}}{(1+|q|)^{\frac{5}{4}+\sigma}}  Z^{N-1} \wht g_{UL}\right\|_{L^2}\lesssim \frac{\ep}{(1+t)^{\frac{1}{4}}}.
	\end{align}
\end{prp}
\begin{proof}
	Thanks to Proposition \ref{estLU} we have
	$$|\partial_q Z^N(\wht g_{UL}+\chi(q)\sigma^0_{UL})|\lesssim |\bar{\partial} Z^N \wht g_1|+\frac{1}{1+s}|Z^N \wht g_1|.$$
	Thanks to \eqref{igN} we have
	$$\int_0^t\left\| w_2'(q)^\frac{1}{2} \bar{\partial} Z^N \wht g_1 \right\|^2_{L^2}ds \lesssim \ep^2(1+t)^{2\rho},$$
	and thanks to \eqref{hg1n} we have
	\begin{align*}
	&\int_0^t\left\| \frac{w_2'(q)^\frac{1}{2}}{1+s}  Z^N \wht g_1\right\|^2_{L^2}ds \\
	&\lesssim \int_0^t\frac{1}{(1+s)}\left\| \frac{w_2(q)^\frac{1}{2}}{1+|q|}  Z^N \wht g_1\right\|^2_{L^2}ds\\ 
	&\lesssim \int_0^t \frac{\ep^2}{(1+s)^{1-2\rho}} \lesssim \ep^2(1+t)^{2\rho},
	\end{align*}
	which concludes the proof of \eqref{wl2dlun}. Then \eqref{hwl2dlun} is a consequence of \eqref{wl2dlun} and Hardy inequality.
	To prove \eqref{wl2dul} we write
	$$|\partial_q Z^{N-1}(\wht g_{UL}+\chi(q)\sigma^0_{UL})|\lesssim \frac{1}{1+s}|Z^N \wht g_1|,$$
	so \eqref{wl2dul} is a direct consequence of \eqref{hg1n}. 
	We prove \eqref{wl2dulbis}. We have
	$$\ld{\chi'(q)Z^{N-1}\sigma^0_{UL}}^2 \lesssim \int \chi'(q)^2\|s\partial_s Z^{N-1}b\|^2_{L^2(\m S^1)}rdr
	\lesssim \frac{\ep^4}{(1+s)^\frac{1}{2}},$$
	thanks to \eqref{estb3},
	and thanks to \eqref{hg2n}
	$$\ld{\frac{w_1(q)^\frac{1}{2}}{(1+|q|)^{\frac{1}{4}+\sigma}(1+s)}Z^N \wht g}\lesssim \frac{1}{(1+t)^{\frac{1}{4}+\sigma}}
	\ld{\frac{w_1(q)^\frac{1}{2}}{1+|q|}Z^N \wht g}\lesssim \frac{\ep^2}{(1+t)^\frac{1}{4}}.$$ 
	This concludes the proof of \eqref{wl2dulbis}.
	Estimates \eqref{hwl2dul} and \eqref{hwl2dulbis} are a direct consequence of the weighted Hardy inequality and \eqref{wl2dul} and \eqref{wl2dulbis}.
\end{proof}

\subsection{Estimates for the good metric coefficient $g_{UU}$}

\subsubsection{Estimates for $G^L$}

\begin{prp}\label{prpG}
	We have the estimates
	\begin{align}
	\label{estG5}&\|rZ^{N-5} G^L\|_{L^2(\m S^1)}\lesssim \frac{\ep^2}{\sqrt{1+t}}+\frac{\ep^2}{(1+|q|)^{1-4\rho}}+\Delta_h,\\
	\label{estG1}&\|rZ^{N-1} G^L\|_{L^2(\m S^1)} \lesssim \frac{\ep^2}{(1+t)^{\frac{1}{2}-\rho}}+\frac{\ep^2}{(1+|q|)^{1-4\rho}}+\Delta_h,\\
	\label{estGN}&\|rZ^{N}G^L\|_{L^2(\m S^1)}\lesssim \ep^2(1+t)^{\rho}+\Delta_h,\\
	\label{estGN1}&\int_0^t(1+\tau) \|r\partial_s Z^N G^L\|^2_{L^2(\m S^1)} d\tau \lesssim (\ep^2+\Delta_h)(1+t)^{2\rho}.
	\end{align}
\end{prp}

\begin{proof}
We have
$$G^L = \frac{1}{r}\Up\left(\frac{r}{t}\right)\int_{\infty}^r\left(2(\partial_q \phi)^2r-h(\theta,2t)\partial_q^2(q\chi(q))\right)dr.$$
Thanks to Proposition \ref{cone} we have $\phi=0$ for $q\geq R+\frac{1}{2}$ so
$G^L = 0$ for $q>R+1$, and we have
$$ G^L = \frac{1}{r}\Up\left(\frac{r}{t}\right)\left(\int_{\infty}^02(\partial_q \phi)^2rdr+h(\theta,2t)
+\int_{0}^r\left(2(\partial_q \phi)^2r-h(\theta,2t)\partial_q^2(q\chi(q))\right)dr\right).$$
We have the estimate, thanks to \eqref{esthbis} (more precisely thanks to Lemma \ref{lmh})
$$\|rG^L\|_{H^{N-5}(\m S^1)}\\
\lesssim \frac{\ep^2}{\sqrt{1+t}}+\Delta_h +\Upr\int_0^r\left( \|(\partial_q \phi)^2\|_{H^{N-5}(\m S^1)}r +\|h(\theta,2t)\|_{H^{N-5}(\m S^1)}\partial_q^2(q\chi(q)) \right)dr.$$
We estimate
\begin{align*}\int_0^r \|(\partial_q \phi)^2\|_{H^{I}(\m S^1)}r dr
&\lesssim \int_0^r \|\partial_\theta^{\frac{I}{2}}\partial_q \phi\|_{L^\infty(\m S^1)}\|\partial_q \phi\|_{H^I(\m S^1)}rdr \\
&\lesssim \int_0^r \frac{\ep r}{\sqrt{1+s}(1+|q|)^{\frac{3}{2}-4\rho}}\|\partial_q Z^{I}\phi\|_{L^2(\m S^1)}\\
&\lesssim \ep  \ld{\partial_q Z^{I} \phi}\left(\int_0^r \frac{1}{(1+|q|)^{3-8\rho}}dr\right)^{\frac{1}{2}}\\
&\lesssim \frac{\ep}{(1+|q|)^{1-4\rho}}\ld{\partial_q Z^{I} \phi}
\end{align*}
where we have used \eqref{bootphi1} to estimate 
$$|\partial^{\frac{I}{2}}_\theta \partial_q \phi|\lesssim \frac{1}{1+|q|}|Z^{\frac{I}{2}+1}\phi|\lesssim \frac{\ep^2}{(1+|q|)^{\frac{3}{2}-4\rho}(1+s)^\frac{1}{2}},$$
for $\frac{I}{2}\leq \frac{N}{2}\leq N-9$.
We obtain
$$\|rG^L\|_{H^{N-5}(\m S^1)}
\lesssim \frac{\ep^2}{\sqrt{1+t}}+\Delta_h+\Upr \frac{\ep}{(1+|q|)^{1-4\rho}}\ld{\partial_q Z^{I} \phi} +\Upr\|Z^{N-5} b(\theta,t)\|_{H^2(\m S^1)}\ch_{q>R},$$
and so thanks to \eqref{phiN1} and \eqref{estb2}
\begin{equation}
\label{tgl1}\|rG^L\|_{H^{N-5}(\m S^1)}\lesssim \frac{\ep^2}{\sqrt{1+t}}+\frac{\ep^2}{(1+|q|)^{1-4\rho}}+\Delta_h
\end{equation}
Similarly, thanks to \eqref{esth} we have
\begin{equation}
\label{tgl1bis}\|rG^L\|_{H^{N-1}(\m S^1)}\lesssim \frac{\ep^2}{(1+t)^{\frac{1}{2}-\rho}}+\frac{\ep^2}{(1+|q|)^{1-4\rho}}+\Delta_h
\end{equation}
and thanks to \eqref{phiN} and \eqref{estb4}
\begin{equation}
\label{tgl}\|rG^L\|_{H^{N}(\m S^1)}\lesssim \ep^2(1+t)^\rho+\Delta_h.
\end{equation}
We now look at the derivatives with respect to $r$ and $t$.
$$\partial_r G^L = O\left(\frac{1}{1+s}G^L\right)
+\Up\left(\frac{r}{t}\right)2(\partial_q \phi)^2-\Upr\frac{h(\theta,2t)\partial_q^2(q\chi(q))}{r},$$
\begin{align*}
\partial_t G^L =& O\left(\frac{1}{1+s}G^L\right)
+\frac{1}{r}\Upr\int_{\infty}^r \partial_t\left(2(\partial_q \phi)^2r-h(\theta,2t)\partial_q^2(q\chi(q))\right)dr\\
=&O\left(\frac{1}{1+s}G^L\right)
+\frac{1}{r}\Upr\int_{\infty}^r (2\partial_s -\partial_r)\left(2(\partial_q \phi)^2r-h(\theta,2t)\partial_q^2(q\chi(q))\right)dr\\
=&O\left(\frac{1}{1+s}G^L\right)-\Up\left(\frac{r}{t}\right)2(\partial_q \phi)^2+\Upr\frac{h(\theta,2t)\partial_q^2(q\chi(q))}{r}\\
&+\frac{1}{r}\Upr\int_{\infty}^r 2\partial_s\left(2(\partial_q \phi)^2r-h(\theta,2t)\partial_q^2(q\chi(q))\right)dr.
\end{align*}
Consequently we have
\begin{equation}\label{dqG}\begin{split}\partial_q G^L=&\Up\left(\frac{r}{t}\right)2(\partial_q \phi)^2-\Upr\frac{h(\theta,2t)\partial_q^2(q\chi(q))}{r}\\
&+  O\left(\frac{1}{1+s}G^L\right)
-\frac{1}{r}\Upr\int_{\infty}^r \partial_s\left(2(\partial_q \phi)^2r-h(\theta,2t)\partial_q^2(q\chi(q))\right)dr,
\end{split}
\end{equation}
$$\partial_s G^L =O\left(\frac{1}{1+s}G^L\right)
+\frac{1}{r}\Upr\int_{\infty}^r \partial_s\left(2(\partial_q \phi)^2r-h(\theta,2t)\partial_q^2(q\chi(q))\right)dr.$$
We calculate
\begin{align*}
\partial_s (r(\partial_q \phi)^2)=&2r\partial_q \phi\partial_s\partial_q \phi+\frac{1}{2}(\partial_q \phi)^2\\
=&\frac{1}{2}\partial_q\phi \Box \phi -\partial_q \phi \left( \partial_s \phi +\frac{1}{r}\partial^2_\theta \phi\right)\\
=&\frac{1}{2}\partial_q \phi \left( (\Box \phi-\Box_g \phi)r -\partial_s \phi -\frac{1}{r}\partial^2_\theta \phi\right).\\
\end{align*}
We have
$$Z^I(\partial_q \phi \partial_s \phi)= \sum_{I_1+I_2\leq I} \partial_q Z^{I_1 }\phi\bar{\partial}Z^{I_2}\phi,$$
If $I_1\leq \frac{N}{2}\leq N-10$ we estimate thanks to \eqref{bootphi1} and \eqref{important}
$$|\partial_q Z^{I_1}\phi|\lesssim \frac{\ep}{(1+|q|)^{\frac{3}{2}-4\rho}(1+s)^\frac{1}{2}},$$
and if $I_2\leq \frac{N}{2}\leq N-10$ we estimate
$$|\bar{\partial} Z^{I_2}\phi|\lesssim\frac{1}{1+s}|Z^{I_2+1}\phi|\lesssim \frac{\ep}{(1+|q|)^{\frac{1}{2}-4\rho}(1+s)^\frac{3}{2}},$$
and so
$$|Z^I(\partial_q \phi \partial_s \phi)|\lesssim \frac{\ep}{(1+|q|)^{\frac{3}{2}-4\rho}(1+s)^\frac{1}{2}}|\bar{\partial} Z^I \phi|+
\frac{\ep}{(1+|q|)^{\frac{1}{2}-4\rho}(1+s)^\frac{3}{2}}|\partial Z^I \phi|.$$
We estimate $Z^I \left(\partial_q \phi \frac{1}{r}\partial^2_\theta \phi\right)$ in the same way.
For the estimate of $Z^I(\Box \phi-\Box_g \phi)$ we refer to Proposition \ref{strphi}. We obtain 
\begin{align*}|Z^I \partial_s (r(\partial_q \phi)^2)|
\lesssim &\frac{\ep}{(1+|q|)^{\frac{3}{2}-4\rho}(1+s)^\frac{1}{2}}|\bar{\partial} Z^{I+1} \phi|
+\frac{\ep}{(1+s)^\frac{3}{2}(1+|q|)^{\frac{1}{2}-4\rho}}|\partial Z^I \phi|\\
&+\frac{\ep\sqrt{1+s}}{(1+|q|^{\frac{3}{2}-4\rho}}\Bigg(\frac{\ep}{\sqrt{1+s}(1+|q|)^{\frac{5}{2}-4\rho}}| Z^I \wht g_{LL}|
+\frac{\ep}{(1+s)^\frac{3}{2}(1+|q|)^{\frac{3}{2}-4\rho}}| Z^I \wht g|\\
&+\frac{\ep(1+|q|)^\frac{1}{2}}{(1+s)^{\frac{3}{2}}}|\partial Z^{I+1} \phi|
+\frac{\ep}{(1+s)^\frac{3}{2}(1+|q|)^{\frac{1}{2}-4\rho}}| Z^I G^L|
+\frac{\ep}{(1+s)\sqrt{1+|q|}}|\bar{\partial}Z^I \phi|\Bigg).
\end{align*}
We estimate, for $I\leq N-1$
\begin{align*}\int \left\|\frac{\ep}{(1+|q|)^{\frac{3}{2}-4\rho}(1+s)^\frac{1}{2}}\bar{\partial} Z^{I+1} \phi\right\|_{L^2(\m S^1)}dr
&\lesssim \int \frac{\ep}{(1+|q|)^{\frac{3}{2}-4\rho}(1+s)^\frac{3}{2}}\|Z^{I+2}\phi\|_{L^2(\m S^1)}dr\\
&\lesssim \ld{\frac{w^{\frac{1}{2}}}{1+|q|}Z^{I+2} \phi}\left(\int \frac{\ep^2}{(1+s)^{4}(1+|q|)^{1-8\rho}}dr\right)^\frac{1}{2}\\
&\lesssim \frac{\ep}{(1+s)^{2-5\rho}}\ld{\frac{w^\frac{1}{2}}{1+|q|}Z^{I+2} \phi},
\end{align*}
and for $I\leq N$ ,
$$\int \left\|\frac{\ep}{(1+|q|)^{\frac{3}{2}-4\rho}(1+s)^\frac{1}{2}}\bar{\partial} Z^{I+1} \phi\right\|_{L^2(\m S^1)}dr
\lesssim  \frac{\ep}{(1+s)}\ld{w'(q)^\frac{1}{2}\bar{\partial}Z^{I+1} \phi}$$
$$
\int \left\|\frac{\ep}{(1+s)^\frac{3}{2}(1+|q|)^{\frac{1}{2}-4\rho}}\partial Z^I \phi\right\|_{L^2(\m S^1)}dr
\lesssim \frac{\ep}{(1+s)^{2-5\rho}}\ld{\partial Z^{I} \phi},
$$
\begin{align*}&\int \left\|\frac{\ep\sqrt{1+s}}{(1+|q|)^{\frac{3}{2}-\rho}}\frac{\ep}{\sqrt{1+s}(1+|q|)^{\frac{5}{2}-4\rho}} Z^I \wht g_{LL}\right\|_{L^2(\m S^1)}dr\\
&\lesssim \ld{\frac{w_2^\frac{1}{2}}{(1+|q|)^\frac{3}{2}}Z^I \wht g_{LL}}\left( \int\frac{\ep^4}{(1+s)(1+|q|)^{4-10\rho-2\sigma}}\right)^\frac{1}{2}dr\\
&\lesssim \frac{\ep^2}{(1+s)^\frac{1}{2}}\ld{\frac{w_2^\frac{1}{2}}{(1+|q|)^\frac{3}{2}}Z^I \wht g_{LL}},
\end{align*}
\begin{align*}
\int \left\|\frac{\ep\sqrt{1+s}}{(1+|q|)^{\frac{3}{2}-4\rho}}\frac{\ep(1+|q|)^\frac{1}{2}}{(1+s)^{\frac{3}{2}}}\partial Z^{I+1} \phi\right\|_{L^2(\m S^1)}dr&\lesssim
 \ld{\partial Z^{I+1} \phi}\left(\int \frac{\ep^4}{(1+s)^{3}(1+|q|)^{2-8\rho}}dr\right)^\frac{1}{2}\\
&\lesssim \frac{\ep^2}{(1+s)^\frac{3}{2}}\ld{\partial Z^{I+1} \phi},
\end{align*}
$$
\int \left\|\frac{\ep\sqrt{1+s}}{(1+|q|)^{\frac{3}{2}-4\rho}}\frac{\ep}{(1+s)^\frac{3}{2}(1+|q|)^{\frac{1}{2}-4\rho}} Z^I G^L\right\|_{L^2}dr
\lesssim \frac{\ep^2}{(1+s)^2}\sup_{r \in \m R} \| rZ^I G^L\|_{L^2(\m S^1)}.
$$
Consequently we can estimate for $I\leq N-1$
\begin{align*}
&\left\|Z^I \int_{\infty}^r \partial_s\left((\partial_q \phi)^2r-h(\theta,2t)\partial_q^2(q\chi(q))\right)dr\right\|\\
&\lesssim \frac{\ep}{(1+s)^{2-5\rho}}\ld{\frac{w^\frac{1}{2}}{1+|q|}Z^{I+2} \phi}+\frac{\ep}{(1+s)^{2-5\rho}}\ld{\partial Z^{I} \phi}
+\frac{\ep^2}{(1+s)^\frac{1}{2}}\ld{\frac{w_2^\frac{1}{2}}{(1+|q|)^\frac{3}{2}}Z^I \wht g_{LL}}\\
&+\frac{\ep^2}{(1+s)^\frac{3}{2}}\ld{\partial Z^{I+1} \phi}+\frac{\ep}{(1+s)^2}\sup_{r\in \m R} \| rZ^I G^L\|_{L^2(\m S^1)}+\|\partial_s Z^Ih\|_{L^2(\m S^1)},\\
\end{align*}
and therefore, for $I\leq N-5$, thanks to \eqref{hphin1}, \eqref{phiN1}, \eqref{hwl2dll2} and \eqref{estb3} we have
\begin{equation}\label{sgl25}
\|r(s\partial_s Z^{N-5} G^L)\|_{L^2(\m S^1)} \lesssim \|rZ^{N-5} G^L\|_{L^2(\m S^1)} +\frac{\ep^2}{(1+s)^{\frac{1}{2}}}.
\end{equation}
For $I\leq N-2$, thanks to \eqref{hphin}, \eqref{phiN}, \eqref{hwl2dll} and \eqref{estb3} we have
\begin{equation}\label{sgl2}
\|r(s\partial_s Z^{N-2} G^L)\|_{L^2(\m S^1)} \lesssim \|rZ^{N-2} G^L\|_{L^2(\m S^1)} +\frac{\ep^2}{(1+s)^{\frac{1}{2}-\rho}},
\end{equation}
and for $I\leq N-1$, thanks to \eqref{hphinn}, we have
\begin{equation}\label{sgl1}
\|r(s\partial_s Z^{N-1} G^L)\|_{L^2(\m S^1)} \lesssim \|rZ^{N-1} G^L\|_{L^2(\m S^1)} +\frac{\ep^2}{(1+s)^{\frac{1}{4}}}.
\end{equation}
Moreover, for $I\leq N$ we have
\begin{equation}\label{sgl}\begin{split}
\|r(\partial_s Z^N G^L)\|_{L^2(\m S^1)} \lesssim &\frac{1}{1+s}\|rZ^N G^L\|_{L^2(\m S^1)} 
+\frac{\ep}{1+s}\ld{w'(q)\bar{\partial}Z^{N+1} \phi}\\
&+\frac{\ep^2}{(1+s)^\frac{1}{2}}\ld{\frac{w_2^\frac{1}{2}}{(1+|q|)^\frac{3}{2}}Z^N \wht g_{LL}}
+\frac{\ep^2}{(1+s)^\frac{3}{2}}\ld{\partial Z^{N+1} \phi}+\|\partial_s Z^Nh\|_{L^2(\m S^1)} .
\end{split}
\end{equation}
We now estimate
\begin{equation}
\label{qgl}\begin{split} &rq\left\|Z^I\left(\Up\left(\frac{r}{t}\right)2(\partial_q \phi)^2-\Upr\frac{h(\theta,2t)\partial_q^2(q\chi(q))}{r}\right)\right\|_{L^2(\m S^1)}\\
&\lesssim \frac{\ep\sqrt{1+s}}{(1+|q|)^{\frac{1}{2}-4\rho}}\|\partial_q Z^I \phi\|_{L^2(\m S^1)}
+\ch_{R<q<R+1}\|Z^I h\|_{L^2(\m S^1)}.
\end{split}
\end{equation}
Consequently we have proved, for $I\leq N-6$, thanks to \eqref{tgl1}, \eqref{sgl25} and \eqref{qgl}
\begin{align*}
\|rZ^{I+1}G^L\|_{L^2(\m S^1)} &\lesssim 
\frac{\ep^2}{(1+t)^{\frac{1}{2}}}+\frac{\ep^2}{(1+|q|)^{1-4\rho}}+\Delta_h+
\|rZ^I G^L\|_{L^2(\m S^1)}\\
&+\frac{\ep\sqrt{1+s}}{(1+|q|)^{\frac{1}{2}-4\rho}}\|\partial_q Z^I \phi\|_{L^2(\m S^1)}
+\ch_{R<q<R+1}\|Z^I h\|_{L^2(\m S^1)}.
\end{align*}
We can estimate $\ch_{R<q<R+1}\|Z^I h\|_{L^2(\m S^1)}$ thanks to Corollary \ref{estih}. Thanks to the term $\ch_{R<q<R+1}$ it has as much decay in $q$ as we want.
By recurrence we have
\begin{equation}
\label{gln5}
\|rZ^{N-5}G^L\|_{L^2(\m S^1)}\lesssim \frac{\ep^2}{\sqrt{1+t}}+\frac{\ep^2}{(1+|q|)^{1-4\rho}}+\Delta_h+\frac{\ep\sqrt{1+s}}{(1+|q|)^{\frac{1}{2}-4\rho}}\|\partial_q Z^I \phi\|_{L^2(\m S^1)}.
\end{equation}
In the same way, thanks to \eqref{tgl1bis}, \eqref{sgl2} and \eqref{qgl} we obtain 
\begin{equation}
\label{gln1}
\|rZ^{N-1}G^L\|_{L^2(\m S^1)}\lesssim \frac{\ep^2}{(1+t)^{\frac{1}{2}-2\rho}}+\frac{\ep^2}{(1+|q|)^{1-4\rho}}+\Delta_h+\frac{\ep\sqrt{1+s}}{(1+|q|)^{\frac{1}{2}-4\rho}}\|\partial_q Z^I \phi\|_{L^2(\m S^1)}.
\end{equation}
Thanks to \eqref{tgl}, \eqref{qgl} and \eqref{sgl1} we obtain
\begin{equation}
\label{gln2}\|rZ^{N}G^L\|_{L^2(\m S^1)}\lesssim\ep^2(1+t)^{\rho} +\Delta_h.
\end{equation}
Moreover, thanks to \eqref{sgl} we obtain
\begin{align*}
\int_0^t(1+\tau)\|r\partial_s Z^N G^L\|^2_{L^2(\m S^1)}d\tau\lesssim & \int_0^t\left(\frac{\ep^2+\Delta_h}{(1+\tau)^{1-2\rho}}+\frac{\ep}{1+\tau}\ld{w'(q)\bar{\partial}Z^{N+1} \phi}^2\right)d\tau\\
&+\int_0^t \left(\ld{\frac{w_2^\frac{1}{2}}{(1+|q|)^\frac{3}{2}}Z^N \wht g_{LL}}^2+(1+\tau)\|\partial_s Z^Nh\|_{L^2(\m S^1)}^2 \right)d\tau
\end{align*}
and consequently
\begin{equation*}
\int_0^t(1+\tau)\|r\partial_s Z^N G^L\|_{L^2(\m S^1)}d\tau \lesssim (\ep^2+\Delta_h)(1+t)^\rho.
\end{equation*}
\end{proof}

\subsubsection{Estimates for $g_{UU}$}

\begin{cor}\label{estuu}
We have the estimates for $q>R+1$ 
\begin{equation}
\label{wextuu}
|\partial Z^{N-7} \wht g_{UU}|\lesssim \frac{\ep}{(1+s)^{\frac{3}{2}-\rho}(1+|q|)^{\frac{1}{2}+\delta}}
\end{equation}
and for $q\leq R+1$
\begin{align}
\label{wintuu1}|\partial_q Z^{N-7}  \wht g_{UU}|&\lesssim \frac{\ep(1+|q|)^{\frac{1}{2}}}{(1+s)^{\frac{3}{2}-\rho}}+\frac{\ep}{(1+s)(1+|q|)^{1-4\rho}},\\
\label{wintuu2}|\partial_q Z^{N-8} \wht g_{UU}|&\lesssim \frac{\ep}{(1+s)^{\frac{3}{2}-3\rho}}+\frac{\ep}{(1+s)(1+|q|)^{1-4\rho}} ,\\
\label{wintuu3}|\partial_q Z^{N-10} \wht g_{UU}|&\lesssim \frac{\ep}{(1+s)^{\frac{3}{2}-\rho}}+\frac{\ep}{(1+s)(1+|q|)^{1-4\rho}}.
\end{align}
\end{cor}

\begin{proof}
	Thanks to Proposition \ref{estLU} we have
	$$|\partial_q Z^I \wht g_{UU}|\lesssim \frac{1}{1+s}|Z^{I+1}\wht g|+ |Z^I G^L|.$$
	Consequently, \eqref{iks6} yields \eqref{wextuu}.
	Thanks to \eqref{estG5}, the Sobolev embedding $H^{1}(\m S^1)\subset L^\infty$ and the estimate \eqref{estdelta} for $\Delta_h$ we obtain
	\begin{equation}
	\label{estGinf}
	|Z^{N-6}G^L|\lesssim \frac{\ep}{(1+s)^\frac{3}{2}}+ \frac{\ep}{(1+s)(1+|q|)^{1-4\rho}}.
	\end{equation}
Consequently \eqref{iks3} yields \eqref{wintuu1}, \eqref{bootg2} yields \eqref{wintuu2} and \eqref{bootg1} yields \eqref{wintuu3}.
\end{proof}
Thanks to Lemma \ref{lmintegration}, since $\delta-\sigma >\frac{1}{2}$ we obtain the following corollary
\begin{cor}\label{iestuu}
We have the estimates for $q>R+1$ 
\begin{equation}
\label{iwextuu}
|Z^{N-7} \wht g_{UU}|\lesssim \frac{\ep}{(1+s)^{\frac{3}{2}-\rho}(1+|q|)^{-\frac{1}{2}+\delta}}
\end{equation}
and for $q\leq R+1$
\begin{align}
\label{iwintuu1}| Z^{N-7}  \wht g_{UU}|&\lesssim \frac{\ep(1+|q|)^{\frac{3}{2}}}{(1+s)^{\frac{3}{2}-\rho}}+\frac{\ep(1+|q|)^{4\rho}}{(1+s)},\\
\label{iwintuu2}| Z^{N-8} \wht g_{UU}|&\lesssim \frac{\ep(1+|q|)}{(1+s)^{\frac{3}{2}-3\rho}}+\frac{\ep(1+|q|)^{4\rho}}{(1+s)},\\
\label{iwintuu3}|  Z^{N-10} \wht g_{UU}|&\lesssim \frac{\ep(1+|q|)}{(1+s)^{\frac{3}{2}-\rho}}+\frac{\ep(1+|q|)^{4\rho}}{(1+s)}.
\end{align}
\end{cor}

We now give the $L^2$ estimates for $\wht g_{UU}$.
\begin{prp}\label{l2uu}
	We have 
	\begin{align}
	\label{wl2duun}
	&\int_0^t(1+\tau)^{-2\rho}\left\| w_1'(q)^\frac{1}{2} \partial Z^N \wht g_{UU}\right\|^2_{L^2}ds \lesssim \ep^2(1+t)^{2\rho}\\
	\label{hwl2duun}
		&\int_0^t(1+\tau)^{-2\rho}\left\| \frac{w_1'(q)^\frac{1}{2}}{1+|q|}  Z^N \wht g_{UU}\right\|^2_{L^2}ds \lesssim \ep^2(1+t)^{2\rho}\\
	\label{wl2duu}&\left\| \frac{w_1(q)^\frac{1}{2}}{(1+|q|)^{\frac{1}{2}+\sigma}} \partial Z^{N-1} \wht g_{UU}\right\|_{L^2}\lesssim \frac{\ep}{(1+t)^\frac{1}{2}},\\
	\label{hwl2duu}&\left\| \frac{w_1(q)^\frac{1}{2}}{(1+|q|)^{\frac{3}{2}+\sigma}} Z^{N-1} \wht g_{UU}\right\|_{L^2}\lesssim \frac{\ep}{(1+t)^\frac{1}{2}}.
\end{align}
\end{prp}
\begin{proof}
	Thanks to Proposition \ref{estLU} we have
	$$|\partial_q Z^{N} \wht g_{UU}|\lesssim |\bar{\partial}Z^N \wht g| +\frac{1}{1+s}|Z^{N}\wht g|+ |Z^{N} G^L|.$$
	We estimate
	\begin{align*}\int_0^t(1+\tau)^{-2\rho}\ld{w_1'(q)^\frac{1}{2}Z^{N} G^L}^2
	&\lesssim \int_0^t (1+\tau)^{-2\rho}\int \frac{1}{(1+|q|)^{1+2\sigma}r^2}\|rZ^N G^L \|^2_{L^2(\m S^1)}drdt\\
	&\lesssim \int_0^t \ep^2(1+\tau)^{-1}\\
	&\lesssim \ep^2(1+t)^{2\rho}.
	\end{align*}
	This estimate together with \eqref{ig2Nbis} yield \eqref{wl2duun}. Then the weighted Hardy inequality yields \eqref{hwl2duun}.
	Thanks to Proposition \ref{estLU} we have
	$$|\partial_q Z^{N-1} \wht g_{UU}|\lesssim \frac{1}{1+s}|Z^{N}\wht g|+ |Z^{N-1} G^L|.$$
	We can estimate
	$$\ld{\frac{w_1(q)^\frac{1}{2}}{(1+|q|)^{\frac{1}{2}+\sigma}} Z^{N-1}G^L}^2
	\lesssim \int \frac{1}{(1+|q|)^{1+\sigma}(1+s)}\|rZ^{N-1}G^L\|^2_{L^2}
	\lesssim \frac{\ep^2}{1+t},$$
	and
	$$\ld{ \frac{w_1(q)^\frac{1}{2}}{(1+|q|)^{\frac{1}{2}+\sigma}(1+s)}Z^N \wht g}\lesssim \frac{1}{(1+t)^{\frac{1}{2}+\sigma}}
	\ld{\frac{w_1(q)^\frac{1}{2}}{1+|q|}Z^N \wht g}\lesssim \frac{\ep}{(1+t)^\frac{1}{2}},$$
	where we have used \eqref{hg2n}, which concludes the proof of \eqref{wl2duu}. We obtain \eqref{hwl2duu} thanks to \eqref{wl2duu} and the weighted Hardy inequality.
\end{proof}
By taking at each time the maximum of the estimates from Corollary \ref{iestll}, \ref{iestul} and \ref{iestuu}, 
and estimating $Z^{N-7}\sigma^0_{UL}=O\left(\frac{\ep^2}{(1+s)^\frac{3}{4}}\right)$ thanks to \eqref{estsigma0}
we obtain the following
\begin{cor}
We have the estimates for $q>R+1$
\begin{equation}
\label{extinftt}| Z^{N-7}  \wht g_{\q T \q T}|\lesssim \frac{\ep}{(1+s)^{\frac{3}{2}-\rho}(1+|q|)^{\delta-\frac{1}{2}}},
\end{equation}
ad for $q\leq R+1$
\begin{align}
\label{inftt1}| Z^{N-7} \wht g_{\q T \q T}|&\lesssim \frac{\ep(1+|q|)^{\frac{3}{4}+\rho}}{(1+s)^{\frac{3}{4}}},\\
\label{inftt2}| Z^{N-8} \wht g_{\q T \q T}|&\lesssim \frac{\ep(1+|q|)^{\frac{1}{4}+3\rho}}{(1+s)^{\frac{3}{4}}},\\
\label{inftt3}| Z^{N-10} \wht g_{\q T \q T}|&\lesssim \frac{\ep(1+|q|)^{\frac{1}{4}+\rho}}{(1+s)^{\frac{3}{4}}}, \\
\label{inftt}|Z^{N-11} \wht g_{\q T\q T}|+(1+|q|)|\partial_q Z^I\wht g_{\q T \q T}| &\lesssim
\frac{(1+|q|)^{\frac{1}{2}+\rho}}{(1+s)}.
\end{align}

\end{cor}

\section{Structure of the equations}
In this section we will study each terms of $\Box_g Z^I g_{\mu \nu}$ in order to perform in the next sections the $L^\infty-L^\infty$ estimates and then the $L^2$ estimates. 

\begin{prp}\label{prpstr}
We can write $\Box_g Z^I g_{\mu \nu}=Z^IR^1_{\mu \nu}+{}^IM_{\mu \nu}+{}^IM^E_{\mu \nu}+{}^IQ_{\mu \nu}$, where $R^1$ is given by Proposition \ref{riccigb}, and
\begin{itemize}
	\item ${}^IM$ is present in the right-hand side of the wave equations satisfied by all the components. It consists principally of terms which have the null structure. It satisfies
\begin{align*}
|{}^IM|
&\lesssim
\frac{\ep}{(1+s)^\frac{3}{2}(1+|q|)^{\frac{1}{2}-4\rho}}|\partial Z^I \phi| +\frac{\ep}{(1+s)^\frac{1}{2}(1+|q|)^{\frac{3}{2}-4\rho}}|\bar{\partial} Z^I \phi|\\
&+\frac{\ep}{(1+s)(1+|q|)^{\frac{1}{2}-\rho}}\left(|\bar{\partial} Z^I \wht g|+\frac{1}{1+s}|Z^I \wht g|\right)\\
&+
\frac{\ep}{(1+s)^{\frac{3}{2}-\rho}}|\partial Z^I \wht g|
+\frac{\ep}{(1+s)(1+|q|)^{\frac{3}{2}-\rho}}|Z^I \wht g_{\q T \q T}|+\frac{1}{1+s}|Z^I G|+|\bar{\partial}Z^I G|\\
&+\ep\min\left(\frac{1}{(1+|q|)(1+s)^{\frac{1}{2}-\rho}},\frac{1}{(1+|q|)^{\frac{1}{2}}(1+s)^\frac{1}{2}}\right)\left(|\bar{\partial}Z^I \wht g_1|+\frac{1}{1+|q|}|Z^I g_{LL}|\right).\\
\end{align*}
\item We have better estimates in the exterior region $q>R$ so we isolate the contribution of this region by introducing a term ${}^IM^E_{\mu \nu}$ which is non zero only in the exterior region $q>R$. We also include in this term the crossed terms between $g_b$ and $\wht g$. ${}^IM^E_{\mu \nu}$ satisfies
\begin{align*}
|{}^IM^E|\lesssim &\frac{\ep}{(1+s)}\left(|\bar{\partial} Z^I \wht g|+ \frac{|q|}{1+s}|\partial Z^I \wht g|+ \frac{1}{1+s}|Z^I \wht g |+ \frac{1}{1+|q|}|\wht g_{\q T \q T}|\right)\\
&+\frac{\ep}{(1+s)(1+|q|)^{2+\delta-\rho}} \left(s|\partial_s Z^I b|
+q|\partial_s \partial_\theta Z^I b|+\frac{q}{s}|Z^I \partial^2_\theta b|\right)\\
&+\min\left(\frac{\ep}{(1+|q|)^{\frac{3}{2}+\delta}(1+s)^{\frac{1}{2}-\rho}}, \frac{\ep}{(1+|q|)^{\frac{3}{2}+\delta-\sigma}(1+s)^\frac{1}{2} }\right)\Big(s|\partial^2_s Z^I b|
+q|\partial^2_s \partial_\theta Z^I b|\\
&+\frac{q}{s}|Z^I \partial_s \partial^2_\theta b|
+\frac{q}{s^2}|Z^I \partial^3_\theta b|\Big).
\end{align*}
\item The terms which do not have the null structure are not presents in all the components of $\Box_g Z^I g_{\mu \nu}$. It is why we introduce ${}^IQ$ such that
$${}^IQ_{\q T \q T}=0, \quad {}^IQ_{L \ba L}=\partial_{\ba L}\wht g_{\ba L \ba L}\partial_L Z^I \wht g_{LL}, \quad 
{}^IQ_{UL}=\partial_{\ba L}\wht g_{\ba L \ba L}Z^I(\partial_U \wht g_{LL}+\partial_L (\wht g_{UL}+\sigma^0_{UL})),$$
$${}^I Q_{\ba L \ba L}=Z^I \left(\partial_q g_{UU}\partial_q \wht g_{\ba L  L}
+\wht g_{L \ba L }\partial_q G^L\right) + \sum_{I_1+I_2= I, \; I_2 \leq I-1} Z^{I_1}g_{LL}\partial_q^2 Z^{I_2}g_{\ba L \ba L},$$
so the new contributions involved in ${}^IQ_{\ba L \ba L}$  are
\begin{align*}
|{}^IQ_{\ba L  \ba L}|&\lesssim 
\frac{\ep}{(1+|q|)(1+s)^{\frac{1}{2}-\rho}}\left(|\bar{\partial}Z^I \wht g_1|+\frac{1}{1+|q|}|Z^I g_{LL}|\right)
+\left(\frac{\ep^2}{(1+s)^\frac{3}{2}}+\frac{\ep}{(1+s)(1+|q|)^{1-4\rho}}\right)|\partial Z^I \wht g_1|\\
+&\ep\min\left(\frac{1}{(1+|q|)(1+s)^{\frac{1}{2}-\rho}},\frac{1}{(1+|q|)^{\frac{1}{2}}(1+s)^\frac{1}{2}}\right)(|\bar{\partial}Z^I \wht g|+|Z^I G^L|)\\
+ & \ch_{q>R} \min\left(\frac{\ep}{(1+s)^{\frac{1}{2}-\rho}(1+|q|)^{\frac{5}{2}+\delta}}, \frac{\ep}{(1+s)^{\frac{1}{2}}(1+|q|)^{\frac{5}{2}+\delta+\sigma}}\right)\frac{q}{s}|Z^I \partial^2_\theta b|
+ \ch_{q>R}\frac{\ep}{1+s}|\partial Z^I \wht g_1|\\
+&\ch_{q>R}\frac{\ep}{(1+|q|)^{\frac{3}{2}+\delta}(1+s)^{\frac{1}{2}-\rho}}\left(s|\partial^2_s Z^I b|
+q|\partial^2_s \partial_\theta Z^I b|+
\frac{q}{s}|Z^I \partial_s \partial^2_\theta b|
+\frac{q}{s^2}|Z^I \partial^3_\theta b|\right).
\end{align*}
\end{itemize}
\end{prp}

\begin{proof}
	We can study the terms in $\Box_g Z^I g_{\mu \nu}$ with simple counting arguments :
the quadratic terms in $\Box_g Z^I g_{\mu \nu}$ are of the form
$$A_{--}=m^{--}m^{--}\partial_{-}Z^{I_1}g_{--}\partial_{-}Z^{I_2}g_{--}$$
or
$$B_{--}=m^{--}m^{--}Z^{I_1}(g-m)_{--}\partial_{-}\partial_{-}Z^{I_2}g_{-- }$$
with $I_1+I_2 \leq I$. The $_{-}$ and $^{-}$ stand for down and up indices. The indices $_{- -}$ in $A_{--}$ or $B_{--}$ appear as down indices in the right hand-side at any place a priori and the other down indices should appear with a repeated up index in $m^{--}$. Consequently in the additional down indices we can not have more than two occurrences of $\ba L$.
 With this technique it may happen that we study terms which are not in the equations but we are certain not to miss any. If some terms happen to be difficult to handle we will of course check if they are or not present in the equations.
\paragraph{The case $\q T \q T$.}
In this case there can not be more than two occurrences of the vector field $\ba L$.
In $A_{\q T \q T}$ : the terms involving two bad derivative are of the form
$$\partial_q Z^{I_1}g_{\q T \q T}\partial_{q} Z^{I_2}g_{\q T \q T}.$$
We may assume $I_1\leq I_2$.
We estimate first 
$$\partial_q Z^{I_1}\wht g_{\q T \q T}\partial_{q} Z^{I_2}\wht g_{\q T \q T}.$$
Thanks to \eqref{inftt}, since $I_1\leq \frac{N}{2}\leq N-11$ this is bounded by
$$\frac{\ep}{(1+s)(1+|q|)^{\frac{1}{2}-\rho}}|\partial_q  Z^I \wht g_{\q T \q T}|.$$
Thanks to Proposition \eqref{estLU} this term is bounded by
\begin{align}
\label{str1}&\frac{\ep}{(1+s)(1+|q|)^{\frac{1}{2}-\rho}}\left(|\bar{\partial} Z^I \wht g|+\frac{1}{1+s}|Z^I \wht g|\right),\\
\label{str2}&+\frac{\ep}{(1+s)(1+|q|)^{\frac{1}{2}-\rho}}|Z^IG|.
\end{align}
We estimate
$$\partial_q Z^{I_1} (g_{\mathfrak{b}})_{\q T \q T}\partial_{q} Z^{I_2}\wht g_{\q T \q T}.$$
This term is only present in the exterior and thanks to \eqref{estg0} this is bounded by
\begin{align}
\label{str3}&\ch_{q>R}\frac{\ep}{(1+s)}\left(|\bar{\partial} Z^I \wht g|+\frac{1}{1+s}|Z^I \wht g|\right),\\
\label{str4}&+\ch_{q>R}\frac{\ep}{(1+s)}|Z^I G|.
\end{align}
We now estimate
$$\partial_q Z^{I_1} \wht g_{\q T \q T}\partial_{q} Z^{I_2}(g_{\mathfrak{b}})_{\q T \q T}.$$
This term is only present in the exterior and is bounded by
\begin{equation}
\label{str5}
\ch_{q>R}\frac{\ep}{(1+s)(1+|q|)^{1+\delta-\rho}}\frac{1}{1+|q|} \left(s|\partial_s Z^I b|
+q|\partial_s \partial_\theta Z^I b|+\frac{q}{s}|Z^I \partial^2_\theta b|\right),
\end{equation}
where we have estimated $\partial_q Z^{I_1} \wht g_{\q T \q T}$ with \eqref{extinftt} in the region $q>R+1$ and \eqref{inftt} in the region $R\leq q \leq R+1$, and we have estimated $\partial_{q} Z^{I_2}(g_{\mathfrak{b}})_{\q T \q T}$ with \eqref{calcqg} and \eqref{calcg0}.
The terms involving only the metric $g_{\mathfrak{b}}$ are all taken into account in the Ricci tensor of $g_{\mathfrak{b}}$.

We decompose the terms of the form $\bar{\partial} Z^{I_1} g\partial_q Z^{I_2}  g$, with $I_1\leq I_2$ in three part
$$\bar{\partial} Z^{I_1} \wht g\partial_q Z^{I_2}  \wht g, \quad \bar{\partial} Z^{I_1} \wht g\partial_q Z^{I_2}  g_{\mathfrak{b}}, \quad \bar{\partial} Z^{I_1} g_{\mathfrak{b}}\partial_q Z^{I_2} \wht g.$$
They can be estimated respectively by
\begin{align}
\label{str6}& \frac{\ep}{(1+s)^{\frac{3}{2}-\rho}}|\partial_q Z^{I} \wht g|,\\
\label{str7} &\ch_{q>R}\frac{\ep}{(1+s)^{\frac{3}{2}-\rho}(1+|q|)^{\frac{1}{2}+\delta}}\frac{1}{1+|q|} \left(s|\partial_s Z^I b|
+q|\partial_s \partial_\theta Z^I b|+\frac{q}{s}|Z^I \partial^2_\theta b|\right),\\
\label{str8} & \ch_{q>R} \frac{\ep(1+|q|)}{(1+s)^2}|\partial_q Z^{I_2} \wht g|,
\end{align} 
where in \eqref{str6} we have used
$$|\bar{\partial }Z^{I_1}\wht g|\lesssim \frac{1}{1+s}|Z^{I_1+1}\wht g|\lesssim \frac{\ep}{(1+s)^{\frac{3}{2}-\rho}},$$
thanks to \eqref{bootg1}. In \eqref{str7} we have used
$$|\bar{\partial }Z^{I_1}\wht g|\lesssim \frac{1}{1+s}|Z^{I_1+1}\wht g|\lesssim \frac{\ep}{(1+s)^{\frac{3}{2}-\rho}(1+|q|)^{\frac{1}{2}+\delta}},$$
thanks to \eqref{iks2}, and we have estimated $\partial_q Z^{I_2}  g_{\mathfrak{b}}$ thanks to \eqref{calcg0}, and \eqref{calcqg}. In \eqref{str8} we have estimated
$$|\bar{\partial }Z^{I_1}g_{\mathfrak{b}}|\lesssim \frac{1}{1+s}|Z^{I_1+1}g_{\mathfrak{b}}|\lesssim \frac{\ep(1+|q|)}{(1+s)^2},$$
thanks to \eqref{estg0}, with $I_1+1\leq \frac{N}{2}+1 \leq N-11$.

The remaining terms are of one of the following form 
$$\partial_{\ba L} Z^{I_1}g_{\ba L \q T}\bar{\partial} Z^{I_2}g_{\q T \q T}, \quad
\partial_{\ba L} Z^{I_1}g_{\q T \q T}\bar{\partial} Z^{I_2}g_{\ba L \q T}.$$
We estimate the first one, beginning with $\partial_q Z^{I_1}\wht g_{\ba L \q T} \bar{\partial} Z^{I_2}\wht g_{\q T \q T}$ with $I_1\leq I_2$ : it can be estimated in two different way, according that we use \eqref{bootg1} or \eqref{ks2bis}
\begin{equation}
\label{str9} \frac{\ep}{(1+|q|)(1+s)^{\frac{1}{2}-\rho}}|\bar{\partial}Z^I \wht g_{\q T \q T}|,
\end{equation}
\begin{equation}
\label{str9bis} \frac{\ep}{(1+|q|)^{\frac{1}{2}}(1+s)^{\frac{1}{2}}}|\bar{\partial}Z^I \wht g_{\q T \q T}|.
\end{equation}
The term $\partial_{\ba L} Z^{I_1}(g_{\mathfrak{b}})_{\ba L \q T}\bar{\partial} Z^{I_2}\wht g_{\q T \q T} $ gives the contributions \eqref{str3} and \eqref{str4}. The term $\partial_{\ba L}Z^{I_1}\wht g_{\ba L \q T} \bar{\partial} Z^{I_2}g_{\mathfrak{b}} $ gives the contribution
\begin{equation}
\label{str10}
\begin{split}
&\ch_{q>R}\min\left(\frac{\ep}{(1+|q|)^{\frac{3}{2}+\delta}(1+s)^{\frac{1}{2}-\rho}},
\frac{\ep}{(1+|q|)^{\frac{3}{2}+\delta-\sigma}(1+s)^{\frac{1}{2}}}\right)\\
&\left(s|\partial^2_s Z^I b|+
q|\partial^2_s \partial_\theta Z^I b|+
\frac{q}{s}|Z^I \partial_s \partial^2_\theta b|
+ |\partial_s \partial_\theta Z^I b|
+\frac{q}{s^2}|Z^I \partial^3_\theta b|\right),
\end{split}
\end{equation} 
where we have estimated $\partial_q Z^{I_1}\wht g_{L\q T}$ thanks to \eqref{iks5} or \eqref{iks6}, and we have estimated $\bar{\partial} Z^{I_2}g_{\mathfrak{b}}$ thanks to \eqref{calcsg}.
Thanks to the estimates \eqref{inftt} and \eqref{extinftt}
the term $\partial_{\ba L} Z^{I_1}g_{\q T \q T}\bar{\partial} Z^{I_2}g_{\ba L \q T}$ can be estimated by \eqref{str1}, \eqref{str3} and \eqref{str5}.

We now estimate $B_{\q T \q T}$.
First we note that in generalised wave coordinates the terms involving $\partial^2 Z^I \wht g$ are absent.
Moreover, thanks to Proposition \ref{prpcross},
we note that there are no terms involving $\partial^3_s \partial_\theta Z^I b$,
$\partial^3_s Z^I b$ or $\partial^4_\theta Z^I b$. 

The terms in $Z^{I_1}(g-m)\partial^2 Z^{I_2} g$, with $I_1\leq I_2<I$ give similar contributions than the terms in $A_{\q T \q T}$, noticing that
$$|\bar{\partial}\partial Z^{I_2} g|\lesssim \frac{1}{1+s}|\partial Z^{I_2+1} g|,\quad
|\partial^2 Z^{I_2}g|\lesssim \frac{1}{1+|q|}|\partial Z^{I_2+1}g|.$$
The terms of the form
$(\partial \bar{\partial} Z^{I_1}g) Z^{I_2}g$, with $I_1\leq I_2$ give contributions 
\begin{align}
\label{str11}& \frac{\ep}{(1+|q|)(1+s)^{\frac{3}{2}-\rho}}| Z^{I_2} \wht g|,\\
\label{str12} &+ \ch_{q>R} \frac{\ep}{(1+s)^2}| Z^{I_2} \wht g|,
\end{align} 
or \eqref{str7}.
The other terms are of the form
$\partial_q^2 Z^{I_1}g_{\q T \q T} Z^{I_2}g_{\q T \q T} $. They give contributions 
\begin{align}
\label{str13}\frac{\ep}{(1+s)(1+|q|)^{\frac{3}{2}-\rho}}|Z^{I_2}\wht g_{\q T \q T}|,\\
\label{str14}\ch_{q>0}\frac{\ep}{(1+s)(1+|q|)}|Z^{I_2}\wht g_{\q T \q T}|,
\end{align}
or \eqref{str5}.
We now estimate the terms involving $G$ : they are
$$Z^I(G^\alpha \partial_\alpha  g_{\q T \q T}), \quad Z^I (g_{\q T \alpha} \partial_{\q T} G^\alpha).$$
They give contributions
\begin{align}
\label{str15}&\frac{1}{1+s}|Z^I G|+|\bar{\partial}Z^I G|\\
\label{str16}&\frac{\ep}{(1+s)(1+|q|)^\frac{1}{2}}\left( \bar{\partial} Z^I \wht g_{\q T \q T}+\frac{1}{1+s}Z^I \wht g_{\q T \q T}\right),
\end{align}
where we have used the estimates \eqref{estgu} and \eqref{estGinf} to estimate $Z^{I_1}G$ for $I_1\leq \frac{N}{2}$.
\paragraph{The case $L\ba L$}
We now turn to $A_{L \ba L}$. The new terms are those who contain three times a $\ba L$ vector field : they must also contain three times a $L$ vector field. They are of the form
$$Z^I(\partial_{\ba L}g_{\ba L L}\partial_{\ba L} g_{LL}),\quad
Z^I (\partial_{\ba L}g_{\ba L L}\partial_{L}g_{\ba L L} ), 
\quad Z^I(\partial_{\ba L}g_{\ba L \ba L}\partial_L g_{LL})
\quad Z^I(\partial_{ L }g_{\ba L \ba L}\partial_L g_{L \ba L}).$$
We treat the first term. 
Thanks to Proposition \ref{estLL}, $Z^I(\partial_{\ba L}\wht g_{\ba L L}\partial_{\ba L} \wht g_{LL})$ is equivalent to
$$Z^I(\partial \wht g_1 \bar{\partial} \wht g_1),$$
and consequently gives \eqref{str6}, or either \eqref{str9} or \eqref{str9bis}.
The term $\partial_{\ba L}Z^{I_1} (g_{\mathfrak{b}})_{\ba L L}\partial_{\ba L} Z^{I_2}g_{LL}$ with $I_1\leq I_2$ gives \eqref{str3}, 
$\partial_{\ba L}Z^{I_1} (g_{\mathfrak{b}})_{\ba L L}\partial_{\ba L} Z^{I_2}g_{LL}$ with $I_2\leq I_1$ gives
\eqref{str7}. The term $\partial_{\ba L}Z^{I_1} (\wht g)_{\ba L L}\partial_{\ba L} Z^{I_2}(g_{\mathfrak{b}})_{LL}$ with $I_1\leq I_2$ gives,  thanks to the estimate \eqref{calcsigma1} on $\sigma^1_{LL}$
\begin{equation}
\label{str17} \ch_{q>R}
\frac{\ep}{(1+|q|)^{\frac{3}{2}+\delta}s^{\frac{1}{2}-\rho}}\frac{1}{1+|q|}\left(|\partial_s Z^I b|+\frac{q}{s}|\partial_s \partial_\theta Z^I b|+\frac{q^2}{s^2}|\partial_\theta Z^I b|\right) 
\end{equation}
and the term $\partial_{\ba L}Z^{I_1} \wht g_{\ba L L}\partial_{\ba L} Z^{I_2}(g_{\mathfrak{b}})_{LL}$ with $I_2\leq I_1$ gives
\eqref{str8}
The second term give contributions similar to the first term, except for
$\partial_{\ba L} \wht g_{L \ba L}\partial_L (g_{\mathfrak{b}})_{LL}$ which give \eqref{str10}.
We treat the third term. The terms $\bar{\partial} Z^{I_1} g_{LL}\partial_q Z^{I_2}g_{\ba L \ba L}$ with $I_1\leq I_2$ give contributions \eqref{str6}, \eqref{str7} and \eqref{str8}. The term $\partial_L Z^{I_1} g_{LL} \partial_{\ba L } Z^{I_2}g_{\ba L \ba L}$, with $I_1\leq I_2<I$ is estimated by
\begin{equation}
\label{str19} \frac{1}{(1+s)^{\frac{3}{2}-\rho}(1+|q|)}|Z^{I} g_{LL}|.
\end{equation}
 The term $\partial_{\ba L} \wht g_{\ba L \ba L} \partial_L Z^I \wht g_{LL}$ is in ${}^IQ_{L \ba L}$, and it can be estimated by
 \begin{equation}
 \label{str19bis}\frac{\ep}{(1+s)^{\frac{1}{2}-\rho}(1+|q|)}|\bar{\partial}Z^I \wht g_{LL}|.
 \end{equation} 
 The term $\partial_{\ba L} Z^{I_1} \wht g_{\ba L \ba L} \partial_L Z^{I_2} (g_{\mathfrak{b}})_{LL}$ can be (roughly) estimated by \eqref{str10}, thanks to the estimate \eqref{calcsigma1} on $\sigma^1_{LL}$.
The forth term give contributions \eqref{str6}, \eqref{str7} and \eqref{str8}.
 We turn to $B_{L \ba L}$. The new terms are of the form
 $$\partial_{\ba L \ba L} Z^{I_1} g_{\ba L L} Z^{I_2}g_{LL}, \quad \partial_{\ba L \ba L} Z^{I_1} g_{L L} Z^{I_2} g_{L \ba L},$$
 with $I_1\leq I_2$.
 The term $ \partial_{\ba L \ba L} Z^{I_1} \wht g_{\ba L L} Z^{I_2}\wht g_{LL}$ gives contributions either
 \begin{equation}
 \label{str20}\frac{\ep}{(1+|q|)^2(1+s)^{\frac{1}{2}-\rho}}|Z^I \wht g_{LL}|
 \end{equation}
or 
 \begin{equation}
 \label{str20bis}\frac{\ep}{(1+|q|)^{\frac{3}{2}}(1+s)^{\frac{1}{2}}}|Z^I \wht g_{LL}|
 \end{equation}
according that we estimate $\wht g_{\ba L L}$ with \eqref{bootg1} or \eqref{ks2}. The crossed terms between $\wht g$ and $g_{\mathfrak{b}}$ give contributions \eqref{str14} or \eqref{str17}.
The second term give contributions \eqref{str11}, \eqref{str12} or \eqref{str7}.

We now estimate the terms involving $G$ :
$$Z^I(G^{\q T}\partial_{\q T} g_{L \ba L}),\quad Z^I(g_{\ba L \q T}\partial_L G^{\q T}), \quad Z^I(g_{L \q T}\partial_{\ba L}G^{\q T}).$$
They give contributions \eqref{str1}, \eqref{str15} or \eqref{str16}, where we note that
$$|\partial_q Z^IG|\lesssim \frac{1}{1+|q|}|Z^IG|+|\bar{\partial}Z^I G|.$$

\paragraph{The case $U\ba L$}
The new terms contain three times the vector field $\ba L$, so they contain twice the vector field $L$ and once the vector field $U$. The terms containing two derivatives $\partial_{\ba L}$ are of the form
$$Z^I(\partial_{\ba L} g_{LU}\partial_{\ba L}g_{\ba LL}), \quad
Z^I(\partial_{\ba L}g_{LL}\partial_{\ba L}g_{U \ba L}).$$
The second term can be estimated in the same way as 
$Z^I(\partial_{\ba L} g_{L  L}\partial_{\ba L } g_{\ba L L})$ in the case $L\ba L$.
We now treat the first term. We consider $Z^I(\partial_{\ba L} g_{LU}\partial_{\ba L}\wht g_{\ba LL})$. We decompose it in
$$Z^I(\partial_{\ba L} (\wht g_{LU}+\sigma^0_{LU})\partial_{\ba L}\wht g_{\ba LL})+Z^I(
\partial_{\ba L} (\sigma^1_{LU})\partial_{\ba L}\wht g_{\ba LL}).$$
Thanks to the Proposition \ref{estLU} the first term is equivalent to
$Z^I(\bar{\partial}\wht g\partial_{\ba L}\wht g_{\ba LL})$ and gives contributions \eqref{str6} or either \eqref{str9} or \eqref{str9bis}.
The second terms give contributions \eqref{str17} or \eqref{str8}.
The term
$Z^I(\partial_{\ba L} \wht g_{LU}\partial_{\ba L}(g_{\mathfrak{b}})_{\ba LL})$
 gives contribution \eqref{str3}, \eqref{str4} and \eqref{str5}.
The new terms involving a good derivative are the following (with $I_1\leq I_2$)
$$\partial_{\ba L} Z^{I_1}\wht g_{\ba L \ba L}\partial_{\q T} Z^{I_2}g_{\q T \q T}, \quad
\partial_{\ba L} Z^{I_1}\wht g_{\ba L \q T}\partial_{\q T}Z^{I_2}g_{\ba L \q T }, \quad 
\partial_{\ba L} Z^{I_1} \wht g_{\q T \q T}\partial_{\q T}Z^{I_2}g_{\ba L \ba L}.$$
The third term can be bounded by \eqref{str1}. The first gives the contribution \eqref{str19bis}.
The second term consist of
$$\partial_{\ba L} \wht g_{\ba L \ba L}\partial_{U} \wht g_{LL}, \quad \partial_{\ba L}\wht g_{\ba L \ba L}\partial_{L} (\wht g_{UL}+\sigma^0_{UL}), \quad \partial_{\ba L}\wht g_{\ba L \ba L}\bar{\partial}\sigma^1_{\q T \q T}.$$
The first two are in ${}^IQ_{U\ba L}$. They can be estimated by
\begin{equation}
\label{str24bis}
\ep \min \left(\frac{1}{(1+|q|)(1+s)^{\frac{1}{2}-\rho}},\frac{1}{\sqrt{1+s}\sqrt{1+|q|}}\right)|\bar{\partial}Z^I \wht g|,
\end{equation}
\begin{equation}
\label{str25}
\begin{split}
&\ch_{q>R}\frac{\ep}{(1+|q|)^{\frac{3}{2}+\delta}(1+s)^{\frac{1}{2}-\rho}}\\
&\left(s|\partial^2_s Z^I b|+
q|\partial^2_s \partial_\theta Z^I b|+
\frac{q}{s}|Z^I \partial_s \partial^2_\theta b|
+ |\partial_s \partial_\theta Z^I b|
+\frac{q}{s^2}|Z^I \partial^3_\theta b|\right).
\end{split}
\end{equation} 
The third can be estimated (loosely) by \eqref{str10}.
We turn to $B_{\ba L U}$. The new terms are of the form
$$\partial_{\ba L}\partial_{\ba L} Z^{I_1}g_{U\ba L} Z^{I_2}g_{L L}, \quad
\partial_{\ba L} \partial_{\ba L}Z^{I_1}(g_{\mathfrak{b}})_{\ba L L}Z^{I_2}\wht g_{UL}, \quad
\partial_{\ba L} \partial_{\ba L} Z^{I_1}(g_{\mathfrak{b}})_{U L} Z^{I_2}\wht g_{\ba L L}.$$
The first two terms can be estimated by \eqref{str20} or \eqref{str20bis}, \eqref{str17} and \eqref{str14}. 
The last term would not have enough decay, but it is actually not present : such a term could only come from
$g_{UU}\partial_{\ba L} H^U,$
and more precisely from $\partial_{\ba L} \wht G^U$. However, according to the definition of $\wht G^U$, this term do not contain terms in $\partial_{\ba L}g_{\mathfrak{b}}$.

We now look at the terms involving $G$. They are of the form
$$Z^I(G^{\q T}\partial_{\q T} g_{U \ba L}), \quad Z^I( g_{\q T U}\partial_{\ba L}G^{\q T}), \quad Z^I(g_{\q T \ba L}\partial_U G^{\q T}).$$
Let's look at the order one term in $G$. They are of the form
$$Z^I(\partial_{\ba L}G^U), \quad Z^I(\partial_U G^L).$$
The first term has been introduced to compensate the bad component $R^0_{U\ba L}$ of the Ricci tensor of $g_{\mathfrak{b}}$. The second term gives a contribution which is 
\eqref{str15}. 
The quadratic terms give the same contributions as in the $L\ba L$ case.

\paragraph{The case $\ba L\ba L$}
We look at the term involving two bad derivatives. They are of the form
$$Z^I(\partial_{\ba L} g_{\ba L \ba L}\partial_{\ba L}g_{LL}), \quad
Z^I(\partial_{\ba L} g_{\ba L L}\partial_{\ba L }g_{L \ba L}), \quad Z^I(\partial_{\ba L}g_{UU}\partial_{\ba L}g_{L\ba L}),\quad Z^I(\partial_{\ba L}g_{UL}\partial_{\ba L}g_{U\ba L}).$$
Using Proposition \eqref{estLL}, the first term 
can be estimated by
$$|\partial_{\ba L}Z^{I_1}g_{\ba L \ba L}|\left(|\bar{\partial}Z^{I_2} g_{\q T \q T}|+\frac{1}{1+s}|Z^{I_2}\wht g|+|Z^{I_2}\sigma^1_{LL}|\right)$$
so this gives contribution \eqref{str9}, \eqref{str12} and \eqref{str17}. The term $Z^I(\partial_{\ba L} g_{L \ba L}\partial_{\ba L} g_{L \ba L})$ would come from $P_{\ba L \ba L}$. We can check that it is not present. 
The third term is in ${}^IQ_{\ba L \ba L}$.
Thanks to the wave coordinate condition,  it is composed of
$$ Z^I(G^L \partial_{\ba L} g_{L \ba L}) \quad Z^I(\bar{\partial} \wht g \partial_{\ba L}g_{L \ba L}) \quad Z^I(\partial_{\ba L} (g_{b})_{UU}\partial_{\ba L} \wht g_{L \ba L}).$$
They give contributions
\begin{align}
\label{str31} &\min\left(\frac{\ep^2}{(1+s)^\frac{3}{2}},\frac{\ep}{(1+s)(1+|q|)^{1-4\rho}}\right)|\partial \wht g_1|, \\
\label{str33} &\ep \min\left(\frac{1}{(1+s)^{\frac{1}{2}-\rho}(1+|q|)}, \frac{1}{\sqrt{1+s}\sqrt{1+|q|}}\right) (|Z^I G|+|\bar{\partial}Z^I \wht g|),\\
\label{str32} & \ch_{q>R} \min\left(\frac{\ep}{(1+s)^{\frac{1}{2}-\rho}(1+|q|)^{\frac{3}{2}+\delta}},\frac{\ep}{(1+s)^{\frac{1}{2}}(1+|q|)^{\frac{3}{2}+\delta-\sigma}}\right) \frac{1}{1+|q|} \left(|\partial_s Z^I b|
+\frac{q}{s}|\partial_s \partial_\theta Z^I b|+\frac{q}{s}|Z^I \partial^2_\theta b|\right),\\
\label{str34} &\ch_{q>0}\frac{\ep}{1+s}|\partial Z^I \wht g_1|
\end{align}
or  \eqref{str24bis}. 
The fourth term is equivalent to the term $Z^{I}(\partial_{\ba L}g_{UL}\partial_{\ba L}g_{L\ba L})$ which has already been treated in the case $U\ba L$.

We go to $B_{\ba L \ba L}$. The new terms are of the form
$$\partial_{\ba L }\partial_{\ba L}Z^{I_1}g_{\ba L \ba L} Z^{I_2}g_{LL}, \quad \partial_{\ba L}\partial_{\ba L} Z^{I_1}(g_{\mathfrak{b}})_{L \ba L} Z^{I_2}g_{L \ba L}, \quad \partial_{\ba L}\partial_{\ba L} Z^{I_1}(g_{\mathfrak{b}})_{UU} Z^{I_2}g_{L \ba L}.$$
The first one is in ${}^IQ_{\ba L \ba L}$ and gives a contribution \eqref{str20} or \eqref{str17}. 
The second term would come from
$\wht g_{L \ba L}\partial_{\ba L} F^L,$
but (see the analysis of the wave coordinate condition), in $F^L$, there is only one order one term involving a derivative $\ba L$ which is $\partial_{\ba L} g_{UU}$, consequently 
the second term is not present. The third one is in $Q_{\ba L \ba L}$ and give a contribution in
\eqref{str32} and \eqref{str34}.
The terms with $G$ consist in
$$Z^I(G^{\q T}\partial_{\q T} g_{\ba L \ba L}), \quad Z^I(g_{\ba L \q T}\partial_{\ba L} G^{\q T}),$$
We calculate $\partial_q G^L$ thanks to \eqref{dqG}
$$\partial_{q} G^L= \frac{1}{r}G^L + 2(\partial_q \phi)^2 +\frac{\partial^2_q(q\chi(q))}{r}h(\theta,2t)+\partial_s G^L.$$
The term $2(\partial_q \phi)^2 +\frac{\partial^2_q(q\chi(q))}{r}h$ in $\partial_q G^L$ is here to compensate the term $(\partial_q \phi)^2$ which comes from the right-hand side of \eqref{sys} and the bad component $R^0_{\ba L \ba L}$ of the Ricci tensor of $g_{\mathfrak{b}}$. Let us note that this term is actually
$\frac{\partial^2_q(q\chi(q))}{r}h(\theta,s)$. However we have
$$\frac{\partial^2_q(q\chi(q))}{r}(h(\theta,s)-h(\theta,2t))=\frac{\partial^2_q(q\chi(q))}{r}O(\partial_s h),$$
so we can neglect this term, compared to the terms which are already present in $R^1$.
The terms which remain give contributions \eqref{str31}, \eqref{str3}, and \eqref{str33}.
\end{proof}

We now give a similar result for $\phi$.

\begin{prp}\label{strphi}
We have
\begin{align*}
|\Box_g Z^I \phi|&\lesssim \frac{\ep}{\sqrt{1+s}(1+|q|)^{\frac{5}{2}-4\rho}}| Z^I \wht g_{LL}|
+\frac{\ep}{(1+s)^\frac{3}{2}(1+|q|)^{\frac{3}{2}-4\rho}}| Z^I \wht g_1|
+\frac{\ep}{(1+s)^\frac{5}{2}(1+|q|)^{\frac{1}{2}-4\rho}}| Z^I \wht g|
\\
&+\frac{\ep}{(1+s)^{\frac{3}{2}-\rho}}|\partial Z^I \phi|
+\frac{\ep}{(1+s)^\frac{3}{2}(1+|q|)^{\frac{1}{2}-4\rho}}| Z^I G^L|
+\frac{\ep}{(1+s)\sqrt{1+|q|}}|\bar{\partial}Z^I \phi|.
\end{align*}
\end{prp}

\begin{proof}
	First we note that since $\phi$ is supported in $q<R+\frac{1}{2}$, the support of $\phi$ and the support of $g_{\mathfrak{b}}-m$ are disjoint, consequently the support of $\phi$ and the support of $F_b$, $\wht G$ and $G^U$ are also disjoint. Therefore we have
	$$\Box_g \phi = g^{\alpha\beta}\partial_\alpha \partial_\beta \phi + 2G^L\partial_{s}\phi.$$
	Consequently $\Box_g Z^I \phi$ is composed of terms
$$Z^{I_1}\wht g_{LL}\partial^2_q Z^{I_2}\phi, \quad Z^{I_1}\wht g_{\q T \q V}\partial_q \bar{\partial} Z^{I_2}\phi,
\quad Z^{I_1}\wht g_{\ba L \ba L}\partial^2_s Z^{I_2}\phi, \quad Z^{I_1}G^L \partial_s Z^{I_2}\phi,$$
with $I_1+I_2= I$ and $I_2<I$.
Therefore \eqref{bootphi1}, \eqref{bootg1} and \eqref{estGinf} yield the estimate of Proposition \ref{strphi}.
\end{proof}

\section{Angle and linear momentum}\label{angle}
The aim of this section is to prove Proposition \ref{prpangle}. Roughly speaking, the estimates of Proposition \ref{prpangle} are obtained by "integrating the constraint equations".
For this, we separate in $R_{\mu \nu}$ the linear terms in $g$ and $G$ from the quadratic terms, which are the same as the quadratic terms in $\Box g_{\mu \nu}$.
We denote by $\wht \Gamma$ the part of the Christoffel symbol of $g$ which involve derivatives of $\wht g$. We note $O((\partial g)^2)$ the quadratic terms : they are determined in Proposition \ref{prpstr}. 
\begin{align*}
R_{00} &= (R_b)_{00} + \partial_0 \wht \Gamma^0_{00}+\partial_i   \wht \Gamma^i_{00}
-\partial_0 \wht \Gamma^0_{00}-\partial_0   \wht \Gamma^i_{0i}+O((\partial g)^2)\\
&=(R_b)_{00} +\partial_i   \wht \Gamma^i_{00}
-\partial_0   \wht \Gamma^i_{0i}+O((\partial g)^2),\\
R_{ii}&= (R_b)_{ii} + \partial_0 \wht \Gamma^0_{ii}+\partial_j   \wht \Gamma^j_{ii}
-\partial_i \wht \Gamma^0_{i0}-\partial_i   \wht \Gamma^j_{ij}+O((\partial g)^2).
\end{align*}
We note that
$$-\partial_0   \wht \Gamma^i_{0i}+ \partial_0 \wht \Gamma^0_{ii}=
-\partial_0\partial_i g_{i0}+O((\partial g)^2).$$
Consequently 
\begin{equation}
\label{trace}
2((\partial_t \phi)^2+|\nabla \phi|^2)=(R_b)_{00}+(R_b)_{11}+(R_b)_{22}
+\partial_i   \wht \Gamma^i_{00}+\partial_j   \wht \Gamma^j_{ii}
-\partial_i \wht \Gamma^0_{i0}-\partial_i   \wht \Gamma^j_{ij}
-\partial_i\partial_0 g_{i0}+O((\partial g)^2).
\end{equation}
Moreover we have
$$(R_b)_{00}+(R_b)_{11}+(R_b)_{22}=\frac{2}{r}\partial^2_q (q\chi(q))h(\theta,s)+O\left(\frac{\ep^2}{r^2}\ch_{R\leq q \leq R+1}\right).$$
We note that when $R\leq q \leq R+1$ we can write $h(\theta,s)=h(\theta,2t)+O\left(\frac{\ep^2}{(1+s)^\frac{3}{2}}\right).$
Consequently, if we integrate \eqref{trace} over $\m R^2$ we obtain
$$\int (\partial_t \phi)^2+|\nabla \phi|^2 =\int h(\theta,2t)d\theta +O\left(\frac{\ep^2}{1+t}\right) + \int O((\partial g)^2).$$
We calculate
$$R_{0i}= (R_b)_{0i} + \partial_0  \wht \Gamma^0_{0i} +\partial_j \wht \Gamma^j_{0i}-\partial_0  \wht \Gamma^0_{0i} -\partial_0 \wht \Gamma^j_{ji}+ O((\partial g)^2),$$
and consequently
\begin{equation}
\label{roi} 2\partial_t \phi \partial_1 \phi = -\frac{\cos(\theta)}{r}\partial^2_q(q \chi(q))h(\theta,s) +\partial^2_q \chi(q) \sigma^0_{UL}\sin(\theta) +O\left(\frac{\ep^2}{r^2}\ch_{R\leq q \leq R+1}\right)+\partial_j \wht \Gamma^j_{01} -\partial_0\partial_1 \wht g_{jj} +O((\partial g)^2).
\end{equation}
By integrating \eqref{roi} over $\m R^2$ we obtain
$$ 2\int \partial_t \phi \partial_1 \phi = -\int \cos(\theta)h(\theta,2t)d\theta  +O\left(\frac{\ep^2}{1+t}\right) + \int O((\partial g)^2).$$
Moreover, thanks to Proposition \ref{prpstr} the quadratic terms in $\Box g_{\mu \nu}$ can be bounded by $$\frac{\ep^2}{(1+s)^\frac{3}{2}(1+|q|)^{\frac{3}{2}-2\rho}},$$
(see also the proof of Proposition \ref{linfN} for more details). Consequently we have
$$\Delta_h = O\left(\frac{\ep^2}{(1+s)^{\frac{1}{2}}}\right).$$

\section{$L^\infty$ estimates}\label{seclinf}

\subsection{Estimate for $I\leq N-9$}

\begin{prp}\label{linfN}
We have the estimates for for $I \leq N-9$ 
\begin{align*}
|Z^I \wht g| &\leq \frac{C_0\ep+C\ep^2 }{(1+s)^{\frac{1}{2}-\rho}},\\
|Z^I \phi| &\leq \frac{C_0\ep+C\ep^2 }{\sqrt{1+s}(1+|q|)^{\frac{1}{2}-4\rho}}.
\end{align*}
\end{prp}
This proposition is a consequence of the following propositions.

\begin{prp}\label{linfphi}
We have the estimate for $I \leq N-9$ and $q\leq R+1$
$$
|\Box Z^I \phi| \lesssim \frac{\ep^2 }{(1+s)^{2-3\rho}(1+ |q|)} , \; q<R+1,\\
$$
and $\Box Z^I \phi=0$ for $q>R+1$.
\end{prp}

\begin{prp}\label{linfwhtg}
We have the estimate for $I \leq N-9$ and $q>R$ 
$$
|\Box Z^I \wht g| \lesssim \frac{\ep^2}{(1+s)^\frac{3}{2}(1+|q|)^{\frac{3}{2}-4\rho-\sigma}},$$
and for $q<R$
$$|\Box Z^I \wht g| \lesssim \frac{\ep^2}{(1+s)^\frac{3}{2}(1+|q|)}.$$
\end{prp}

\begin{rk}The estimate in the region $q>R$ is not sharp for the decay in $q$.
\end{rk}

The following lemma is a direct consequence of the $L^\infty-L^\infty$ estimate and is proved at the end of the section.
\begin{lm}
\label{linf2}Let $\beta,\alpha \geq 0$, such that $\beta-\alpha \geq \rho>0$.
Let $u$ be such that
$$|\Box u|\lesssim \frac{1}{(1+s)^{\frac{3}{2}-\alpha}(1+|q|)}, \; for \;q<0
\quad |\Box u|\lesssim \frac{1}{(1+s)^{\frac{3}{2}-\alpha}(1+|q|)^{1+\beta}},
\; for \; q>0,$$
and $(u,\partial_t u)|_{t=0}=0.$
Then we have the estimate
$$|u|\lesssim \frac{(1+t)^{\alpha+\rho}}{\sqrt{1+s}}.$$
\end{lm}

We first assume Proposition \ref{linfphi} and \ref{linfwhtg}, and prove Proposition \ref{linfN}.

\begin{proof}[Proof of Proposition \ref{linfN}]
We have 
$$|\Box Z^I \phi| \lesssim \frac{\ep^2 } {(1+s)^{2-3\rho}(1+ |q|)}
\lesssim \frac{\ep^2}{(1+s)^{2-4\rho}(1+|q|)^{1+\rho}},$$
therefore the $L^\infty-L^\infty$ estimate, combined with Proposition \ref{flat1} for the contribution of the initial data yields
$$|Z^I \phi|\leq \frac{C_0\ep}{\sqrt{1+s}\sqrt{1+|q|}}+\frac{C\ep^2}{\sqrt{1+s}(1+|q|)^{\frac{1}{2}-4\rho}},$$
where $C$ is a constant depending on $\rho$.

The estimate for $\wht g$ follows from Lemma \ref{linf2} with $\alpha= 0$, $\beta=\frac{3}{2}+\delta-\sigma$ combined with Proposition \ref{flat1} 
$$| Z^I \wht g|\leq \frac{C_0\ep}{\sqrt{1+s}\sqrt{1+|q|}}+\frac{C\ep^2}{(1+s)^{\frac{1}{2}-\rho}},$$
which concludes the proof of Proposition \ref{linfN}.
\end{proof}

\begin{proof}[Proof of Proposition \ref{linfphi}]
We have, thanks to Proposition \ref{strphi}
\begin{align*}
|\Box Z^I \phi|&\lesssim \frac{\ep}{\sqrt{1+s}(1+|q|)^{\frac{5}{2}-4\rho}}| Z^I \wht g_{LL}|
+\frac{\ep}{(1+s)^\frac{3}{2}(1+|q|)^{\frac{3}{2}-4\rho}}| Z^I \wht g|\\
&+\frac{\ep}{(1+s)^{\frac{3}{2}-\rho}}|\partial Z^{I+1} \phi|
+\frac{\ep}{(1+s)^\frac{3}{2}(1+|q|)^{\frac{1}{2}-4\rho}}| Z^I G^L|
+\frac{\ep}{(1+s)\sqrt{1+|q|}}|\bar{\partial}Z^I \phi|\\
&\lesssim  \frac{\ep}{\sqrt{1+s}(1+|q|)^{\frac{3}{2}-4\rho}}| Z^I \wht g_{LL}|
+\frac{\ep}{(1+s)^\frac{3}{2}(1+|q|)^{\frac{1}{2}-4\rho}}| Z^I \wht g|\\
&+\frac{\ep}{(1+s)^{\frac{3}{2}-\rho}(1+|q|)}|Z^{I+2} \phi|
+\frac{\ep}{(1+s)^\frac{3}{2}(1+|q|)^{\frac{1}{2}-4\rho}}| Z^I G^L|
+\frac{\ep}{(1+s)^2(1+|q|)^\frac{1}{2}}|Z^{I+1}\phi|.
\end{align*}
For $I\leq N-9$ we can estimate $Z^I \wht g_{LL}$ thanks to \eqref{iwintll2},
$$|Z^I \wht g_{LL}|\lesssim \frac{\ep(1+|q|)}{(1+s)^{\frac{3}{2}-3\rho}},$$
we estimate $Z^I \wht g$ thanks to \eqref{bootg2}
$$|Z^I \wht g|\lesssim \frac{\ep}{(1+s)^{\frac{1}{2}-3\rho}},$$
we estimate
$Z^{I+2} \phi$ thanks to \eqref{bootphi2}
$$|Z^{I+2} \phi|\lesssim \frac{\ep}{(1+s)^{\frac{1}{2}-2\rho}},$$
and we estimate $Z^IG^L$ thanks to Proposition \eqref{prpG}, with Proposition \ref{prpangle} to estimate $\Delta_h$
$$|Z^I G^L|\lesssim \frac{\ep^2}{(1+s)(1+|q|)^\frac{1}{2}}.$$
Consequently we obtain
\begin{align*}
|\Box Z^I \phi|\lesssim& \frac{\ep^2}{(1+s)^{2-3\rho}(1+|q|)^{\frac{3}{2}-4\rho}}
+\frac{\ep^2}{(1+s)^{2-3\rho}(1+|q|)}\\
&+\frac{\ep^2}{(1+s)^\frac{5}{2}(1+|q|)^{1-4\rho}}
+\frac{\ep^2}{(1+s)^2(1+|q|)^{1-2\rho}}\\
\lesssim & \frac{\ep^2}{(1+s)^{2-3\rho}(1+|q|)}.
\end{align*}

\end{proof}

\begin{proof}[Proof of Proposition \ref{linfwhtg}]
We start with the region $q<R$. We estimate first
$Q_{\ba L \ba L}$ which contain the limiting contributions. Thanks to Proposition \ref{prpstr} we have
\begin{align*}
|{}^IQ_{\ba L  \ba L}|\lesssim& 
\frac{\ep}{(1+|q|)(1+s)^{\frac{1}{2}-\rho}}\left(|\bar{\partial}Z^I g_{\q T \q T}|+\frac{1}{1+|q|}|Z^I g_{LL}|\right)
+\left(\frac{\ep^2}{(1+s)^\frac{3}{2}}+\frac{\ep}{(1+s)(1+|q|)^{1-4\rho}}\right)|\partial Z^I \wht g_1|\\
&+\ep\min\left(\frac{1}{(1+|q|)(1+s)^{\frac{1}{2}-\rho}},\frac{1}{(1+|q|)^{\frac{1}{2}}(1+s)^\frac{1}{2}}\right)(|\bar{\partial}Z^I \wht g|+|Z^I G^L|)\\
\lesssim &\left(\frac{\ep}{(1+s)^{\frac{3}{2}-\rho}}+\frac{\ep}{(1+s)(1+|q|)^{1-4\rho}}\right)|\partial Z^{I} \wht g_1|
+\frac{\ep}{(1+s)^{\frac{1}{2}-\rho}(1+|q|)^2}|Z^I g_{LL}|+s.t.
\end{align*}
where $s.t.$ denotes similar terms.
We estimate $\partial Z^I \wht g_1$ in two ways : thanks to \eqref{bootg2} we have
$$|\partial Z^I \wht g_1|\lesssim \frac{\ep}{(1+|q|)(1+s)^{\frac{1}{2}-3\rho}},$$
and thanks to \eqref{ks2} we have
$$|\partial Z^I \wht g_1|\lesssim \frac{\ep}{(1+|q|)^{\frac{1}{2}-\sigma}(1+s)^{\frac{1}{2}}}.$$
Consequently
$$\left(\frac{\ep^2}{(1+s)^{\frac{3}{2}-\rho}}+\frac{\ep}{(1+s)(1+|q|)^{1-4\rho}}\right)|\partial Z^I \wht g_1|\lesssim 
 \frac{\ep^2}{(1+s)^{2-4\rho}(1+|q|)}+\frac{\ep^2}{(1+s)^\frac{3}{2}(1+|q|)^{\frac{3}{2}-4\rho-\sigma}}.$$
We estimate $Z^I g_{LL}$ thanks to \eqref{wintll2}
$$\frac{\ep}{(1+s)^{\frac{1}{2}-\rho}(1+|q|)^2}|Z^I g_{LL}|\lesssim\frac{\ep^2}{(1+s)^{2-4\rho}(1+|q|)}.$$
Therefore in the region $q<R$ we have
\begin{equation}
\label{estqbal}
|{}^IQ_{\ba L \ba L}|\lesssim \frac{\ep^2}{(1+s)^\frac{3}{2}(1+|q|)^{\frac{3}{2}-4\rho-\sigma}}.
\end{equation}
We now estimate the other contributions in $\Box Z^I \wht g$ : they are given  thanks to Proposition \ref{prpstr} by
$$\frac{\ep}{(1+s)^{\frac{3}{2}-\rho}(1+|q|)}|Z^{I+2}\wht g|
+\frac{\ep}{(1+s)^\frac{3}{2}(1+|q|)^{\frac{3}{2}-4\rho}}|Z^{I+1}\phi|
+\frac{1}{1+s}|Z^{I+1}G|.$$
We estimate $Z^{I+2} \wht g$ thanks to \eqref{bootg2}, we estimate $Z^{I+1} \phi$ thanks to \eqref{bootphi2}, 
we estimate $Z^{I+1}G$ thanks to Proposition \ref{prpG}, with Proposition \ref{prpangle} which is now proved to estimate $\Delta_h$. We obtain
$$|\Box Z^I \wht g|\lesssim  \frac{\ep^2}{(1+s)^\frac{3}{2}(1+|q|)^{\frac{3}{2}-4\rho-\sigma}}+\frac{\ep^2}{(1+s)^2(1+|q|)^\frac{1}{2}}.$$
We now look at the region $q>R$. We estimate the new contributions in $Q_{\ba L \ba L}$ which are given by
\begin{align*}  & \ch_{q>R} \min\left(\frac{\ep}{(1+s)^{\frac{1}{2}-\rho}(1+|q|)^{\frac{5}{2}+\delta}}, \frac{\ep}{(1+s)^{\frac{1}{2}}(1+|q|)^{\frac{5}{2}+\delta+\sigma}}\right)\frac{q}{s}|Z^I \partial^2_\theta b|
+ \ch_{q>R}\frac{\ep}{1+s}|\partial Z^{I+1} \wht g_1|\\
+&\ch_{q>R}\frac{\ep}{(1+|q|)^{\frac{3}{2}+\delta}(1+s)^{\frac{1}{2}-\rho}}\left(s|\partial^2_s Z^I b|
+q|\partial^2_s \partial_\theta Z^I b|+
\frac{q}{s}|Z^I \partial_s \partial^2_\theta b|
+\frac{q}{s^2}|Z^I \partial^3_\theta b|\right)\\
\lesssim &\ch_{q>R}\frac{\ep}{(1+s)(1+|q|)}|Z^{I+1}\wht g_1|\\
&+ \ch_{q>R} \min\left(\frac{\ep}{(1+s)^{\frac{1}{2}-\rho}(1+|q|)^{\frac{5}{2}+\delta}}, \frac{\ep}{(1+s)^{\frac{1}{2}}(1+|q|)^{\frac{5}{2}+\delta+\sigma}}\right)\left(q|\partial_s\partial_\theta Z^{I+1}b|+ \frac{q}{s}|Z^{I+1} \partial^2_\theta b|\right)
\end{align*}
We estimate $\partial_s \partial_\theta Z^I b$ thanks to \eqref{estb3} (and the Sobolev embedding $H^1(\m S^1)\subset L^\infty(\m S^1)$)
$$|\partial_s \partial_\theta Z^{I+1} b|\lesssim \frac{\ep^2}{(1+s)^{2-\frac{\sigma}{2}}},$$
we estimate $Z^I \partial^2_\theta b$ thanks to \eqref{estb2}.
$$|  Z^I \partial^2_\theta b|\lesssim \ep^2,$$
and we estimate $Z^{I+1}\wht g_1$ thanks to \eqref{iks6}
$$|Z^{I+1} \wht g_1|\lesssim\frac{\ep }{(1+|q|)^{\frac{1}{2}+\delta-\sigma}\sqrt{1+s}}.$$
Consequently, we obtain for $q>R$
\begin{equation}
\label{qbalext}
|Q_{\ba L \ba L}|\lesssim \frac{\ep^2}{(1+s)^\frac{3}{2}(1+|q|)^{\frac{3}{2}-4\rho-\sigma}}
+ \frac{\ep^2 }{(1+|q|)^{\frac{3}{2}+\delta-\sigma}(1+s)^\frac{3}{2}}.
\end{equation}
The estimate for $M^E$ can be done exactly in the same way, which concludes
the proof of Proposition \ref{linfwhtg}.

\end{proof}

\begin{proof}[Proof of Lemma \ref{linf2}]
Let $t_0 >0$. We consider times $t\leq t_0$.
In the region $r\leq 2t$ we have $|q|\leq t \leq t_0$ and $s \leq 3t \leq 3t_0$. Therefore
$$|\Box u|\lesssim \frac{(1+t_0)^{\alpha+\rho}}{(1+|q|)^{1+\frac{\rho}{2}}(1+s)^{\frac{3}{2}+\frac{\rho}{2}}}.$$
In the region $r\geq 2t$, we have 
$\frac{r}{2}\leq |q|\leq r$ and $r\leq s\leq \frac{3r}{2}$, therefore
$$|\Box u|\lesssim \frac{1}{(1+r)^{\frac{3}{2}-\alpha+1+\beta}}
\lesssim \frac{(1+t_0)^{\alpha+\rho}}{(1+r)^{\frac{5}{2}-\alpha+\beta}}
\lesssim\frac{(1+t_0)^{\alpha+\rho}}{(1+|q|)^{1+\frac{\rho}{2}}(1+s)^{\frac{3}{2}+\frac{\rho}{2}}},$$
provided $\frac{5}{2}+\rho \leq \frac{5}{2}+\beta-\alpha$, i.e. $\beta -\alpha\geq \rho$.
Consequently, the $L^\infty-L^\infty$ estimate yields, for $t\leq t_0$
$$|u|\lesssim \frac{(1+t_0)^{\alpha+\rho}}{\sqrt{1+s}}.$$
If we take $t=t_0$ we have proved
$$|u|\lesssim \frac{(1+t)^{\alpha+\rho}}{\sqrt{1+s}},$$
which concludes the proof of Lemma \ref{linf2}.
\end{proof}

We now give the $L^\infty$ estimate for $k$, defined by \eqref{eqg2}.
\begin{cor}
	\label{linfk}
	We have the estimate
	$$|Z^{N-9}k|\lesssim \frac{\ep^2}{(1+t)^{\frac{1}{2}-\rho}}.$$
\end{cor}

\begin{proof}
	This is a direct consequence of Lemma \ref{linf2} since the initial data for $k$ are $0$ and $\Box k$ satisfies the same estimate as $\Box \wht g$.
\end{proof}

\subsection{Estimate for $I\leq N-7$}

\begin{prp}\label{linfN5}
We have the estimates for for $I \leq N-7$
\begin{align*}
|Z^I \wht g| &\leq \frac{C_0\ep+C\ep^2 }{(1+s)^{\frac{1}{2}-3\rho}},\\
|Z^I \phi| &\leq \frac{C_0\ep+C\ep^2 }{(1+s)^{\frac{1}{2}-2\rho}}.
\end{align*}
\end{prp}
This proposition is a straightforward consequence of Lemma \ref{linf2}, Proposition \ref{flat1} and the following propositions.

\begin{prp}\label{linfphiN5}
We have the estimate for $I \leq N-7$ 
$$
|\Box Z^I \phi| \lesssim \frac{\ep^2}{(1+s)^{\frac{3}{2}-\rho}(1+|q|)} , \; q<R+1,\\
$$
and $\Box Z^I \phi=0$ for $q>R+1$.
\end{prp}

\begin{prp}\label{linfwhtgN5}
We have the estimate for $I \leq N-5$ and $q<R$
$$
|\Box Z^I \wht g| \lesssim \frac{\ep^2}{(1+s)^{\frac{3}{2}-2\rho}(1+|q|)},$$
and for $q>R$
$$
|\Box Z^I \wht g| \lesssim \frac{\ep^2}{(1+s)^{\frac{3}{2}}(1+|q|)^{\frac{3}{2}-4\rho-\sigma}}.$$
\end{prp}

\begin{proof}[Proof of Proposition \ref{linfphiN5}]
We have, thanks to Proposition \ref{strphi}
\begin{align*}
|\Box Z^I \phi|
&\lesssim  \frac{\ep}{\sqrt{1+s}(1+|q|)^{\frac{5}{2}-4\rho}}| Z^I \wht g_{LL}|
+\frac{\ep}{(1+s)^\frac{3}{2}(1+|q|)^{\frac{3}{2}-4\rho}}| Z^I \wht g|\\
&+\frac{\ep}{(1+s)^{\frac{3}{2}-\rho}(1+|q|)}|Z^{I+2} \phi|
+\frac{\ep}{(1+s)^\frac{3}{2}(1+|q|)^{\frac{1}{2}-4\rho}}| Z^I G^L|
+\frac{\ep}{(1+s)^2(1+|q|)^\frac{1}{2}}|Z^{I+1}\phi|.
\end{align*}
For $I\leq N-7$ we can estimate $Z^I \wht g_{LL}$ thanks to \eqref{iwintll1bis},
$$|Z^I \wht g_{LL}|\lesssim \frac{\ep(1+|q|)^{\frac{3}{2}}}{(1+s)^{\frac{3}{2}-\rho}},$$
we estimate $Z^I \wht g$ thanks to \eqref{iks3}
$$|Z^I \wht g|\lesssim \frac{\ep(1+|q|)^\frac{1}{2}}{(1+s)^{\frac{1}{2}-\rho}},$$
we estimate
$Z^{I+2} \phi$ thanks to \eqref{iks1}
$$|Z^{I+2} \phi|\lesssim \frac{\ep(1+|q|)^{\frac{1}{2}}}{(1+s)^{\frac{1}{2}}},$$
and we estimate $Z^IG^L$ thanks to Proposition \eqref{prpG}
$$|Z^I G^L|\lesssim \frac{\ep^2}{(1+s)(1+|q|)^\frac{1}{2}}.$$
Consequently we obtain
\begin{align*}
|\Box Z^I \phi|\lesssim &\frac{\ep^2}{(1+s)^{2}(1+|q|)^{1-4\rho}}+\frac{\ep}{(1+s)^{2-\rho}(1+|q|)^{1-4\rho}}
+\frac{\ep^2}{(1+s)^{2-\rho}\sqrt{1+|q|}}\\
&+\frac{\ep^2}{(1+s)^\frac{5}{2}(1+|q|)^{1-4\rho}}
+\frac{\ep^2}{(1+s)^\frac{5}{2}}
\end{align*}
which concludes the proof of Proposition \ref{linfphiN5}.
\end{proof}

\begin{proof}[Proof of Proposition \ref{linfwhtgN5}]
We start with the region $q<R$. 
We estimate first $Q_{\ba L \ba L}$ in the same way than in the proof of Proposition \ref{linfwhtg} by
$$
\left(\frac{\ep}{(1+s)^\frac{3}{2}}+\frac{\ep}{(1+s)(1+|q|)^{1-4\rho}}\right)|\partial Z^I \wht g_1|+\frac{\ep}{(1+|q|)^2(1+s)^{\frac{1}{2}-\rho}}|Z^I g_{LL}|
$$
We estimate $\partial Z^I \wht g_1$  thanks to \eqref{ks3} 
$$|\partial Z^I \wht g_1|\lesssim \frac{\ep}{(1+|q|)^{\frac{1}{2}}(1+s)^{\frac{1}{2}-\rho}}.$$
Consequently
$$\left(\frac{\ep}{(1+s)^\frac{3}{2}}+\frac{\ep}{(1+s)(1+|q|)^{1-4\rho}}\right)|\partial Z^I \wht g_1|\lesssim 
\frac{\ep^2}{(1+s)^{2-\rho}(1+|q|)^\frac{1}{2}}+\frac{\ep^2}{(1+s)^{\frac{3}{2}-\rho}(1+|q|)^{\frac{3}{2}-4\rho}}.$$
We estimate $Z^I g_{LL}$ thanks to \eqref{iwintll1bis},
$$\frac{\ep}{(1+|q|)^2(1+s)^{\frac{1}{2}-\rho}}|Z^I g_{LL}|\lesssim \frac{\ep^2}{(1+|q|)^\frac{1}{2}(1+s)^{2-2\rho}}.$$
Consequently
\[|Q_{\ba L \ba L}|\lesssim \frac{\ep^2}{(1+s)^{\frac{3}{2}-2\rho}(1+|q|)}.\]
We now estimate the other contributions in $\Box Z^I \wht g$ : they are given  thanks to Proposition \ref{prpstr} by
$$\frac{\ep}{(1+s)^{\frac{3}{2}-\rho}(1+|q|)}|Z^{I+2}\wht g|
+\frac{\ep}{(1+s)^\frac{3}{2}(1+|q|)^{\frac{3}{2}-4\rho}}|Z^{I+1}\phi|+\frac{1}{1+s}|Z^{I+1}G|.$$
We estimate $|Z^{I+2} \wht g|$ thanks to \eqref{iks3},
 we estimate $Z^{I+1} \phi$ thanks to \eqref{iks1} and we estimate $Z^{I+1}G$ thanks to Proposition \ref{prpG} with $I+1\leq N-6$. We obtain
$$|\Box Z^I \wht g|\lesssim |Q_{\ba L \ba L}|
+\frac{\ep^2}{(1+s)^{2-2\rho}(1+|q|)^\frac{1}{2}}
+\frac{\ep^2}{(1+s)^2(1+|q|)^{1-4\rho}}
+\frac{\ep^2}{(1+s)^2(1+|q|)^\frac{1}{2}}.
$$
In the exterior region, the same estimates as for $I\leq N-9$ are valid.
\end{proof}

\section{Weighted energy estimate}\label{weighte}
\subsection{On Minkowski space-time}
We consider the wave equation on Minkowski space-time 
$\Box u=f$.
We introduce the energy-momentum tensor associated to $\Box$ 
$$Q_{\alpha \beta}=\partial_\alpha u \partial_\beta u-\frac{1}{2}m_{\alpha \beta}m^{\mu \nu}\partial_\mu u \partial_\nu u.$$
We have
$$D^\alpha Q_{\alpha \beta}=f\partial_\beta u.$$
We also note $T=\partial_t$, and introduce the deformation tensor of $T$
$$\pi_{\alpha \beta}=D_\alpha T_\beta +D_\beta T_\alpha=0$$
where $D$ is the covariant derivative.
We have
\begin{equation}
\label{eqm}
D^\alpha (Q_{\alpha \beta} T^\beta) =f\partial_t u +Q_{\alpha \beta}\pi^{\alpha \beta}=f\partial_t u.
\end{equation}
We remark that
$$Q_{TT}=\frac{1}{2}\left((\partial_t u)^2+|\nabla u|^2\right).$$
\begin{prp}\label{prpweightem}Let $w$ be any of our weight functions.
 We have the following weighted energy estimate for $u$
 $$\frac{d}{dt}\left(\int Q_{TT} w(q)dx\right)+\frac{1}{2}\int w'(q)\left((\partial_s u)^2+\left(\frac{\partial_\theta u}{r}\right)^2\right)dx
 \leq  \int w(q)|f\partial_t u|dx.$$
\end{prp}

\begin{proof}
We multiply \eqref{eqm} by $w(q)$ and integrate it on an hypersurface of constant $t$.
We obtain
\begin{equation}
\label{energi1m}-\frac{d}{dt}\left(\int Q_{TT} w(q)\right)=\int w(q)f\partial_t u +\int Q_{T\alpha}D^\alpha w.
\end{equation}
We have
$$Q_{T\alpha}D^\alpha w=
-2w'(q)m^{\alpha \ba L}Q_{T\alpha}
=w'(q)Q_{TL}=w'(q)\left(\partial_t u(\partial_t u+\partial_r u)-\frac{1}{2}(-(\partial_t u)^2
+|\nabla u|^2)\right),$$
so
$$Q_{T \alpha}D^\alpha w
=\frac{1}{2}\left((\partial_s u)^2+\left(\frac{\partial_\theta u}{r}\right)^2\right)w'(q)
$$
which concludes the proof of Proposition \ref{prpweightem}.
\end{proof}

\subsection{On the curved space-time}

We consider the equation 
$$\Box_g u=f,$$
where $g=g_{\mathfrak{b}} +\wht g$ is our space-time metric, satisfying the bootstrap assumptions.
We now introduce the energy-momentum tensor associated to $\Box_g$
$$Q_{\alpha \beta}=\partial_\alpha u \partial_\beta u-\frac{1}{2}g_{\alpha \beta}g^{\mu \nu}\partial_\mu u \partial_\nu u.$$
We have
$$D^\alpha Q_{\alpha \beta}=f\partial_\beta u.$$
We also note $T=\partial_t$, and introduce the deformation tensor of $T$
$$\pi_{\alpha \beta}=D_\alpha T_\beta +D_\beta T_\alpha$$
where $D$ is the covariant derivative.
We have
\begin{equation}
\label{eq}
D^\alpha (Q_{\alpha \beta} T^\beta) =f\partial_t u +Q_{\alpha \beta}\pi^{\alpha \beta}.
\end{equation}
We remark that
$$Q_{TT}=\frac{1}{2}\left((\partial_t u)^2+|\nabla u|^2\right)+O(\ep(\partial u)^2).$$
\begin{prp}\label{prpweighte}Let $w$ be any of our weight functions.
 We have the following weighted energy estimate for $u$
 $$\frac{d}{dt}\left(\int Q_{TT} w(q)dvol_g\right)+C\int w'(q)\left((\partial_s u)^2+\left(\frac{\partial_\theta u}{r}\right)^2dvol_g\right)
 \lesssim \frac{\ep}{1+t}\int w(q)(\partial u)^2dvol_g + \int w(q)|f\partial_t u|dvolg,$$
 where $dvol_g=\sqrt{|det(g)|}dx$,
 and since $-\frac{1}{2}\leq |det(g)|\leq \frac{1}{2}$
 $$\frac{d}{dt}\left(\int Q_{TT} w(q)dvol_g\right)+C\int w'(q)\left((\partial_s u)^2+\left(\frac{\partial_\theta u}{r}\right)^2dx\right)
 \lesssim \frac{\ep}{1+t}\int w(q)(\partial u)^2dx + \int w(q)|f\partial_t u|dx.$$

\end{prp}

\begin{proof}
We multiply \eqref{eq} by $w(q)$ and integrate it on an hypersurface of constant $t$.
We obtain
\begin{equation}
\label{energi1}-\frac{d}{dt}\left(\int Q_{TT} w(q)dvol_g\right)=\int w(q)\left(f\partial_t u +Q_{\alpha \beta}\pi^{\alpha \beta}\right)dvol_g+\int Q_{T\alpha}D^\alpha wdvol_g.
\end{equation}
We have
$$Q_{T\alpha}D^\alpha w=
-2w'(q)g^{\alpha \ba L}Q_{T\alpha}
=w'(q)Q_{TL}+w'(q)\left(g^{\alpha \ba L}-m^{\alpha \ba L}\right)Q_{T\alpha}.$$
We calculate
\begin{align*}
Q_{TL}=&\partial_t u(\partial_t u+\partial_r u)-\frac{1}{2}\left(-(\partial_t u)^2
+|\nabla u|^2\right)+O\left( g_{LL}(\partial u)^2
+(g-m)_{\q T \q V}\partial u \bar{\partial }u +(g-m)_{\ba L \ba L}(\bar{\partial}u)^2\right).\\
=&\frac{1}{2}\left((\partial_s u)^2+\left(\frac{\partial_\theta u}{r}\right)^2\right)+O\left( g_{LL}(\partial u)^2
+(g-m)_{\q T \q V}\partial u \bar{\partial }u +(g-m)_{\ba L \ba L}(\bar{\partial}u)^2\right).
\end{align*}
We estimate the metric coefficients in the following way : first we estimate $g_{\mathfrak{b}}$ thanks to \eqref{estg0}
$$|g_{\mathfrak{b}}-m|\lesssim \frac{\ep(1+|q|)}{1+s},$$
thanks to \eqref{wintll3} we estimate
$$| \wht g_{LL}|\leq \frac{(1+|q|)}{(1+s)^{\frac{3}{2}-\rho}},$$
thanks to \eqref{iks2bis} we estimate
$$|\wht g_{\q T \q V}|\lesssim \frac{\ep\sqrt{1+|q|}}{\sqrt{1+s}},$$
and thanks to \eqref{bootg1} we estimate
$$|\wht g_{\ba L \ba L}|\lesssim \frac{\ep}{(1+s)^{\frac{1}{2}-\rho}}.$$
Consequently we have
\[
Q_{T\alpha}D^\alpha w=\left((\partial_s u)^2+\left(\frac{\partial_\theta u}{r}\right)^2\right)(1+O(\ep))w'(q)
+O\left(\frac{\ep (1+|q|)}{1+t}(\partial u)^2\right)w'(q)
+O\left(\frac{\ep \sqrt{1+|q|}}{\sqrt{1+t}}\bar{\partial} u\partial u\right)w'(q)
\]
so
$$Q_{T\alpha}D^\alpha w=\left((\partial_s u)^2+\left(\frac{\partial_\theta u}{r}\right)^2\right)(1+O(\ep))w'(q)
+O\left(\frac{\ep (1+|q|)}{1+t}(\partial u)^2\right)w'(q),$$
and  since $|w'(q)|\lesssim \frac{w(q)}{1+|q|}$
\begin{equation}\label{glt}
Q_{T\alpha}D^\alpha w=\left((\partial_s u)^2+\left(\frac{\partial_\theta u}{r}\right)^2\right)(1+O(\ep))w'(q)
+O\left(\frac{\ep w(q)}{(1+t)}(\partial u)^2\right).
\end{equation}

We now estimate the deformation tensor of $T$. We have
$$\pi_{\alpha \beta}= \q L_T g_{\alpha \beta}=\partial_t g_{\alpha \beta}.$$
 We obtain
\begin{align*}
\pi_{LL}&=\partial_T g_{LL}=O\left(\frac{\ep}{1+t}\right),\\
\pi_{UL}&=\partial_T g_{UL}=O\left(\frac{\ep}{1+t}\right),\\
\pi_{L \ba L}&=\partial_T g_{\ba L L}=O\left(\frac{\ep}{1+t}\right)+O\left(\frac{\ep}{(1+|q|)(1+t)^{\frac{1}{2}-\rho}}\right),\\
\pi_{U \ba L}&=\partial_T g_{U \ba L}= O\left(\frac{\ep}{1+t}\right)+O\left(\frac{\ep}{(1+|q|)(1+t)^{\frac{1}{2}-\rho}}\right),\\
\pi_{\ba L \ba L}&= \partial_{T}g_{\ba L \ba L}= O\left(\frac{\ep}{1+t}\right)+O\left(\frac{\ep}{(1+|q|)(1+t)^{\frac{1}{2}-\rho}}\right),\\
\pi_{UU}&=\partial_T g_{UU}= O\left(\frac{\ep}{1+t}\right),
\end{align*}
Consequently, the terms $Q^{LL}\pi_{LL}$, $Q^{UL}\pi_{UL}$ and $Q^{UU}\pi_{UU}$ give contributions of the form
\begin{equation}\label{cont1}
\frac{\ep}{(1+t)}(\partial u)^2.
\end{equation}
We can calculate
\begin{align*}Q_{L\ba L}&= \partial_{L}u\partial_{\ba L}u-\frac{1}{2}g_{L\ba L}\left(2g^{L\ba L}\partial_L u \partial_{\ba L}u +(\partial_U u)^2\right)+O\left( g_{LL}(\partial u)^2
+(g-m)_{\q T \q V}\partial u \bar{\partial }u +(g-m)_{\ba L \ba L}(\bar{\partial}u)^2\right)\\
&=(\partial_U u)^2+O\left(\frac{\ep(1+|q|)}{(1+s)}(\partial u)^2\right) + O\left( \ep (\bar{\partial} u)^2 \right).
\end{align*}
Consequently $Q^{L \ba L}\pi_{L \ba L}$ gives the contribution
\begin{equation}\label{cont2}
\frac{\ep}{(1+t)}(\partial u)^2 + \frac{\ep}{(1+|q|)^{\frac{3}{2}-\rho}}(\bar{\partial }u)^2.
\end{equation}
The terms $Q^{\ba L \ba L}\pi_{\ba L \ba L}$ and $Q^{\ba L U }\pi_{\ba L U}$ also give the contribution \eqref{cont2}.

Thanks to \eqref{energi1}, \eqref{glt}, \eqref{cont1} and \eqref{cont2} what we obtain is
\begin{equation}\begin{split}\label{enee}
&\frac{d}{dt}\left(\int Q_{TT} w(q)dvol_g\right)+\frac{1}{2}\int w'(q)\left((\partial_s u)^2+\left(\frac{\partial_\theta u}{r}\right)^2dvol_g\right)\\
\lesssim& \frac{\ep}{1+t}\int w(q)(\partial u)^2dvol_g
+\ep\int \frac{w(q)}{(1+|q|)^{\frac{3}{2}-\rho}}(\bar{\partial}u)^2dvol_g
+\ep \int w(q)|f\partial_t u |dvol_g.
\end{split}
\end{equation}
All our weight functions satisfy
$$\frac{w(q)}{(1+|q|)^{\frac{3}{2}-\rho}}\lesssim w'(q),$$
therefore, for $\ep$ small enough, we can subtract from our inequality the term 
$$ \ep\int \frac{w(q)}{(1+|q|)^{\frac{3}{2}-\rho}}(\bar{\partial}u)^2dvol_g,$$
and we obtain
\begin{equation*}
\frac{d}{dt}\left(\int Q_{TT} w(q)dvol_g\right)+C\int w'(q)\left((\partial_s u)^2+\left(\frac{\partial_\theta u}{r}\right)^2\right)dvol_g
 \lesssim \frac{\ep}{1+t}\int w(q)(\partial u)^2dvol_g + \int w(q)|f\partial_t u|dvol_g.
\end{equation*}
 This conclude the proof of Proposition \ref{prpweighte}.
\end{proof}
\section{Higher order $L^2$ estimates}\label{sechigh}

\subsection{Estimate of $\partial Z^N \wht g$} \label{secl2gn}
\begin{prp}\label{prpl2whtg}
We have
$$\ld{w_1^\frac{1}{2}\partial Z^N \wht g} \leq C_0 \ep + C\ep^\frac{3}{2}(1+t)^{2\rho}$$
and
$$\int_0^t \ld{w_1'(q)^\frac{1}{2}\bar{\partial}Z^N \wht g}^2d\tau \lesssim C_0^2 \ep^2 + C\ep^3(1+t)^{4\rho}.$$
\end{prp}

\begin{cor}
	\label{corl2whtg} The proof of Proposition \ref{prpl2whtg} gives us also
	$$\int_0^t (1+\tau)^{-2\rho}\ld{w_1'(q)^\frac{1}{2}\bar{\partial}Z^N \wht g}^2 d\tau\lesssim C_0^2 \ep^2 + C\ep^3(1+t)^{2\rho}.$$
\end{cor}

\begin{cor}\label{col2k}
We have
$$\ld{w_1^\frac{1}{2}\partial Z^N k }\leq  C\ep^\frac{3}{2}(1+t)^{2\rho}$$
and
$$\int_0^t \ld{w_1'(q)^\frac{1}{2}\bar{\partial}Z^N k}^2d\tau \lesssim  C\ep^3(1+t)^{4\rho}.$$
$$\int_0^t (1+\tau)^{-2\rho}\ld{w_1'(q)^\frac{1}{2}\bar{\partial}Z^N k}^2d\tau \lesssim  C\ep^3(1+t)^{2\rho}.$$
\end{cor}

\begin{proof}[Proof of Proposition \ref{prpl2whtg}]
We use the weighted energy estimate in the curved metric. Thanks to Proposition \ref{prpweighte} we have
\begin{align*}&\frac{d}{dt}\left(\int (\partial Z^N \wht g)^2w_1(q)dvol_g\right)+C\int w_1'(q)(\bar{\partial}Z^N \wht g)^2dx\\
 \lesssim &\frac{\ep}{1+t}\int w_1(q)(\partial Z^N \wht g)^2dx + \int w_1(q)(\partial_t Z^N \wht g) \Box_g Z^N \wht gdvol_g.
 \end{align*}
And consequently
\begin{equation}\label{ene}\begin{split}
&\left\|w_1^\frac{1}{2} \partial Z^N \wht g(t)\right\|^2_{L^2}+\int_0^t \left\|w_1'(q)^\frac{1}{2}\bar{\partial}Z^N \wht g(\tau)\right\|^2_{L^2}d\tau\\
&\lesssim \left\|w_1^\frac{1}{2} \partial Z^N \wht g(0)\right\|_{L^2}+\int_0^t \frac{\ep}{1+\tau}\left\|w_1^\frac{1}{2} \partial Z^N \wht g(\tau)\right\|^2_{L^2}d\tau + \left|\int_0^t\int w_1(\partial_\tau Z^N \wht g) \Box_g Z^N \wht g dvol_gd\tau\right|.
\end{split}
\end{equation}
First, thanks to \eqref{g2N} we have
\begin{equation}
\label{gn0}
\int_0^t \frac{\ep}{1+\tau}\left\|w_1^\frac{1}{2} \partial Z^N \wht g(t)\right\|^2_{L^2}\lesssim \int_0^t \frac{\ep^3}{(1+\tau)^{1-4\rho}} \lesssim \ep^3(1+t)^{4\rho}.
\end{equation}
We will decompose $\Box_g Z^N \wht g=A^N +B^N + C^N$ where
$$\ld{w_1^\frac{1}{2} A^N}\lesssim \frac{\ep^2}{(1+t)^{1-2\rho}},$$
$$\int_0^t \ep^{-1}(1+\tau)\ld{w_1^\frac{1}{2}B^N}^2 d\tau\lesssim \ep^3(1+t)^{4\rho},$$
and $C^N$ is dealt with in a specific manner (like integration by part).  We note that since $-\frac{1}{2}\leq \sqrt{|det(g)|}\leq \frac{3}{2}$, this factor do not matter when we study $A^N$ or $B^N$. However we need to keep it when we do integration by parts, so when we study the $C^N$ terms. 
A term $A^N$ will give the contribution
\begin{equation}\label{an}
\int_0^t \left|\int w_1(\partial_t Z^N \wht g) A^N dvol_g\right|\lesssim \int_0^t \|w_1^\frac{1}{2} A^N\|_{L^2}\ld{w_1^\frac{1}{2}\partial Z^N \wht g}\lesssim \int_0^t \frac{\ep^3}{(1+\tau)^{1-4\rho}}\lesssim \ep^3(1+t)^{4\rho},
\end{equation}
and a term $B^N$ will give the contribution
\begin{equation}\label{bn}\begin{split}
\int_0^t \left|\int w_1(\partial_t Z^N \wht g) B^N dvol_g\right|&\lesssim \int_0^t \ep^{-1}(1+\tau)\|w_1^\frac{1}{2} B^N\|_{L^2}^2d\tau
+\int_0^t \frac{\ep}{1+\tau}\ld{w_1^\frac{1}{2}\partial Z^N \wht g}^2\\
&\lesssim \ep^3(1+t)^{4\rho}+\int_0^t \frac{\ep^3}{(1+\tau)^{1-4\rho}}\lesssim \ep^3(1+t)^{4\rho}.
\end{split}
\end{equation}

We estimate $\Box_g Z^N \wht g$ thanks to Proposition \ref{prpstr}.
$$|\Box_g Z^N \wht g|\lesssim |{}^N Q|+|{}^N M|+|{}^N M^E|.$$
We start with ${}^IQ_{\ba L \ba L}$
\begin{align*}
|{}^NQ_{\ba L  \ba L}|&\lesssim 
\frac{\ep}{(1+|q|)(1+s)^{\frac{1}{2}-\rho}}\left(|\bar{\partial}Z^N \wht g_1|+\frac{1}{1+|q|}|Z^N g_{LL}|\right)
+\left(\frac{\ep^2}{(1+s)^\frac{3}{2}}+\frac{\ep}{(1+s)(1+|q|)^{1-4\rho}}\right)|\partial Z^N \wht g_1|\\
+&\ep\min\left(\frac{1}{(1+|q|)(1+s)^{\frac{1}{2}-\rho}},\frac{1}{(1+|q|)^{\frac{1}{2}}(1+s)^\frac{1}{2}}\right)(|\bar{\partial}Z^N \wht g|+|Z^N G^L|)\\
+ & \ch_{q>R} \min\left(\frac{\ep}{(1+s)^{\frac{1}{2}-\rho}(1+|q|)^{\frac{5}{2}+\delta}}, \frac{\ep}{(1+s)^{\frac{1}{2}}(1+|q|)^{\frac{5}{2}+\delta+\sigma}}\right)\frac{q}{s}|Z^N \partial^2_\theta b|
+ \ch_{q>R}\frac{\ep}{1+s}|\partial Z^N \wht g_1|\\
+&\ch_{q>R}\frac{\ep}{(1+|q|)^{\frac{3}{2}+\delta}(1+s)^{\frac{1}{2}-\rho}}\left(s|\partial^2_s Z^N b|
+q|\partial^2_s \partial_\theta Z^N b|+
\frac{q}{s}|Z^N \partial_s \partial^2_\theta b|
+\frac{q}{s^2}|Z^N \partial^3_\theta b|\right).
\end{align*}
We estimate the contributions term by term
\begin{equation}\label{l21}\begin{split}
&\int_0^t \ep^{-1}(1+\tau)\ld{w_1^\frac{1}{2}\frac{\ep}{(1+|q|)(1+s)^{\frac{1}{2}-\rho}}\bar{\partial }Z^N \wht g_1}^2d\tau\\
&\lesssim \ep(1+t)^{2\rho}\int_0^t \ld{w_2'(q)^\frac{1}{2}\bar{\partial}Z^N \wht g_1}^2d\tau\\
&\lesssim \ep^3(1+t)^{4\rho},
\end{split}
\end{equation}
where we have used $\frac{w_1}{(1+|q|)^2} \lesssim w_2'(q)$ and the bootstrap assumption \eqref{igN}.
We estimate
\begin{equation}
\label{l22}\begin{split}&\int_0^t \ep^{-1}(1+\tau)\left\|\frac{\ep w_1^\frac{1}{2}}{(1+|q|)^2(1+s)^{\frac{1}{2}-\rho}} Z^N \wht g_{LL}\right\|_{L^2}^2d\tau\\
&\lesssim \ep(1+t)^{2\rho}\int_0^t \left\|\frac{w_1^\frac{1}{2}}{(1+|q|)^2} Z^N \wht g_{LL}\right\|^2_{L^2}d\tau\\
&\lesssim \ep^3(1+t)^{4\rho}
\end{split}
\end{equation}
thanks to \eqref{hwl2dlln}. We estimate
\begin{equation}
\label{l23}\left\|w_1^\frac{1}{2}\frac{\ep}{(1+s)(1+|q|)^{\frac{1}{2}}}\partial Z^N \wht g_1\right\|_{L^2}
\lesssim \frac{\ep}{1+t}\|w_1^\frac{1}{2}\partial Z^N \wht g\|_{L^2}\lesssim \frac{\ep^2}{(1+t)^{1-2\rho}}.
\end{equation}
We have
\begin{equation}\label{l26}\begin{split}&\int_0^t \ep^{-1}(1+\tau)\ld{w_1^\frac{1}{2}\frac{\ep}{\sqrt{1+s}\sqrt{1+|q|}}\bar{\partial}Z^I \wht g}d\tau\\
&\lesssim \ep \int_0^t \ld{w_1'(q)^\frac{1}{2}\bar{\partial}Z^N \wht g}^2d\tau\\
&\lesssim \ep^3(1+t)^{4\rho},
\end{split}
\end{equation}
\begin{equation}
\label{l25}\left\|\ep \frac{w_1^\frac{1}{2}}{\sqrt{1+s}\sqrt{1+|q|}} Z^N G^L\right\|_{L^2}
\lesssim \frac{\ep}{(1+t)}\left( \int \frac{1}{(1+|q|)^{1+2\sigma}} \|rZ^N G^L\|^2x_{L^2(\m S^1)} dr\right)^{\frac{1}{2}}\lesssim
\frac{\ep^2}{(1+t)^{1-\rho}},
\end{equation}
thanks to \eqref{estGN}.
We estimate
\begin{equation}\label{l25bbis}\begin{split}&\left\| w_1^\frac{1}{2} \ch_{q>R} \frac{\ep}{(1+s)^{\frac{3}{2}-\rho}(1+|q|)^{\frac{3}{2}+\delta}} \partial^2_\theta Z^N b\right\|_{L^2}\\
&\lesssim \ep\left(\int \frac{1}{(1+|q|)^{1+2\sigma}(1+s)^{3-2\rho}}\|\partial^2_\theta Z^N b\|^2_{L^2(\m S^1)}rdr 
\right)^\frac{1}{2}
\lesssim \frac{\ep^2}{(1+t)^{1-\rho}},
\end{split}
\end{equation}
where we have used \eqref{estb4}.
We can bound $M^E$ 
\begin{align*}
|M^E|\lesssim &\frac{1}{1+s}\left(|\partial Z^N \wht g| +\frac{1}{1+|q|}|Z^N \wht g|\right)\\
&+\frac{\ep}{(1+s)(1+|q|)^{2+\delta-\rho}} \left(s|\partial_s Z^I b|
+q|\partial_s \partial_\theta Z^I b|+\frac{q}{s}|Z^I \partial^2_\theta b|\right)\\
&+\frac{\ep}{(1+|q|)^{\frac{3}{2}+\delta}(1+s)^{\frac{1}{2}-\rho}}\left(s|\partial^2_s Z^I b|
+q|\partial^2_s \partial_\theta Z^I b|+
\frac{q}{s}|Z^I \partial_s \partial^2_\theta b|
+\frac{q}{s^2}|Z^I \partial^3_\theta b|\right).
\end{align*}
Consequently the estimate for $M^E$ will also give the remaining of the estimate of $Q_{\ba L \ba L}$.
We estimate
\begin{equation}\label{l25bis}\ld{w_1^\frac{1}{2}\frac{\ep}{1+s}\left(|\partial Z^N \wht g| +\frac{1}{1+|q|}|Z^N \wht g|\right)}
\lesssim \frac{\ep}{1+t}\ld{w_1^\frac{1}{2}\partial Z^N \wht g} \lesssim \ep^2\frac{1}{(1+t)^{1-2\rho}}.
\end{equation}
$$\ld{w_1^\frac{1}{2}\frac{\ep}{(1+|q|)^{\frac{3}{2}+\delta-\rho}} \partial_s \partial_\theta Z^I b}
\lesssim \ep \left(\int \frac{ 1}{(1+|q|)^{1+2\sigma-2\rho}}\|\partial_s  Z^I b\|^2_{H^1(\m S^1)} rdr\right)^{\frac{1}{2}}.$$
and so
\begin{equation}
\label{l27}\begin{split}
&\int_0^t \ep^{-1}(1+\tau)\ld{w_1^\frac{1}{2}\frac{\ep}{(1+|q|)^{\frac{3}{2}+\delta-\rho}} \partial_s \partial_\theta Z^I b}^2d\tau\\
&\lesssim \int_0^t \int_{\tau+R}^\infty  \frac{ (1+\tau)r}{(1+|q|)^{1+2\sigma-2\rho}}\|\partial_s Z^I b\|_{H^1}^2dr d\tau\\
&\lesssim \int_R^\infty \frac{1}{(1+|q|)^{1+2\sigma-2\rho}} \left(  \int_0^{2t+q} (1+s)^2\|\partial_s Z^I b\|_{H^1}^2ds\right)dq\\
&\lesssim (1+t)^{2\rho},
\end{split}
\end{equation}
where we have used \eqref{estf1bis1} and \eqref{estf2}.
The term involving $\partial^2_\theta Z^N b$ has already been estimated. We now estimate
$$\ld{w_1^\frac{1}{2}\frac{\ep}{(1+|q|)^{\frac{3}{2}+\delta}(1+s)^{\frac{1}{2}-\rho}}s\partial^2_s \partial_\theta Z^I b }
\lesssim \ep \left(\int \frac{ 1}{(1+|q|)^{1+2\sigma}}
s^{1+2\rho}\|\partial^2_s  Z^I b\|^2_{H^1(\m S^1)} rdr\right)^{\frac{1}{2}}
,$$
and so
\begin{equation}
\label{l28}\int_0^t \ep^{-1}(1+\tau)\ld{w_1^\frac{1}{2}\frac{\ep}{(1+|q|)^{\frac{3}{2}+\delta}(1+s)^{\frac{1}{2}-\rho}}s\partial^2_s  \partial_\theta Z^I b }^2d\tau
\lesssim \ep\int_0^t (1+s)^{3+2\rho}\|\partial^2_s Z^I b\|^2_{H^1(\m S^1)}ds\lesssim \ep^2(1+t)^{4\rho},
\end{equation}
where we have used \eqref{estf11} and \eqref{estf2bis}.
$$\ld{w_1^\frac{1}{2}\frac{\ep}{(1+|q|)^{\frac{3}{2}+\delta}(1+s)^{\frac{1}{2}-\rho}}\partial_s \partial^2_\theta Z^I b }
\lesssim \ep \left(\int \frac{ 1}{(1+|q|)^{1+2\sigma}}
(1+s)^{2\rho}\|\partial_s  Z^I b\|^2_{H^2(\m S^1)} rdr\right)^{\frac{1}{2}},$$
and so
\begin{equation}
\label{l210}\int_0^t \ep^{-1}(1+\tau)\ld{w_1^\frac{1}{2}\frac{\ep}{(1+|q|)^{\frac{3}{2}+\delta}(1+s)^{\frac{1}{2}-\rho}}|\partial_s \partial^2_\theta Z^I b| }^2d\tau
\lesssim \ep\int_0^t (1+s)^{1+2\rho}\|\partial_s \partial^2_\theta Z^I b\|^2_{H^1(\m S^1)}ds\lesssim \ep^2(1+t)^{4\rho},
\end{equation}
where we have used \eqref{estb5}.
We now turn to the term involving $\partial^3_\theta Z^N b.$ Unfortunately, we don't have a good estimate for $\partial^3_\theta Z^N b$.
 To treat terms of the form
 $$(\partial \wht g_1) \frac{q}{r^2}\chi(q)\partial^3_\theta Z^N b,$$
 we remark that
\begin{align*}&\left|(\partial \wht g_1) \frac{q}{r^2}\chi(q)\partial^3_\theta Z^N b-\Box_g ( (r^2g^{\theta \theta})^{-1}q(\partial \wht g_1) \chi(q)\partial_\theta Z^N b)\right|\\&\lesssim 
\frac{\ep}{(1+s)^{\frac{3}{2}-\rho}(1+|q|)^{\frac{3}{2}+\delta}}|\partial^2_\theta Z^N b|
+\frac{\ep}{(1+s)^{\frac{3}{2}-\rho}(1+|q|)^{\frac{1}{2}+\delta}}|\partial_s\partial^2_\theta Z^N b|
+\frac{\ep}{(1+s)^{\frac{1}{2}-\rho}(1+|q|)^{\frac{1}{2}+\delta}}|\partial^2_s\partial_\theta Z^N b|.
\end{align*}
and all these terms have already been estimated.
We easily check that $q(r^2g^{\theta \theta})^{-1}(\partial \wht g_1) \chi(q)\partial_\theta Z^N b$
satisfy the same estimates as the one we want $\wht g$ to satisfy : for instance
$$\int_{\m R^2} w_1 |(r^2g^{\theta \theta})^{-1}\partial \wht g_1 \chi(q)\partial_\theta Z^N b|^2 dx\lesssim \int \frac{s^{2\rho}}{(1+|q|)^{1+2\rho}}\|\partial_\theta Z^N b\|^2_{L^2(\m S^1)}dr \lesssim \ep^4(1+t)^{4\rho}.$$

The terms $Q_{U\ba L}$ and $Q_{L \ba L}$ can be estimated in the same manner as $Q_{\ba L \ba L}$.

${}^IM$ can be estimated by
$$\frac{\ep}{1+s}|\partial Z^N \wht g|+ \frac{\ep}{\sqrt{1+|q|}\sqrt{1+s}}|\bar{\partial} Z^N \wht g|+\frac{\ep}{(1+s)(1+|q|)}|Z^N \wht g|
+\frac{1}{1+s}|Z^N G|+ |\bar{\partial} Z^N G|.
$$
All these terms have already been estimated, except
the last two. They can be estimated by
\begin{align*} &\chi'(q)( |\partial_s \partial_\theta Z^Nb|+s|\partial^2_s Z^Nb|)
+\frac{1}{1+s}|Z^NG^L|
+\frac{1}{1+s}\Upr\left|Z^N\bar{\partial} \int_{\infty}^r (\partial_q \phi)^2rdr\right|\\
&+\frac{1}{(1+r)^2}(1-\chi(q))\Upr|\partial_\theta Z^N h|+\frac{1}{r}(1-\chi(q))\Upr|\partial_s Z^N h|.
\end{align*}
We estimate $\frac{1}{1+s}\Upr\left|Z^N\bar{\partial} \int_{\infty}^r (\partial_q \phi)^2r'dr'\right|$.
We have
\begin{align*}\left\|\bar{\partial}Z^N \int (\partial_q \phi)^2r'dr'\right\|_{L^2(\m S^1)}
&\lesssim \int \frac{\ep}{\sqrt{1+s}(1+|q|)^{\frac{3}{2}-4\rho}}\|\bar{\partial}Z^N \phi\|_{L^2(\m S^1)}r'dr'\\
&\lesssim \ep \ld{w'(q)^\frac{1}{2}\bar{\partial} \partial Z^N \phi},
\end{align*}
so
\begin{equation}\label{l211}\int_0^t \ep^{-1}(1+\tau)\ld{w_1^\frac{1}{2}\frac{1}{1+s}\Up\bar{\partial}Z^N \int (\partial_q \phi)^2r'dr'}^2 d\tau
\lesssim \int_0^t \ep\ld{w'(q)^\frac{1}{2}\bar{\partial} \partial Z^N \phi}^2d\tau\lesssim \ep^3(1+t)^{2\rho}.
\end{equation}
The contributions \eqref{l23}, \eqref{l25} and \eqref{l25bbis} correspond to $A^N$, and the contributions \eqref{l21}, \eqref{l22}, \eqref{l25bis}, \eqref{l26}, \eqref{l27}, \eqref{l28}, \eqref{l210} and \eqref{l211} correspond to $B^N$.

The terms we will now estimate correspond to $C^N$.
We now estimate the contribution of
$$\frac{1}{(1+r)^2}(1-\chi(q))\Upr\partial_\theta Z^N h.
$$
which appears in $\bar{\partial} Z^N G$. The estimates for $Z^N h$ are given by Corollary \ref{estih}, but we do not have bootstrap assumptions for  $\partial_\theta Z^N h$. Consequently, we will estimate this term with integration by parts in the energy estimate. We calculate
\begin{align*}
&\int_0^t \int w_1 \frac{1}{(1+r)^2}\Uprt(1-\chi(q))(\partial_\theta Z^N h) (\partial_\tau Z^N \wht g)dvol_gd\tau\\
=&\int_0^t \int  w_1 \frac{1}{(1+r)}\Uprt(1-\chi(q))(Z^N h) (\partial_\tau\bar{\partial} Z^N \wht g)dvol_g d\tau
\\
&+\int_0^t \int  w_1 \frac{1}{(1+r)^2}\Uprt(1-\chi(q))(Z^N h) (\partial_\tau Z^N \wht g)(\partial_\theta \sqrt{|det(g)|})dx d\tau
\\
=&\left[\int  w_1 \frac{1}{(1+r)}\Uprt(1-\chi(q))(Z^N h) (\bar{\partial} Z^N \wht g) dvol_g\right]_0^t\\
&-\int_0^t\int w_1\Uprt \frac{1}{(1+r)}(1-\chi(q))(\partial_s Z^N h)(\bar{\partial} Z^N \wht g)dvol_gd\tau\\
&-\int_0^t \int \partial_\tau\left (w_1(1-\chi(q))\Uprt \sqrt{|det(g)|}\right)\frac{1}{(1+r)} (Z^N h)(\bar{\partial} Z^N \wht g)dxd\tau\\
&+\int_0^t \int  w_1 \frac{1}{(1+r)^2}\Uprt(1-\chi(q))(Z^N h) (\partial_\tau Z^N \wht g)(\partial_\theta \sqrt{|det(g)|})dx d\tau.
\\
\end{align*}
We estimate
\begin{equation}\label{l212}\begin{split}&\left|\left[\int  w_1 \frac{1}{(1+r)}\Uprt(1-\chi(q))(Z^N h) (\bar{\partial} Z^N \wht g) dvol_g\right]_0^t\right|\\
&\lesssim
\ld{ w_1^\frac{1}{2} \frac{1}{1+r}\Upr(1-\chi(q))Z^N h}\ld{w_1^\frac{1}{2}\partial Z^N \wht g} \lesssim \ep (1+t)^{3\rho},
\end{split}
\end{equation}
\begin{align*}
&\left|\int_0^t\int w_1\frac{1}{(1+r)}\Uprt(1-\chi(q))(\partial_s Z^N h)(\bar{\partial} Z^N \wht g)dvol_gd\tau\right|\\
&\lesssim \int_0^t \ld{w_1^\frac{1}{2} \frac{1}{(1+r)}(1-\chi(q))\Uprt\partial_s Z^N h}
\ld{w_1^\frac{1}{2}\partial Z^N \wht g}\\
&\lesssim \int_0^t \ep^{-1}(1+\tau)\int w_1 \frac{(1-\chi(q))^2}{r^2}\Uprt\|\partial_s Z^N h\|^2_{L^2(\m S^1)}
rdrd\tau + \int_0^t \frac{\ep}{(1+\tau)}\ld{w_1^\frac{1}{2}\partial Z^N \wht g}^2d\tau\\
&\lesssim \ep^3(1+t)^{4\rho}+\int_0^t (1+s)\|\partial_s Z^N h\|^2_{L^2(\m S^1)}ds
\end{align*}
and consequently
\begin{equation}
\label{l213}\left|\int_0^t\int w_1\frac{1}{(1+r)}\Uprt(1-\chi(q))(\partial_s Z^N h)(\bar{\partial} Z^N \wht g) dvol_g d\tau\right|\lesssim \ep^3(1+t)^{4\rho}
\end{equation}
\begin{align*}
&\left|
\int_0^t \int \partial_\tau \left(w_1(1-\chi(q))\Uprt \sqrt{|det(g)|}\right)\frac{1}{(1+r)} (Z^N h)\bar{\partial} Z^N \wht g dxt\tau\right|\\
&\lesssim \int_0^t \ld{w^\frac{1}{2}_1(1+|q|)^{-\frac{1}{2}} \Uprt\frac{1}{(1+r)}(1-\chi(q)) Z^N h}
\ld{w'_1(q)^\frac{1}{2}\bar{\partial} Z^N \wht g}\\
&\lesssim \int_0^t \ep^{-1}\int w_1 \frac{(1-\chi(q))^2}{r^2(1+|q|)}\Uprt\| Z^N h\|^2_{L^2(\m S^1)}
rdr d\tau+ \int_0^t \ep\ld{w'_1(q)^\frac{1}{2}\bar{\partial} Z^N \wht g}^2d\tau\\
&\lesssim \ep^3(1+t)^{4\rho}+\frac{1}{\ep}\int_0^t \frac{1}{1+\tau}\| Z^N h\|^2_{L^2(\m S^1)}d\tau
\end{align*}
and consequently
\begin{equation}
\label{l214}\left|
\int_0^t \int \partial_\tau \left(w_1(1-\chi(q))\Uprt\sqrt{|det(g)|}\right)\frac{1}{(1+r)} Z^N h\bar{\partial} Z^N \wht gdxt\tau\right|\lesssim \ep^3(1+t)^{4\rho}.
\end{equation}
The last term can be estimated in the same way.
The term $\chi'(q) s\partial^2_s Z^N b$, which is also present in $R^1_{\mu \nu}$ can be estimated in the following way :
we estimate first the term
$$\chi'(q)s\partial_s  Z^N f_2,$$
where we use the decomposition of $\partial_s Z^N b$ \eqref{decb}.
We calculate
\begin{align*}
& \int w_1(q)\chi'(q)s(\partial_s Z^N f_2) (\partial_t Z^N \wht g) \sqrt{|det(g)|}rdrd\theta\\
 &=\frac{1}{2}\partial_t \int w_1(q)\chi'(q)s (Z^N f_2 )(\partial_t Z^N \wht g) \sqrt{|det(g)|}rdrd\theta
 -\frac{1}{2}\int \partial_t( w_1(q)\chi'(q)s r\sqrt{|det(g)|}) (Z^N f_2) (\partial_t Z^N \wht g) drd\theta\\
 &-\frac{1}{2}\int  w_1(q)\chi'(q)s (Z^N f_2) (\partial^2_t Z^N \wht g)\sqrt{|det(g)|}rdrd\theta\\
 &-\frac{1}{2}\int \partial_r( w_1(q)\chi'(q)s r\sqrt{|det(g)|}) (Z^N f_2) (\partial_t Z^N \wht g) drd\theta
 -\frac{1}{2}\int  w_1(q)\chi'(q)s (Z^N f_2) (\partial_t \partial_r Z^N \wht g )\sqrt{|det(g)|}rdrd\theta\\
 &=\frac{1}{2}\partial_t \int w_1(q)\chi'(q)s (Z^N f_2 )(\partial_t Z^N \wht g)\sqrt{|det(g)|} rdrd\theta
  -\int \partial_s( w_1(q)\chi'(q)s\sqrt{|det(g)|})(Z^N f_2) (\partial_t Z^N \wht g) rdrd\theta\\
&  -\partial_t  \int w_1(q)\chi'(q)s (Z^N f_2) (\partial_s Z^N \wht g) \sqrt{|det(g)|}rdrd\theta
  +\int \partial_t(w_1(q)\chi'(q)sZ^N f_2)(\partial_s Z^N \wht g)\sqrt{|det(g)|} rdrd\theta\\
&=\partial_t A   -\int \partial_s( w_1(q)\chi'(q)s\sqrt{|det(g)|})(Z^N f_2) (\partial_t Z^N \wht g) rdrd\theta
  -\int \partial_t(w_1(q)\chi'(q)sZ^N f_2\sqrt{|det(g)|})(\partial_s Z^N \wht g) rdrd\theta
\end{align*}
where we have noted
$$A=\frac{1}{2}\int w_1(q)\chi'(q)s Z^N f_2 \partial_t Z^N \wht g\sqrt{|det(g)|} rdrd\theta- \int w_1(q)\chi'(q)s Z^N f_2 \partial_s Z^N \wht g \sqrt{|det(g)|}rdrd\theta.$$
We estimate, noticing that in the region $\chi'(q)\neq 0$ we have $t\sim s \sim r$, $q$ is bounded from above and from bellow and that $|\partial_s (s\sqrt{|det(g)|})|\lesssim 1$,
\begin{align*}
\int_0^t\left|\int w_1(q)\chi'(q)(Z^N f_2) (\partial_t Z^N \wht g) rdrd\theta\right|
&\lesssim\int_0^t \left(\int |\chi'(q)|( Z^N f_2)^2rdrd\theta\right)^\frac{1}{2}\left(\int |\chi'(q)|(\partial_t Z^N \wht g) rdrd\theta\right)^\frac{1}{2}\\
&\lesssim \int_0^t\frac{\ep}{1+t}\|w_1^\frac{1}{2}\partial Z^N \wht g\|^2_{L^2}
+\int_0^t\int \frac{1}{\ep}(1+\tau)^2|\chi'(q)|\| Z^N f_2\|^2_{L^2(\m S^1)}drd\tau\\
\end{align*}
and so
\begin{equation}
\label{l215}
\int_0^t \left|\int w_1(q)\chi'(q)(Z^N f_2) (\partial_t Z^N \wht g) rdd\theta\right| \lesssim \ep^3(1+t)^{4\rho}
+\int_0^t (1+s)^2\|Z^N f_2\|^2_{L^2(\m S^1)}ds\lesssim \ep^3(1+t)^{4\rho},
\end{equation}
where we have used \eqref{estf2}. Noticing that $|\partial_\tau(w_1(q)\chi'(q)s\sqrt{|det(g)|})| \lesssim s$ we estimate
\begin{align*}
&\int_0^t\left| \int \partial_\tau(w_1(q)\chi'(q)s\sqrt{|det(g)|})(Z^N f_2)(\partial_s Z^N \wht g) rdrd\theta\right|\\
&\lesssim \int_0^t\left(\int |\chi'(q)|( sZ^N f_2)^2\theta\right)^\frac{1}{2}\left(\int |\chi'(q)|(\partial_s Z^N \wht g )^2rdd\theta\right)^\frac{1}{2}+s.t.\\
&\lesssim \ep \int_0^t\|w_1'(q)^\frac{1}{2}\bar{\partial} Z^N \wht g\|^2_{L^2}d\tau
+\int_0^t\int \frac{1}{\ep}(1+\tau)^3\chi'(q)^2\| Z^N f_2\|^2_{L^2(\m S^1)}dr d\tau
\end{align*}
and so
\begin{equation}
\label{l216}\int_0^t\left| \int \partial_t(w_1(q)\chi'(q)s)Z^N f_2\partial_s Z^N \wht g rdrd\theta\right|d\tau\lesssim \ep^3(1+t)^{4\rho},
\end{equation}
where we have used \eqref{estf2}.
\begin{align*}
&\int_0^t\left| \int w_1(q)\chi'(q)s(\partial_s Z^N f_2)(\partial_s Z^N \wht g) rdrd\theta\right|d\tau\\
&\lesssim \int_0^t\left(\int |\chi'(q)|(s \partial_s Z^N f_2)^2rdrd\theta\right)^\frac{1}{2}\left(\int |\chi'(q)|(\partial_s Z^N \wht g)^2 rdrd\theta\right)^\frac{1}{2}d\tau\\
&\lesssim \ep \int_0^t\|w_1'(q)^\frac{1}{2}\bar{\partial} Z^N \wht g\|^2_{L^2}d\tau
+\int_0^t\int \frac{1}{\ep}(1+\tau)^3|\chi'(q)|\| \partial_s Z^N f_2\|^2_{L^2(\m S^1)}dr d\tau
\end{align*}
and so
\begin{equation}
\label{l216bis}\int_0^t\left| \int w_1(q)\chi'(q)s(\partial_s Z^N f_2)(\partial_s Z^N \wht g) rdrd\theta\right|d\tau\lesssim \ep^3(1+t)^{4\rho},
\end{equation}
where we have used \eqref{estf2bis}.
We now turn to the estimate of $A$
\begin{equation}\label{l217}\begin{split}
|A|&\lesssim \left|\int w_1(q)\chi'(q)s(Z^N f_2) (\partial Z^N \wht g) rdrd\theta\right|\\
&\lesssim (1+t)^\frac{3}{2}\|Z^N f_2\|_{L^2(\m S^1)}\|w_1^\frac{1}{2}\partial Z^N \wht g\|_{L^2} \lesssim \ep^3(1+t)^{3\rho},
\end{split}
\end{equation}
where we have used \eqref{estf2b}.
We now estimate the contribution of
$$\chi'(q)s\partial_s Z^Nf_1.$$
We have
\begin{align*}
&\int_0^t\left|\int w_1(q)\chi'(q)s(\partial_s Z^Nf_1)( \partial_t Z^N \wht g) rdrd\theta\right|d\tau\\
&\lesssim \int_0^t\int\frac{1}{\ep}(1+\tau)^{4}|\chi'(q)|\|\partial_s Z^N f_1\|^2_{L^2(\m S^1)}drd\tau+\int_0^t\frac{\ep}{(1+\tau)}\|w_1^\frac{1}{2} \partial_t Z^N \wht g \|^2_{L^2}d\tau
\end{align*}
and consequently
\begin{equation}\label{l218}
\int_0^t\left|\int w_1(q)\chi'(q)s(\partial_s Z^Nf_1) (\partial_t Z^N \wht g) rdrd\theta\right|\lesssim \ep^3(1+t)^{4\rho}+\ep^ {-1}\int_0^t (1+s)^4\|\partial_s Z^N f_1\|^2_{L^2(\m S^1)}\lesssim \ep^3(1+t)^{4\rho},
\end{equation}
where we have used \eqref{estf1}.
Estimates \eqref{an}, \eqref{bn}, \eqref{l212}, \eqref{l213}, \eqref{l214}, \eqref{l215}, \eqref{l216}, \eqref{l216bis}, \eqref{l217} and \eqref{l218} conclude the proof of Proposition \ref{prpl2whtg}.
\end{proof}

\begin{proof}[Proof of Corollary \ref{corl2whtg}]
	We use the energy estimate 
\begin{align*}&\frac{d}{dt}\left(\int (\partial Z^N \wht g)^2w_1(q)dvol_g\right)+C\int w'(q)(\bar{\partial}Z^N \wht g)^2\\
	&\lesssim \frac{\ep}{1+t}\int w_1(q)(\partial Z^N \wht g)^2 + \left|\int w_1(q)\partial_t Z^N \wht g \Box_g Z^N \wht gdvol_g\right|,
	\end{align*}
	we multiply it by $(1+t)^{-2\rho}$, and notice that
	$$\frac{d}{dt}\left((1+t)^{-2\rho}\int (\partial Z^N \wht g)^2w_1(q)\right)\lesssim
	(1+t)^{-2\rho}\frac{d}{dt}\left(\int (\partial Z^N \wht g)^2w_1(q)\right).$$
	Then Corollary \ref{corl2whtg} can be proved with exactly the same steps as Proposition \ref{prpl2whtg}.
\end{proof}

\begin{proof}[Proof of Corollary \ref{col2k}]
	We perform the energy estimate for $k$
	\begin{align*}
	&\frac{d}{dt}\left(\int (\partial Z^N k)^2w_1(q)dvol_g\right)+C\int w'(q)(\bar{\partial}Z^N k)^2\\
&	\lesssim \frac{\ep}{1+t}\int w_1(q)(\partial Z^N k)^2 + \left|\int w_1(q)\partial_t Z^N \wht g \Box_g Z^N kdvol_g\right|,
	\end{align*}
	then the fact that the initial data for $k$ are $0$, and that $\Box_g k$ satisfy the same estimates as $\Box_g \wht g$ yield Corollary \ref{col2k}.
\end{proof}



\subsection{Estimate of $\partial Z^N \wht g_{1}$}\label{l2g1}
We need the following corollary of Proposition \ref{prpstr}.
\begin{cor}\label{corstr}
We have
$$\Box_g Z^N \wht g_{1} =Z^NR^1+ {}^N M +{}^NM^E+{}^NQ_{\q T \q V}+ O\left(\Upr\frac{1}{r^2}\partial_\theta k\right),$$
\end{cor}

\begin{proof}
	We expressed the 2-forms $dq^2$ in the coordinate $(t,x_1,x_2)$
	$$
	dq^2= (dr-dt)^2=(\cos(\theta)dx^1+\sin(\theta)dx^2-dt)^2\\
	$$
	Therefore, we will have, in the coordinates
	$x_1,x_2$
	\begin{equation}
	\label{noncom}\Box \left( \Up\left(\frac{r}{t}\right)g_{\ba L \ba L}dq^2\right)_{\mu \nu}
	-\Box\left(\Up\left(\frac{r}{t}\right)g_{\ba L \ba L}\right)(dq^2)_{\mu \nu}
	=\Up\left(\frac{r}{t}\right)\frac{1}{r^2}\left(u^1_{\mu \nu}(\theta)g_{\ba L \ba L}
	+u^2_{\mu \nu}(\theta)\partial_\theta g_{\ba L \ba L}\right)
	\end{equation}
	where $u^1_{\mu \nu}$ and $u^2_{\mu \nu}$ are some trigonometric functions.
\end{proof}

\begin{prp}\label{prpwhtg1n}
We have 
$$\ld{w_2^\frac{1}{2}\partial Z^N \wht g_{1}}\leq C_0 \ep + C\ep^{\frac{5}{4}}(1+t)^\rho,$$
$$\int_0^t \ld{w'_2(q)^\frac{1}{2}\bar{\partial} Z^N \wht g_{1}}^2 \lesssim C_0 \ep^2+C\ep^{\frac{5}{2}}(1+t)^{2\rho}.$$
\end{prp}

\begin{proof}
We use the weighted energy estimate in the background metric $g$
\begin{align*}&\|w_{2}^\frac{1}{2} \partial Z^N \wht g_{1}(t)\|^2_{L^2}+\int_0^t \|w_2'(q)^\frac{1}{2}\bar{\partial}Z^N \wht g_{1}\|^2_{L^2}\\
&\lesssim \|w_{2}^\frac{1}{2} \partial Z^N \wht g_{1}(0)\|^2_{L^2}+\int_0^t \frac{\ep}{1+t}\|w_{2}^\frac{1}{2} \partial Z^N \wht g_{1}(t)\|^2_{L^2} + \left|\int_0^t \int w_2\partial_t Z^N \wht g_{1} \Box_g Z^N \wht g_{1}dvol_gdt\right|.
\end{align*}
We decompose $\Box_g Z^N \wht g_{\q T\q T}= A^N +B^N +C^N$ with
$$\ld{w_2^\frac{1}{2}A^N}\lesssim \frac{\ep^2}{(1+t)^{1-\rho}}$$
$$\int_0^t\ep^{-1}(1+s)\ld{w_2^\frac{1}{2}B^N}^2ds \lesssim \ep^\frac{5}{2}(1+t)^{2\rho}$$
and $C^N$ is dealt with in a specific manner.
We start with ${}^N M$
\begin{align*}
|{}^NM|
&\lesssim
\frac{\ep}{(1+s)^\frac{3}{2}(1+|q|)^{\frac{1}{2}-4\rho}}|\partial Z^N \phi| +
\frac{\ep}{(1+s)(1+|q|)^{\frac{1}{2}-\rho}}\left(|\bar{\partial} Z^N \wht g|+\frac{1}{1+s}|Z^N \wht g|\right)\\
&+
\frac{\ep}{(1+s)^{\frac{3}{2}-\rho}}|\partial Z^N \wht g|
+\frac{\ep}{(1+s)(1+|q|)^{\frac{3}{2}-\rho}}|Z^N \wht g_{\q T \q T}|+\frac{1}{1+s}|Z^N G|+|\bar{\partial}Z^N G|\\
&+\ep\min\left(\frac{1}{(1+|q|)(1+s)^{\frac{1}{2}-\rho}},\frac{1}{(1+|q|)^{\frac{1}{2}}(1+s)^\frac{1}{2}}\right)\left(|\bar{\partial}Z^N \wht g_1|+\frac{1}{1+|q|}|Z^N g_{LL}|\right).
\end{align*}
We estimate
$$\ld{w_2^\frac{1}{2}\frac{\ep}{(1+s)(1+|q|)^{\frac{1}{2}-\rho}}\bar{\partial} Z^N \wht g}
\lesssim \frac{\ep}{(1+t)}\ld{ w_1'(q)^\frac{1}{2} \bar{\partial}Z^N \wht g},$$
where we used $\frac{w_2(q)^\frac{1}{2}}{(1+|q|)^{\frac{1}{2}-\rho}}\leq w'_1(q)^\frac{1}{2}$ and so
$$\int_0^t \ep^{-1}(1+\tau)\ld{w_2^\frac{1}{2}\frac{\ep}{(1+s)(1+|q|)^{\frac{1}{2}-\rho}}\bar{\partial} Z^N \wht g}^2d\tau
\lesssim \int_0^t \ep (1+t)^{-2\rho}\ld{ w_1'(q)^\frac{1}{2} \bar{\partial}Z^N \wht g}^2
\lesssim \ep^3(1+t)^{2\rho}.$$
We estimate
$$\ld{w_2^\frac{1}{2}\frac{\ep}{(1+s)^2(1+|q|)^{\frac{1}{2}-\rho}} Z^N \wht g}
\lesssim \frac{\ep}{(1+t)^{1+2\rho}}\ld{ \frac{w_1(q)^\frac{1}{2}}{1+|q|} Z^N \wht g}\lesssim \frac{\ep^3}{1+t},$$
where we have used \eqref{ig2Nbis},
$$\ld{w_2^\frac{1}{2}\frac{\ep}{(1+s)^{\frac{3}{2}-\rho}}\partial Z^N \wht g}
\lesssim \frac{\ep}{(1+t)^{1+2\rho}}\ld{ w_1(q)^\frac{1}{2} \partial Z^N \wht g}\lesssim \frac{\ep^3}{1+t},$$
$$\ld{w_2^\frac{1}{2}\frac{\ep}{(1+s)(1+|q|)^{\frac{3}{2}-\rho}}Z^N \wht g_{\q T \q T}}
\lesssim \frac{1}{1+t}\ld{\frac{w_2^\frac{1}{2}}{1+|q|} Z^N \wht g_1}\lesssim \frac{\ep^3}{(1+t)^{1-\rho}},$$
where we have used \eqref{hg1n}.
$$\ld{w_2^\frac{1}{2}\frac{\ep}{(1+|q|)^{\frac{1}{2}}(1+s)^\frac{1}{2}}\bar{\partial}Z^N \wht g_{\q T \q T}}
\lesssim \frac{\ep}{(1+t)^\frac{1}{2}}\ld{w_2'(q)^\frac{1}{2}\bar{\partial}Z^N \wht g_{\q T \q T}},$$
so
$$\int_0^t \ep^{-1}(1+\tau)\ld{w_2^\frac{1}{2}\frac{\ep}{(1+|q|)^{\frac{1}{2}}(1+s)^\frac{1}{2}}\bar{\partial}Z^N \wht g_{1}}^2d\tau
\lesssim \int_0^t \ep \ld{w_2'(q)^\frac{1}{2}\bar{\partial}Z^N \wht g_{1}}^2d\tau
\lesssim \ep^3(1+t)^{2\rho},$$
where we have used \eqref{igN}.
$$\int_0^t \ep^{-1}(1+\tau)\ld{w_2^\frac{1}{2}\frac{\ep}{(1+|q|)^{\frac{3}{2}}(1+s)^\frac{1}{2}}Z^N \wht g_{LL}}^2d\tau
\lesssim \int_0^t \ld{\frac{w_2^\frac{1}{2}}{(1+|q|)^\frac{3}{2}}Z^N \wht g_{LL}}^2d\tau\lesssim \ep(1+t)^{2\rho},$$
where we have used \eqref{hwl2dll}.
We estimate the term involving $G$ in the same way than in the previous section : see \eqref{l211} to \eqref{l218}.
We now estimate the contribution of terms coming from the non commutation of the wave operator with the null frame.
It is sufficient to estimate
$$\ld{w_2^\frac{1}{2}\Upr\frac{1}{r}\bar{\partial} k}
\lesssim \frac{1}{(1+t)^{\frac{1}{2}+\sigma}}\ld{\frac{w_1^\frac{1}{2}}{(1+|q|)^\frac{1}{2}} \bar{\partial }k},$$
which yields
$$\int_0^t \ep^{-\frac{1}{2}}(1+\tau)\ld{w_2^\frac{1}{2}\Upr \frac{1}{r}\bar{\partial} k}^2d\tau
\lesssim \ep^{-\frac{1}{2}} \int_0^t (1+t)^{-2\rho}\ld{\frac{w_1^\frac{1}{2}}{(1+|q|)^\frac{1}{2}} \bar{\partial }k}^2d\tau
\lesssim \ep^\frac{5}{2}(1+t)^{2\rho},$$
thanks to Corollary \ref{col2k}.

We now estimate the contribution of $M^E$. :
\begin{align*}|M^E|\lesssim &\frac{\ep}{(1+s)}\left(|\bar{\partial} Z^N \wht g|+ \frac{|q|}{1+s}|\partial Z^N \wht g|+ \frac{1}{1+s}|Z^N \wht g |+ \frac{1}{1+|q|}|\wht g_{\q T \q T}|\right)\\
&+\frac{\ep}{(1+s)(1+|q|)^{2+\delta-\rho}} \left(s|\partial_s Z^N b|
+q|\partial_s \partial_\theta Z^N b|+\frac{q}{s}|Z^N \partial^2_\theta b|\right)\\
&\min\left(\frac{\ep}{(1+|q|)^{\frac{3}{2}+\delta}(1+s)^{\frac{1}{2}-\rho}}, \frac{\ep}{(1+|q|)^{\frac{3}{2}+\delta-\sigma}(1+s)^\frac{1}{2}}\right)\Big(s|\partial^2_s Z^N b|
+q|\partial^2_s \partial_\theta Z^N b|\\
&+
\frac{q}{s}|Z^N \partial_s \partial^2_\theta b|
+\frac{q}{s^2}|Z^N \partial^3_\theta b|\Big).
\end{align*}
We have
$$\ld{w_2^\frac{1}{2}\frac{\ep}{(1+s)}\bar{\partial} Z^N \wht g}
\lesssim \ld{w_1^\frac{1}{2}\frac{\ep}{(1+s)(1+|q|)^{\sigma}}\bar{\partial} Z^N \wht g}
\lesssim \frac{\ep}{(1+t)^{\frac{1}{2}+\rho}}\ld{w_1'(q)^\frac{1}{2}\bar{\partial} Z^N \wht g},$$ and so
$$\int_0^t \ep^{-1}(1+\tau)\ld{w_2^\frac{1}{2}\frac{\ep}{(1+s)}\bar{\partial} Z^N \wht g}^2d\tau
\lesssim \ep \int_0^t (1+\tau)^{-2\rho}\ld{w_1'(q)^\frac{1}{2}\bar{\partial} Z^N \wht g}^2d\tau\lesssim \ep^3(1+t)^{2\rho}.$$
We proceed in a similar way for the other terms involving $\wht g$.
The term 
$\frac{\ep}{(1+s)(1+|q|)^{2+\delta-\rho}}s|\partial_s\partial_\theta Z^N b|$ 
can be estimated like \eqref{l27}. We estimate
\begin{align*}
\ld{w_2^\frac{1}{2}\frac{\ep}{(1+s)^2(1+|q|)^{1+\delta-\rho}}Z^N \partial^2_\theta b}
&\lesssim \left( \int \frac{\ep^2}{(1+s)^4(1+|q|)^{2\sigma-\rho}}\|Z^N b\|_{H^2(\m S^1)}^2rdr\right)^\frac{1}{2}\\
&\lesssim \frac{\ep^3}{1+t}.
\end{align*}
We now estimate
$$\ld{w_2^\frac{1}{2}\frac{\ep}{(1+|q|)^{\frac{3}{2}+\delta-\sigma}(1+s)^{\frac{1}{2}}}s\partial^2_s \partial_\theta Z^N b }
\lesssim \ep \left(\int \frac{ 1}{(1+|q|)^{1+2\sigma}}
s\|\partial^2_s  Z^N b\|^2_{H^1(\m S^1)} rdr\right)^{\frac{1}{2}}
,$$
and so
$$\int_0^t \ep^{-1}(1+\tau)\ld{w_2^\frac{1}{2}\frac{\ep}{(1+|q|)^{\frac{3}{2}+\delta-\sigma}(1+s)^{\frac{1}{2}}}s\partial^2_s \partial_\theta ^N b }^2d\tau
\lesssim \ep\int_0^t (1+s)^{3}\|\partial^2_s Z^N b\|^2_{H^1(\m S^1)}ds\lesssim \ep^2(1+t)^{2\rho}.$$
The other terms can be estimated in the same way.

We now treat the terms $Q_{L\ba L}$ and $Q_{U\ba L}$.
We start with $\partial_q\wht g_{\ba L \ba L}\partial_s Z^N \wht g_{LL}$, that we estimate by integration by parts
\begin{align*}
&\int w_2 (\partial_q\wht g_{\ba L \ba L})(\partial_s Z^N \wht g_{LL})(\partial_t Z^N \wht g_1) \sqrt{|det(g)|}dx\\
=& \frac{1}{2}\partial_t\int w_2 (\partial_q \wht g_{\ba L \ba L})(Z^N \wht g_{LL})(\partial_t Z^N \wht g_1) \sqrt{|det(g)|}dx
-\int w_2\partial_s (r\sqrt{|det(g)|}\partial_q \wht g_{\ba L \ba L})(Z^N \wht g_{LL})(\partial_t Z^N \wht g_1) drd\theta\\
&-\partial_t \int w_2(\partial_q \wht g_{\ba L \ba L})(Z^N \wht g_{LL})(\partial_s Z^N \wht g_1)\sqrt{|det(g)|}dx
+\int \partial_t( w_2\sqrt{|det(g)|}(\partial_q \wht g_{\ba L \ba L})Z^N \wht g_{LL})\partial_s Z^N \wht g_1dx.
\end{align*}
We estimate
\begin{align*}\left|\int w_2 (\partial_q \wht g_{\ba L \ba L})(Z^N \wht g_{LL})(\partial_t Z^N \wht g_1) dvol_g\right|&\lesssim 
\int \frac{\ep w_2}{(1+|q|)(1+s)^{\frac{1}{2}-\rho}}|Z^N \wht g_{LL}\partial_t Z^N \wht g_1|\\
&\lesssim \frac{\ep}{(1+t)^{\frac{1}{2}-\rho}}\ld{\frac{w_2^\frac{1}{2}}{1+|q|}Z^N \wht g_{LL}}\ld{w_2^\frac{1}{2}\partial_t Z^N \wht g_1}\lesssim \frac{\ep^3}{(1+t)^{\frac{1}{2}-3\rho}},
\end{align*}
\begin{align*}
\left|\int w_2\partial_s(\sqrt{|det(g)|} \partial_q \wht g_{\ba L \ba L})(Z^N \wht g_{LL})(\partial_t Z^N \wht g_1) dx\right|
&\lesssim \int \frac{\ep w_2}{(1+s)^{\frac{3}{2}-\rho}(1+|q|)}|Z^N \wht g_{LL}\partial_t Z^N \wht g_1|\\
&\lesssim \frac{\ep}{(1+t)^{\frac{3}{2}-\rho}}\ld{\frac{w_2^\frac{1}{2}}{1+|q|}Z^N \wht g_{LL}}\ld{w_2^\frac{1}{2}\partial_t Z^N \wht g_1}\\
&\lesssim \frac{\ep^3}{(1+t)^{\frac{3}{2}-3\rho}},
\end{align*}
$$
\left|\int  w_2(\partial_q \wht g_{\ba L \ba L})(\partial_t Z^N \wht g_{LL})(\partial_s Z^N \wht g_1)dvol_g\right|\\
\lesssim \frac{\ep}{(1+t)^{\frac{1}{2}-\rho}}\ld{w'_2(q)^\frac{1}{2}\partial Z^N \wht g_{LL}}\ld{w_2'(q)^\frac{1}{2} \partial_s Z^N \wht g_1},
$$
and consequently
$$\int_0^t \left|\int  w_2(\partial_q \wht g_{\ba L \ba L})(\partial_t Z^N \wht g_{LL})(\partial_s Z^N \wht g_1)\sqrt{|det(g)|}dx\right| \lesssim \ep^3(1+t)^{2\rho}.$$
The term $\partial_q \wht g_{\ba L \ba L}\partial_U Z^N \wht g_{LL}$ is similar to estimate.
We now turn to $\partial_q \wht g_{\ba L \ba L}\partial_s Z^N \sigma^0_{UL}$. We follow the same calculation, noticing that we have the estimate for $\partial_q \wht g_{\ba L \ba L}$ 
$$|w_2^\frac{1}{2}\partial_q g_{\ba L \ba L}|\lesssim  \frac{\ep}{(1+s)^{\frac{1}{2}-\rho}(1+|q|)^{\frac{1}{2}+\sigma}},$$
and consequently
\begin{align*}\left|\int w_2 (\partial_q \wht g_{\ba L \ba L})(Z^N \sigma^0_{UL})(\partial_t Z^N \wht g_1) dx\right|&\lesssim 
\int \frac{\ep w_2}{(1+|q|)^{\frac{1}{2}+\sigma}(1+s)^{\frac{1}{2}-\rho}}|Z^N \sigma^0_{UL}\partial_t Z^N \wht g_1|\\
&\lesssim \ep\ld{w_2^\frac{1}{2}\partial_t Z^N \wht g_1}
\left(\int \frac{1}{(1+|q|)^{1+\sigma}}(1+s)^{2+2\rho}\|\partial_s Z^N b\|_{L^2}^2\right)^\frac{1}{2}\\
&\lesssim \ep^3(1+t)^{2\rho},
\end{align*}
where we have used \eqref{estf2} and \eqref{estf1b}.
\begin{align*}
\left|\int w_2\partial_s(\sqrt{|det(g)|} \partial_q \wht g_{\ba L \ba L})(Z^N \sigma^0_{UL})(\partial_t Z^N \wht g_1) dx\right|
&\lesssim \int \frac{\ep w_2}{(1+s)^{\frac{3}{2}-\rho}(1+|q|)^{\frac{3}{2}+\delta}}|Z^N \sigma^0_{UL}\partial_t Z^N \wht g_1|\\
&\lesssim \ep\ld{w_2^\frac{1}{2}\partial_t Z^N \wht g_1}
\left(\int \frac{1}{(1+|q|)^{1+\sigma}}(1+s)^{2\rho}\|\partial_s Z^N b\|_{L^2}^2dr\right)^\frac{1}{2}
\end{align*}
and consequently
\begin{align*}&\int_0^t\left|\int w_2\partial_s(\sqrt{|det(g)|} \partial_q \wht g_{\ba L \ba L})(Z^N \sigma^0_{UL})(\partial_t Z^N \wht g_1) dx\right|d\tau\\
&\lesssim \int_0^t \frac{\ep}{1+\tau}\ld{w_2^\frac{1}{2}\partial_t Z^N \wht g_1}^2+\ep\int_0^t (1+\tau)^{1+2\rho}\|\partial_s Z^N b\|_{L^2}^2d\tau\\
&\lesssim\ep^3(1+t)^{2\rho}.
\end{align*}
\begin{align*}
&\left|\int  \partial_t(w_2\sqrt{|det(g)|}(\partial_q \wht g_{\ba L \ba L} )(Z^N \sigma^0_{UL}))(\partial_s Z^N \wht g_1)dx\right|\\
&\lesssim \int \frac{\ep w_2}{(1+s)^{\frac{1}{2}-\rho}(1+|q|)^{\frac{1}{2}+\sigma}}\left(\frac{1}{1+|q|}|Z^N \sigma^0_{UL}|
+|\partial_s Z^N \sigma^0_{UL}|\right)
|\partial_s Z^N \wht g_1|\\
&\lesssim \ep\ld{w'_2(q)^\frac{1}{2}\partial_s Z^N \wht g_1}\left(\int \frac{1}{(1+|q|)^\sigma}\left(\frac{1}{1+|q|}(1+s)^{2+2\rho}\|\partial_s Z^N b\|^2_{L^2(\m S^2)}+(1+s)^{2+2\rho}\|\partial^2_s Z^N b\|^2_{L^2(\m S^1)}\right)dr\right)^\frac{1}{2}
\end{align*}
and consequently
\begin{align*}&\int_0^t \left|\int  \partial_t(w_2\sqrt{|det(g)|}(\partial_q \wht g_{\ba L \ba L} )(Z^N \sigma^0_{UL}))(\partial_s Z^N \wht g_1)dx\right|d\tau\\ &\lesssim 
\ep\int_0^t\ld{w'_2(q)^\frac{1}{2}\partial_s Z^N \wht g_1}^2d\tau+\int_0^t \left((1+s)^{2+2\rho}\|\partial_s Z^N b\|_{L^2(\m S^1)}^2+
(1+s)^{3}\|\partial^2_s Z^N b\|_{L^2(\m S^1)}^2\right)ds\\
&\lesssim \ep^3(1+t)^{2\rho}.
\end{align*}

We now turn to the term $s\chi'(q)\partial^2_s Z^N b$. We cannot do the same reasoning as before because of the estimates
\eqref{l216bis} and \eqref{l217}, which are a consequence of the additional loss in $t^{\rho}$ in \eqref{estf2}.
 Instead, we remark the calculation
\begin{align*}
&\Box_g (g^{L\ba L})^{-1}\left(\Upr(\chi(q)-1)s\partial_s Z^N b\right)\\
=&s\chi'(q)\partial^2_s Z^N b
+O\left( \Upr \partial_s \partial_\theta Z^N b \right)+O\left(\frac{1}{r}\Upr \partial_s \partial^2_\theta Z^N b\right)\\
& + O\left(\wht g_{LL}\Upr s\partial_s^3Z^N b\right).
\end{align*}
We estimate
\begin{align*}
\int_0^t (1+\tau)\ep^{-1}\ld{w_2^\frac{1}{2} \Upr \partial_s \partial_\theta Z^N b}^2d\tau
&\lesssim \int_0^t \int (1+\tau)\ep^{-1}\Upr^2\frac{1}{(1+|q|)^{1+2\sigma}}\|\partial_s \partial_\theta Z^N b\|_{L^2(\m S^1)}^2 r drd\tau\\
&\lesssim \ep^{-1}\int_0^t (1+s)^2\|\partial_s \partial_\theta Z^N b\|_{L^2(\m S^1)}^2ds\\
&\lesssim \ep^3(1+t)^{2\rho},
\end{align*}
where we have used \eqref{estf2} and \eqref{estf1bis1},
\begin{align*}
\int_0^t (1+\tau )\ld{w_2^\frac{1}{2}\frac{1}{r}\Upr \partial_s \partial^2_\theta Z^N b}^2d\tau
&\lesssim \int_0^t \int (1+\tau)\ep^{-1}\Upr^2\frac{1}{r^2(1+|q|)^{1+2\sigma}}\|\partial_s \partial^2_\theta Z^N b\|_{L^2(\m S^1)}^2 r drd\tau\\
&\lesssim \int_0^t \ep^{-1}\|\partial_s \partial^2_\theta Z^N b\|_{L^2(\m S^1)}^2,\\
&\lesssim \ep^3(1+t)^{2\rho},
\end{align*}
where we have used \eqref{estb5bis}
\begin{align*}
&\int_0^t (1+\tau)\ep^{-1}\ld{w_2^\frac{1}{2}\wht g_{LL}\Upr s\partial_s^3Z^N b}^2d\tau \\
&\lesssim \int_0^t (1+\tau)\ep^{-1}\int\frac{1}{(1+|q|)^{1+2\sigma}}\left(\Upr\frac{\ep (1+|q|)}{(1+s)^{\frac{3}{2}-\rho}}s\|\partial_s^3 Z^N b\|_{L^2(\m S^1)}\right)^2rdrd\tau\\
&\lesssim \int_0^t \ep^{-1}(1+s)^{3-2\sigma+2\rho}\|\partial_s^3 Z^N b\|_{L^2(\m S^1)}^2\\
&\lesssim \ep^3(1+t)^{2\rho},
\end{align*}
where we have used \eqref{estb5ter} and the fact that $2\rho \leq \sigma$.
We now check that $ g_{L\ba L}\left(\Upr(\chi(q)-1)s\partial_s Z^N b\right)$ satisfies the same estimates as $\wht g_1$. We have
\begin{align*}
&\ld{w_2^\frac{1}{2}g_{L\ba L}\partial_q\left(\Upr(\chi(q)-1)s\partial_s Z^N b\right)}^2\\
\lesssim &\int \frac{1}{(1+|q|)^{1+\sigma}} \|\partial_s Z^N b\|_{L^2(\m S^1)}^2 rdr
+\int \chi'(q)s^2\|\partial_s Z^N b\|^2_{L^2(\m S^1)} rdr+s.t.\\
\lesssim &\ep^2(1+t)^{2\rho},
\end{align*}
where we have used \eqref{estf1b} and \eqref{estf2b}.
We estimate
\begin{align*}
&\int_0^t \ld{w_2'(q)^\frac{1}{2}\bar{\partial}  g_{L\ba L}\left(\Upr(\chi(q)-1)s\partial_s Z^N b\right)}^2d\tau
\lesssim &\int_0^t \int \frac{1}{(1+|q|)^{2+2\sigma}}s^2\|\partial_s^2 Z^N b\|_{L^2(\m S^1)}^2rdrd\tau  \\
\lesssim &\int_0^t (1+s)^3\|\partial_s^2Z^N b\|_{L^2(\m S^1)}^2ds\\
\lesssim &\ep^2(1+t)^{2\rho},
\end{align*}
which concludes the proof of Proposition \ref{prpwhtg1n}.
\end{proof}
\subsection{Estimates of $\partial Z^N \phi$ and $\partial^2 Z^N \phi$}
\begin{prp}\label{l2phin}
We have
$$\ld{w^\frac{1}{2}\partial Z^N \phi} 
+\ld{w^\frac{1}{2}\partial^2 Z^N \phi}\leq C_0 \ep + C\ep^\frac{3}{2}(1+t)^\rho,$$
$$\int_0^t \ld{w'(q)^\frac{1}{2}\bar{\partial }Z^N \phi}^2
+\ld{w'(q)^\frac{1}{2}\bar{\partial }\partial Z^N \phi}^2 \leq C_0 \ep + C \ep^\frac{3}{2}(1+t)^\rho.$$
\end{prp}
\begin{proof}
As in the previous section, we use the weighted energy estimate in the metric $g$.
Thanks to Proposition \ref{strphi}, we have
\begin{align*}
|\Box_g Z^N \phi|&\lesssim \frac{\ep}{\sqrt{1+s}(1+|q|)^{\frac{5}{2}-4\rho}}| Z^N \wht g_{LL}|
+\frac{\ep}{(1+s)^\frac{3}{2}(1+|q|)^{\frac{3}{2}-4\rho}}| Z^N \wht g|\\
&+\frac{\ep}{(1+s)^{\frac{3}{2}-\rho}}|\partial Z^N \phi|
+\frac{\ep}{(1+s)^\frac{3}{2}(1+|q|)^{\frac{1}{2}-4\rho}}| Z^N G^L|
+\frac{\ep}{(1+s)\sqrt{1+|q|}}|\bar{\partial}Z^N \phi|.
\end{align*}
We estimate
$$\ld{w^\frac{1}{2}\frac{\ep}{\sqrt{1+s}(1+|q|)^{\frac{5}{2}-4\rho}} Z^N \wht g_{LL}}
\lesssim \frac{\ep}{\sqrt{1+t}}\ld{\frac{w_2^\frac{1}{2}}{(1+|q|)^\frac{3}{2}} Z^N \wht g_{LL}}
\lesssim \frac{\ep}{\sqrt{1+t}}\ld{w'_2(q)^\frac{1}{2} \bar{\partial}Z^N \wht g_{LL}},$$
so
\begin{equation}
\label{l2phi1}\int_0^t \ep^{-1}(1+t)\ld{w^\frac{1}{2}\frac{\ep}{\sqrt{1+s}(1+|q|)^{\frac{5}{2}-4\rho}} Z^N \wht g_{LL}}^2 \lesssim \ep^3(1+t)^{2\rho},
\end{equation}
\begin{equation}\label{l2phi2}\ld{w^\frac{1}{2}\frac{\ep}{(1+s)^\frac{3}{2}(1+|q|)^{\frac{3}{2}-4\rho}} Z^N \wht g}
\lesssim \frac{1}{(1+s)^\frac{3}{2}}\ld{\frac{w_1^\frac{1}{2}}{1+|q|}Z^N \wht g}\lesssim \frac{\ep^2}{(1+t)^{\frac{3}{2}-2\rho}},
\end{equation}
\begin{equation}\label{l2phi3}
\ld{w^\frac{1}{2}\frac{\ep}{(1+s)^{\frac{3}{2}-\rho}}\partial Z^N \phi}
\lesssim \frac{\ep^2}{(1+t)^{\frac{3}{2}-2\rho}},
\end{equation}
\begin{equation}\label{l2phi4}\begin{split}\ld{\frac{\ep}{(1+s)^\frac{3}{2}(1+|q|)^{\frac{1}{2}-4\rho}}| Z^N G^L|}
&\lesssim \ep \left( \int \frac{1}{(1+s)^3(1+|q|)^{1-8\rho}}\|Z^N G^L \|^2_{L^2(\m S^2)}rdr\right)^\frac{1}{2}\\
&\lesssim \ep \left( \int \frac{1}{(1+s)^4(1+|q|)^{1-8\rho}}\|rZ^N G^L \|^2_{L^2(\m S^2)}dr\right)^\frac{1}{2}\\
&\lesssim \frac{\ep^2}{(1+t)^\frac{3}{2}},
\end{split}
\end{equation}
$$\ld{w^\frac{1}{2}\frac{\ep}{(1+s)\sqrt{1+|q|}}\bar{\partial}Z^N \phi}
\lesssim \frac{\ep}{(1+s)^{1-\mu}}\ld{w'(q)^\frac{1}{2}\bar{\partial}Z^N \phi},$$
so
\begin{equation}
\label{l2phi5}\int_0^t \ep^{-1}(1+t) \ld{w^\frac{1}{2}\frac{\ep}{(1+s)\sqrt{1+|q|}}\bar{\partial}Z^N \phi}^2
\lesssim \ep \int_0^t\ld{w'(q)^\frac{1}{2}\bar{\partial}Z^N \phi}^2
\lesssim \ep^3(1+t)^{2\rho}.
\end{equation}
Estimates \eqref{l2phi1}, \eqref{l2phi2}, \eqref{l2phi1}, \eqref{l2phi4} and \eqref{l2phi5} conclude the first part of Proposition \ref{l2phin}.
We now estimate $\partial Z^N \phi$. 
The terms are all similar or easier to estimate, since no terms with two derivatives of $g$ are involved.
The terms involving a derivative of $G^L$ can be estimated by
$$\frac{\ep}{(1+s)^\frac{3}{2}(1+|q|)^{\frac{3}{2}-4\rho}}| Z^N G^L|, \quad
\frac{\ep}{(1+s)^\frac{3}{2}(1+|q|)^{\frac{1}{2}-4\rho}}|\bar{\partial} Z^N G^L|.$$
The contribution of the second term can be estimated (very loosely) with an integration by part as in the last section (see estimates \eqref{l211}, \eqref{l212}, \eqref{l213} and \eqref{l214}).
\end{proof}
\subsection{Estimate of $\partial Z^{N+1}\phi$}
\begin{prp}
We have
$$\ld{w^\frac{1}{2}\partial Z^{N+1}\phi}\leq C_0 \ep + C \ep^\frac{3}{2}(1+t)^{\frac{1}{2}+\rho},$$
$$\int \frac{1}{1+t}\ld{w'(q)^\frac{1}{2}\bar{\partial}Z^{N+1}\phi}^2\leq C_0^2\ep^2 +C\ep^3(1+t)^\rho.$$
\end{prp}
\begin{proof}
We use the weighted energy estimate in the metric $g$ and multiply it by $\frac{1}{1+t}$. We obtain
\begin{align*}
&\frac{d}{dt}\left(\frac{1}{1+t}\int (\partial Z^{N+1}\phi)^2w_1(q)dvol_g\right)+C\frac{1}{1+t}\int w'(q)(\bar{\partial}Z^{N+1} \phi)^2\\
&\lesssim \frac{\ep}{1+t}\frac{1}{1+t}\int w_1(q)(\partial Z^{N+1} \phi)^2 + \frac{1}{1+t}\left|\int w_1(q)\partial_t Z^{N+1}\phi \Box_g Z^{N+1} \phi dvol_g\right|.
\end{align*}

To estimate $\Box_g Z^{N+1} \phi$ we use Proposition \ref{strphi} and remark that
$$|Z^{N+1}g| \lesssim (1+s)|\bar{\partial} Z^N g|+(1+|q|)|\partial Z^N g|.$$
Consequently
\begin{align*}
|\Box_g Z^{N+1} \phi|&\lesssim \frac{\ep\sqrt{1+s}}{(1+|q|)^{\frac{5}{2}-4\rho}}|\bar{\partial} Z^N \wht g_{LL}|
+\frac{\ep}{\sqrt{1+s}(1+|q|)^{\frac{3}{2}-4\rho}}|\partial Z^N \wht g_{LL}|\\
&+\frac{\ep}{(1+s)^\frac{1}{2}(1+|q|)^{\frac{3}{2}-4\rho}}|\bar{\partial} Z^N \wht g_1|
+\frac{\ep}{(1+s)^\frac{3}{2}(1+|q|)^{\frac{1}{2}-4\rho}}|\partial Z^N \wht g|\\
&+\frac{\ep}{(1+s)^{\frac{3}{2}-\rho}}|\partial Z^{N+1} \phi|
+\frac{\ep}{(1+s)^\frac{1}{2}(1+|q|)^{\frac{1}{2}-4\rho}}| \bar{\partial}Z^N G^L|\\
&+\frac{\ep}{(1+s)^\frac{3}{2}(1+|q|)^{\frac{1}{2}-4\rho}}| Z^N G^L|
+\frac{\ep}{(1+s)\sqrt{1+|q|}}|\bar{\partial}Z^{N+1}\phi|.
\end{align*}
We estimate the first term. It comes from
a term of the form
$s\bar{\partial} Z^N \wht g_{LL}\partial_q^2 \phi.$
We estimate its contribution with an integration by parts
\begin{align*}
&\int_0^t \frac{1}{1+\tau}\int w(q) s(\bar{\partial} Z^N \wht g_{LL})(\partial_q^2 \phi) (\partial_t Z^{N+1}\phi) dvol_g d\tau\\
=& \left[ \frac{1}{1+t}\int w(q)  s(\bar{\partial} Z^N \wht g_{LL})(\partial_q^2 \phi) (Z^{N+1}\phi) dvol_g\right]_0^t
-\int \int \partial_t\left(w(q)\sqrt{|det(g)|}  \frac{s}{1+\tau}(\bar{\partial}Z^N \wht g_{LL})\partial_q^2 \phi \right) Z^{N+1}\phi dx d\tau\\
=&\left[ \frac{1}{1+\tau}\int w(q)  s(\bar{\partial} Z^N \wht g_{LL})(\partial_q^2 \phi)( Z^{N+1}\phi)\right]_0^t
+\int \int \partial_t\left(w(q) \sqrt{|det(g)|} \frac{s}{1+\tau} (Z^N \wht g_{LL})\partial_q^2 \phi\right)\bar{\partial}Z^{N+1} \phi dxd\tau\\
&+O \left(\int_0^t \int |\partial_t(w(q) \bar{\partial}( \sqrt{|det(g)|} \frac{s}{1+\tau} \partial_q^2 \phi )Z^N \wht g_{LL}) Z^{N+1} \phi |dxd\tau \right)\\
&+O\left(\left[\int \partial_t\left(w(q)\sqrt{|det(g)|} \frac{s}{1+\tau}Z^N \wht g_{LL}\partial_q^2 \phi\right) Z^{N+1} \phi\right]_0^t \right).
\end{align*}
We estimate
\begin{align*}&\frac{1}{1+t}\left| \int w(q)  s(\bar{\partial} Z^N \wht g_{LL})(\partial_q^2 \phi)( Z^{N+1}\phi)dvol_g\right|\\
&\lesssim \frac{1}{1+t}\int w(q) \frac{\ep\sqrt{1+s}}{(1+|q|)^{\frac{5}{2}-4\rho}}|\bar{\partial} Z^N \wht g_{LL}|| Z^{N+1}\phi|dx\\
&\lesssim \ep\frac{1}{\sqrt{1+t}}\ld{w_2(q)^\frac{1}{2}\partial Z^N \wht g_1}\ld{w(q)^\frac{1}{2}\partial Z^{N+1}\phi}\\
&\lesssim \ep^3(1+t)^{2\rho},
\end{align*}
\begin{align*}&\left|\int_0^t \int \partial_t\left(w(q) \sqrt{|det(g)|} \frac{s}{1+\tau} Z^N \wht g_{LL}\partial_q^2 \phi\right)\bar{\partial}Z^{N+1} \phi dxd\tau\right|\\
&\lesssim \int_0^t \int \frac{1}{(1+s)^\frac{1}{2}(1+|q|)^{\frac{5}{2}-4\rho}}\left(|\partial Z^N g_{LL}|+\frac{1}{1+|q|}|Z^N g_{LL}|\right)
|\bar{\partial}Z^{N+1} \phi|dxd\tau \\
&\lesssim \ep\int_0^t \ld{w'_2(q)\left(|\partial Z^N g_{LL}|+\frac{1}{1+|q|}|Z^N g_{LL}|\right)}^2d\tau
+\ep\int_0^t \frac{1}{1+\tau}\ld{w'(q)\bar{\partial}Z^{N+1} \phi}^2d\tau\\
&\lesssim \ep^3(1+t)^{2\rho},
\end{align*}
where we have used \eqref{iZphiN}, \eqref{wl2dlln} and \eqref{hwl2dlln},
\begin{align*}
&\left|\int_0^t \int \frac{1}{(1+s)^2}\partial_t(w(q)  \sqrt{|det(g)|}s Z^N \wht g_{LL}\partial_q^2 \phi) Z^{N+1} \phi dxd\tau \right|\\
&\lesssim \int_0^t \int  \frac{\ep}{(1+s)^\frac{3}{2}(1+|q|)^{\frac{5}{2}-4\rho}}\left(|\partial Z^N g_{LL}|+\frac{1}{1+|q|}|Z^N g_{LL}|\right)|Z^{N+1} \phi|dxd\tau\\
&\lesssim \ep\int_0^t \ld{w'_2(q)\left(|\partial Z^N g_{LL}|+\frac{1}{1+|q|}|Z^N g_{LL}|\right)}^2d\tau
+\ep\int_0^t \frac{1}{(1+\tau)^3}\ld{w^\frac{1}{2}\frac{Z^{N+1}\phi}{1+|q|}}^2d\tau
\lesssim \ep^3(1+t)^{2\rho}.
\end{align*}
The last term can be estimated as the first. For the estimate of the other terms, we refer to the following section.
\end{proof}
\subsection{Estimation of $\partial S Z^N \phi$}
In this section we prove better estimates for $\partial S Z^N \phi$. These better estimates allow to exploit the better decay of $\partial_s Z^N \phi$ noticing
$$\partial_s Z^N \phi = \frac{1}{s}SZ^N \phi + \frac{q}{s}\partial_q Z^N \phi^.$$
This fact is used in Section \ref{secprph} to estimate $\partial_s Z^N h$.
\begin{prp}\label{l2Sphin}
	We have
	$$\ld{w^\frac{1}{2}\partial (SZ^N \phi-s\partial_q \phi Z^Ng_{LL})} 
	\leq C_0 \ep + C\ep^\frac{3}{2}(1+t)^\rho,$$
	$$\int_0^t \ld{w'(q)^\frac{1}{2}\bar{\partial }(SZ^N \phi-s\partial_q \phi Z^Ng_{LL})}
	\leq C_0 \ep + C \ep^\frac{3}{2}(1+t)^\rho.$$
\end{prp}
We do the weighted energy estimate in the metric $g$. We have, thanks to Proposition \ref{strphi}
\begin{align*}
|\Box_g SZ \phi|&\lesssim \frac{\ep\sqrt{1+s}}{(1+|q|)^{\frac{5}{2}-4\rho}}|\partial_s Z^N \wht g_{LL}|
+\frac{\ep}{\sqrt{1+s}(1+|q|)^{\frac{3}{2}-4\rho}}|\partial Z^N \wht g_{LL}|\\
&+\frac{\ep}{(1+s)^\frac{1}{2}(1+|q|)^{\frac{3}{2}-4\rho}}|\bar{\partial} Z^N \wht g_1|
+\frac{\ep}{(1+s)^\frac{3}{2}(1+|q|)^{\frac{1}{2}-4\rho}}|\partial Z^N \wht g|\\
&+\frac{\ep}{(1+s)^{\frac{3}{2}-\rho}}|\partial SZ^{N} \phi|
+\frac{\ep}{(1+s)^\frac{1}{2}(1+|q|)^{\frac{1}{2}-4\rho}}| \partial_s Z^N G^L|\\
&+\frac{\ep}{(1+s)^\frac{3}{2}(1+|q|)^{\frac{1}{2}-4\rho}}| Z^N G^L|
+\frac{\ep}{(1+s)\sqrt{1+|q|}}|\bar{\partial}SZ^{N}\phi|
\end{align*}
We estimate the first term.
It can be written $s\partial_s Z^N g_{LL} \partial_q^2 \phi$. We remark the following calculation
\begin{align*}
\Box_g s\partial_q \phi Z^Ng_{LL}&=Z^N g_{LL}\Box_g (s\partial_q\phi)+s\partial_q \phi \Box_g Z^N g_{LL}+2g^{\alpha \beta}\partial_\alpha (s\partial_q \phi)\partial_{\beta} Z^N g_{LL}\\ 
&=O\left(\frac{\ep}{(1+s)^\frac{1}{2}(1+|q|)^{\frac{5}{2}-4\rho}}\right)Z^N g_{LL}+s\partial_q \phi \Box_g Z^N g_{LL}\\
&+2\partial_s Z^N g_{LL}\partial_q (s\partial_q \phi)
+2\partial_q Z^N g_{LL}\partial_s(s\partial_q \phi)
+\frac{2}{r^2}\partial_\theta Z^N g_{LL}\partial_\theta(s\partial_q \phi)\\
&+2\wht g  \bar{\partial}(s\partial_q \phi)\partial_q Z^N g_{LL}+2\wht g_{L \ba L} \partial_q(s\partial_q \phi)\partial_s Z^N g_{LL} +2\wht g_{UL}\partial_q(s\partial_q \phi)\frac{\partial_\theta}{r} Z^N g_{LL}.
\end{align*}
In $\Box_g Z^N g_{LL}$, considering the support condition on $\phi$, it is sufficient to study the contribution of
$M_{LL}$ : the only dangerous terms are the one of the form
$\partial_{\ba L} \wht g_{\ba L \q T}\partial_{\q T}Z^N g_{\q T \q T}$  which can only be
$\partial_{\ba L}\wht g_{\ba L L}\partial_L Z^N g_{LL}$, 
and the contribution of the commutator of the wave operator with the null frame, which are more precisely in this case
$$\frac{1}{r^2}\Upr Z^N \wht g +\frac{1}{r}\Up\left(\frac{r}{t}\right)\bar{\partial} \wht g_{L\q V}.$$
Consequently, the terms in $\Box_g (SZ^N \phi -  s\partial_q \phi Z^Ng_{LL})$ are similar to one of the three following terms
$$O\left(\frac{\ep}{(1+s)^\frac{1}{2}(1+|q|)^{\frac{5}{2}-4\rho}}\right)Z^N g_{LL},
\quad O\left(\frac{\ep}{(1+s)^\frac{1}{2}(1+|q|)^{\frac{3}{2}-4\rho}}\right)\bar{\partial}Z^N \wht g_1, \quad \wht g_1 \partial_q(s\partial_q \phi)\partial_s Z^N g_{LL} .$$
We estimate the first term
\begin{equation}\label{l2sphi}\begin{split}
\int_0^t \ep^{-1}(1+\tau)\left\|w^\frac{1}{2}O\left(\frac{\ep}{(1+s)^\frac{1}{2}(1+|q|)^{\frac{5}{2}-4\rho}}\right)Z^N g_{LL}\right\|_{L^2}^2d\tau&
\lesssim \ep \int \left\|\frac{w_2^\frac{1}{2}}{(1+|q|)^{\frac{5}{2}-4\rho-\sigma }} Z^N g_{LL}\right\|^2_{L^2}d\tau\\
&\lesssim \ep^3(1+t)^{2\rho},
\end{split}
\end{equation}
where we have used \eqref{hwl2dlln}. The second term is similar.
We now estimate the contribution of the third term.
\begin{align*}
&\int w(q)\wht g_1 (\partial_q(s\partial_q \phi))(\partial_s Z^N g_{LL})(\partial_t S Z^N \phi) dvol_g\\
&=\partial_t \int w(q)\wht g_1 (\partial_q(s\partial_q \phi) )(Z^N g_{LL})(\partial_t S Z^N \phi) dvol_g
-\int \partial_s\left( w(q)\sqrt{|det(g)|}\wht g_1 \partial_q(s\partial_q \phi)r\right)(Z^N g_{LL})(\partial_t S Z^N \phi) d\theta dr\\
&- \int w(q)\wht g_1 (\partial_q(s\partial_q \phi))( Z^N g_{LL})(\partial_t\partial_s S Z^N \phi) dvol_g\\
&= \partial_t \int w(q)\wht g_1 (\partial_q(s\partial_q \phi))( Z^N g_{LL})(\partial_t S Z^N \phi)dvol_g 
-\int \partial_s\left( w(q)\sqrt{|det(g)|}\wht g_1 \partial_q(s\partial_q \phi)r\right)(Z^N g_{LL})(\partial_t S Z^N \phi) d\theta dr\\
&-\partial_t \int w(q)\wht g_1 (\partial_q(s\partial_q \phi))( Z^N g_{LL})(\partial_s S Z^N \phi)dvol_g +\int \partial_t\left(w(q)\sqrt{|det(g)|}\wht g_1 \partial_q(s\partial_q \phi) Z^N g_{LL}\right)(\partial_s S Z^N \phi)dx\\
&=\partial_t A -\int \partial_s\left( w(q)\wht g_1 \partial_q(s\partial_q \phi)r\right)Z^N g_{LL}\partial_t S Z^N \phi d\theta dr+\int \partial_t\left(w(q)\wht g_1 \partial_q(s\partial_q \phi) Z^N g_{LL}\right)\partial_s S Z^N \phi
\end{align*}
with 
$$A=\int w(q)\wht g_1( \partial_q(s\partial_q \phi) )(Z^N g_{LL})(\partial_t S Z^N \phi)dvol_g+\int w(q)\wht g_1( \partial_q(s\partial_q \phi))( Z^N g_{LL})(\partial_s S Z^N \phi)dvol_g.$$
We estimate
\begin{align*}
&\left|\int w(q)\wht g_1 \partial_q(s\partial_q \phi) Z^N g_{LL}\partial_t S Z^N \phi\right|\\
&\lesssim \int \frac{\ep^2}{(1+|q|)^{2-4\rho}}|Z^N g_{LL}||\partial_t SZ^N \phi|\\
&\lesssim \ep^2\ld{\frac{w_2^\frac{1}{2}}{1+|q|}Z^N \wht g_1}\ld{w^\frac{1}{2}\partial_t SZ^N \phi}\lesssim \ep^4(1+t)^{2\rho}.
\end{align*}
The second term in $A$ obey a similar estimate so
$$|A|\lesssim \ep^4(1+t)^{2\rho}.$$
We now estimate
\begin{align*}|\partial_s\left( w(q) \sqrt{|det(g)|}\wht g_1 \partial_q(s\partial_q \phi)r\right)|\lesssim
\frac{1}{(1+|q|)^2}|Z \wht g_1||Z^3 \phi|r &\lesssim\frac{1}{(1+|q|)^2} \frac{\ep(1+|q|)^{\frac{1}{2}}}{(1+s)^\frac{1}{2}}\frac{\ep}{(1+s)^\frac{1}{2}(1+|q|)^{\frac{1}{2}-4\rho}}r \\
&\lesssim \frac{\ep^2}{(1+s)(1+|q|)^{2-4\rho}}r.
\end{align*}
Consequently
\begin{align*}
&\int_0^t\left|\int \partial_s\left( w(q)\sqrt{|det(g)|}\wht g \partial_q(s\partial_q \phi)r\right)Z^N g_{LL}\partial_t S Z^N \phi d\theta dr\right|d\tau\\
&\lesssim \int_0^t\int w(q)\frac{\ep^2}{(1+s)(1+|q|)^{2-4\rho-\sigma}}|Z^Ng_{LL}||\partial_t S Z^N \phi|dxd\tau\\
&\lesssim \int_0^t\frac{\ep^2}{(1+\tau)}\left\|\frac{w_2(q)^\frac{1}{2}}{(1+|q|)^{\frac{3}{2} }}Z^N g_{LL}\right\|_{L^2}\|w^\frac{1}{2} \partial_t S Z^N \phi\|_{L^2}d\tau\\
&\lesssim \ep\int_0^t \left\|\frac{w_2(q)^\frac{1}{2}}{(1+|q|)^{\frac{3}{2} }}Z^N g_{LL}\right\|_{L^2}^2d\tau
+\int_0^t \frac{\ep}{(1+\tau)^2}\|w^\frac{1}{2} \partial_t S Z^N \phi\|_{L^2}^2d\tau\\
&\lesssim \ep^3(1+t)^{2\rho}.
\end{align*}
We now estimate
\begin{align*}
&\int_0^t\left|\int \partial_t\left(w(q)\sqrt{|det(g)|}\wht g_1 \partial_q(s\partial_q \phi) Z^N g_{LL}\right)\partial_s S Z^N \phi\right|\\
&\lesssim \int w(q)\frac{\ep^2}{(1+|q|)^{3-4\rho-\sigma}}|Z^N g_{LL}||\partial_s SZ^N\phi|dx+s.t.\\
&\lesssim\ep^2\int_O^t\left\|\frac{w(q)^\frac{1}{2}}{(1+|q|)^{\frac{5}{2}-4\rho-\sigma-\mu}}Z^N g_{LL}\right\|_{L^2}\left\|\frac{w^\frac{1}{2}}{(1+|q|)^{\frac{1}{2}+\mu}} \partial_s S Z^N \phi\right\|_{L^2}\\
&\lesssim \ep \int_0^t\left\|\frac{w_2'(q)^\frac{1}{2}}{1+|q|} Z^N  g_{LL} \right\|^2_{L^2}
+\ep\int_0^t\left\|w'(q)^\frac{1}{2} \partial_s S Z^N \phi\right\|^2_{L^2}\\
&\lesssim \ep^3(1+t)^{2\rho}.
\end{align*}

We now estimate the other contributions in $\Box_g SZ^N \phi$.
The term $\frac{\ep}{\sqrt{1+s}(1+|q|)^{\frac{3}{2}-4\rho}}|\partial Z^N \wht g_{LL}|$ can be estimated like \eqref{l2sphi}.
We estimate
$$
\int_0^t \ep(1+\tau)^{-1}\ld{w\frac{\ep}{\sqrt{1+s}(1+|q|)^{\frac{3}{2}-4\rho}}\bar{\partial} Z^N \wht g_1}^2d\tau
\lesssim \int_0^t \ep\ld{w_2^\frac{1}{2}(q)'\bar{\partial}Z^N \wht g_1}^2\lesssim \ep^2(1+t)^{2\rho},$$
$$\ld{w\frac{\ep}{(1+s)^\frac{3}{2}(1+|q|)^{\frac{1}{2}-4\rho}}\partial Z^N \wht g}
\lesssim \frac{\ep}{(1+t)^\frac{3}{2}}\ld{w_1^\frac{1}{2}\partial Z^N \wht g}\lesssim \frac{\ep^2}{(1+t)^{\frac{3}{2}-2\rho}}.$$
We decompose
$$\frac{\ep}{(1+s)\sqrt{1+|q|}}|\bar{\partial}SZ^{N}\phi|\leq\frac{\ep}{(1+s)\sqrt{1+|q|}}|\bar{\partial}(SZ^N \phi -s\partial_q \phi Z^N g_{LL})|+ \frac{\ep}{(1+s)\sqrt{1+|q|}}|\bar{\partial} (s\partial_q \phi Z^N g_{LL})|.$$
We estimate
$$\int_0^t \ep(1+\tau)^{-1}\ld{\frac{\ep}{(1+s)\sqrt{1+|q|}}\bar{\partial}(SZ^N \phi -s\partial_q \phi Z^N g_{LL})}^2d\tau
\lesssim \ep^3(1+t)^{2\rho},$$
thanks to \eqref{iSphiN}, and
\begin{align*}
&\int_0^t \ep(1+\tau)^{-1}\ld{\frac{\ep}{(1+s)\sqrt{1+|q|}}\bar{\partial} (s\partial_q \phi Z^N g_{LL})}^2d\tau\\
&\lesssim \int_0^t \ep(1+\tau)^{-1}\ld{\frac{\ep}{\sqrt{1+s}(1+|q|)^{\frac{5}{2}-4\rho}}\bar{\partial}Z^N g_{LL}}^2d\tau +s.t.\\
&\lesssim \int_0^t \ep\ld{w_2'(q)^\frac{1}{2}\bar{\partial}Z^N g_{LL}}^2\lesssim \ep^3(1+t)^{2\rho}.
\end{align*}
The other term can be estimated in a similar way. For the term involving $\bar{\partial}G^L $ we refer to \eqref{l211}, \eqref{l212}, \eqref{l213} and \eqref{l214}.
\section{Lower order $L^2$ estimates}\label{seclow}
\subsection{Estimate of $\partial Z^{N-1} \phi$}
\begin{prp}\label{prpn1phi}
	We have
	$$\ld{w^\frac{1}{2}\partial Z^{N-1}\phi}\lesssim C_0 \ep + C \ep^\frac{3}{2}.$$
\end{prp}
\begin{proof}
	We perform the energy estimate in the Minkowski metric. Estimates \eqref{l2phi2} to \eqref{l2phi5} are quite loose, so the
	only term we need to estimate is
	$$\ld{\frac{\ep}{\sqrt{1+s}(1+|q|)^{\frac{5}{2}-4\rho}} Z^{N-1} \wht g_{LL}}\lesssim \frac{\ep}{\sqrt{1+t}}\ld{\frac{w_2^\frac{1}{2}}{(1+|q|)^2}Z^{N-1} \wht g_{LL}}\lesssim \frac{\ep^2}{(1+t)^{\frac{3}{2}-\rho}},$$
	where we have used \eqref{hwl2dll}.
\end{proof}
\subsection{Estimate of $\partial Z^{N-3}\wht g$}
\begin{prp}\label{prpn3g}
	We have 
	$$\ld{w^\frac{1}{2}\partial Z^{N-3} \wht g}\leq C_0 \ep + C\ep^{\frac{3}{2}}(1+t)^\rho,$$
	$$\ld{w^\frac{1}{2}\partial Z^{N-3} k}\leq  C\ep^{\frac{3}{2}}(1+t)^\rho.$$
\end{prp}
\begin{proof}
We use the weighted energy estimate in the metric $g$
$$\frac{d}{dt}\ld{w^\frac{1}{2}(q)\partial Z^{N-3} \wht g}^2+C\ld{w'(q)^\frac{1}{2}\bar{\partial}Z^{N-3}\wht g}^2
 \lesssim \left| \int w(q)\Box_g Z^{N-3}\wht g\partial_t Z^{N-3}\wht g \right|+ \frac{\ep}{1+t}\ld{w^\frac{1}{2}\partial Z^{N-3} \wht g}^2.$$
We use Proposition \eqref{prpstr} to estimate $\Box_g Z^{N-3}\wht g$. We start with $Q_{\ba L \ba L}$.
\begin{align*}
|Q_{\ba L \ba L}|&\lesssim
\frac{\ep}{(1+|q|)(1+s)^{\frac{1}{2}-\rho}}\left(|\bar{\partial}Z^{N-3} \wht g_1|+\frac{1}{1+|q|}|Z^{N-3} g_{LL}|\right)
+\left(\frac{\ep^2}{(1+s)^\frac{3}{2}}+\frac{\ep}{(1+s)(1+|q|)^{1-4\rho}}\right)|\partial Z^{N-3} \wht g_1|\\
+&\ep\min\left(\frac{1}{(1+|q|)(1+s)^{\frac{1}{2}-\rho}},\frac{1}{(1+|q|)^{\frac{1}{2}}(1+s)^\frac{1}{2}}\right)(|\bar{\partial}Z^{N-3} \wht g|+|Z^{N-3} G^L|)\\
+ & \ch_{q>R} \min\left(\frac{\ep}{(1+s)^{\frac{1}{2}-\rho}(1+|q|)^{\frac{5}{2}+\delta}}, \frac{\ep}{(1+s)^{\frac{1}{2}}(1+|q|)^{\frac{5}{2}+\delta+\sigma}}\right)\frac{q}{s}|Z^{N-3} \partial^2_\theta b|
+ \ch_{q>R}\frac{\ep}{1+s}|\partial Z^{N-3} \wht g_1|\\
+&\ch_{q>R}\frac{\ep}{(1+|q|)^{\frac{3}{2}+\delta}(1+s)^{\frac{1}{2}-\rho}}\left(s|\partial^2_s Z^{N-3} b|
+q|\partial^2_s \partial_\theta Z^{N-3} b|+
\frac{q}{s}|Z^{N-3} \partial_s \partial^2_\theta b|
+\frac{q}{s^2}|Z^{N-3} \partial^3_\theta b|\right)\\
&\lesssim
\frac{\ep}{(1+s)^{\frac{3}{2}-\rho}(1+|q|)}|Z^{N-2}\wht g|+\frac{\ep}{(1+|q|)^2(1+s)^{\frac{1}{2}-\rho}}|Z^{N-3} g_{LL}|\\
&+\left(\frac{\ep^2}{(1+s)^\frac{3}{2}}+\frac{\ep}{(1+s)(1+|q|)^{1-4\rho}}\right)|\partial Z^{N-3} \wht g_1|\\
+&\frac{\ep}{(1+s)^{\frac{1}{2}-\rho}(1+|q|)}|Z^{N-3} G^L|+\ch_{q> R}\frac{\ep}{1+s}|\partial Z^{N-3} \wht g|,
\end{align*}
where we have used the fact that
$$|\partial^3_\theta Z^{N-3}b|\lesssim \|\partial^3_\theta Z^{N-3}b\|_{H^1(\m S^1)}\lesssim \|Z^{N-1}b\|_{H^1(\m S^1)} \lesssim \ep,$$
thanks to \eqref{estb2}, and
$$|\partial^2_s \partial_\theta Z^{N-3}b|\lesssim \frac{1}{1+s}\|\partial_s Z^{N-1} b\|_{H^1(\m S^1)}\lesssim \frac{\ep}{(1+s)^{3-\frac{1}{4}}},$$
thanks to \eqref{estb3}, to say that to estimate the terms involving $b$, it is sufficient to estimate
$\ch_{q> R}\frac{\ep}{1+s}|\partial Z^I \wht g|$.

We estimate
\begin{equation}
\label{l2n31}
\ld{w^\frac{1}{2}\frac{\ep}{(1+s)^{\frac{3}{2}-\rho}(1+|q|)}Z^{N-2}\wht g}
\lesssim \frac{\ep}{(1+s)^{\frac{3}{2}-\rho-\sigma}}\ld{\frac{w_1^\frac{1}{2}}{1+|q|}Z^{N-2}\wht g}\lesssim \frac{\ep^2}{(1+t)^{\frac{3}{2}-3\rho-\sigma}},
\end{equation}
thanks to \eqref{hg2n}.
\begin{equation}\label{l2n31bis}\ld{w^\frac{1}{2}\frac{\ep}{(1+|q|)^2(1+s)^{\frac{1}{2}-\rho}}Z^{N-3} g_{LL}}
\lesssim \frac{\ep}{(1+s)^{\frac{1}{2}-\rho-\sigma}}\ld{ \frac{w_1^\frac{1}{2}}{(1+|q|)^2}Z^{N-3} g_{LL}}
\lesssim \frac{\ep^2}{(1+t)^{\frac{3}{2}-4\rho -\sigma}},
\end{equation}
thanks to Proposition \ref{l2ll}.
We estimate 
\begin{equation}
\label{l2n32}\ld{w^\frac{1}{2}\frac{\ep}{(1+s)\sqrt{1+|q|}}\partial Z^{N-3} \wht g_1}\lesssim \frac{1}{1+t}\ld{w^\frac{1}{2}\partial Z^{N-3} \wht g} \lesssim \frac{\ep^2}{(1+t)^{1-\rho}},
\end{equation}
\begin{align*}&\ld{\frac{\ep}{(1+s)^{\frac{1}{2}-\rho}(1+|q|)}Z^{N-3} G^L}\\
&\lesssim \left(\int\frac{\ep}{(1+s)^{2-2\rho}(1+|q|)^2}\|rZ^{N-3}G^L\|_{L^2(\m S^1)}^2 dr\right)^\frac{1}{2}\\
&\lesssim \left(\int \frac{\ep}{(1+s)^{2-2\rho}(1+|q|)^2}\left(
\frac{\ep^2}{(1+t)^{\frac{1}{2}-\rho}}+\frac{\ep^2}{(1+|q|)^{1-4\rho}}+\frac{\ep\sqrt{1+s}}{(1+|q|)^{\frac{1}{2}-4\rho}}\|\partial_q Z^I \phi\|_{L^2(\m S^1)}\right)^2dr\right)^\frac{1}{2},
\end{align*}
so
\begin{equation}
\label{l2n33}\ld{\frac{\ep}{(1+s)^{\frac{1}{2}-\rho}(1+|q|)}Z^{N-3} G^L}\lesssim \frac{\ep^2}{(1+t)^{1-\rho}}.
\end{equation}
We estimate
\begin{equation}
\label{l2n34}
\ld{w^\frac{1}{2}\ch_{q> R}\frac{\ep}{1+s}|\partial Z^{N-3} \wht g|}\lesssim \frac{\ep}{1+t}\ld{w^\frac{1}{2}\partial Z^{N-3} \wht g}\lesssim \frac{\ep}{(1+t)^{1-\rho}}.
\end{equation}

We now go to the other terms in $\Box_g Z^{N-3}\wht g$.
The contribution of ${}^IM^E$ can be estimated by \eqref{l2n34}.
To estimate the contribution of $R^1$, it is sufficient to consider
$$\ch_{R\leq q \leq R+1}\partial_s Z^{I+2}b =O\left(\ch_{R\leq q \leq R+1}\frac{\ep^2}{(1+s)^{2-\frac{1}{4}}}\right),$$
and we have
\begin{equation}
\label{l2n35}\ld{w^\frac{1}{2}\ch_{R\leq q \leq R+1}\partial_s Z^{I+2}b}\lesssim \frac{\ep^2}{(1+t)^{\frac{3}{2}-\frac{1}{4}}}.
\end{equation}
The terms in ${}^IM$ can be estimated by \eqref{l2n31} and \eqref{l2n31bis}, except
$$\frac{\ep}{(1+s)(1+|q|)^{\frac{3}{2}-\rho}}|Z^I \wht g_{\q T \q T}|, \quad \frac{1}{1+s}|Z^{I+1} G^L|.$$
The first term can be estimated by \eqref{l2n32}, and the second one by
\begin{equation}
\label{l2n36}
\begin{split}
&\ld{w^\frac{1}{2}\frac{1}{1+s}Z^{N-2}G^L}\\
&\lesssim \left(\int \frac{1}{(1+s)^3}\|rZ^{N-2}G^L\|_{L^2(\m S^1)}^2 dr\right)^\frac{1}{2}\\
&\lesssim\left(\int \frac{1}{(1+s)^3}\left(
\frac{\ep^2}{(1+t)^{\frac{1}{2}-\rho}}+\frac{\ep^2}{(1+|q|)^{1-4\rho}}+\frac{\ep\sqrt{1+s}}{(1+|q|)^{\frac{1}{2}-4\rho}}\|\partial_q Z^{N-2} \phi\|_{L^2(\m S^1)}\right)^2 dr\right)^\frac{1}{2}\\
&\lesssim \frac{\ep^2}{(1+t)^{\frac{3}{2}-\rho}}.
\end{split}
\end{equation}
Estimates \eqref{l2n31}, \eqref{l2n31bis}, \eqref{l2n32}, \eqref{l2n33}, \eqref{l2n34}, \eqref{l2n35} and \eqref{l2n36} yield 
$$\ld{w^\frac{1}{2}\Box_g Z^{N-3}\wht g} \lesssim \frac{\ep^2}{(1+t)^{1-\rho}}.$$
We also have
$$\ld{w^\frac{1}{2}\Box_g Z^{N-3}k} \lesssim \frac{\ep^2}{(1+t)^{1-\rho}},$$
which concludes the proof of Proposition \ref{prpn3g}.
 \end{proof}
\subsection{Estimate of $\partial Z^{N-4} \wht g_1$}\label{secl2gn4}
\begin{prp}\label{prpn4g}
	We have 
	$$\ld{w_1^\frac{1}{2}\partial Z^{N-4}\wht g_1}\lesssim C_0 \ep + C\ep^\frac{5}{4}.$$
\end{prp}
\begin{proof}
We use the weighted energy estimate in the flat metric
$$\frac{d}{dt}\ld{w_1(q)^\frac{1}{2}\partial Z^{N-3} \wht g_1}^2+C\ld{w_1'(q)^\frac{1}{2}\bar{\partial}Z^{N-3}\wht g}^2
\lesssim  \left|\int w_1(q)\Box_g Z^{N-4}\wht g_1\partial_t Z^{N-4}\wht g\right|.$$
We estimate 
$$\Box Z^{N-4} \wht g_1= (\Box - \Box_g)Z^{N-4} \wht g_1+{}^{N-4}M + {}^{N-4} Q_{\q T \q V} + {}^{N-4} M^E + O\left(\frac{1}{r^2}\Upr \partial_\theta Z^{N-4} k\right).$$
The terms in ${}^{N-4} M$, ${}^{N-4} Q_{\q T \q V}$ and the terms in $(\Box_g-\Box) Z^{N-4} \wht g_1$ of the form $\wht g\partial^2 Z^{N-4} \wht g_1$ can be estimated by \eqref{l2n31}, \eqref{l2n31bis} and \eqref{l2n36}, except the term
$$\frac{\ep}{(1+s)(1+|q|)^{\frac{3}{2}-\rho}}|Z^{N-4} \wht g_{\q T \q T}|.$$
We estimate it
$$\ld{w_1^\frac{1}{2}\frac{\ep}{(1+s)(1+|q|)^{\frac{3}{2}-\rho}}Z^{N-4} g_{\q T \q T}}
\lesssim \frac{\ep}{(1+t)}\ld{ \frac{w_1^\frac{1}{2}}{(1+|q|)^{\frac{5}{4}+\sigma}}Z^{N-4} g_{\q T \q T}}\lesssim \frac{\ep^2}{(1+t)^\frac{5}{4}},$$
thanks to \eqref{hwl2dll}, \eqref{hwl2dul} and \eqref{hwl2duu}.
Terms coming from the non commutation with the null frame can be estimated by
$$\ld{\frac{w_1^\frac{1}{2}}{(1+s)^2}Z^{N-3}k}\lesssim \frac{1}{(1+s)^{1+\sigma}}\ld{\frac{w^\frac{1}{2}}{1+|q|}Z^{N-3}k}
\lesssim \frac{\ep^\frac{3}{2}}{(1+t)^{1+\sigma-\rho}}.$$
The terms in $M_E$, and the terms in $(\Box - \Box_g)Z^{N-4}\wht g_1$ of the form $g_{\mathfrak{b}}\partial^2 Z^{N-4} \wht g_1$ can be estimated by
$$\ch_{q>R}\frac{\ep}{(1+s)^2} Z^{N-3}\wht g, \quad \ch_{q>R}\frac{\ep(1+|q|)}{(1+s)^2} \partial_q Z^{N-3}\wht g.$$
We estimate
$$\ld{w_1^\frac{1}{2}\ch_{q>R}\frac{\ep}{(1+s)^2} Z^{N-3}\wht g} \lesssim \frac{\ep}{(1+s)^{1+\sigma}}\ld{\frac{w^\frac{1}{2}}{1+|q|}Z^{N-3} \wht g}\lesssim \frac{\ep^2}{(1+t)^{1+\sigma-\rho}}.$$
We now estimate the terms coming from the non commutation with the null frame
$$\ld{\frac{w_1^\frac{1}{2}}{(1+s)^2}Z^{N-3}k}\lesssim \frac{1}{(1+s)^{1+\sigma}}\ld{\frac{w^\frac{1}{2}}{1+|q|}Z^{N-3}k}
\lesssim \frac{\ep^\frac{3}{2}}{(1+t)^{1+\sigma-\rho}}.$$
Consequently we have proved
$$\ld{w_1^\frac{1}{2}\Box Z^{N-4}\wht g_1}\lesssim \frac{\ep^\frac{3}{2}}{(1+t)^{1+\sigma-\rho}},$$
which concludes the proof of Proposition\ref{prpn4g}.
\end{proof}
\subsection{Estimate of $\partial Z^{N-10}\wht g_1$}
\begin{prp}\label{prpn10g}
	We have 
	$$\ld{\partial Z^{N-10}\wht g_1}\lesssim C_0 \ep + C\ep^\frac{5}{4}.$$
\end{prp}
\begin{proof}
	We perform the energy estimate in the flat metric.
	We note that in the exterior region the result is already given by Proposition \ref{prpn4g}. In the interior, the only place where the weight $w_1$ was needed, was in the estimate of the term coming from the non commutation with the null frame,
	$$\Upr \frac{\partial_\theta Z^{N-10}k}{(1+r)^2}.$$
	We estimate it thanks to Corollary \ref{linfk}.
	For $I\leq N-10$ we have
	$$\ld{\Upr\frac{1}{(1+s)^2}Z^{I+1}k}\lesssim \left(\int_\frac{t}{2}^{2t} \frac{\ep^2}{(1+s)^{5-2\rho}}rdr\right)^\frac{1}{2}\lesssim
	\frac{\ep^2}{(1+t)^{\frac{3}{2}-\rho}}.$$
\end{proof}

\section{Choice of $b$}
\subsection{Proof of Proposition \ref{prph}}\label{secprph}
In this section we will choose $h$ for the next iteration. The heuristic choice would be $2\int_0^\infty \Upr\partial_t \phi \partial_r \phi rdr$. However we have to modify this choice in order for $h$ to satisfy two important conditions
\begin{itemize}
	\item $\partial_t h$ must be $O\left(\frac{1}{(1+t)^2}\right)$,
	\item $\partial_t h$ must be at the same level of regularity than $\partial_\theta \partial \phi$ and $\partial_\theta g$.
\end{itemize}
We can achieve the second point by choosing $h$ to be equal to $2\int_0^\infty \Upr g^{t\alpha}\sqrt{|\det{g}|}\partial_\alpha \phi \partial_r \phi dr$, where \textbf{in this section $det(g)$ denotes the determinant of $g$ in the coordinates $t,r,\theta$.}
However, with this choice, $\partial_t h$ contains a term of the form $\int \partial_r g_{LL}(\partial_q \phi)^2rdr$ which does not have the decay $O\left(\frac{1}{(1+t)^2}\right)$ (we can note that the regularity condition is satisfied by such a term because the $\partial_r$ which falls on $g_{LL}$ can be put on the other factors if necessary with an integration by part). To deal with such a term we will set
$$\check{h}(\theta,2t)
=2\int_0^\infty \Upr(1+\beta)g^{t\alpha}\sqrt{|\det{g}|}\partial_\alpha \phi \partial_r \phi dr,$$
with the metric $g$ expressed in coordinates $t,r,\theta$, and $\beta$ is a factor whose role is to compensate the term $\int \partial_r g_{LL}(\partial_q \phi)^2rdr$. We have to be careful with the choice of $\beta$, because it should not induce terms that do not have  the required regularity. Fortunately, the wave coordinate condition implies that the only term with a decay of $\frac{1}{(1+s)^\frac{3}{2}}$ in $\partial_q g_{LL}$ is $\frac{\wht g_{L\ba L}}{r}$, which is more regular than a derivative of $g$. To define precisely $\beta$ we need the following Corollary of Proposition \ref{estLL}
\begin{cor}\label{corbeta}
	In the region $q\leq R+1$ we can write
$$\partial_q \wht g_{LL} = \frac{1}{r}\wht g_{L\ba L}+ F_1+F_2,$$
where
$$F_1=\wht g_{\q T \q T}\partial_r \wht g_1+\wht g_1 \partial_U \wht g_1
+\partial_U \wht g_{\q T \q T} +\frac{1}{r} \wht g_{\q T \q T},$$
and
$$F_2= (\partial_s + \frac{1}{4}g_{LL}\partial_r) \wht g_{LL}+\wht g_1 (\partial_{s} +\frac{1}{4}g_{LL}\partial_r) \wht g_1.$$
\end{cor}
\begin{proof}
	We recall from Proposition \ref{estLL} that
	$$\partial_r \wht g_{LL} = \frac{1}{r}\wht g_{L\ba L}+ \wht g_{\q T \q T}\partial_q \wht g_1 + \wht g_1 \partial_{\q T} \wht g_1
	+\partial_U \wht g_{\q T \q T} +\frac{1}{r} \wht g_{\q T \q T}.$$
	To obtain Corollary \ref{corbeta}, we just reorder the terms, noticing that $\partial_q = \partial_r-\partial_s$, and neglecting cubic terms which have a similar decay.
\end{proof}

We can now define $\beta$ by $\beta=0$ at $t=T$ and
$$\partial_s \beta +\frac{1}{4}g_{LL}\partial_r \beta= -\frac{1}{2}\frac{\wht g_{L \ba L}}{r}-\frac{1}{2}F_2.$$

\begin{prp}\label{prph1}
	We have the estimates
		$$\|Z^N h(\theta,t)\|_{L^2(\m S^1)} \lesssim \ep^2(1+t)^\rho,$$
		$$\|Z^{N-1}h(\theta,t)\|_{L^2(\m S^1)}\lesssim \ep^2.$$
\end{prp}
\begin{proof}
	We have
	\begin{align*}&\left|Z^I\left(\Upr(1+\beta)g^{t\alpha}\sqrt{|\det{g}|}\partial_\alpha \phi \partial_r \phi\right)\right)|\\
	&\lesssim \frac{\ep\sqrt{1+s}}{(1+|q|)^{\frac{3}{2}-4\rho}}|\partial Z^I \phi|
	+ \frac{\ep^2}{(1+|q|)^{3-8\rho}}(|Z^I \wht g| +|Z^I \beta|).
	\end{align*}
	We can estimate
	$$\left\|\int \frac{\ep\sqrt{1+s}}{(1+|q|)^{\frac{3}{2}-4\rho}}|\partial Z^I \phi|dr \right\|_{L^2(\m S^1)}\lesssim \ep\ld{\partial Z^I \phi}.$$
For $I=N$ we estimate it with \eqref{phiN} and for $I=N-1$ with \eqref{phiN1}.
	$$\left\|\int  \frac{\ep^2}{(1+|q|)^{3-8\rho}}|Z^I \wht g|\right\|_{L^2(\m S^1)} \lesssim \frac{\ep^2}{\sqrt{1+t}}\ld{\frac{w_1}{1+|q|}Z^I \wht g}
	\lesssim \frac{\ep^3}{(1+s)^{\frac{1}{2}-2\rho}},$$
	thanks to \eqref{hg2n}.
	We now estimate $\beta$.
	We have
	$$\partial_s Z^I \beta +\frac{1}{4}g_{LL}\partial_r Z^I\beta= -\frac{1}{2}\frac{Z^I\wht g_{L \ba L}}{r}-\frac{1}{2}Z^I F_2 +Z^{I-J}g_{LL}\partial_q Z^J \beta $$
	and consequently
	$$(\partial_s +\frac{1}{4}g_{LL}\partial_r)( Z^I\beta+Z^I \wht g_1)= O\left(\frac{1}{r}Z^I \wht g_1\right)+O\left(Z^{I-J}g_{LL}\partial_q Z^J \beta\right).$$
	Thanks to \eqref{iks2bis} we easily obtain that $Z^{N-13}\beta = O\left(\ep\frac{\sqrt{1+|q|}}{\sqrt{1+s}}\right).$
	It is equivalent to integrate with respect to $s$ than with an affine parameter $s'$ along the integral curve of $\partial_s +g_{LL}\partial_q \beta$. We obtain
	\begin{equation}
	\label{beta}Z^I \beta(q,s,\theta) = O\left(Z^I \wht g_1\right)
	+\int_s^{2T-q} \left(\frac{Z^I\wht g_1}{s'+q} + \frac{Z^Ig_{LL}}{\sqrt{s'}\sqrt{1+|q|}}\right) ds'.
	\end{equation}
	We estimate
	\begin{align*}&\left(\int \left(\frac{1}{(1+|r-t|)^{1+\mu}} \int_{t+r}^{2T-t+r} \frac{Z^I\wht g_1(\rho,r-t,\theta)}{\rho+r-t} d\rho\right)^2 drd\theta\right)^\frac{1}{2}\\
	&\lesssim \int_t^{2T} \left(\int \ch_{r+t\leq \rho\leq 2T-t+r} \left(\frac{1}{(1+|r-t|)^{1+\mu}}\frac{Z^I\wht g_1(\rho,r-t,\theta)}{\rho+r-t}\right)^2 drd\theta \right)^\frac{1}{2}d\rho \\
	&\lesssim \left(\int_t^{2T} \frac{1}{\rho^\frac{5}{4}}d\rho \right)^\frac{1}{2}
	\left(\int_t^{2T} \int  \rho^\frac{5}{4}\ch_{r+t\leq \rho\leq 2T-t+r} \left(\frac{1}{(1+|r-t|)^{1+\mu}}\frac{Z^I\wht g_1(\rho,r-t,\theta)}{\rho+r-t}\right)^2 dr d\rho \right)^\frac{1}{2}.
	\end{align*}
	We make the change of variable $r'-t'=r-t$, $r'+t'=\rho$. We obtain
	\begin{align*}&\left(\int_\frac{t}{2}^{t+R} \left(\frac{1}{(1+|r-t|)^{1+\mu}} \int_{t+r}^{2T-t+r} \frac{Z^I\wht g_1(\rho,r-t,\theta)}{\rho+r-t} d\rho\right)^2 drd\theta\right)^\frac{1}{2}\\
	&\lesssim \left(\int_t^{2T} \frac{1}{\rho^\frac{5}{4}}d\rho \right)^\frac{1}{2}
	\left(\int_t^T  \int_{t'+\frac{t}{2}}^{R+t'}  (r'+t')^\frac{5}{4} \left(\frac{1}{(1+|r'-t'|)^{1+\mu}}\frac{Z^I\wht g_1(r'+t',r'-t',\theta)}{r'}\right)^2 dr d\rho \right)^\frac{1}{2}\\
	&\lesssim \frac{1}{t^\frac{1}{8}}
	\left(\int_t^T \frac{(2t'+R)^\frac{5}{4}}{(t'+\frac{t}{2})^3} 
	\left\|\frac{\ch_{q\leq R}}{(1+|q|)^{1+\mu}}Z^I\wht g_{1}\right\|^2_{L^2}dt' \right)^\frac{1}{2}.
	\end{align*}
	In the same way we estimate
		\begin{align*}&\left(\int_\frac{t}{2}^{t+R} \left(\frac{1}{(1+|r-t|)^{1+\mu}} \int_{t+r}^{2T-t+r} \frac{Z^I\wht g_{LL}(\rho,r-t,\theta)}{\sqrt{\rho}\sqrt{1+|r-t|}} d\rho\right)^2 drd\theta\right)^\frac{1}{2}\\
		&\lesssim \left(\int_t^{2T} \frac{1}{\tau^{1+\frac{1}{2}-4\rho}}d\tau \right)^\frac{1}{2}
		\left(\int_t^T  \int_{t'+\frac{t}{2}}^{R+t'}  (r'+t')^{1+\frac{1}{2}-4\rho} \left(\frac{1}{(1+|r'-t'|)^{\frac{3}{2}+\mu}}\frac{Z^I\wht g_{LL}(r'+t',r'-t',\theta)}{\sqrt{r'}}\right)^2 dr d\rho \right)^\frac{1}{2}\\
		&\lesssim \frac{1}{t^{\frac{1}{2}-2\rho}}
		\left(\int_t^T \frac{(2t'+R)^{1+\frac{1}{2}-4\rho}}{(t'+\frac{t}{2})^\frac{3}{2}} 
		\left\|\frac{\ch_{q\leq R}}{(1+|q|)^{1+\mu}}Z^I\wht g_{LL}\right\|^2_{L^2}dt' \right)^\frac{1}{2},
		\end{align*}
	and consequently
	\begin{equation}\label{estbeta}\begin{split}
\left(\int \frac{1}{(1+|q|)^{2+2\mu}}(Z^I \beta)^2drd\theta \right)^\frac{1}{2}
\lesssim &\frac{1}{t^\frac{1}{8}}
\left(\int_t^T \frac{(2t'+R)^\frac{5}{4}}{(t'+\frac{t}{2})^3} 
\left\|\frac{\ch_{q\leq R}}{(1+|q|)^{1+\mu}}Z^I\wht g_{1}\right\|^2_{L^2}dt' \right)^\frac{1}{2}\\
&+\frac{1}{t^{\frac{1}{2}-2\rho}}
\left(\int_t^T \frac{(2t'+R)^{1+\frac{1}{2}-4\rho}}{(t'+\frac{t}{2})^\frac{3}{2}} 
	\left\|\frac{\ch_{q\leq R}}{(1+|q|)^{1+\mu}}Z^I\wht g_{L L}\right\|^2_{L^2}dt' \right)^\frac{1}{2}\\
\lesssim& \frac{1}{t^\frac{1}{8}}\left(\int_t^T \frac{1}{t'^{3-\frac{5}{4}-2\rho}}dt'\right)^\frac{1}{2} 
+\frac{1}{t^{\frac{1}{2}-2\rho}}\left(\int_t^T (1+t')^{-4\rho}\ld{\frac{w_2'(q)}{1+|q|}Z^I g_{LL}}^2dt'\right)^\frac{1}{2}\\
\lesssim& \frac{\ep}{t^{\frac{1}{2}-2\rho}},
\end{split}
	\end{equation}
where we have used \eqref{hg1n} (we assume $\mu \geq \frac{1}{2}+\sigma$). It is easy to convince oneself, with Section \ref{l2g1} of
$$\int_t^T (1+t')^{-4\rho}\ld{\frac{w_2'(q)}{1+|q|}Z^I g_{LL}}^2dt'\lesssim \ep.$$
Consequently
$$
\left\|\int \frac{\ep^2}{(1+|q|)^{3-8\rho}}|Z^I \beta|\right\|_{L^2(\m S^1)} \lesssim \frac{\ep^3}{(1+t)^{\frac{1}{2}-2\rho}}.$$
This concludes the proof of Proposition \ref{prph1}.
\end{proof}

\begin{prp}\label{prph2}
	We have the estimates
	$$\|\partial_t Z^{N-1}h(\theta,t)\|_{H^{-1}(\m S^1)}\lesssim \frac{\ep^2}{(1+t)^{2-\frac{1}{2}\sigma}},$$
		$$\|\partial_t Z^{N}h(\theta,t)\|_{H^{-1}(\m S^1)}\lesssim \frac{\ep^2}{(1+t)^{1-\rho}},$$
	$$\int_0^t (1+s)\|\partial_t Z^N h\|^2_{L^2(\m S^1)} \lesssim \ep^4(1+t)^{2\rho},$$ 
		$$\int_0^t (1+s)^{3-2\rho}\|\partial_t^3 Z^N h\|^2_{H^{-2}(\m S^1)} \lesssim \ep^4(1+t)^{2\rho},$$ 
	and we can decompose $$\partial_s h=h_1+h_2$$ with
	$$\|Z^{N}h_1\|_{H^{-2}(\m S^1)}\lesssim \frac{\ep^2}{(1+t)^{2-\frac{1}{2}\sigma}},$$
	$$\int_0^t (1+s)^2\|Z^{N}h_1\|^2_{H^{-1}(\m S^1)}\lesssim \ep^2(1+t)^{2\rho},$$
	$$\int_0^t (1+s)^4\|\partial_t Z^{N}h_1\|^2_{H^{-2}(\m S^1)} \lesssim \ep^2(1+t)^{2\rho},$$
	$$\int_0^t (1+s)^3\|\partial_t Z^{N}h_1\|^2_{H^{-1}(\m S^1)} \lesssim \ep^2(1+t)^{2\rho},$$
	$$\|\partial_t Z^N h_2\|_{H^{-2}(\m S^1)}\lesssim \frac{\ep^2}{(1+s)^{\frac{3}{2}-\rho}},$$
	$$\int_0^t (1+s)^{3-2\rho}\|Z^{N}h_2\|^2_{H^{-1}(\m S^1)}\lesssim \ep^2(1+t)^{2\rho},$$
$$\int_0^t (1+s)^3\|\partial_t Z^{N}h_1\|_{H^{-1}(\m S^1)}^2 \lesssim \ep^2(1+t)^{2\rho}.$$
\end{prp}

\begin{proof}

Since $\phi$ satisfies $\Box_g \phi= 0$ we have in coordinates $t,r,\theta$
$$\partial_\mu (g^{\mu \nu}\sqrt{|\det (g)|}\partial_\nu \phi )=0.$$
We can neglect the contributions of $\Upr$, because when it is different from $1$, we are far from the light cone and $q\sim s$.
We calculate 
\begin{align*}
2\partial_t h(\theta,2t)
=& \int_0^\infty \partial_t \beta g^{t\alpha}\sqrt{|\det{g}|}\partial_\alpha \phi \partial_r \phi dr\\
&+\int_0^\infty (1+\beta)\partial_t\left(g^{t\alpha}\sqrt{|\det{g}|}\partial_\alpha \phi\right)\partial_r \phi dr
 + \int_0^\infty (1+\beta)g^{t\alpha}\sqrt{|\det{g}|}\partial_\alpha \phi \partial_r\partial_t \phi dr\\
=&\int_0^\infty \partial_t \beta g^{t\alpha}\sqrt{|\det{g}|}\partial_\alpha \phi \partial_r \phi dr\\
&-\int_0^\infty(1+\beta)\partial_r\left(g^{r\alpha}\sqrt{|\det{g}|}\partial_\alpha \phi\right)\partial_r \phi dr
- \int_0^\infty(1+\beta)\partial_\theta\left(g^{\theta\alpha}\sqrt{|\det{g}|}\partial_\alpha \phi\right)\partial_r \phi dr\\
&+\frac{1}{2}\int_0^\infty (1+\beta)g^{tt}\sqrt{|\det{g}|}\partial_r(\partial_t \phi)^2 \phi dr\\
&+\int_0^\infty (1+\beta)g^{tr}\sqrt{|\det{g}|}\partial_r \phi \partial_ r\partial_t \phi dr
+ \int_0^\infty (1+\beta)g^{t\theta}\sqrt{|\det{g}|}\partial_\theta \phi \partial_ r\partial_t \phi dr\\
=&\int_0^\infty \partial_t \beta g^{t\alpha}\sqrt{|\det{g}|}\partial_\alpha \phi \partial_r \phi dr
+\int_0^\infty \partial_r \beta g^{rr}\sqrt{|\det{g}|}(\partial_r \phi)^2 dr\\
&+\frac{1}{2}\int_0^\infty (1+\beta)g^{rr}\sqrt{|\det{g}|}\partial_r (\partial_r \phi)^2 dr \\
&-\int_0^\infty (1+\beta)g^{rt}\sqrt{|\det{g}|}\partial_r \partial_t \phi\partial_r \phi dr-\int_0^\infty(1+\beta)\partial_\theta\left(g^{\theta\alpha}\sqrt{|\det{g}|}\partial_\alpha \phi\right)\partial_r \phi dr\\
&-\int_0^\infty (1+\beta)\partial_r(g^{rt}\sqrt{|\det{g}|}) \partial_t \phi\partial_r \phi dr
-\int_0^\infty(1+\beta)\partial_r\left(g^{r\theta}\sqrt{|\det{g}|}\partial_\theta \phi\right)\partial_r \phi dr\\
&-\frac{1}{2}\int_0^\infty \partial_r \left((1+\beta)g^{tt}\sqrt{|\det{g}|}\right)(\partial_t \phi)^2  dr\\
&+\int_0^\infty (1+\beta)g^{tr}\sqrt{|\det{g}|}\partial_r \phi \partial_ r\partial_t \phi dr + \int_0^\infty (1+\beta)g^{t\theta}\sqrt{|\det{g}|}\partial_\theta \phi \partial_ r\partial_t \phi dr\\
=&\int_0^\infty \partial_t \beta g^{t\alpha}\sqrt{|\det{g}|}\partial_\alpha \phi \partial_r \phi dr
+\frac{1}{2}\int_0^\infty \partial_r \beta \sqrt{|\det{g}|}\left( g^{rr}(\partial_r \phi)^2-g^{tt}(\partial_t \phi)^2\right)dr\\
&-\frac{1}{2}\int_0^\infty (1+\beta)\left(\partial_r(g^{rr} \sqrt{|\det{g}|})(\partial_r \phi)^2+2\partial_r(g^{rt}\sqrt{|\det{g}|}) \partial_t \phi\partial_r \phi+\partial_r(g^{tt}\sqrt{|\det{g}|})(\partial_t \phi)^2 \right)dr \\
&\int_0^\infty(1+\beta)\partial_\theta\left(g^{\theta\alpha}\sqrt{|\det{g}|}\partial_\alpha \phi\right)\partial_r \phi dr
+\int_0^\infty \partial_r \beta g^{r\theta}\sqrt{|\det{g}|}\partial_\theta \phi\partial_r \phi dr\\
&+\int_0^\infty(1+\beta)g^{r\theta}\sqrt{|\det{g}|}\partial_\theta \phi\partial_r^2 \phi dr
 +\int_0^\infty (1+\beta)g^{t\theta}\sqrt{|\det{g}|}\partial_\theta \phi \partial_ r\partial_t \phi dr
\end{align*}

We analyse the different contribution to $\partial_t h$ :
$$\partial_t h= \int A_1+A_2+...$$
where
\[A_1=  \partial_s \beta \sqrt{|\det(g)|}
(\partial_q \phi)^2\left(-g^{tt}+g^{tr}+\frac{1}{2}(g^{rr}-g^{tt})\right)=\partial_s \beta \sqrt{|\det(g)|}
(\partial_q \phi)^2(-4g^{\ba L\ba L}-4g^{L\ba L}),
\]
\[A_2= \partial_s \beta \sqrt{|\det(g)|}\partial_s \phi \partial_q \phi\left(2g^{tr}+\frac{1}{2}(2g^{rr}+2g^{tt})\right)
=\partial_s \beta \sqrt{|\det(g)|}\partial_s \phi \partial_q \phi 4g^{ L  L},
\]
\[A_3=\partial_s \beta \sqrt{|\det(g)|}(\partial_s \phi )^2\left(g^{tt}+g^{tr}+\frac{1}{2}(g^{rr}-g^{tt})\right)
=\partial_s \beta \sqrt{|\det(g)|}(\partial_s \phi)^2 2g^{ L L},
\]
\[A_4=  \partial_q \beta \sqrt{|\det(g)|}
(\partial_q \phi)^2\left(g^{tt}-g^{tr}+\frac{1}{2}(g^{rr}-g^{tt})\right)=\partial_q \beta \sqrt{|\det(g)|}
(\partial_q \phi)^22  g^{\ba L \ba L},
\]
\[A_5= \partial_q \beta \sqrt{|\det(g)|}\partial_s \phi \partial_q \phi\left(-2g^{tr}+\frac{1}{2}(2g^{rr}+2g^{tt})\right)
=\partial_q \beta \sqrt{|\det(g)|}\partial_s \phi \partial_q \phi 4g^{ \ba L \ba L},
\]
\[A_6=\partial_q \beta \sqrt{|\det(g)|}(\partial_s \phi )^2\left(-g^{tt}-g^{tr}+\frac{1}{2}(g^{rr}-g^{tt})\right)
=\partial_s \beta \sqrt{|\det(g)|}(\partial_s \phi)^2 (-2g^{LL}-4g^{L\ba L}),\]
\[A_7=(1+\beta)(\partial_q \phi)^2\partial_r \left(\sqrt{|\det g|}(g^{rr}-2g^{rt}+g^{tt})\right)
=(1+\beta)(\partial_q \phi)^2\partial_r \left(\sqrt{|\det g|}4g^{\ba L \ba L}\right),
\]
\[A_8=(1+\beta)\partial_q \phi\partial_s \phi \partial_r \left(\sqrt{|\det g|}(g^{rr}-g^{tt})\right)
=-(1+\beta)\partial_s \phi \partial_q \phi\partial_r \left(\sqrt{|\det g|}4g^{\ba L L}\right),
\]
\[A_9=(1+\beta)(\partial_s \phi)^2 \partial_r \left(\sqrt{|\det g|}(g^{rr}+2g^{rt}+g^{tt})\right)
=(1+\beta)(\partial_s \phi)^2\partial_r \left(\sqrt{|\det g|}4g^{ L L}\right),
\]
\[A_{10}=(1+\beta)\partial_\theta \left(g_{UU}\frac{\sqrt{|\det g|}}{r}\partial_{U}\phi\right)\partial_r \phi,
\]
\[A_{11}=(1+\beta)\partial_\theta \left(g_{U\ba L}\frac{\sqrt{|\det g|}}{r}\partial_s\phi\right)\partial_r \phi,
\]
\[A_{12}=(1+\beta)\partial_\theta \left(g_{UL}\frac{\sqrt{|\det g|}}{r}\partial_q\phi\right)\partial_r \phi,
\]
\[A_{13}=\partial_s \beta (g^{Ut}+g^{Ur})\sqrt{|\det g|}\partial_U\phi \partial_r \phi=\partial_s \beta g^{UL}\sqrt{|\det g|}\partial_U\phi \partial_r \phi,\]
\[A_{14}=\partial_q \beta (-g^{Ut}+g^{Ur})\sqrt{|\det g|}\partial_U\phi \partial_r \phi
=\partial_q \beta g^{U\ba L}\sqrt{|\det g|}\partial_U\phi \partial_r \phi
,\]
\[A_{15}=(1+\beta)\sqrt{|\det{g}|}(g^{Ur}+g^{Ut})\partial_U \phi \partial_r \partial_s \phi =   (1+\beta)\sqrt{|\det{g}|}g^{UL}\partial_U \phi \partial_r \partial_s \phi
,\]
\[A_{16}=(1+\beta)\sqrt{|\det{g}|}(g^{Ur}-g^{Ut})\partial_U \phi \partial_r \partial_q \phi =   (1+\beta)\sqrt{|\det{g}|}g^{U\ba L}\partial_U \phi \partial_r \partial_q \phi
.\]
We remark (see \eqref{determinant}) that 
$$\sqrt{|\det g|}g^{\ba L L }\sim -\frac{1}{2}g_{UU}r,$$
so the term $A_8$ has the required decay.
We recall that we choose $\beta$ such that $\beta=0$ at time $T$ and
$$2\partial_s \beta+\frac{1}{2}g_{LL}\partial_q \beta =- \frac{\wht g_{\ba L \ba L}}{r}-F_2.$$ 
We remark that thanks to Corollary \ref{corbeta},
$$2\partial_s \beta+\frac{1}{2}g_{LL}\partial_q \beta=-\partial_r  g_{LL}+F_1
=-\partial_r  g_{LL}+
\wht g_{\q T \q T}\partial_r  g_1+\wht g_1 \partial_U  g_1
+\partial_U  g_{\q T \q T} +\frac{1}{r} \wht g_{\q T \q T}$$
Consequently we have
\begin{align*}
A_1+A_7 +A_4=&\partial_s \beta \sqrt{|\det(g)|}
(\partial_q \phi)^2(-4g^{\ba L\ba L}-4g^{L\ba L})+(1+\beta)(\partial_q \phi)^2\partial_r \left(\sqrt{|\det g|}4g^{\ba L \ba L}\right)\\
&+\partial_q \beta \sqrt{|\det(g)|}
(\partial_q \phi)^2 2 g^{\ba L \ba L}\\
=&(\partial_q \phi)^2\sqrt{|\det(g)|}(2\partial_s \beta +\partial_r g_{LL}+\frac{1}{2}g_{LL}\partial_q \beta
)+O\left(r(\partial_q \phi)^2(\partial_r g_{LL})\wht g_1\right)+s.t.\\
=&O\left(r(\partial_q \phi)^2\partial_U g_{\q T \q T}\right)+O\left(r(\partial_q \phi)^2(\partial_r g_{LL})\wht g_1\right)+s.t.
\end{align*}
We note that applying a vector field to $h$ corresponds to applying a vector field to the integrand. We note also that we can get rid of a $\partial_r$ derivative on a term $Z^N g$ or $\partial Z^N \phi$ by integration by part.
Then we can distribute the vector fields and use the $L^\infty$ estimate for the terms $Z^J \wht g$ and $Z^J \phi$ with $J \leq N/2$.
We obtain that 
$$\partial_s Z^I h =\int (B_1+B_2+...)dr$$ where
$$
B_1=O\left(\frac{\ep}{(1+s)^\frac{1}{2}(1+|q|)^{\frac{3}{2}-4\rho}}\right)\partial_s \partial^\alpha_\theta Z^I \phi,$$
$$B_2=O\left(\frac{\ep}{(1+s)^\frac{3}{2}(1+|q|)^{\frac{1}{2}-4\rho}}\right)\partial_q \partial^\alpha_\theta Z^I \phi,$$
$$B_3=O\left(\frac{\ep}{(1+s)^\frac{1}{2}(1+|q|)^{\frac{3}{2}-4\rho}}\right)\frac{\partial_\theta}{r} \partial^\alpha_\theta Z^I \phi,$$
$$B_4= O\left(\frac{\ep^2}{(1+s)(1+|q|)^{3-8\rho}}\right) \partial^\alpha_\theta Z^I g_{\q T \q T},$$
$$B_5= O\left(\frac{\ep^2}{(1+s)^2(1+|q|)^{2-8\rho}}\right) \partial^\alpha_\theta Z^I \wht g,$$
$$B_6=O\left(\frac{\ep^2}{(1+s)^\frac{3}{2}(1+|q|)^{\frac{7}{2}-8\rho-\sigma}}\right) Z^I \beta,$$
$$B_7=O\left(\frac{\ep^2}{(1+s)^{\frac{3}{2}}(1+|q|)^{\frac{3}{2}-8\rho}}\right) \partial^\alpha_\theta Z^I \wht g_1,$$
where $\alpha=0,1$.
We estimate the term for $I=N$.
We start with $B_1$. We write 
$$\partial_s Z^N \phi= \frac{1}{s}S Z^N \phi + \frac{q}{s}\partial_q Z^N \phi
=\frac{1}{s}(SZ^N  \phi - s\partial_q\phi Z^N g_{LL}) +\partial_q \phi Z^N g_{LL} 
 + \frac{q}{s}\partial_q Z^N \phi.$$
 The last two terms are similar to $B_4$ and $B_2$ respectively. We note
 $$\wht B_1= O\left(\frac{\ep}{(1+s)^\frac{3}{2}(1+|q|)^{\frac{3}{2}-4\rho}}\right) (S\partial^\alpha_\theta Z^N  \phi - s\partial_q\phi \partial^\alpha_\theta Z^N g_{LL}).$$
We estimate
\begin{align*}
\left\|\int \wht B_1dr\right\|_{H^{-1}(\m S^1)}&\lesssim
\int \frac{\ep}{(1+s)^{\frac{3}{2}}(1+|q|)^{\frac{3}{2}-4\rho}}\|SZ^N  \phi - s\partial_q\phi Z^N g_{LL}\|_{L^2(\m S^1)}dr\\
&\lesssim \frac{\ep}{(1+t)^2} \left( \int \frac{1}{(1+|q|)^{2+2\rho}}|SZ^N  \phi - s\partial_q\phi Z^N g_{LL}|^2 rdrd\theta\right)^{\frac{1}{2}}\left( \int \Up\left(\frac{r}{t}\right)\frac{1}{(1+|q|)^{1-10\rho}}dr\right)^{\frac{1}{2}}\\
&\lesssim \frac{\ep}{(1+t)^{2-10\rho}}\| w^\frac{1}{2} \partial (SZ^N  \phi - s\partial_q\phi Z^N g_{LL})\|_{L^2}\\
&\lesssim  \frac{\ep^2}{(1+t)^{2-\frac{1}{2}\sigma}},
\end{align*}
\begin{align*}
\left\|\int B_2dr\right\|_{H^{-1}(\m S^1)}&\lesssim \int
\frac{\ep}{(1+s)^\frac{3}{2}(1+|q|)^{\frac{1}{2}-4\rho}}\|\partial_q Z^I \phi\|_{L^2(\m S^1)}dr\\
&\lesssim \frac{\ep}{(1+t)^2}\left(\int (\partial_q Z^I \phi)^2 rdrd\theta\right)^\frac{1}{2}\left(\int\Up\left(\frac{r}{t}\right) \frac{1}{(1+|q|)^{1-8\rho}}dr\right)^\frac{1}{2}\\
&\lesssim \frac{\ep^2}{(1+t)^{2-3\rho}},
\end{align*}
We estimate $B_3$ in $H^{-2}$ first
$$
\left\|\int  B_3dr\right\|_{H^{-2}(\m S^1)}\lesssim
\int \frac{\ep}{(1+s)^{\frac{3}{2}}(1+|q|)^{\frac{3}{2}-4\rho}}\|Z^N \phi\|_{L^2(\m S^1)}dr
\lesssim  \frac{\ep^2}{(1+t)^{2-5\rho}}.$$
and now in $H^{-1}$
\begin{align*}
\left\|\int B_3dr\right\|_{H^{-1}(\m S^1)}&\lesssim
\int \frac{\ep}{(1+s)^{\frac{1}{2}}(1+|q|)^{\frac{3}{2}-4\rho}}\|\bar{\partial}Z^N \phi\|_{L^2(\m S^1)}dr\\
&\lesssim \frac{\ep}{(1+t)} \left(\int \frac{1}{(1+|q|)^{1+2\mu}}|\bar{\partial}Z^N \phi|^2 rdrd\theta\right)^\frac{1}{2}
\left(\int \frac{1}{(1+|q|)^{2-8\rho-2\mu}}dr\right)^\frac{1}{2}\\
&\lesssim \frac{\ep}{1+t}\|w'(q)^\frac{1}{2}\bar{\partial}Z^N \phi\|_{L^2}
\end{align*}
and so
$$\int_0^t (1+s)^2 \left\|\int B_3dr\right\|_{H^{-1}(\m S^1)}^2ds\lesssim \ep^4(1+t)^{2\rho}.$$
We also have
$$\left\|\int B_3dr\right\|_{H^{-1}(\m S^1)}\lesssim \frac{\ep^2}{(1+t)^{1-\rho}}.$$
We now turn to $B_4$.
\begin{align*}
 \left\|\int B_4 dr\right\|_{H^{-1}(\m S^1)}
 &\lesssim \int
 \frac{\ep^2}{(1+s)(1+|q|)^{3-8\rho}} \| Z^N g_{L \q T}\|_{L^2(\m S^1)}dr\\
 &\lesssim \frac{\ep^2}{(1+s)^\frac{3}{2}}\left(\int \frac{1}{(1+|q|)^{2-2\sigma -16\rho}}dr\right)^\frac{1}{2}\left\| \frac{w_2(q)}{(1+|q|)^{\frac{3}{2}}}  Z^N \wht g_{L \q T}\right\|_{L^2}.
\end{align*}
First of all, thanks to \eqref{hg1n} we have
$$\left\|\int B_4dr\right\|_{H^{-1}(\m S^1)}\lesssim \frac{\ep^2}{(1+t)^{\frac{3}{2}-\rho}}.$$
To have a more precise estimate use
Propositions \ref{l2ll}, \ref{l2lu} and \ref{l2uu}. We have to decompose $B_4= B_4^{(1)}+B_4^{(2)}$ with
$$\int_0^t (1+\tau)^{3-2\rho} \left\|\int B_4^{(1)} dr\right\|_{H^{-1}(\m S^1)} d\tau\lesssim \ep^2(1+t)^{2\rho},$$
and
$$B_4^{(2)}=O\left(\frac{\ep^2}{(1+s)(1+|q|)^{3-8\rho}} \partial^\alpha_\theta (\sigma^0_{UL})(1-\chi(q))\right)
=O\left(\frac{\ep^2}{(1+s)(1+|q|)^{3-8\rho}}(1-\chi(q))s\partial_s Z^N b\right).$$
We again decompose $B^{(2)}_4= \wht B^{(2)}_4+ \wht B^{(3)}_4$ with
$$\wht B_4^{(2)}=O\left(\frac{\ep^2}{(1+s)(1+|q|)^{3-8\rho}} (1-\chi(q))sZ^N f_1\right)$$
and
$$\wht B_4^{(3)}=O\left(\frac{\ep^2}{(1+s)(1+|q|)^{3-8\rho}} (1-\chi(q))sZ^N f_2\right).$$
We have
$$\left\|\int \wht B_4^{(2)}dr\right\|_{L^2(\m S^1)} \lesssim \int \frac{\ep^2}{(1+|q|)^{3-8\rho}}\|Z^N f_1\|_{L^2(\m S^1)}dr\lesssim \frac{\ep^4}{(1+s)^{2-\frac{\sigma}{2}}}$$
and
$$\int_0^t (1+s)^{3-2\rho}\left\|\int \wht B_4^{(3)}dr\right\|_{L^2(\m S^1)}^2ds\lesssim \ep^4\int_0^t (1+s)^{3-2\rho}\|Z^N f_2\|^2_{L^2}ds\lesssim \ep^6(1+t)^{2\rho}.$$
\begin{align*}
\left\|\int B_5dr\right\|_{H^{-1}(\m S^1)}&\lesssim\int \frac{\ep^2}{(1+s)^2(1+|q|)^{2-8\rho}}\| Z^I \wht g\|_{L^2(\m S^1)}\\
&\lesssim \frac{\ep^2}{(1+s)^\frac{5}{2}}\left\|\frac{1}{(1+|q|)^{1+\mu+\sigma}}\partial Z^I \wht g\right\|_{L^2} \left(\int \frac{1}{(1+|q|)^{2-16\rho-\mu}}\right)^\frac{1}{2}\\
&\lesssim \frac{\ep^2}{(1+s)^\frac{5}{2}}\|w_1^\frac{1}{2}\partial Z^N \wht g\|_{L^2}
\lesssim \frac{\ep^2}{(1+s)^{\frac{5}{2}-2\rho}}.
\end{align*}
To estimate $B_6$ we use \eqref{estbeta}.

\begin{align*}
\left\|\int B_6dr\right\|_{L^2}&\lesssim \int \left(\frac{\ep^2}{(1+s)^\frac{3}{2}(1+|q|)^{\frac{7}{2}-8\rho-\sigma}}\right) Z^I \beta dr\\
&\lesssim \frac{\ep^2}{(1+t)^\frac{3}{2}}\left(\int \frac{1}{(1+|q|)^{2+2\mu}}|Z^I \beta|^2dr\right)^\frac{1}{2}\\
&\lesssim \frac{\ep^3}{(1+t)^{2-2\rho}}.
\end{align*}
We estimate
\begin{align*}
\left\|\int B_7 dr\right\|_{H^{-1}(\m S^1)}
&\lesssim \int
\frac{\ep^2}{(1+s)^\frac{3}{2}(1+|q|)^{\frac{3}{2}-8\rho}} \| Z^N \wht g_1\|_{L^2(\m S^1)}dr\\
&\lesssim \frac{\ep^2}{(1+s)^2}\left(\int \frac{1}{(1+|q|)^{1-2\sigma -8\rho}}dr\right)^{\frac{1}{2}}\left\| \frac{w_2(q)^\frac{1}{2}}{1+|q|}  Z^N \wht g_1\right\|_{L^2} \\
&\lesssim \frac{\ep^2}{(1+s)^{2-\frac{\sigma}{2}}}.
\end{align*}
We set
$$Z^N h_1= \int \wht B_1+B_2+B_3 + \wht B^{(2)}_4+B_5+B_6+B_7, \quad Z^N h_2= \int B_4^{(1)}+ \wht B^{(3)}_4.$$

We now turn to the estimate of $\partial_s^2 Z^N h$. We start with $\partial_s Z^N h_1$. We have
\begin{align*}\partial_s \wht B_1 =&O\left(\frac{\ep}{(1+s)^\frac{5}{2}(1+|q|)^{\frac{3}{2}-4\rho}}\right) (SZ^N  \phi - s\partial_q\phi Z^N g_{LL})\\
&+O\left(\frac{\ep}{(1+s)^\frac{3}{2}(1+|q|)^{\frac{3}{2}-4\rho}}\right) \partial_s (SZ^N  \phi - s\partial_q\phi Z^N g_{LL}).
\end{align*}
We estimate
\begin{align*}
&\left\|\int \partial_s \wht B_1dr\right\|_{H^{-1}(\m S^1)}\\
\lesssim&
\int \frac{\ep}{(1+s)^{\frac{3}{2}}(1+|q|)^{\frac{3}{2}-4\rho}}\|\partial_s (SZ^N  \phi - s\partial_q\phi Z^N g_{LL}) \phi\|_{L^2(\m S^1)}dr +\frac{\ep^2}{(1+t)^{3-5\rho-\mu}} \\
\lesssim &\frac{\ep}{(1+t)^2} \left( \int \frac{1}{(1+|q|)^{1+2\mu}}|\partial_s (SZ^N  \phi - s\partial_q\phi Z^N g_{LL})|^2 rdrd\theta\right)^{\frac{1}{2}}\left( \int \Up\left(\frac{r}{t}\right)\frac{1}{(1+|q|)^{2-8\rho -2\mu}}dr\right)^{\frac{1}{2}}\\
&+\frac{\ep^2}{(1+t)^{3-5\rho-\mu}}\\
\lesssim& \frac{\ep}{(1+t)^{2}}\| w'(q) \partial_s (SZ^N  \phi - s\partial_q\phi Z^N g_{LL})\|_{L^2}+\frac{\ep^2}{(1+t)^{3-5\rho-\mu}}.
\end{align*}
To estimate $\partial_s \int B_2$ and $\partial_s \int B_3$ we note that
$$\partial_s \int B_2dr = \int \left(\frac{1}{t}S B_2 +\frac{r}{t}\partial_r B_2\right)
=\frac{1}{t} \int (SB_2 -B_2).$$ 
Consequently we obtain

$$\left\|\partial_s \int  B_2dr\right\|_{H^{-1}(\m S^1)}\lesssim \frac{\ep^2}{(1+t)^{3-3\rho}},$$
$$
\left\|\partial_s \int  B_3dr\right\|_{H^{-2}(\m S^1)}
\lesssim  \frac{\ep^2}{(1+t)^{3-5\rho}},$$
To estimate $\partial_s B_3$ in $H^{-1}$ we write
$$\partial_s B_3 =O\left(\frac{\ep}{(1+s)^\frac{3}{2}(1+|q|)^{\frac{3}{2}-4\rho}}\right)\frac{\partial_\theta}{r} \partial^\alpha_\theta (S Z^N \phi - s\partial_q\phi Z^N g_{LL})+O(\partial_s B_2)+O(\partial_s B_4).$$
Consequently
$$
\left\|\partial_s \int B_3dr\right\|_{H^{-1}(\m S^1)}\lesssim \frac{\ep}{(1+t)^\frac{3}{2}}\|w'(q)^\frac{1}{2}\bar{\partial}(SZ^N \phi -s\partial_q\phi Z^N g_{LL})\|_{L^2},
$$
$$
 \left\|\int \partial_s \wht B^{(2)}_4 dr\right\|_{H^{-1}(\m S^1)}
\lesssim \int \frac{\ep^2}{(1+s)(1+|q|)^{3-8\rho}} \|sZ^N \partial_s f_1\|_{L^2(\m S^1)}dr \lesssim \frac{\ep^4}{(1+t)^{2-\rho}},$$
\begin{align*}
\left\|\int \partial_s B_5dr\right\|_{H^{-1}(\m S^1)}&\lesssim\int \frac{\ep^2}{(1+s)^2(1+|q|)^{2-8\rho}}\| \partial_s Z^I \wht g\|_{L^2(\m S^1)}dr\\
&\lesssim \frac{\ep^2}{(1+s)^\frac{5}{2}}\|w_1'(q)^\frac{1}{2}\bar{\partial} Z^N \wht g\|_{L^2}.
\end{align*}
Since $\partial_s \beta \sim \frac{\wht g_{L\ba L}}{r}+ \bar{\partial} \wht g_{L\q T}$, $\partial_s B_6$ gives contributions similar to $\partial_s B_7$.
\begin{align*}
\left\|\int \partial_s B_7 dr\right\|_{H^{-1}(\m S^1)}
&\lesssim \int
\frac{\ep^2}{(1+s)^\frac{3}{2}(1+|q|)^{\frac{3}{2}-8\rho}} \| \partial_s Z^N \wht g_1\|_{L^2(\m S^1)}dr\\
&\lesssim \frac{\ep^2}{(1+s)^2}\left(\int \frac{1}{(1+|q|)^{1-2\sigma -8\rho}}dr\right)^\frac{1}{2}\left\| w_2'(q)^\frac{1}{2} \partial_s  Z^N \wht g_1\right\|^2_{L^2}.
\end{align*}
Consequently we have proven
\begin{align*}
&\int_0^t (1+s)^4\left\| \partial_s Z^N h_1\right\|_{H^{-2}(\m S^1)}^2ds\\
\lesssim& \int_0^t \ep^2\ld{w'(q)^\frac{1}{2}\bar{\partial}(SZ^N \phi -s\partial_q\phi Z^N g_{LL})}^2
+\frac{\ep^2}{(1+s)}\ld{w_1'(q)^\frac{1}{2}\bar{\partial} Z^N \wht g}\\
&+\ep^2\ld{ w_2'(q)^\frac{1}{2} \partial_s  Z^N \wht g_1}^2 +\frac{\ep^4}{(1+s)^{2-\frac{\sigma}{2}}}ds\\
	\lesssim& \ep^4(1+t)^{2\rho},
\end{align*}
and also
$$\int_0^t (1+s)^3\left\| \partial_s Z^N h_1\right\|_{H^{-1}(\m S^1)}^2ds
\lesssim  \ep^4(1+t)^{2\rho}.$$
We now turn to $h_2$.
We have
\begin{align*}
&\left\|\int \partial_s (B_4^{(2)}+\wht B_4^{(3)}) \right\|_{H^{-1}(\m S^1)}\\
&\lesssim \int
\frac{\ep^2}{(1+s)(1+|q|)^{3-8\rho}} \| \partial_s Z^N g_{\q T \q T}\|_{L^2(\m S^1)}
+\int
\frac{\ep^2}{(1+s)(1+|q|)^{3-8\rho}} \|s \partial_s Z^N b\|_{L^2(\m S^1)}
\\
&\lesssim \frac{\ep^2}{(1+s)^\frac{3}{2}}\left\| w_2'(q)^\frac{1}{2} \bar{\partial }Z^N \wht g_{\q T \q T}\right\|_{L^2} 
+\int
\frac{\ep^2}{(1+|q|)^{3-8\rho}} \| \partial_s Z^N b\|_{L^2(\m S^1)}
\end{align*}
and consequently
$$\int_0^t (1+s)^3\|\partial_s Z^N h_2\|_{H^{-1}(\m S^1)}^2
\lesssim\ep^2 \int_0^t \ld{ w_2'(q)^\frac{1}{2} \bar{\partial }Z^N \wht g_{\q T \q T}}^2
+\ep^2\int_0^t (1+s)^3\|\partial_s Z^N b\|_{L^2(\m S^1)}^2\lesssim \ep^4(1+t)^{2\rho}.$$

We now do the estimate of $\partial_s Z^N h$ in $L^2$. We can write the terms of our decomposition in the form
$$
\wht B_1=O\left(\frac{\ep}{(1+s)^\frac{1}{2}(1+|q|)^{\frac{3}{2}-4\rho}}\right)\bar{\partial} S Z^N \phi,$$
$$B_2=O\left(\frac{\ep}{(1+s)^\frac{3}{2}(1+|q|)^{\frac{1}{2}-4\rho}}\right) \partial Z^{N+1}\phi,$$
$$ B_3=O\left(\frac{\ep}{(1+s)^\frac{1}{2}(1+|q|)^{\frac{3}{2}-4\rho}}\right)\bar{\partial} Z^{N+1}\phi,$$
$$ (B_4+B_7)= O\left(\frac{\ep^2}{(1+|q|)^{3-8\rho}}\right) \bar{\partial} Z^N \wht g_1,$$
$$ B_5=O\left(\frac{\ep^2}{(1+|q|)^{2-8\rho}(1+s)}\right) \bar{\partial} Z^N \wht g,$$
$$ B_6=O\left(\frac{\ep^2}{(1+s)^\frac{3}{2}(1+|q|)^{\frac{7}{2}-8\rho-\sigma}}\right) Z^N \beta.$$
We estimate
$$\left\|\int  \wht B_1dr\right\|_{L^2(\m S^1)}
\lesssim \frac{\ep}{1+t}\| w'(q) \bar{\partial} S Z^N \phi\|_{L^2},$$
$$\left\|\int  B_2dr\right\|_{L^2(\m S^1)}
\lesssim \frac{\ep}{(1+t)^2}\left(\int \Upr \frac{1}{(1+|q|)^{1-8\rho}}dr\right)^\frac{1}{2}\| w(q) \partial Z^{N+1} \phi\|_{L^2} \lesssim \frac{\ep^2}{(1+t)^{\frac{3}{2}-4\rho}},$$
$$
\left\|\int  B_3dr\right\|_{L^2(\m S^1)}
\lesssim \frac{\ep}{1+t}\ld{w'(q)^\frac{1}{2}\bar{\partial}Z^{N+1}\phi}$$
$$\left\|\int  (B_4+B_7)dr\right\|_{L^2(\m S^1)}
\lesssim \frac{\ep^2}{(1+t)^\frac{1}{2}}\left\| w_2'(q) \bar{\partial }Z^N \wht g_1\right\|^2_{L^2},$$
$$\left\|\int B_5dr\right\|_{L^2(\m S^1)}
\lesssim \frac{\ep^2}{(1+t)^\frac{3}{2}}\left\| w_1'(q) \bar{\partial }Z^N \wht g\right\|^2_{L^2},$$
$$\left\|\int   B_6dr\right\|_{L^2(\m S^1)}
\lesssim \frac{\ep^2}{(1+t)^{2-\rho}}.$$
Consequently
\begin{align*}\int_0^t (1+s)\|\partial_s h\|^2_{L^2}ds \lesssim &
\ep\int_0^t \left(\frac{\ep^2}{1+s}\ld{w'(q)^\frac{1}{2}\bar{\partial}Z^{N+1}\phi}^2+\frac{\ep^4}{(1+s)^{2-8\rho}}\right)ds\\
&+\int_0^t \left(\ep^2\ld{ w_2'(q) \bar{\partial }Z^N \wht g_1}^2
+\frac{\ep^2}{(1+s)^2}\ld{w_1'(q) \bar{\partial }Z^N \wht g}^2\right)ds\\
\lesssim &\ep^4(1+t)^{2\rho}.
\end{align*}

We finish with the estimate of $\partial_s^3 h(\theta,s)$ in $H^{-2}(\m S^1)$. 
We claim that it satisfies the same estimate as $\partial_s h(\theta,s)$ in $H^{-1}(\m S^1)$. Indeed, to estimate
$$\partial_t^2 B_1= \int O\left(\frac{\ep}{(1+s)^\frac{3}{2}(1+|q|)^{\frac{3}{2}-4\rho}}\right) \partial_t^2 S  \partial^\alpha_\theta Z^I \phi,$$
we write
$$g_{00}\partial_t^2 \partial^\alpha_\theta SZ^I \phi = g_{00}\partial_t^2 \partial^\alpha_\theta SZ^I\phi -\Box_g  S\partial^\alpha_\theta Z^I \phi + \Box_g S\partial^\alpha_\theta Z^I \phi
=O\left(\partial_r \partial S\partial^\alpha_\theta Z^I \phi\right) + O\left(\partial_U \partial \partial^\alpha_\theta SZ^I \phi\right).$$
The term $\Box_g S\partial^\alpha_\theta Z^I \phi$ contains quadratic term and we can neglect it. We can get rid of a $\partial_r$ with an integration by parts, and we are reduced to estimate
$$\int  O\left(\frac{\ep}{(1+s)^\frac{3}{2}(1+|q|)^{\frac{3}{2}-4\rho}}\right)\partial S\partial^\alpha_\theta Z^I \phi,$$
which can be estimated in the same way as $\wht B_1$.
Since $g$ also satisfy a wave equation, we can treat the other terms in a similar way.
This concludes the proof of Proposition \ref{prph2}.
\end{proof}

\begin{prp}\label{prph3}
	We have 
	$$\|\partial_t Z^{N-11} h\|_{L^2(\m S^1)} \lesssim \frac{\ep^2}{(1+s)^2}.$$
\end{prp}

\begin{proof}
	We use the decomposition $\partial_s h = \int B_1 + B_2 +...$
	We estimate, 
	$$B_1=O\left(\frac{\ep}{(1+s)^\frac{1}{2}(1+|q|)^{\frac{3}{2}-4\rho}}\right)\partial_s \partial^\alpha_\theta Z^I \phi
	=O\left(\frac{\ep}{(1+s)^\frac{3}{2}(1+|q|)^{\frac{3}{2}-4\rho}}\right)Z^{I+2} \phi,$$
	so thanks to \eqref{bootphi1}
	$$B_1=O\left(\frac{\ep^2}{(1+s)^2(1+|q|)^{2-8\rho}}\right).$$
	In the same way
$$B_2=O\left(\frac{\ep}{(1+s)^\frac{3}{2}(1+|q|)^{\frac{1}{2}-4\rho}}\right)\partial_q \partial^\alpha_\theta Z^I \phi
=O\left(\frac{\ep^2}{(1+s)^2(1+|q|)^{2-8\rho}}\right),$$
$$B_3=O\left(\frac{\ep}{(1+s)^\frac{1}{2}(1+|q|)^{\frac{1}{2}-4\rho}}\right)\frac{\partial_\theta}{r} \partial^\alpha_\theta Z^I \phi=O\left(\frac{\ep^2}{(1+s)^2(1+|q|)^{2-8\rho}}\right).$$
We estimate
$$B_4= O\left(\frac{\ep^2}{(1+s)(1+|q|)^{3-8\rho}}\right) \partial^\alpha_\theta Z^I g_{\q T \q T}
.$$
Thanks to estimates \eqref{iwintll3}, \eqref{iwintuu3} and \eqref{iwintlu3} we can estimate
$$\partial_\theta Z^I g_{\q T \q T} = O\left(\frac{\ep(1+|q|)^{\frac{1}{2}-\rho}}{(1+s)}\right)+O\left(s\partial_s Z^I b\right).
$$
Consequently
$$B_4= O\left(\frac{\ep^3}{(1+s)^2(1+|q|)^{\frac{5}{2}-9\rho}}\right) +  O\left(\frac{\ep^2}{(1+|q|)^{3-8\rho}}\partial_s Z^I b\right)= O\left(\frac{\ep^3}{(1+s)^2(1+|q|)^{\frac{5}{2}-9\rho}}\right).$$
where we have used the estimate \eqref{estb1} for $\partial_s Z^I b$.
We estimate
$$B_5= O\left(\frac{\ep^2}{(1+s)^2(1+|q|)^{2-8\rho}}\right) \partial^\alpha_\theta Z^I \wht g
=O\left(\frac{\ep^2}{(1+s)^{\frac{5}{2}-\rho}(1+|q|)^{2-8\rho}}\right),
$$
$$B_6=O\left(\frac{\ep^2}{(1+s)^\frac{3}{2}(1+|q|)^{\frac{7}{2}-8\rho-\sigma}}\right) Z^I \beta.$$
The equation \eqref{beta} gives, for $I\leq N-6$
$$	Z^I \beta(q,s,\theta) = O\left(\frac{(1+|q|)^{\frac{1}{2}+\sigma}}{(1+s)^{\frac{3}{2}}}\right)
	+\int_s^{2T-q} \frac{\ep(1+|q|)^{\frac{1}{2}+\sigma}}{(\tau+q)(1+\tau)^{\frac{1}{2}}} d\tau=O\left(\frac{\ep(1+|q|)^{\frac{1}{2}+\sigma}}{(1+s)^{\frac{1}{2}}}\right).$$
where we have used \eqref{iks2}.
Consequently we obtain
$$B_6=O\left(\frac{\ep^2}{(1+s)^2(1+|q|)^{3-8\rho-\sigma}}\right).$$
By integrating we obtain
$$Z^I h= O\left(\frac{\ep^2}{(1+s)^2}\right).$$
\end{proof}

\begin{prp}\label{comph}
	We have
	$$\left\|Z^{N-1}\left(h(\theta,2t)+2\int \partial_r \phi(t,r,\theta)\partial_t \phi(t,r,\theta) rdr\right)\right\|_{L^2} \lesssim \frac{\ep^3}{(1+t)^{\frac{1}{2}-2\rho}},$$
	$$\left\|Z^{N-5}\left(h(\theta,2t)+2\int \partial_r \phi(t,r,\theta)\partial_t \phi(t,r,\theta) rdr\right)\right\|_{L^2} \lesssim \frac{\ep^3}{\sqrt{1+t}}.$$
\end{prp}
\begin{proof}
	We have
	$$h(\theta,2t)-\int( \partial_q \phi)^2 rdr= \int O(\bar{\partial} \phi \partial \phi)r+O\left(\beta(\partial \phi)^2\right)r + O\left(\wht g_1 (\partial \phi)^2\right)r = \int C$$
	with 
	$$C\lesssim 
	\frac{\ep^2}{(1+|q|)^{\frac{3}{2}-4\rho-\sigma}}|\partial Z^I \phi|
	+ \frac{\ep^2}{(1+|q|)^{3-8\rho}}(|Z^I \wht g_1| +|Z^I \beta|).$$
	We have already estimated, for $I\leq N-1$ (see the proof of Proposition \ref{prph})
	$$\left\|\int \frac{\ep^2}{(1+|q|)^{3-8\rho}}(|Z^I \wht g_1| +|Z^I \beta|)\right\|_{L^2(\m S^1)}\lesssim \frac{\ep^3}{(1+t)^{\frac{1}{2}-2\rho}}.$$
	We have
	$$\left\|\int \frac{\ep^2}{(1+|q|)^{\frac{3}{2}-4\rho-\sigma}}|\partial Z^I \phi|\right\|_{L^2(\m S^1)}\lesssim \frac{\ep^2}{\sqrt{1+t}}\ld{\partial Z^I \phi}\left(\int \frac{1}{(1+|q|)^{2-8\rho-2\sigma}}\right)^\frac{1}{2}	\lesssim \frac{\ep^3}{(1+t)^{\frac{1}{2}-\rho}}.$$
For $I\leq N-5$ we easily see from the previous calculation
$$\left\|\int \frac{\ep^2}{(1+|q|)^{\frac{3}{2}-4\rho-\sigma}}|\partial Z^I \phi|\right\|_{L^2(\m S^1)}\lesssim \frac{\ep^3}{(1+t)^{\frac{1}{2}}},$$
$$\left\|\int \frac{\ep^2}{(1+|q|)^{3-8\rho}}|Z^I \wht g_1| \right\|_{L^2(\m S^1)}\lesssim\frac{\ep^3}{(1+t)^{\frac{1}{2}}},$$
where we use \eqref{hg1n1}.
Thanks to \eqref{estbeta} we have
	\begin{equation*}\begin{split}
	\left(\int \frac{1}{(1+|q|)^{2+2\mu}}(Z^{N-5}\beta)^2drd\theta \right)^\frac{1}{2}
	&\lesssim \frac{1}{t^\frac{1}{8}}
	\left(\int_t^T \frac{(2t'+R)^\frac{5}{4}}{(t'+\frac{t}{2})^3} 
	\left\|\frac{\ch_{q\leq R}}{(1+|q|)^{1+\mu}}Z^{N-5}\wht g_{L \ba L}\right\|^2_{L^2}dt' \right)^\frac{1}{2}\\
	&\lesssim \frac{1}{t^\frac{1}{8}}\left(\int_t^T \frac{1}{t'^{3-\frac{5}{4}}}dt'\right)^\frac{1}{2} \leq \frac{\ep}{t^{\frac{1}{2}}},
	\end{split}
	\end{equation*}
	and consequently
		$$\left\|\int \frac{\ep^2}{(1+|q|)^{3-8\rho}}|Z^{N-5} \beta|\right\|_{L^2(\m S^1)}\lesssim \frac{\ep^3}{(1+t)^{\frac{1}{2}-\rho}}.$$
\end{proof}

We now prove Proposition \ref{prph}.
\begin{proof}
We extend $h$ to $\infty$ by setting
$$h'(\theta,s)=\psi(s)h(\theta,s)+(1-\psi(s))h(\theta,2T),$$
where $\psi$ is a cut-off function such that $\psi=1$ for $s\leq 2T-1$ and $\psi=0$ for $s>2T$.
The fact that $h'$ satisfies the estimates of Proposition \ref{prph1} is straightforward.
For the estimates of Propositions \ref{prph2} and \ref{prph3} we just have to notice that  
$$\partial_s  h'(\theta,s)= \psi(s)\partial_s h(\theta,s) +\psi'(s)(h(\theta,s)-h(\theta,2T)).$$
Since in the region $s\sim 2T$, we have $h(\theta,s)-h(\theta,2T)=O(\partial_s h),$
we easily see that $\partial_s h'$ satisfy the same estimates as $\partial_s h$.
For the estimates \eqref{esth'} and \eqref{esth2'} we write
$$h'(\theta,s)=h(\theta,s)+(1-\psi(s))(h(\theta,2T)-h(\theta,s)),$$
and notice that
$$\|Z^{N-1}(h(\theta,2T)-h(\theta,s))\|_{L^2(\m S^1)}\lesssim \int_s^{2T}\|Z^{N}(\partial_s h)\|_{H^{-1}(\m S^1)}\lesssim \int_s^{2T}\frac{\ep^2}{(1+s)^{\frac{3}{2}-\rho}}\lesssim \frac{\ep^2}{T(1+s)^{\frac{1}{2}-\rho}}.$$
Without loss of generality, we can assume $T\geq \frac{D}{\ep},$ and consequently
$$\|Z^{N-1}(h(\theta,2T)-h(\theta,s))\|_{L^2(\m S^1)}\lesssim \frac{\ep^3}{(1+s)^{\frac{1}{2}-\rho}}.$$
In a similar way
$$\|Z^{N-5}(h(\theta,2T)-h(\theta,s))\|_{L^2(\m S^1)}\lesssim \frac{\ep^3}{(1+s)^{\frac{1}{2}}},$$
which concludes the proof of Proposition \ref{prph}.
\end{proof}

\subsection{Proof of Proposition \ref{prpb}}\label{secprpb}

We want to find three coefficients $b_0(s),b_1(s),b_2(s)$ and a solution $b(\theta,s)$ of
\begin{align*}&\frac{2a(\theta+f)}{(1+b)^2}+\frac{1}{(1+b)^2}-1-2\frac{\partial^2_\theta b}{(1+b)}-\frac{(\partial_\theta b)^2}{(1+b)^2}\\
&=\Pi h'(\theta,s)+b_0+b_1\cos(\theta)+b_2\sin(\theta)
\end{align*}
satisfying $\int \frac{b}{1+b}d\theta =0$
and $1+\partial_\theta f=(1+b)^{-1}$.
We make a change of unknown $\beta= \frac{b}{1+b}$. We calculate
$$\partial^2_\theta \beta = \frac{\partial^2_\theta b}{(1+b)^2}-2\frac{(\partial_\theta b)^2}{(1+b)^3}.$$
The problem is therefore equivalent to finding $b_0(s),b_1(s),b_2(s)$ and $\beta$ a solution of
\begin{equation}\label{eqbeta}-2\partial^2 \beta -2\beta = (1-\beta)(\Pi h'(\theta,s)+b_0+b_1\cos(\theta)+b_2\sin(\theta))-2(1-\beta)^3a_0(1+f)+R(\partial_\theta \beta,\beta)
\end{equation}
with $\int \beta = 0$ and $f$ defined by $\partial_\theta f =-\beta$, and we have denoted by $R$ a quadratic form.
 We will do this with a fixed point argument. We consider $F: H^2(\m S^1)\mapsto  H^2(\m S^1)$ which maps $\beta$ such that $\|\beta\|_{H^2}\leq \ep$ and $\int \beta =0$ to $\beta'$ solution of 
\begin{equation}
\label{linb}-2\partial^2 \beta' -2\beta' = (1-\beta)(\Pi h'(\theta,s)+b_0+b_1\cos(\theta)+b_2\sin(\theta))-2(1-\beta)^3a_0(1+f)+R(\partial_\theta \beta,\beta)
\end{equation}
with $b_1,b_2,b_0$ chosen such that
\begin{align}
\label{b1}&\int \cos(\theta)\left((1-\beta)(\Pi h'(\theta,s)+b_0+b_1\cos(\theta)+b_2\sin(\theta))-2(1-\beta)^3a_0(1+f))+R(\partial_\theta \beta,\beta)\right)d\theta =0,\\
\label{b2}&\int \sin(\theta)\left((1-\beta)(\Pi h'(\theta,s)+b_0+b_1\cos(\theta)+b_2\sin(\theta))-2(1-\beta)^3a_0(1+f))+R(\partial_\theta \beta,\beta)\right)d\theta =0,\\
\label{b3}&\int \left( (1-\beta)(\Pi h'(\theta,s)+b_0+b_1\cos(\theta)+b_2\sin(\theta))-2(1-\beta)^3a_0(1+f))+R(\partial_\theta \beta,\beta) \right)d\theta=0.
\end{align}
We first show that we can find three such coefficients. We have, thanks to Sobolev embedding
$$|b|+|\partial_\theta b|\leq \ep.$$
Consequently the three integral conditions \eqref{b1}, \eqref{b2} and \eqref{b3} can be written
$$ \frac{1}{2}b_1 = O(\ep)b_0 +  O(\ep)b_1+ O(\ep)b_2 + \int \cos(\theta)\left((1-\beta)\Pi h'(\theta,s)-2(1-\beta)^3a(1+f)+R(\partial_\theta \beta,\beta)\right)d\theta.$$
$$\frac{1}{2}b_2 = O(\ep)b_0 +  O(\ep)b_1+ O(\ep)b_2 + \int \sin(\theta)\left((1+\beta)\Pi h'(\theta,s)-2(1-\beta)^3a(1+f)+R(\partial_\theta \beta,\beta)\right)d\theta.$$
$$b_0 = \int\left((1-\beta)\Pi h'(\theta,s)-2(1-\beta)^3a(1+f)+R(\partial_\theta \beta,\beta^2)\right)d\theta. $$
This system is invertible : we have a unique solution which satisfies the estimate
$$|b_0|+|b_1|+|b_2|\lesssim \| h'\|_{L^2(\m S^1)} + \ep \|\beta\|_{H^1(\m S^1)} \lesssim \ep^2.$$
Thanks to \eqref{b1} and \eqref{b2}, we are allowed to solve \eqref{linb}. There exists a unique solution $\beta'\in H^2(\m S^1)$, and it satisfies
$$\|\beta'\|_{H^2} \lesssim \|h'\|_{L^2(\m S^1)} + \ep \|\beta\|_{H^1(\m S^1)} +|a_0| \lesssim \ep^2$$
Moreover, thanks to \eqref{b3} we have $\int \beta'=0$.
We see easily that the map $F$ is contracting. Consequently it admits a unique fixed point $\beta(\theta,s)$, satisfying
$$\|\beta\|_{H^{2}} \lesssim \|h'\|_{L^{(2)}}+|a|.$$
Moreover there exists $b_0,b_1,b_2$ such that $\beta$ satisfy \eqref{eqbeta}.
In addition we have
$$\|\beta\|_{H^{k+2}} \lesssim \|h'\|_{H^{k}}+|a|,$$
and deriving \eqref{eqbeta} , \eqref{b1}, \eqref{b2} and \eqref{b3} with respect to $s$ we obtain
$$\|\partial_s^l \beta \|_{H^{k+2}} \lesssim \|\partial_s^l h'\|_{H^{k}},$$
$$|\partial^l_s b_0|+|\partial^l_s b_1|+ |\partial^l_s b_2|\lesssim   \|\partial_s^l h'\|_{L^2}.$$

\subsection{Proof of Proposition \ref{prpfin}}\label{secprpfin}

From the initial data of Theorem \ref{thinitial}, we can construct a solution of \eqref{sys} up to time $T=1$. Moreover, this solution exists in the entire region $q>R+1$, $t>0$ (see Appendix \ref{ext}).
Let $b^{(2)}$ be defined by Proposition \ref{prpb} and set
$$h^{(2)}=h'(\theta,s)+b_0+b_1\cos(\theta)+b_2\sin(\theta).$$
By performing the change of variable of Section \ref{change} with $ b^{(2)}$, and looking at the data on $t=0$, we obtain a solution of the constraint equation with the desired asymptotic behaviour (see Appendix \ref{reguini})
$$g^{(2)}=g_{b^{(2)}}+\wht g^{(2)}.$$
We consider the solution $(g^{(2)},\phi^{(2)})$ in generalized wave coordinates
$$(H^{(2)})^\alpha = (g^{(2)})^{\lambda \beta}(\Gamma^{(2)})^\alpha_{\lambda \beta}=(F^{(2)})^\alpha+(G^{(2)})^\alpha +(\wht G^{(2)})^\alpha,$$
with
$$(F^{(2)})^\alpha = \Box_{g_{b^{(2)}}}x^\alpha.$$
\begin{align*}
U_\alpha (G^{(2)})^\alpha&=-(\sigma^0_{U\ba L})^{(2)}\chi'(q),\\
L_\alpha (g^{(2)})^\alpha& =  \frac{1}{r}\Up\left(\frac{r}{r}\right)\int_{-\infty}^r \left(2(\partial_q \phi^{(2)})^2r -h^{(2)}(\theta,2t)\chi'(q)\right)dr,\\
\ba L_\alpha (G^{(2)})^\alpha &= 0,
\end{align*}
where $(\sigma^0_{U\ba L})^{(2)}=s(1+b^{(2)})\partial_s f^{(2)}$ with
$$1+\partial_\theta f^{(2)}=(1+b^{(2)})^{-1},$$
and $(\wht G^{(2)})^\alpha$ contains the terms of the form $\wht g^{(2)} \partial^l_s\partial^k_\theta b^{(2)}$, where
$l+k-2\geq 1$ or $l\geq 2$.
$b^{(2)}$ satisfy the hypothesis \eqref{estb1} to \eqref{estf2bis}.
We assume that on $[0,T^{(2)}]$, $g^{(2)}$ satisfy the bootstrap estimates.
Thanks to all we have done so far, 
we know that $(g^{(2)},\phi^{(2)})$
satisfies the improved bootstrap estimates, except \eqref{esth} which remained to be proved.
For this, thanks to Proposition \ref{comph}, all we have to do is to compare $\phi$ and $\phi^{(2)}$.
We can pass from $(g,\phi)$ to $(g^{(2)},\phi^{(2)})$ by a change of variable, that we note $\Psi$.

We note $x^{(2)}=\Psi(x)$ the new generalized wave coordinates.
We have 
$$|\nabla s^{(2)}|_g = \left|( g^{(2)})^{ LL}\right|,$$
and so
\begin{align*}
|s^{(2)}-s|&\lesssim \int_{s}^q \left|g_{b^{(2)}}^{LL}-g_{\mathfrak{b}}^{LL}\right|+\left|\wht g^{(2)}\right|dq'\\
&\lesssim  \int_s^q O\left(\frac{q'}{s}\partial_\theta (b-b^{(2)})\right)dq'
+\frac{\ep}{(1+|q|)^{\frac{1}{2}+\delta}(1+s)^{\frac{1}{2}-\rho}}\\
&\lesssim \ep\sqrt{1+s},
\end{align*}
where we have used that $\delta-\rho \geq \frac{1}{2}$.
and in the interior
\begin{align*}
|s^{(2)}-s|&\lesssim \int_{s}^q \left|\wht g^{(2)}\right|dq'+|s^{(2)}-s|_{q=0}\\
&\lesssim  \ep\sqrt{1+s}+\int_s^q \frac{\ep}{(1+s)^{\frac{1}{2}-\rho}}dq'\\
&\lesssim \ep\sqrt{1+s} +\ep \frac{1+|q|}{(1+s)^{\frac{1}{2}-\rho}}.
\end{align*}
We have in the exterior region
$$|\nabla (q^{(2)}-q(1+b^{(2)})^{-1})|_g = \left|(\wht g^{(2)})^{ \ba L \ba L}\right|,$$
so in the exterior region we obtain
\begin{align*}
|q^{(2)}-q|&\lesssim |q||b^{(2)}-b|+\int_{s}^q \left|(\wht g^{(2)})^{\ba L \ba L}\right|dq\\
&\lesssim  |q||b^{(2)}-b|+\int_s^q 
\frac{\ep}{(1+|q'|)^{-\frac{1}{2}+\delta-\sigma}(1+s)^{\frac{3}{2}-\rho}}dq'\\
&\lesssim \frac{\ep|q|}{\sqrt{1+s}}+\frac{\ep}{(1+s)^{\delta-\sigma-\rho}}
\end{align*}
where we used
$$|b^{(2)}-b|\lesssim \|h^{(2)}-h\|_{L^2(\m S^1)}\lesssim \frac{\ep^2}{\sqrt{1+s}}.$$
We recall that $\delta-\sigma-\rho \geq \frac{1}{2}$. In the interior we obtain
\begin{align*}
|q^{(2)}-q|&\lesssim \int_{s}^q \left|(\wht g^{(2)})^{\ba L \ba L}\right|dq'+|q^{(2)}-q|_{q=0}\\
&\lesssim  \frac{\ep}{\sqrt{1+s}}+\int_s^q \frac{\ep(1+|q'|)^{\frac{3}{2}}}{(1+s)^{\frac{3}{2}}}dq'\\
&\lesssim \frac{\ep (1+|q|)^{\frac{3}{2}}}{(1+s)^\frac{1}{2}}.
\end{align*}
We have
$$|\nabla \theta^{(2)}|_g = \frac{1}{r}\left|( g^{(2)})^{ UU}\right|,$$
so in the exterior
\begin{align*}
|\theta^{(2)}-\theta|&\lesssim \int_{s}^q \left(\frac{1}{r}\left|g_{b^{(2)}}^{UU}-g_{\mathfrak{b}}^{UU}\right|+\frac{1}{r}\left|\wht g^{(2)}\right|\right)dq'\\
&\lesssim  \int_s^q O\left(\frac{q'}{s^2}\partial_\theta^2 (b-b^{(2)})\right)dq'
+\frac{\ep}{(1+|q|)^{\frac{1}{2}+\delta}(1+s)^{\frac{3}{2}-\rho}}\\
&\lesssim \frac{\ep}{\sqrt{1+s}},
\end{align*}
and in the interior
\begin{align*}
|\theta^{(2)}-\theta|&\lesssim \int_{s}^q \frac{1}{r}\left|\wht g^{(2)}\right|dq'+|\theta^{(2)}-\theta|_{q=0}\\
&\lesssim  \frac{\ep}{\sqrt{1+s}}+\int_s^q \frac{\ep}{(1+s)^{\frac{3}{2}-\rho}}dq'\\
&\lesssim \frac{\ep}{\sqrt{1+s}} +\ep \frac{1+|q|}{(1+s)^{\frac{3}{2}-\rho}}.
\end{align*}
With these estimates it is easy to see that $\phi^{(2)}$ satisfy the same improved estimates as $\phi$. 
We have
$$\phi^{(2)}(x)=\phi(\Psi(x)).$$
We calculate 
$$\partial \phi^{(2)}=(\partial s^{(2)})\partial_s \phi(\Psi(x))
 +(\partial q^{(2)})\partial_q \phi(\Psi(x))+ (\partial \theta^{(2)})\partial_\theta \phi(\Psi(x)).$$
We compare $\phi$ and $\phi^{(2)}$ to improve \eqref{esth}.
$\phi$ is non zero only in the interior region. We obtain
\begin{align*}|\partial \phi^{(2)}-(\partial \phi)(\Psi(x))|\lesssim&
|\wht g||\bar{\partial} \phi|+|\wht g_{LL}||\partial_q \phi|\\
\lesssim& \frac{\ep}{(1+s)^{\frac{1}{2}-\rho}}\frac{\ep}{(1+s)^\frac{3}{2}(1+|q|)^{\frac{1}{2}-4\rho}}+ \frac{1+|q|}{(1+s)^{\frac{3}{2}-\rho}}\frac{\ep}{(1+s)^\frac{1}{2}(1+|q|)^{\frac{3}{2}-4\rho}}\\
 \lesssim& \frac{\ep^2}{(1+|q|)^{\frac{1}{2}-5\rho}(1+s)^{2-\rho}}.
\end{align*}
We estimate
\begin{align*}
|(\partial \phi)(\Psi(x))-\partial \phi(x)|&\lesssim\left|\int_0^1 (x^{(2)}-x).\nabla \partial \phi(x+\tau(x^{(2)}-x) )d\tau\right|\\
&\lesssim \int_0^1 \left(\ep(1+s)^{\frac{1}{2}}|\bar{\partial}\partial \phi(x+\tau(x^{(2)}-x))|+\frac{\ep (1+|q|)^{\frac{3}{2}+\sigma}}{(1+s)^\frac{1}{2}}
|\partial^2\phi(x+\tau(x^{(2)}-x) )|\right)d\tau\\
&\lesssim \frac{\ep^2}{(1+s)(1+|q|)^{1-4\rho-\sigma}}
\end{align*}
In the same way we have
$$|\partial Z^I \phi^{(2)}-(\partial Z^I \phi)(\Psi(x))|\lesssim
\frac{\ep}{(1+s)^{\frac{3}{2}-\rho}}| Z^{I+1} \phi|
+\frac{\ep}{\sqrt{1+s}(1+|q|)^{\frac{3}{2}-4\rho}}|Z^I g_{LL}|+
 \frac{\ep}{(1+s)^\frac{3}{2}(1+|q|)^{\frac{1}{2}-4\rho}}|Z^I \wht g|
$$
and
\begin{align*}
&|(\partial Z^I \phi)(\Psi(x))-\partial Z^I\phi(x)|\\
&\lesssim\left|\int_0^1 (x^{(2)}-x).\nabla \partial Z^I \phi(x+\tau(x^{(2)}-x) )d\tau\right|\\
&\lesssim \int_0^1 \left(\ep(1+s)^{\frac{1}{2}}|\bar{\partial}\partial Z^I \phi(x+\tau(x^{(2)}-x))|+\frac{\ep (1+|q|)^{\frac{3}{2}+\sigma}}{(1+s)^\frac{1}{2}}
|\partial^2 Z^I \phi(x+\tau(x^{(2)}-x) )|\right)d\tau\\
&\lesssim \int_0^1\frac{\ep (1+|q|)^{\frac{1}{2}+\sigma}}{(1+s)^\frac{1}{2}}
\left|\partial Z^{I+1} \phi(x+\tau(x^{(2)}-x) )\right|d\tau.
\end{align*}
Consequently, we have
\begin{align*}
&\left\|Z^I \left(\int (\partial_q \phi^{(2)})^2 rdr-\int (\partial_q \phi)^2 rdr \right) \right\|_{L^2(\m S^1)}\\
&\lesssim \int \frac{\ep}{(1+|q|)^{\frac{3}{2}-4\rho}(1+s)^\frac{1}{2}}\frac{\ep (1+|q|)^{\frac{1}{2}+\sigma}}{(1+s)^\frac{1}{2}}
\|\partial Z^{I+1} \phi\|_{L^2(\m S^1)} rdr\\
&+\int\frac{\ep^2}{(1+s)(1+|q|)^{3-8\rho}}\|Z^I g_{LL}\|_{L^2(\m S^1)} rdr+s.t.\\
&\lesssim \frac{\ep}{(1+t)^\frac{1}{2}}\left(\ld{\partial Z^{I+1} \phi}+\ld{w_2\frac{Z^Ig_{LL}}{1+|q|}}\right)+ s.t.
\end{align*}
Consequently we have,
$$\left\|Z^{N-5} \left(\int (\partial_q \phi^{(2)})^2 rdr-\int (\partial_q \phi)^2 rdr \right) \right\|_{L^2(\m S^1)}\lesssim \frac{\ep^3}{(1+t)^{\frac{1}{2}}},$$
	and 
$$\left\|Z^{N-1} \left(\int (\partial_q \phi^{(2)})^2 rdr-\int (\partial_q \phi)^2 rdr \right) \right\|_{L^2(\m S^1)}\lesssim \frac{\ep^3}{(1+t)^{\frac{1}{2}-\rho}}.$$
\appendix
\section{Global existence of solutions in the exterior}\label{ext}

We denote by $\bar{C}$ the complementary of the domain of dependence of $B(0,R+1)$. 
Let $g_{\mathfrak{a}}$ be defined by Theorem \ref{thinitial}.
In generalized wave coordinates $\Box_g x^\alpha = \Box_{g_{\mathfrak{a}}} x^{\alpha}$ the system $R_{\mu \nu}=0$ can be written, with the
decomposition $g=g_{\mathfrak{a}} + \wht g$
$$
\Box_g\wht g_{\mu \nu} = P_{\mu \nu}(g)(\partial \wht g, \partial \wht g) + \wht P_{\mu \nu} (\wht g, g_{\mathfrak{a}}),$$
where we used the fact that $g_{\mathfrak{a}}$ is Ricci flat.
We perform a bootsrap argument : let $T$ be such that we have a solution $g$ of this equation on $\bar{C}_T$ where $\bar{C}_T$ is the restriction of $\bar{C}$ to times less than $T$, and assume that
\begin{align}
\label{extboot1}&\|v^\frac{1}{2}\partial Z^N \wht g\|_{L^2(\bar{C}\cap \Sigma_t)} \lesssim \ep (1+t)^\rho,\\
\label{extboot2}&\|v_1^\frac{1}{2} \partial Z^{N-2} \wht g\|_{L^2(\bar{C}\cap \Sigma_t)} \lesssim \ep.
\end{align}
where
$$v(q)=(1+|q|)^{2+2\delta'},$$
$$v_1(q)=(1+|q|)^{2+2\delta'-2\sigma}.$$
Thanks to Klainerman-Sobolev estimates (which are still valid in a region $\bar{C}\cap \Sigma_t$) we have
\begin{align*}
&|\partial Z^{N-2}\wht g|\lesssim \frac{\ep}{(1+s)^{\frac{1}{2}-\rho}(1+|q|)^{\frac{3}{2}+\delta'}},\\
&|\partial Z^{N-3}\wht g|\lesssim \frac{\ep}{(1+s)^{\frac{1}{2}}(1+|q|)^{\frac{3}{2}+\delta'-\sigma}},
\end{align*}
\begin{align}
\label{extiks1}&|Z^{N-2}\wht g|\lesssim \frac{\ep}{(1+s)^{\frac{1}{2}-\rho}(1+|q|)^{\frac{1}{2}+\delta'}},\\
\label{extiks2}&|\partial Z^{N-3}\wht g|\lesssim \frac{\ep}{(1+s)^{\frac{1}{2}}(1+|q|)^{\frac{1}{2}+\delta'-\sigma}}.
\end{align}
Thanks to the wave coordinate condition
\begin{align}
\label{exttt1}&|Z^{N-3}\wht g_{\q T \q T}|\lesssim \frac{\ep}{(1+s)^{\frac{3}{2}-\rho}(1+|q|)^{-\frac{1}{2}+\delta'}},\\
\label{exttt2}&| Z^{N-4}\wht g_{\q T \q T}|\lesssim \frac{\ep}{(1+s)^{\frac{3}{2}}(1+|q|)^{-\frac{1}{2}+\delta'-\sigma}}.
\end{align}
To improve \eqref{extboot1} we perform the energy estimate in the background metric, in the region $\bar{C}_t$. We obtain
\begin{align*}
&\int_{\Sigma_t} wQ_{TT}+ \int_{\partial \bar{C}} wQ_{T\q L} +C\int_0^t \int w'(q) (\bar{\partial} \wht g)^2\\
&\lesssim \int_0^t \frac{1}{1+\tau}\int Q_{TT} + \int_0^t \int |\partial_t Z^I \wht g||\Box_g Z^I \wht g|
\end{align*}
where $\q L$ is a null vector tangent to $\partial \bar{C}$ and $Q$ is the energy momentum tensor for $\Box_g$
$$Q_{\alpha \beta}=\partial_\alpha Z^I \wht g \partial_\beta Z^I \wht g-\frac{1}{2}g_{\alpha \beta}g^{\mu \nu}\partial_\mu Z^I \wht g \partial_\nu Z^I \wht g.$$
Consequently
$$Q_{T\q L}=T(Z^I \wht g)\q L(Z^I \wht g)-\frac{1}{2}g_{T\q L}(\q L(Z^I \wht g) \ba {\q L} (Z^I \wht g)+e_\theta(Z^I \phi)^2)$$
where $\ba {\q L}$ is null such that $g(\q \ba L,\q L)=-2$ and $e_\theta$ is tangent to $\bar{C}$ and orthogonal to $\q L$ and $\ba {\q L}$.
We have
\begin{align*}Q_{T\q L}=&\frac{1}{2}g_{T\q L}\ba {\q L} (Z^I \wht g)\q L(Z^I \wht g)+\frac{1}{2}g_{T \ba {\q L}}\q L(Z^I \wht g)\q L(Z^I \wht g)
+g_{Te_\theta}\q L(Z^I \wht g)e_\theta(Z^I \wht g)\\
&-\frac{1}{2}g_{T\q L}(\q L(Z^I \wht g)\ba {\q L} (Z^I \wht g)+e_\theta(Z^I \wht g)^2)\\
=&\frac{1}{2}g_{T \ba {\q L}}\q L(Z^I \wht g)\q L(Z^I \wht g)-\frac{1}{2}g_{T\q L}e_\theta(Z^I \phi)^2+g_{Te_\theta}\q L(Z^I \wht g)e_\theta(Z^I \wht g)\\
\geq &(1-C\ep)(\q L(Z^I \wht g)^2+e_\theta (Z^I \wht g)^2)\geq 0.
\end{align*}
Since in all our proof, the bootstrap condition \eqref{esth} was not needed in the exterior region, we easily see from section \ref{secl2gn} that we will be able to improve \eqref{extboot1}.

To improve \eqref{extboot2} we perform the energy estimate in the flat metric. $\wht Q$ is now the flat energy-momentum tensor. We now have to be careful with
\begin{align*}\wht Q_{T\q L}= &\partial_t Z^I \wht g \q L(Z^I \wht g)-\frac{1}{2}m_{T\q L}\left(\partial_s(Z^I \wht g)\partial_q (Z^I \wht g)+\frac{1}{r^2}(\partial_\theta Z^I \wht g)^2\right),\\
\end{align*}
which may not be positive.
Since $\q L=  (1+O(g-m))\partial_s + O(g_{L L})\partial_q + O(g_{UL})\partial_U,$
we have
\begin{align*}\wht Q_{T\q L}&=(1+O(\ep))\left((\partial_s Z^I \wht g)^2+\frac{1}{r^2}(\partial_\theta Z^I \wht g)^2\right)
+O(g_{ L L})(\partial_q Z^I \wht g)^2\\
&\geq (1-\ep)\left((\partial_s Z^I \wht g)^2+\frac{1}{r^2}(\partial_\theta Z^I \wht g)^2\right)
-\frac{\ep^3}{(1+s)^{\frac{5}{2}-3\rho}}
\end{align*}
where we have used \eqref{extiks1}. Consequently the energy estimate yields
\begin{align*}
&\int_{\Sigma_t} w(\partial Z^I  \wht g)^2+ \int_{\partial \bar{C}} w(\bar \partial Z^I \wht g)^2 +C\int_0^t \int w'(q) (\bar{\partial} \wht g)^2\\
&\lesssim \int_0^t \int |\partial_t Z^I \wht g||\Box_g Z^I \wht g|+ \int_{0}^t \frac{\ep^3}{(1+\tau)^{\frac{3}{2}-3\rho}} d\tau.
\end{align*}
We then easily see from Section \ref{secl2gn4} that we can improve the bootstrap assumption (we can check that the cubic non linearities without null structure in $Q_{\ba L \ba L}$ are not present).

\section{Regularity of the initial data}\label{reguini}
To obtain solutions of the constraint equation with an asymptotic behaviour $g=g_{\mathfrak{b}}+\wht g$, we take the exterior solution constructed in the previous section (we denote by $s',q',\theta'$ the coordinates used for this construction), make the change of variable
$$s'=(1-\chi(r))s+\chi(r)\left((1+b(\theta,s))s-(\partial_\theta b(\theta,s))^2(1+b(\theta,s))^{-1}q\right),$$
$$q'=(1-\chi(r))q+\chi(r)(1+b(\theta,s))^{-1}q,$$
$$\theta'= (1-\chi(r))\theta + \chi(r)\left(\theta - \frac{q}{r}\frac{\partial_\theta b(\theta,s)}{(1+b(\theta,s))^2}+ f(\theta,s)\right),$$
and consider the space-like hypersurface, given by $t=0$. 
We denote by $\Sigma_b$ this hypersurface, and consider $\bar{g}=g|_{\Sigma_b}$, and $K$ the second fundamental form of the embedding $\Sigma_b \subset M$. $(\bar{g},K)$ is a solution to the constraint equations.

\begin{prp}\label{prpinitial}
There exists
	$$ (g_{\alpha \beta})_0,(g_{\alpha \beta})_1 \in H^{N+1}_{\delta}\times H^{N}_{\delta+1}$$ such that the initial data for $g$ given by
	$$g=g_{\mathfrak{b}}+g_0, \;\partial_t g = \partial_t g_{\mathfrak{b}} +g_1,$$
	are such that
	\begin{itemize}
		\item $\bar{g}_{ij}=g_{ij}, K_{ij}=\q L_{\beta}g_{ij}$ satisfy the constraint equations \eqref{contrmom} and \eqref{contrham}.
		\item the following generalized wave coordinates condition is satisfied at $t=0$
		$$g^{\lambda\beta}\Gamma^\alpha_{\lambda \beta}=g_{\mathfrak{b}}^{\lambda\beta}(\Gamma_b)^\alpha_{\lambda \beta}+G^\alpha + \wht G^\alpha,$$
	\end{itemize}
	where $G^\alpha$ is defined by \eqref{gu}, \eqref{gl} and \eqref{gbal} and $\wht G$ is the sum of all the crossed term of the form
	$g_0\frac{\partial_\theta}{r} g_{\mathfrak{b}}$ and $g_0\partial_s g_{\mathfrak{b}}$ in
	$g^{\lambda\beta}\Gamma^\alpha_{\lambda \beta}-g_{\mathfrak{b}}^{\lambda\beta}(\Gamma_b)^\alpha_{\lambda \beta}$.
	Moreover we have the estimate
	$$\|g_0\|_{H^{N+1}_{\delta}} +\|g_1\|_{H^N_{\delta+1}}\lesssim \ep.$$
\end{prp}

\begin{proof}
There are two issues to consider for the regularity of $(\bar{g},K)$.
\begin{itemize}
	\item We have $t'\sim t-b(\theta,s)r$, so $|t'|\rightarrow \infty$ as $r\rightarrow \infty$ in $\Sigma_b$. Consequently we have to be careful with the logarithmic growth in $t'$ of the higher energy of $\wht g$.
	\item In $\partial^N_\theta \wht g$, we have terms of the form $\partial_\theta^{N+2}b(\theta,s) \partial_\theta \wht g$ : we have also to be careful with the logarithmic growth in $s$ of $\|\partial_\theta^{N+2}b(\theta,s) \|_{L^2(\m S^1)}$.
\end{itemize}
We treat the first issue.
We can estimate $\int_{\Sigma_b} w(q)(\partial Z^N \wht g)^2 rdrd\theta$ by performing the energy estimate on the domain
delimited by $\Sigma_0$ and $\Sigma_b$. We denote by $\Omega_b$ this domain.
We have
$$\int_{\Sigma_b} w(q)(\partial Z^N \wht g)^2\lesssim \int_{\Sigma_0}w(q)(\partial Z^N \wht g)^2
+ \int_{\Omega_b} \frac{\ep}{1+s}w(q)(\partial Z^N \wht g)^2.$$
We note that in the region $\Omega_b\cap \{q>R\}$ we have $|q|>Ct$. Since $w(q)\leq v(q)(1+|q|)^{\delta-\delta'}$ we have
\begin{align*}
\int_{\Sigma_b} w(q)(\partial Z^N \wht g)^2&\lesssim \int_{\Sigma_0}w(q)(\partial Z^N \wht g)^2
+ \int_{\Omega_b} \frac{\ep}{(1+t)^{1+\delta-\delta'}}v(q)(\partial Z^N \wht g)^2\\
&\lesssim \int_{\Sigma_0}w(q)(\partial Z^N \wht g)^2 +\ep^3.
\end{align*}
We treat the second issue with the help of the weight $w$: 
\begin{align*}
\int w(r)(\partial^{N+2}b(\theta,s) \partial\partial_\theta \wht g)^2
&\lesssim \int \frac{\ep}{(1+r)^{1+\delta-\delta'}}\|\partial_\theta^{N+2}b(\theta,r) \|^2_{L^2(\m S^1)}\\
&\lesssim  \int \frac{\ep^3}{(1+r)^{1+\delta-\delta'-2\rho}}\lesssim \ep^3
\end{align*}
$$
\int w(r)(\partial^{N+2}\partial_s b(\theta,s)\partial_\theta \wht g)^2
\lesssim \int \frac{\ep}{(1+r)^{\delta-\delta'}}\|\partial_s \partial_\theta^{N+2}b(\theta,r) \|^2_{L^2(\m S^1)}dr\\
\lesssim  \ep^3.$$

We now discuss the regularity of $\partial_t g_{0i}$.
 The generalized wave coordinate condition can be written
 $$g^{\lambda \beta}\Gamma^\alpha_{\lambda \beta}=(g_{\mathfrak{b}})^{\lambda \beta}(\Gamma_b)^\alpha_{\lambda \beta}+G^\alpha + \wht G^\alpha,$$
 Therefore, if we write it for $\alpha=i$ we obtain a relation for $\partial_tg_{0i}$ and if we write it for $\alpha=0$, we obtain a relation for
 $\partial_t g_{00}$. However, if we write $g=g_{\mathfrak{b}} +\wht g$, the term
 $$g^{\lambda \beta}\Gamma^\alpha_{\lambda \beta}-(g_{\mathfrak{b}})^{\lambda \beta}(\Gamma_b)^\alpha_{\lambda \beta}
 $$
 contains crossed terms of the form
 $$\wht g \partial_U g_{\mathfrak{b}} \sim \wht g \frac{\partial^3_\theta b(\theta)}{r}+s.t., \quad \wht g \partial_s g_{\mathfrak{b}} \sim \wht g\partial_s^2 \partial_\theta b+s.t.$$
 which do not belong in $H^N_{\delta+1}$ because we are missing a derivative on $b$. However these terms are removed thanks to the addition of the term $\wht G$ in the generalized wave coordinate condition. Consequently $\partial_t \wht g_{00}$ and $\partial_t \wht g_{0i}$ are given by a sum of terms the form
 $$K, \;\nabla g_0,\;g_{\mathfrak{b}}K,\; g_{\mathfrak{b}} \nabla g_0, \; \frac{\chi(r)g_{\mathfrak{b}}}{r}g_0.$$ 
 With this choice, $\partial_t \wht g_{0i}$ and $\partial_t \wht g_{00}$
 belong to $H^N_{\delta+1}$.
\end{proof}

\paragraph{Acknowledgement} This paper has benefit from the insight of many people. The author would like to thank in particular Jérémie Szeftel, Qian Wang, Spyros Alexakis, Mihalis Dafermos and Igor Rodnianski for the interesting conversations.

\bibliographystyle{plain}
\bibliography{stab}
\end{document}